\documentclass[a4paper,reqno,11pt]{amsart}
\usepackage{amsmath, amsfonts, amssymb, amsthm, amscd}
\usepackage[a4paper,scale={0.72,0.74},marginratio={1:1},footskip=7mm,headsep=10mm]{geometry}
\usepackage{tabularx}
\usepackage{subfigure}
\usepackage{graphicx}
\usepackage{psfrag}
\usepackage{perpage}
\usepackage{url}
\usepackage{color}
\usepackage{mathrsfs}
\usepackage{dsfont} 
\usepackage{mathrsfs}
\usepackage{hyperref}
\usepackage{color}
\usepackage{multicol}
\usepackage{mathtools}
\usepackage{mathdots}
\usepackage{tikz,tikz-cd} \usetikzlibrary{decorations.pathreplacing}
\usetikzlibrary{arrows,decorations.pathmorphing,backgrounds,positioning,fit,petri} 
\usepackage[utf8]{inputenc}
\usepackage[T1]{fontenc}
\usepackage{microtype}

\usepackage{hyperref}
\usepackage[vcentermath]{youngtab}

\makeatletter
\def\@secnumfont{\bfseries\scshape}

\def\section{\@startsection{section}{1}%
  \z@{.7\linespacing\@plus\linespacing}{.5\linespacing}%
  {\normalfont\large\bfseries\scshape\centering}}

\def\subsection{\@startsection{subsection}{2}%
  \z@{.5\linespacing\@plus.7\linespacing}{-.5em}%
  {\normalfont\bfseries\scshape}}
  
\def\subsubsection{\@startsection{subsubsection}{3}%
  \z@{.5\linespacing\@plus.7\linespacing}{-.5em}%
  {\normalfont\scshape}}

\def\specialsection{\@startsection{section}{1}%
  \z@{\linespacing\@plus\linespacing}{.5\linespacing}%
  {\normalfont\centering\large\bfseries\scshape}}
\makeatother

\makeatletter

\renewenvironment{proof}[1][\proofname]{\par
\pushQED{\qed}%
\normalfont \topsep4\p@\@plus4\p@\relax
\trivlist
\item[\hskip\labelsep
\bfseries
#1\@addpunct{.}]\ignorespaces
}{%
\popQED\endtrivlist\@endpefalse
}
\makeatother

\makeatletter
\newcommand \Dotfill {\leavevmode \leaders \hb@xt@ 6pt{\hss .\hss }\hfill \kern \z@}
\makeatother

\makeatletter
\def\@tocline#1#2#3#4#5#6#7{\relax
  \ifnum #1>\c@tocdepth 
  \else
    \par \addpenalty\@secpenalty\addvspace{#2}%
    \begingroup \hyphenpenalty\@M
    \@ifempty{#4}{%
      \@tempdima\csname r@tocindent\number#1\endcsname\relax
    }{%
      \@tempdima#4\relax
    }%
    \parindent\z@ \leftskip#3\relax \advance\leftskip\@tempdima\relax
    \rightskip\@pnumwidth plus4em \parfillskip-\@pnumwidth
    #5\leavevmode\hskip-\@tempdima
      \ifcase #1
       \or\or \hskip 1.65em \or \hskip 3.3em \else \hskip 4.95em \fi%
      #6\nobreak\relax
    \Dotfill
    \hbox to\@pnumwidth{\@tocpagenum{#7}}\par
    \nobreak
    \endgroup
  \fi}
\makeatother

\makeatletter
\def\l@section{\@tocline{1}{0pt}{1pc}{}{\scshape}}
\renewcommand{\tocsection}[3]{%
\indentlabel{\@ifnotempty{#2}{\ignorespaces#1 #2.\hskip 0.7em}}#3}
\def\l@subsection{\@tocline{2}{0pt}{1pc}{5pc}{}}

\def\l@subsubsection{\@tocline{3}{0pt}{1pc}{7pc}{}}

\makeatother

\numberwithin{equation}{section}

\newtheoremstyle{mytheorem}{.7\linespacing\@plus.3\linespacing}{.7\linespacing\@plus.3\linespacing}%
     {\itshape}
     {}
     {\bfseries}
     {. }
     {0.3ex}
     {\thmname{{\bfseries #1}}\thmnumber{ {\bfseries #2}}\thmnote{ (#3)}}  

\theoremstyle{mytheorem}

\newtheorem{theorem}{Theorem}[section]

\newtheorem{proposition}[theorem]{Proposition}

\newtheorem{remark}[theorem]{Remark}
\newtheorem{definition}[theorem]{Definition}

\def\bs{\boldsymbol}

\newcommand\bW{\boldsymbol W}
\newcommand\bK{\boldsymbol K}
\newcommand\be{\boldsymbol e}

\newcommand\bp{\boldsymbol p}
\newcommand\bq{\boldsymbol q}
\newcommand\bx{\boldsymbol x}
\newcommand\by{\boldsymbol y}
\newcommand\ba{\boldsymbol a}
\newcommand\bb{\boldsymbol b}

\newcommand\bw{\boldsymbol w}

\newcommand\bPi{\boldsymbol \Pi}

\newcommand\red{\textcolor{red}}

\usepackage{bbold} 

\newcommand{\bbE}{{\ensuremath{\mathbb E}} }

\newcommand{\bbN}{{\ensuremath{\mathbb N}} }

\newcommand{\bbP}{{\ensuremath{\mathbb P}} }
\newcommand{\bbQ}{{\ensuremath{\mathbb Q}} }
\newcommand{\bbR}{{\ensuremath{\mathbb R}} }

\newcommand{\bbZ}{{\ensuremath{\mathbb Z}} }

\newcommand{\cA}{{\ensuremath{\mathcal A}} }

\newcommand{\cE}{{\ensuremath{\mathcal E}} }

\newcommand{\cL}{{\ensuremath{\mathcal L}} }

\newcommand{\cS}{{\ensuremath{\mathcal S}} }

\newcommand{\cX}{{\ensuremath{\mathcal X}} }
\newcommand{\cY}{{\ensuremath{\mathcal Y}} }
\newcommand{\cZ}{{\ensuremath{\mathcal Z}} }

\newcommand{\ga}{\alpha}
\newcommand{\gb}{\beta}

\newcommand{\gd}{\delta}

\newcommand{\gz}{\zeta}

\newcommand{\gl}{\lambda}

\newcommand{\go}{\omega}

\newcommand\sfB{\mathsf B}

\newcommand\sfG{\mathsf G}

\newcommand\sfK{\mathsf K}
\newcommand\sfL{\mathsf L}

\newcommand\sfQ{\mathsf Q}
\newcommand\sfR{\mathsf R}
\newcommand\sfS{\mathsf S}
\newcommand\sfT{\mathsf T}

\newcommand\sfW{\mathsf W}

\newcommand\sfX{\mathsf X}
\newcommand\sfZ{\mathsf Z}

\newcommand\sff{\mathsf f}
\newcommand\sfg{\mathsf g}
\newcommand\sfh{\mathsf h}

\newcommand\sft{\mathsf t}

\newcommand\sfw{\mathsf w}

\renewcommand{\tilde}{\widetilde}          
\DeclareMathSymbol{\leqslant}{\mathalpha}{AMSa}{"36} 
\DeclareMathSymbol{\geqslant}{\mathalpha}{AMSa}{"3E} 
\DeclareMathSymbol{\emptysetet}{\mathalpha}{AMSb}{"3F}     
\newcommand{\dd}{\text{\rm d}}             

\newcommand{\sumtwo}[2]{\sum_{\substack{#1 \\ #2}}} 
\newcommand{\crossrsk}[4]{\small{{\begin{array}{ccc} & #1& \\ #2 & \cross & #3 \\ & #4&\end{array}}}}

\newcommand{\R}{\mathbb{R}}
\newcommand{\C}{\mathbb{C}}
\newcommand{\Z}{\mathbb{Z}}
\newcommand{\N}{\mathbb{N}}

\newcommand{\PEfont}{\mathrm}

\newcommand{\p}{\ensuremath{\PEfont P}}
\newcommand{\e}{\ensuremath{\PEfont E}}
\newcommand{\E}{\e}
\renewcommand{\P}{\p}

\newcommand\bP{\ensuremath{\bs{\mathrm{P}}}}

\newcommand{\ind}{\mathds{1}}
\renewcommand{\ind}{\mathbb{1}}

\renewcommand{\epsilon}{\varepsilon}
\renewcommand{\theta}{\vartheta}
\renewcommand{\rho}{\varrho}

\newenvironment{myenumerate}{%
\renewcommand{\theenumi}{\arabic{enumi}}%
\renewcommand{\labelenumi}{{\rm(\theenumi)}}%
\begin{list}{\labelenumi}
	{%
	\setlength{\itemsep}{0.4em}%
	\setlength{\topsep}{0.5em}%
	\setlength\leftmargin{2.45em}%
	\setlength\labelwidth{2.05em}%
	\setlength{\labelsep}{0.4em}%
	\usecounter{enumi}%
	}%
	}%
{\end{list}
}

{\end{list}
}

{\end{list}
}

\renewenvironment{enumerate}{
\begin{myenumerate}}%
{\end{myenumerate}}

\newenvironment{myitemize}{%
\begin{list}{$\bullet$}%
 	{%
	\setlength{\itemsep}{0.4em}%
	\setlength{\topsep}{0.5em}%
	\setlength\leftmargin{2.45em}%
	\setlength\labelwidth{2.05em}%
	\setlength{\labelsep}{0.4em}%
	}%
	}%
{\end{list}}

\renewenvironment{itemize}{
\begin{myitemize}}%
{\end{myitemize}}

\MakePerPage[2]{footnote} 

\def\dd{\mathrm{d}}

\newcommand\rs{\sfR\sfS}
\newcommand\rsk{\sfR\sfS\sfK}
\newcommand\grsk{\sfg\sfR\sfS\sfK}
\newcommand\GT{\sfG\sfT}
\newcommand\gGT{\sfg\sfG\sfT}
\newcommand{\Ai}{\mathcal{A}i}

\newcommand{\cross}
{\begin{picture}(12,10)(-2,0)
\put(-8,0){\Large$\longrightarrow$}
\put(0,0){$\big\downarrow$}
\end{picture}}

\begin{document}

\title
[algebraic structures in KPZ]
{Some algebraic structures \\ in KPZ universality}

\begin{abstract} We review some algebraic and 
combinatorial structures that underlie models in the KPZ universality class.
Emphasis is placed on the Robinson-Schensted-Knuth correspondence and its geometric lifting due to A.N.Kirillov.
We present how these combinatorial constructions are used to analyse the structure of solvable models in the KPZ 
class and lead to computation of their statistics via connecting to representation theoretic objects such as
 Schur, Macdonald and Whittaker functions, Young tableaux and Gelfand-Tsetlin patterns. 
 We also present how fundamental representation theoretic concepts, such
 as the Cauchy identity, the Pieri rule and the branching rule, can be used,
  alongside $\rsk$ correspondences, and can be combined with probabilistic ideas, in order 
  to construct integrable stochastic dynamics on two dimensional arrays of Gelfand-Tsetlin type, 
  in ways that couple different one dimensional stochastic processes. 
  For example, interacting particle systems, on the one hand, and
  processes related to eigenvalues of random matrices, on the other, thus illuminating the emergence of random matrix distributions
  in interacting stochastic processes.
 The goal of the notes is to expose some of the overarching principles, which have driven a significant number
 of developments in the field.

\end{abstract}

\author[N.Zygouras]{Nikos Zygouras}
\address{Department of Mathematics\\
University of Warwick\\
Coventry CV4 7AL, UK}
\email{N.Zygouras@warwick.ac.uk}

\date{\today}


\maketitle

\tableofcontents
\section{Introduction}\label{sec:intro}

The Kardar-Parisi-Zhang equation is the nonlinear (Hamilton-Jacobi type) stochastic partial differential equation 
\begin{align}\label{KPZ}
\frac{\partial h}{\partial t}=\frac{1}{2}\Delta h +\frac{1}{2}|\nabla h|^2 +\dot{W} (t,x),\qquad x\in \bbR^d, t>0,
\end{align}
proposed in \cite{KPZ86} as a universal object governing the 
fluctuations of randomly growing interfaces. $\dot{W}$ is the space-time white noise, which is a distribution valued
Gaussian process, delta correlated in space and time as
\begin{align*}
\bbE\big[ \dot{W}(t,x) \dot{W}(s,y)\,\big] = \delta(t-s)\,\delta(x-y).
\end{align*}

Through a dynamical renormalization analysis, based also on earlier studies on the stochastic
 Burgers equation by \cite{FNS77}, it was predicted that, in dimension one, the fluctuations of models within this class are governed by a $t^{1/3}$ scaling, while one expects to observe spatial correlations at scales $t^{2/3}$. Roughly speaking,
 this means that asymptotically for large time $t$, the solution to \eqref{KPZ} behaves as
 \begin{align}\label{GUEasy}
 h(t,x) \approx \mu t + t^{1/3} \mathfrak{h}\big(\frac{x}{t^{2/3}}\big),
 \end{align}
 where the term $\mu t$ represents a macroscopic, deterministic behaviour and is model dependent
 and $\mathfrak{h}(\cdot)$ is a random function representing the fluctuations of the system.
 
 Even though the exponents $1/3$ and $2/3$ were predicted in \cite{KPZ86}, the arguments there do not allow for a prediction on the nature of the process $ \mathfrak{h}(\cdot)$. The indication of what this should be came through the analysis of discrete models,
  which showed the remarkable fact that the statistics of $\mathfrak{h}(\cdot)$ are related to statistics in random matrix theory. The first work that set the grounds of this link is
 the work of Baik-Deift-Johansson \cite{BDJ99} on the statistics of the {\it longest increasing subsequence 
 in a random permutation} of length $n$ (also known as {\it ``Ulam's problem''}), which
  exhibited rigorously not only the $n^{1/3}$ scaling but also
  the link to the {\it Tracy-Widom GUE law}\footnote{GUE stands for {\it Gaussian Unitary Ensemble}}, 
  which describes the asymptotic statistics of the largest eigenvalue of
  an ensemble of random, hermitian
   matrices with gaussian entries. At the same time the work of Pr\"ahofer and Spohn 
  \cite{PS00, PS02} introduced the {\it Airy process} as the law of the process
 $\mathfrak{h}(\cdot)$ after subtraction of a certain parabola. 
 Crucial in these breakthroughs was a combinatorial algorithm known as the {\it Robinson-Schensted-Knuth corespondence} ($\rsk$). This is a correspondence between (generalised) permutations and a pair of {\it Young tableaux} (with the latter being an object of central importance in representation theory). The link to probability and KPZ emerges from the 
 fact that the length of a longest increasing subsequence in a permutation is encoded via a certain observable in
 the Young tableaux (the common length of their  {\it first rows}) when 
 viewed as the image of a permutation under the Robinson-Schensted-Knuth correspondence.
 The significance of this correspondence, though, goes into much greater depths as it encodes a profound
 structure shared by models in the KPZ class, which otherwise would remain hidden.
 
 Around the same time as the works of Baik-Deift-Johansson and Pr\"ahofer-Spohn,
  the works of Okounkov \cite{O01}, Okounkov-Reshetikhin \cite{OR03} 
  on {\it ``Schur Measures''} and {\it ``Schur-Processes''}, respectively, consolidated the link between probabilistic 
  models in the KPZ class to integrable and representation theoretic objects and in particular to the theory of 
  symmetric functions (centred around Schur functions). Moreover, the works of Johansson \cite{J01b}, Okounkov \cite{O00} and 
  Borodin-Okounkov-Olshanksi \cite{BOO00} extended the link between permutations and Young tableaux beyond just the equality of the length of the longest increasing 
  subsequence and the length of the first row. This amounted to clarifying the link between the rest of the lines of the Young tableaux and observables in random permutations
  and also performing asymptotics that consolidated the link to statistics of eigenvalues of random matrices.
  Reference \cite{BO01} provides an interesting account of relations between the Robinson-Schensted-Knuth correspondence, measures on partitions, random matrix distributions and the representation theory of the infinite symmetric group.
  
  These breakthroughs followed a decade of intense activity where KPZ / Tracy-Widom fluctuations and Airy processes were discovered within a large class of models, ranging from interacting particle systems, domino tilings, Aztec diamonds and growth models. The analysis of such models was 
  facilitated by combinatorial
 algorithms, such as $\rsk$, and constructions which connected to ensembles of non-intersecting paths
 and led to the expression of probabilities in terms of {\it determinantal measures}. These developments created
 what is now known as the framework of {\it determinantal processes} \cite{B11}.
 The asymptotic analysis of such
 measures and processes
  was then possible by imitating or using techniques developed earlier in the context of random 
 matrix  theory. One can look at the book \cite{BDS16} or the reviews \cite{J01,J05, J17} for reference to those 
 developments. 
 
However important the outcomes and the methods developed during this period may have been, they
were restricted to models at {\it ``zero temperature''}. Here,
temperature is to be understood as a parameter of the system which, according to its value, imposes a certain rigidity or relaxation.
For example, if one considers a system of particles moving on a line, then zero temperature
is to be interpreted as the particles moving only towards one direction (e.g. right) while {\it positive temperature}
should be understood as particles being able to move towards 
both right and left directions (but with a drift towards one of the directions, e.g. right). 
Or, if one considers {\it polymer models}, then zero temperature corresponds to considering
 a single (polymer) path that achieves the maximum in a random variational problem (known as {\it last passage
 percolation}), while at positive temperature we consider a {\it thermal average} over all admissible
 paths. Even though the breakthroughs during the first decade of the millennium confirmed the KPZ
 prediction for several models within this class, the restriction of those methods to ``zero temperature''
 presented a limitation, which, among others, did not allow to prove that the solution to the KPZ equation itself (interpreted in a certain rigorous sense \cite{BG95})  obeys these predictions.
 
 At this point, a second wave of breakthroughs took place, which allowed not only to handle models at positive
 temperature, but also to expand significantly the class of models, which exhibit KPZ type fluctuations as in \eqref{GUEasy}, with new stochastic models as well as
 new methods and links between disparate scientific areas. Some of the highlights of this progress can be summarised to be: (i) Tracy and Widom's solution of the asymmetric exclusion process (ASEP) \cite{TW08a, TW08b, TW09} via the Bethe Ansatz, which subsequently led to the confirmation of the Tracy-Widom fluctuations
 for the solution to the KPZ equation at a single point \cite{SS10, ACQ11} and to simultaneous
  works with similar results in the physics literature \cite{D10, CLeDR10};
 (ii) the analysis of the combinatorial structure, via Kirillov's {\it geometric lifting} of $\rsk$ \cite{K01},
  of two random polymer models, the {\it Brownian} or {\it O'Connell-Yor polymer} \cite{O12}
 and the {\it log-gamma polymer} \cite{COSZ14} (the latter was
 originally introduced by Sepp\"al\"ainen \cite{S12}, inspired by
 earlier works on the ASEP \cite{BS10}) and (iii) the introduction of the {\it Macdonald processes} by Borodin-Corwin \cite{BC14} as a deformation of Okounkov-Reshetikhin's ``Schur processes''.
 \vskip 2mm
 Exposing all the developments, methods and results obtained in the understanding of KPZ universality,
  especially after this second wave of developments, would be impossible in a single review. 
 The purpose of these notes is to expose only some of the key algebraic and 
 combinatorial structures underlying discrete models in the KPZ universality class,
 which provide the necessary tools to study their statistical properties. The emphasis will be on and around the
 Robinson-Schensted-Knuth correspondence and its geometric lifting due to Kirillov
 ({\bf Sections \ref{sec:RSK}, \ref{sec:grsk}})
 and how these lead to connections with
 representation theoretic objects and related special, symmetric functions such as Schur, Whittaker and Macdonald
 functions, which we will introduce in {\bf Section} \ref{basics}. Our hope with these notes is to unravel some of the
  overarching principles that run as a backbone in the
 theme and can be used as a motivation and inspiration for further developments. 
 In particular, we would like to draw some parallel between the first set of 
 breakthroughs at the end of nineties and the second set from around 2010 onwards. In this effort we have chosen to
 present $\rsk$ and geometric $\rsk$ correspondences in an intertwined fashion highlighting conceptual similarities between the two 
 and showing how one can lift or reduce properties of one to the other. The main probabilistic examples that
we will use, in order to demonstrate the power of the Robinson-Schensted-Knuth correspondence (and its geometric
lifting), will be the last passage percolation ({\bf Section \ref{geomLPP}}) and the {\it log-gamma} 
random polymer model ({\bf Section \ref{sec:loggamma}}), whose analysis we will present in a parallel fashion, indicating similarities and differences. 
We will also link to dynamics of particle systems ({\bf Section \ref{sec:dynamics}}), 
inspired partly by viewing the $\rsk$ algorithm as a dynamical procedure, rather than a bijective map,
and partly by certain representation theoretic structures, such as the {\it Cauchy identity},
the {\it Pieri rule} and the {\it branching rule}. In particular, we will try to show how these three principles can be used
as a guide to produce interesting stochastic dynamics that couple different particle systems.
Finally, we will also give an idea of the general principles that drive the asymptotic
analysis, towards Tracy-Widom and Airy process laws
 ({\bf Section \ref{sec:Fred}}),  where we will also expose the framework of {\it determinantal point processes}.
In {\bf Section \ref{sec:Fred}} we will also discuss multi-point correlations and present some recent constructions of the {\it universal}
models of {\it Airy line ensemble, KPZ fixed point, Airy sheet and Directed Landscape}.
 The whole discussion will start in
  {\bf Section \ref{sec:examples}} with some of the main examples of models in the KPZ class. The list
  here is far from exhaustive, especially taking into account the large number of models that have
  been constructed in the recent years. 
However, the models presented in Section \ref{sec:examples} are chosen as being suitable (and distinguished)
 discretizations of the ``solution'' to the KPZ equation, which after a number of simplifications lead to the fundamental problem of the length of the longest increasing subsequence in a random permutation - a problem that
has undoubtedly been the cornerstone of all the developments around the integrability of KPZ.  
 \vskip 2mm
 Before closing this introduction, let us mention that 
 there is a number of other reviews that expose different aspects of this multifaceted
 field, to which these notes have a complementary purpose and focus, and which can provide source of 
 further reading and comparison. 
 Without being exhaustive let us highlight the following ones.
  Krug and Spohn \cite{KS92} provide a physical background of stochastic growth. Johansson 
 \cite{J01,J05, J17}, Kriecherbauer-Krug \cite{KK10} as well as the book of Baik-Deift-Suidan \cite{BDS16}
 provide a comprehensive account of mostly the early breakthroughs that led to the link between KPZ models, 
 determinantal structures and random matrices. The {\it ``Lectures on Integrable Probability} ''
  by Borodin and Gorin \cite{BG16} and Borodin and Petrov \cite{BP14} expose the recent developments in relation to Macdonald Processes \cite{BC14}, while the lecture notes by Borodin and Petrov \cite{BP16a} expose links to other distinguished integrable, statistical mechanics models such as the stochastic six-vertex model, which
  we will not discuss here. Reviews \cite{S10a, S17} expose more probabilistic techniques 
  (large deviations, Busemann functions, Burke's property) that have been developed 
  to capture the $t^{1/3}$ fluctuations of the solvable corner growth (with the solvability there encoded via the knowledge
  of the ``stationary process'') as well as the asymmetric exclusion process. 
  For reviews more focused  on the KPZ equation itself and related Airy processes one could refer to
   \cite{C12, QR14, QS15}.
  Finally, for martingale approaches to general (non solvable) random polymer models one can refer to 
  \cite{C17, CSY04} as well as \cite{denH09}; in the latter,
  questions on polymers outside the KPZ scope are also explored.
 
 \section{Examples of discrete models in the KPZ class}\label{sec:examples}
 The mere existence of a solution to \eqref{KPZ} and the sense that this could be given are far from obvious. The reason is that due to the presence of the white noise, any possible solution should look locally like Brownian motion in space, meaning that its spatial derivative exists only as a distribution, i.e. a Dirac delta-like function.
  The difficulty then arises when one tries to consider the nonlinear term $|\partial_x h|^2$, 
  since one cannot give a meaning to the square of a delta function.
The problem of well posedness of the KPZ equation in dimension one
\footnote{In dimensions higher than one the situation is more complicated. Progress has recently been made \cite{CSZ20, CD20, Gu20, LZ20, CNN20, CSZ21}}
 has now been settled with methods emanating both from stochastic analysis
\cite{H13, GP17, K16} and particle systems \cite{BG95, GJ14}. Even so, these approaches do not capture
 the asymptotics \eqref{GUEasy}.
 In order to study the statistical properties of the solution to the \eqref{KPZ} and in particular \eqref{GUEasy} one resorts to suitable discrete models and combinatorial or integrable methods which yield the corresponding 
 asymptotics for these. From thereon one can use approximations of the KPZ by discrete systems, eg via \cite{BG95},
 in order to transfer asymptotics of discrete models to KPZ. 
 
 The characterising features of the KPZ class are the following:
 \begin{itemize}
 \item {\bf Local smoothing}: this is manifested by the $\Delta h$ term and captures the tendency of the
 interface towards a flat profile. 
 \item {\bf Gradient growth:} this is manifested by the term $|\nabla h|^2$ and captures an opposite to the previous mechanism,
 which is that peaks of the interface will tend to grow faster.
 \item {\bf Local randomness:} this is manifested by the white noise $\dot W(t,x)$, which is uncorrelated in space and time.
 \end{itemize}
 
Our focus in this section is to present some of the discrete models, which exhibit these features and 
 are amenable to analysis. We will start with the formal representation of the solution to the KPZ equation via 
 the Feynman-Kac formula and then go on to a number of simplifications, each one of which produces a model that can be solved 
 with the algebraic and combinatorial methods that we will present. 
 The simplest model in this sequence of reductions will be the problem of {\it longest increasing subsequence},
 which was the first one solved \cite{BDJ99}.

 \subsection{Directed Polymer in Random Medium.}\label{sec:DPRM}
 Let us assume for the moment that the white noise in \eqref{KPZ}
  is replaced by a {\it smooth}, random function $V(t,x)$ and let us perform the change of variables
 \begin{align*}
 h(t,x)=\log u(t,x).
 \end{align*}
 This transformation is known as the Hopf-Cole transformation and transforms \eqref{KPZ} 
 (with the white noise replaced by the smooth potential $V$) to
 \begin{align}\label{SHE}
 \frac{\partial u}{\partial t} =\frac{1}{2}\frac{\partial^2 u}{\partial x^2}+ V u ,\qquad x\in \bbR, t>0,
 \end{align} 
 which is a linear, parabolic equation, known as the {\it Stochastic Heat Equation}.
 Assuming an initial condition $f(\cdot)$, then \eqref{SHE} can be solved via the Feynman-Kac formula
(see \cite{D18}, Chapter 4) as
 \begin{align}\label{FK}
 u(t,x) = \E_x\Big[ f(B(t)) \,e^{\int_0^t V(t-s,B(s)) \,\dd s}  \Big],
 \end{align} 
 where $B(\cdot)$ is one-dimensional Brownian motion and $\E_x$ denotes expectations with respect to its law,
 when starting from location $x\in\R$.
 We may assume that the initial condition to \eqref{SHE} is a delta function, in order to obtain
  its fundamental solution. More precisely, we assume that
  $f(x)=\delta_0(x)$ and then making the time change $s\mapsto t-s$ in \eqref{FK} and using the time reversal 
  invariance of the Brownian Bridge, \eqref{FK} has the same distribution as
   \begin{align}\label{FKd}
 \E_0\Big[ \delta_x(B(t)) \,e^{\int_0^t V(s,B(s)) \,\dd s}  \Big].
 \end{align} 
 Attempting to repeat the same argument back in the case of white noise $\dot W$, instead of a smooth potential $V$, runs into the problem of giving a meaning to the quantity $\exp\big(\int_0^t \dot W(t-s,B(s)) \dd s\big)$, as it amounts to integrating the white noise field  which 
  takes values $\pm \infty$. It turns out that this can be done in dimension one by first mollifying spatially the noise and applying a variation of the Feynman-Kac formula suitable for white-in-time potential and then removing the mollification. This was done by Bertini-Cancrini in
 \cite{BC95}. However, motivated by
 \eqref{FK}, one may also be led to study a discrete, lattice analogue where 
 \begin{itemize}
\item[ (i)] the space-time white noise $\big(\dot W(t,x)\big)_{t>0,x\in \bbR}$ is replaced by
 a family of independent, identically distributed (i.i.d.) variables $(\go(n,x))_{n\in\bbN, x\in \bbZ}$, 
\item[ (ii)] Brownian motion $(B(t))_{t>0}$ is replaced by its discrete analogue, which is a simple, symmetric random walk $(S_n)_{n\geq 1}$ ,
\item[ (iii) ] integration is replaced by summation.
\end{itemize}
 Making these replacements,
 \eqref{FKd} is replaced by
 \begin{align}\label{DPRMp2p}
 Z_{N,\beta}^\go(x):=\E_0\Big[\ind_{\{S_N=x\}}\,e^{\beta\sum_{n=1}^N \go(n,S_n)}\Big].
 \end{align}
 This is the {\it point-to-point partition function} of the model of {\it Directed Polymer in Random Medium}, where we have also included an {\it inverse temperature}
 parameter $\beta$, which tunes the strength of the disorder ($\beta=\infty$ will amount to what we have called ``zero temperature''; we will come back to this in a subsequent paragraph). Notice that,
 contrary to \eqref{FKd}, \eqref{DPRMp2p} is perfectly defined and one expects that its large time asymptotics
 should agree with \eqref{GUEasy}. This means that the asymptotic behaviour
 \begin{align}\label{GUEpoly}
 \log Z_{N,\beta}^\go(x)\approx \sff(\beta) N+ \sigma(\beta)N^{1/3} \,
 \mathfrak{h}\big(\rho(\beta)N^{-2/3}x\big), \qquad \text{as}\,\, N\to\infty,
 \end{align}
should hold,
 exhibiting the same universal exponents $1/3$ and $2/3$ and fluctuation process $\mathfrak{h}(\cdot)$
 as for the KPZ equation \eqref{GUEasy} ($\sff(\beta), \sigma(\beta),\rho(\beta)$ will be model specific constants). 
  In particular, one expects these asymptotics to be universal, irrespective of the distribution of the random 
  variables $(\go_{n,x})_{n\in \bbN,x\in \bbZ}$, as long as they possess enough (five) moments
  (the conjectures on the number of required moments is more recent, see
   \cite{D07, BBP07, GDBR15, DZ16}). The non universal parameter
  $\sff(\beta)$ is known as the {\it free energy}. 
 \subsection{Corner growth.}\label{sec:corner}
 KPZ was proposed in an attempt to describe fluctuations in randomly growing processes. Let us describe a prototypical such process, which is known as the 
 {\it corner growth process}. At time $t=0$ consider a corner like the one depicted in Figure \ref{corner}. After a random time, a unit square
 fills the corner with bottom vertex at $(0,0)$. Now, two more corners with bottom vertices $(1,1)$ and $(-1,1)$ are formed and each one is filled with a unit square after random times, which are independent of each other as well 
 as the previous filling time. At each step the corners of the interface are filled with a unit square and the
 time that it takes for each corner to be
 filled is independent of all other times. Let us now map the features of the process to the terms of \eqref{KPZ}
 \begin{itemize}
 \item The fact that unit squares fill corners is consistent with the smoothing effect of the Laplacian;
 \item Parts of the interface which are very stretched (that is they have very few corners, e.g. the corner marked by the string ``$1000$'' )
 in Figure \ref{corner}), grow slower than other parts with many corners (e.g. the sequence of corners marked by the string ``$10100$'' in Figure \ref{corner}).
  This is consistent with the growth of the interface being proportional to $|\partial_x h|^2$ 
  (strictly speaking this slower effect should correspond to a negative sign in front of $|\partial_x h|^2$ 
  but in terms of statistical properties this is immaterial upon considering $-h$, instead);
 \item The randomness and independence of the waiting times until corners are filled is consistent with the presence of the space-time white noise $\dot W$.
 \end{itemize}
 
  \begin{figure}[t]
 \begin{center}
\begin{tikzpicture}[scale=.6]
\draw (0,0) -- (10.5,10.5) ;
\draw (0,0) -- (-10.5,10.5);
\draw[ultra thick, red] (-10.5, 10.5)--(-6,6) -- (-5,7) -- (-3,5) -- (-2,6) -- (-1,5)--(2,8)--(3,7)--(4,8)--(5,7)--(7,9)--(8,8)--(10.5,10.5); 
\draw[dashed] (-1,1) -- (5,7);
\draw[dashed] (-2,2) -- (3,7);
\draw[dashed] (-3,3) -- (-1,5);
\draw[dashed] (-4,4) -- (-3,5);
\draw[dashed] (-5,5) -- (-4,6);
\draw[dashed] (1,1) -- (-3,5);
\draw[dashed] (2,2) -- (-1,5);
\draw[dashed] (3,3) -- (0,6);
\draw[dashed] (4,4) -- (1,7);
\draw[dashed] (5,5) -- (3,7);
\draw[dashed] (6,6) -- (5,7);
\draw[dashed] (7,7) -- (6,8);
\node at (-9.7, 9.3) {$1$}; \node at (-8.7, 8.3) {$1$}; \node at (-7.7, 7.3) {$1$}; \node at (-6.7, 6.3) {$1$};
\node at (-5.3, 6.2) {$0$}; \node at (-4.7, 6.2) {$1$}; \node at (-3.8, 5.3) {$1$}; \node at (-2.3, 5.2) {$0$};
\node at (-1.7, 5.2) {$1$}; \node at (-0.2, 5.3) {$0$}; \node at (0.8, 6.3) {$0$}; \node at (1.8, 7.3) {$0$};
\node at (2.3, 7.3) {$1$}; \node at (3.7, 7.3) {$0$};  \node at (4.2, 7.3) {$1$};  \node at (5.7, 7.3) {$0$};
\node at (6.7, 8.3) {$0$}; \node at (7.3, 8.3) {$1$};  \node at (8.7, 8.3) {$0$}; \node at (9.4, 9.0) {$0$};
\draw[dotted, thick] (9.7,9.3) -- (10.2, 9.8);
\draw[dotted, thick] (-10.1,9.7) -- (-10.5, 10.1);

\draw (-12,0)--(12,0);
\draw [fill] (-9.7, 0) circle [radius=0.2];  \draw  [fill] (-8.7, 0)  circle [radius=0.2]; \draw  [fill] (-7.7, 0)  circle [radius=0.2]; \draw  [fill] (-6.7, 0)  circle [radius=0.2]; 
\draw (-5.7, 0)  circle [radius=0.2]; \draw  [fill] (-4.7, 0)  circle [radius=0.2]; \draw  [fill]  (-3.7, 0)  circle [radius=0.2]; \draw (-2.7, 0)  circle [radius=0.2];
\draw  [fill ](-1.7, 0)  circle [radius=0.2];; \draw (-0.7, 0)  circle [radius=0.2];; \draw (0.8, 0)  circle [radius=0.2];; \draw (1.8, 0)  circle [radius=0.2];;
\draw  (1.8, 0)  circle [radius=0.2]; \draw [fill] (2.8, 0)  circle [radius=0.2]; \draw (3.8, 0)  circle [radius=0.2]; \draw [fill] (4.8, 0)  circle [radius=0.2];
\draw  (5.8, 0)  circle [radius=0.2]; \draw (6.8, 0)  circle [radius=0.2]; \draw  [fill]  (7.8, 0)  circle [radius=0.2]; \draw (8.8, 0)  circle [radius=0.2];
\draw (9.8, 0)  circle [radius=0.2];

\draw[dotted, thick] (10.5,0.2) -- (12,0.2);
\draw[dotted, thick] (-10.5,0.2) -- (-12,0.2);
\end{tikzpicture}
 \end{center}  
\caption{ \small  A corner growth process and its mapping to TASEP }
\label{corner}
\end{figure}

   \begin{figure}[t]
 \begin{center}
\begin{tikzpicture}[scale=.6]
\draw (0,0) -- (10.5,10.5) ;
\draw (0,0) -- (-10.5,10.5);
\draw[dashed] (-1,1) -- (6,8);
\draw[dashed] (-2,2) -- (5,9);
\draw[dashed] (-3,3) -- (4,10);
\draw[dashed] (-4,4) -- (3,11);
\draw[dashed] (-5,5) -- (2,12);
\draw[dashed] (-6,6) -- (1,13);
\draw[dashed] (-7,7) -- (0,14);

\draw[dashed] (1,1) -- (-6,8);
\draw[dashed] (2,2) -- (-5,9);
\draw[dashed] (3,3) -- (-4,10);
\draw[dashed] (4,4) -- (-3,11);
\draw[dashed] (5,5) -- (-2,12);
\draw[dashed] (6,6) -- (-1,13);
\draw[dashed] (7,7) -- (0,14);

\foreach \x in {0,...,6} {
 \draw[fill] (\x,\x+1) circle [radius=0.1]; 
 }
 \foreach \x in {1,...,7} {
  \draw[fill] (\x-2,\x+1) circle [radius=0.1]; 
 }
  \foreach \x in {2,...,8} {
  \draw[fill] (\x-4,\x+1) circle [radius=0.1]; 
 }
   \foreach \x in {3,...,9} {
  \draw[fill] (\x-6,\x+1) circle [radius=0.1]; 
 }
  \foreach \x in {4,...,10} {
  \draw[fill] (\x-8,\x+1) circle [radius=0.1]; 
 }
  \foreach \x in {5,...,11} {
  \draw[fill] (\x-10,\x+1) circle [radius=0.1]; 
 }
  \foreach \x in {6,...,12} {
  \draw[fill] (\x-12,\x+1) circle [radius=0.1]; 
 }

\draw [ultra thick, red] (0,1)--(2,3) -- (1,4) --(4,7) -- (0,11) --(1,12)--(0,13);
\end{tikzpicture}
 \end{center}  
\caption{ \small  Last passage percolation path }
\label{LPPsquare}
\end{figure}

 \subsection{Interacting particle systems - exclusion process.}\label{secTASEP}
 We can map the interface and the dynamics of the corner growth process, described in the previous paragraph,
  to an interacting particle system as follows: To each downward edge of the corner interface we assign the number $1$ and to each upward the number $0$. This configuration can also be projected onto the one-dimensional lattice $\frac{1}{2}\bbZ$,
 with $1$'s corresponding to particles and $0$'s corresponding to empty sites. All corners are encoded by a sequence $(\cdots 10 \cdots)$
 and when a corner (encoded via a string ``$10$'')
  is filled then the configuration of zeros and ones changes to ``$01$''.
  Projected down to the line $\frac{1}{2}\bbZ$ this change corresponds to a particle (encoded via a ``$1$'' in the string) jumping to the empty site (encoded via a ``$0$'' in the string) on its immediate right. Recall that the corner growth can only grow at corners, which means that a particle can only jump if its right neighbouring site is empty. In the case that the waiting times are exponentially distributed, then this 
  particle process is Markovian and known as the Totally Asymmetric Simple Exclusion Process (TASEP).  
 
 For each $x\in \frac{1}{2}\bbZ$ let us denote by 
 \begin{align*}
 \eta_t(x):=\ind_{\{\text{a particle occupies site $x$ at time $t$}\}}.
 \end{align*}
 Then we see that $\eta_t(x) = (h(t,x)-h(t,x-1)+1)/2$. This is in fact the discrete version of another
  standard PDE transformation, which transforms a Hamilton-Jacobi equation to a conservation law. 
  More specifically, setting
  $\rho:=\partial_x h$ in \eqref{KPZ} leads via differentiating the equation in space to
 \begin{align}\label{Burgers}
 \frac{\partial \rho}{\partial t} = \frac{1}{2}\frac{\partial^2 \rho}{\partial x^2}+ \rho \frac{\partial \rho}{\partial x} +\frac{\partial \dot W}{\partial x}(t,x).
 \end{align}
 This is the Stochastic Burgers equation and describes the density fluctuations of the TASEP.
  
 \subsection{Last Passage Percolation.}\label{sec:LPP}
 We can describe the time $\tau_{x,y}$ that it takes for the corner growth interface to cover a corner with bottom site
  $(x,y)$ ($x$ being the horizontal cartesian coordinate and $y$ the vertical one) in terms of a variational problem:
  Notice that in order for the site $(x,y)$ to be covered it is necessary that both its neighbouring sites $(x-1,y-1)$ and $(x+1,y-1)$
  are already covered. Once both of these are covered, they form a corner which will then 
  be covered after an additional time $w_{x,y}$. We have,
  therefore, the recursive equation
  \begin{align}\label{LPPrec}
  \tau_{x,y} = \max(\tau_{x-1,y-1}, \tau_{x+1,y-1})+ w_{x,y}.
  \end{align}
  Iterating this and denoting by ${\bf\Pi}_{x,y}$ the set of directed, up-left or up-right paths from $(1,1)$ to $(x,y)$, we derive the variational formula
  \begin{align*}
  \tau_{x,y}=\max_{\pi\in {\bf\Pi}_{x,y}} \sum_{v\in \pi} w_{v}.
  \end{align*}
  This is depicted in Figure \ref{LPPsquare}.
  This quantity is known as the last passage percolation time and its statistics are linked
   to the statistics of the height via
  \begin{align*}
  \bbP\big(h(t,x) \geq y \big) = \bbP\big(\tau_{x,y} \leq t\big).
  \end{align*}
  
   \begin{figure}[t]
 \begin{center}
\begin{tikzpicture}[scale=.6]
\draw (0,0) -- (0,10)--(10,10)--(10,0)--(0,0) ;
\draw  [fill ](1,1)  circle [radius=0.1]; \draw (1,1)--(1,10.5);  \draw (1,1)--(10.5, 1);
\draw  [fill ](1.4, 8.1)  circle [radius=0.1]; \draw (1.4,8.1)--(1.4,10.5); \draw (1.4,8.1)--(2.1,8.1);
\draw  [fill ](2.1, 1.8)  circle [radius=0.1]; \draw(2.1,1.8)--(2.1,8.1); \draw(2.1,1.8)--(10.5, 1.8);
\draw  [fill ](3.4, 4.1)  circle [radius=0.1]; \draw(3.4, 4.1) --(3.4,10.5); \draw(3.4,4.1)--(7.9, 4.1);
\draw  [fill ](6.2, 5.6)  circle [radius=0.1]; \draw  (6.2, 5.6)--(6.2,10.5); \draw (6.2, 5.6)--(10.5, 5.6);
\draw  [fill ](5.2, 4.6)  circle [radius=0.1]; \draw(5.2, 4.6)--(10.5, 4.6); \draw(5.2, 4.6)--(5.2, 6.8);
\draw  [fill ](4.3, 6.8)  circle [radius=0.1]; \draw (4.3, 6.8)--(5.2, 6.8);   \draw (4.3, 6.8)--(4.3, 10.5);
\draw  [fill ](7.9, 2.8)  circle [radius=0.1]; \draw(7.9, 2.8)--(7.9,4.1); \draw(7.9, 2.8)--(10.5,2.8); 
\draw  [fill ](7.2, 8.1)  circle [radius=0.1]; \draw (7.2, 8.1)--(7.2,10.5);  \draw (7.2, 8.1)--(8.4,8.1);
\draw  [fill ](8.4, 7.4)  circle [radius=0.1]; \draw (8.4, 7.4)--(8.4, 8.1); \draw(8.4, 7.4)--(10.5, 7.4); 
\draw[ultra thick, red] (0,0)--(1,1)--(2.1,1.8)--(3.4, 4.1)--(5.2, 4.6)--(6.2, 5.6)--(7.2, 8.1) --(10,10) ;
\end{tikzpicture}
 \end{center}  
\caption{ \small  An optimal path in the Hammersley process. }
\label{fig:hammer}
\end{figure}

  \subsection{The Hammersley Process and Longest Increasing Subsequences.}\label{sec:Hammer}
  One may consider the following, degenerate last passage percolation problem: In the square with 
  side length $N$ and lower-left corner $(0,0$) 
  we have a Poisson Point Process with intensity $1$ and we ask what is the maximum number of Poisson points that can be collected by
  going from $(0,0)$ to $(N,N)$ via an up-right path, see Figure \ref{fig:hammer}. This is known as Hammersley problem.
  
  One can read the length of such maximal path (with which we mean the maximal number of points collected by an up-right path) 
  as follows: From each point draw a horizontal and vertical line going rightwards and upwards.
  If two such rays meet, they cancel each other. As is suggested from the picture, the length of the longest path 
  (i.e. the maximal number of points collected by an up-right path) equals the number of rays that reach either the top or right side of the square. 
  \begin{remark}{\rm
  From this construction one may realise the difficulty associated to this problem: if one only looks at the number of rays that reach the sides of the square, one cannot see
   the cancellations that take place inside it. So, in order to handle this problem, we would need to develop a method that would allow us to track down the 
   cancellations. The combinatorial methods that we will expose in these notes do exactly this.}
   \end{remark}
   One can also map the Hammersley problem to the problem of longest increasing subsequence in a random permuation. This is done as follows:
   Order the horizontal and vertical coordinates of the Poisson points in the square as $1,2,3...$ according to the order of their projections.
   We then write the coordinates $(x,y)$ of each point in the form of a {\bf biletter} ${x \choose y}$. 
   For example, in the case of Figure \ref{fig:hammer}, 
   we represent all the points in the form of a double array as
   \begin{align*}
   \left(\begin{array}{cccccccccc}
   1&2&3&4&5&6&7&8&9&10\\
   1&10&2&4&7&5&6&9&3&8
   \end{array}\right),
   \end{align*}
   and we see that the Poisson points are mapped to a permutation. Moreover,
   the length of the longest upright path through these points (which in this example is $6$)
    equals the length of the/a longest increasing subsequence in the permutation (which in this case is $1,2,4,5,6,9$).
   \vskip 2mm
   The problem of longest increasing subsequence is also related to last passage percolation as follows: A permutation $\sigma\in S_N$ of $\{1,2...,N\}$ is encoded through the 
   {\bf permutation matrix} $(a_{ij})_{1\leq i,j\leq N}$ where $a_{ij}=\delta_{i,\sigma(i)}$, with $\delta_{ij}$ the Kronecker delta. The length of the 
   longest increasing subsequence in the permutation $\sigma$ can be easily 
   seen to equal to $\max_{\pi:(1,1)\to (N,N)} \sum_{(i,j)\in \pi} a_{ij}$, where $\pi:(1,1)\to(N,N)$
   is a down-right path from entry $(1,1)$ to entry $(N,N)$.
   
   \section{Some central special functions }\label{basics}
   The combinatorial and algebraic techniques that we aim at exposing in the notes 
   allow to express the laws of observables of models with {\it integrable randomness} (hence the 
   recently coined name Integrable Probability)
   in terms of special functions. These are typically symmetric functions of combinatorial or representation theoretic origins
   but also of analytic nature, being eigenfunctions of difference or differential operators. The main such functions that will appear in
   these notes are 
   \begin{itemize}
   \item[(i)] {\bf Schur functions}, which will emerge in Section \ref{geomLPP} describing the law of a solvable last passage percolation model,
   \item[(ii)] {\bf Macdonald polynomials} and their {\bf q-Whittaker} degeneration, which will emerge in Section \ref{sec:dynamics} describing the laws 
   of particles in solvable interacting particle systems
   \item[(iii)] {\bf Whittaker functions}, which will emerge in Section \ref{sec:loggamma}, where they will be used to express the Laplace transform of 
   the partition function (recall \eqref{DPRMp2p}) of a solvable directed polymer model. 
  \end{itemize}
  \subsection{Basic symmetric functions}
  A {\bf partition} of a number $n$ is a sequence of non-increasing numbers $\lambda_1\geq \lambda_2\geq\cdots$ such that 
  $\lambda_1+\lambda_2+\cdots=n$.  More generally, we call a {\it partition} a sequence $\lambda:=(\lambda_1,\lambda_2,...)$, 
  which is decreasing, i.e. $\lambda_1\geq \lambda_2\geq\cdots$, and has only a finite number of non-zero terms. The non-zero 
  terms of $\lambda$ are called {\bf parts} and its number is denoted by $\ell(\lambda)$.
  A partition can be depicted by {\bf Young diagrams}. 
  These are arrays of left justified unit boxes, the first row of which has $\lambda_1$ boxes, the second row 
  $\lambda_2$ boxes etc. For example:
   \[
\lambda = (4,3,1) \qquad\longleftrightarrow\qquad \yng(4,3,1)
\]
   The boxes in a Young diagram are usually filled with (integer) numbers giving rise to either a {\bf standard Young tableau}, if the content of the
   boxes are strictly increasing along rows and columns, or a {\bf semistandard Young tableau}, if the contents are strictly increasing along columns but weakly increasing along rows.  The vector $(\lambda_1,\lambda_2,...)$ of the lengths of the rows of the Young tableau $T$ is called the {\bf shape} of the tableau
   and we denote it by $sh(T)$. 
   
   We will work with (the ring of) symmetric polynomials with integer coefficients, in some indeterminates $x_1,x_2,..., x_n$. The term ``symmetric''
   refers to the fact that the polynomials stay invariant under permutation of the indeterminates. We can also extend the notion of symmetric polynomials
   to the notion of symmetric functions by considering formal power series in an infinite number of indeterminates, which are invariant under permutations
   of the  indeterminates. We would prefer to refrain from exposing the formal framework, for which we refer to \cite{M98}, and rather go straight to the
    basic objects and examples.
    
    The most basic symmetric functions are the {\bf monomial symmetric functions}, which also form an algebraic basis for the the ring of symmetric functions.
    These are defined as follows: For a partition 
    $\lambda=(\lambda_1,\lambda_2,...)$, we 
    define the monomial $x^\lambda=x_1^{\lambda_1}x_2^{\lambda_2}\cdots$. The monomial symmetric function $m_\lambda$,
     indexed by $\lambda$, is the sum of all distinct monomials obtained from $x^\lambda$ by permuting the indeterminates $x_1,x_2,...$
     For example $m_{(2,1,1)}(x_1,x_2,...)=\sum_{i\neq j\neq k} x_i^2 x_j x_k$. When $\lambda=(1^r)=(\underbrace{\, 1,1,...,1\,}_r)$, then 
     $m_{(1^r)}$ becomes the {\bf elementary symmetric function} of degree $r$, which is also expressed as 
     \begin{align}\label{elementary}
     e_r(x_1,x_2,...)=\sum_{i_1<i_2<\cdots <i_r} x_{i_1}x_{i_2} \cdots x_{i_r}.
     \end{align}
    When $\lambda =(r)$, that is, the partition has only one part of length $r$, then $m_{(r)}$ becomes the {\bf power symmetric function}
    \begin{align}\label{power}
    p_r(x_1, x_2,...)=\sum_i x_i^r.
    \end{align}
    Finally, the {\bf complete symmetric functions} of degree $r$ are the sum of all monomials of degree $r$, that is
    \begin{align}\label{complete}
    h_r(x_1,x_2,...)=\sum_{|\lambda|=r} m_\lambda(x_1,x_2,...) = \sum_{i_1\leq i_2\leq \cdots \leq i_r} x_{i_1}x_{i_2} \cdots x_{i_r},
    \end{align}
    where $|\lambda|=\lambda_1+\lambda_2+\cdots$.
    
    The notions of elementary, power and complete symmetric functions can be extended so that these function are indexed by partitions.
    More precisely, the power symmetric function indexed by a partition $\lambda=(\lambda_1,\lambda_2,...)$ 
    is defined as $p_\lambda:=p_{\lambda_1}p_{\lambda_2}\cdots$ and, similarly, $e_\lambda:=e_{\lambda_1}e_{\lambda_2}\cdots$
    and  $h_\lambda:=h_{\lambda_1}h_{\lambda_2}\cdots$
    
    We close this basics discussion on symmetric functions by an expansion of the complete symmetric functions in terms of the power 
    symmetric functions, which brings in the numerical function $z_\lambda$. The latter plays an important role in the definition of Schur 
    and Macdonald (as well as other) functions via orthogonalisation. The relation, we allude to, can be proved via generating series, 
    see \cite{M88}, relation (1.8), and is
    \begin{align}\label{zlambda}
    h_n= \sum_{|\lambda|=n} z_\lambda^{-1} p_\lambda, \qquad \text{with} \qquad z_\lambda=\prod_{r\geq 1} r^{m_r} \cdot m_r!
    \end{align} 
    where the numbers $m_r$, $r\geq 1$, denote the number of parts in the partition $\lambda$ which have length equal to $r$.
     \subsection{Schur functions}\label{Schur-intro} 
     We will mostly restrict attention when Schur functions involve only a finite number of indeterminates,
     in which case we talk about {\it Schur polynomials}.
     We present three different ways to define Schur polynomials, each one of which has its own
     benefits and uses in this text. 
     
     The first one, which was the original definition given by Schur, is via a determinant. 
     In particular, for a partition $\lambda=(\lambda_1,...,\lambda_n)$, with $n$ parts, 
     \begin{align}\label{def:Schur1}
     s_\lambda(x_1,...,x_n):=\frac{\det \big(\,x_i^{\lambda_j+n-j}\,\big)_{1\leq i,j \leq n} }{\det \big(\, x_i^{n-j} \,\big)_{1\leq i,j \leq n}}.
     \end{align}
     This determinantal expression is crucial in expressing the law of observables of integrable models in terms of determinants
     and enabling the asymptotic analysis as will be seen in Section \ref{sec:Fred}. 
     
     The second definition is combinatorial. It will arise in Section \ref{geomLPP}, via the combinatorial methods that we will develop
     in Section \ref{sec:RSK}, as expressing the laws of observables of integrable models. In the combinatorial definition, Schur functions
     appear as generating functions of Young tableaux. In particular,
     \begin{align}\label{Schur-gen-tableau}
     s_\lambda(x_1,...,x_n)=\sum_{\sfT \colon sh(\sfT)=\lambda} x_1^{\sharp 1's} x_2^{\sharp 2's}\cdots x_n^{\sharp n's},
     \end{align}
     where the sum is over all Young tableaux with shape $\lambda$ and $\sharp 1's$ denotes the number of boxes in the tableau
     filled in with $1$, $\sharp 2's$ denotes the number of boxes in the tableau
     filled in with $2$ and so on.
     
     The third definition is via an orthogonalisation procedure, which is of Gram-Schmidt type. 
     This approach was generalised by Macdonald in \cite{M88}, as we will
     see below, in order to define the Macdonald polynomials and it does not restrict to polynomials. 
     To define the Schur polynomials in this way, we need, first, to introduce an inner product. This is defined  
     via its evaluation on the power symmetric functions (remember that these form a basis) as
     \begin{align}\label{Schur-inner}
     \langle p_\lambda, p_\mu \rangle = z_\lambda \delta_{\lambda, \mu},
     \end{align}
     where $z_\lambda$ is defined in \eqref{zlambda} and 
     $\delta_{\lambda, \mu}$ is the Kronecker delta and it is equal to $1$ if $\lambda=\mu$ and zero otherwise.
     Given this inner product, it can be shown, see \cite{M98}, that Schur functions are uniquely determined by
     their expansion in terms of the monomial symmetric functions:
     \begin{align*}
     \hskip -3cm {\rm (A)} \hskip 2cm s_\lambda = m_\lambda + \sum_{\mu<\lambda} K_{\lambda \, \mu} m_\mu,
     \end{align*}
     and their orthogonality with respect to the inner product as
     \begin{align*}
     \hskip -2cm {\rm (B)} \hskip 2cm \langle s_\lambda, s_\mu \rangle =0, \qquad \text{if} \qquad \lambda\neq \mu.
     \end{align*}
In (A) the sum is over all partitions $\mu$ such that $\mu<\lambda$, where the partial ordering $<$ on partitions is defined 
as $\mu<\lambda$, if $\mu_1+\cdots+\mu_i<\lambda_1+\cdots+\lambda_i$ for all $i\geq 1$.
     
     A fundamental identity involving Schur functions, with representation theoretic significance 
     (we refer to \cite{Bum04} for details on this aspect), is the so-called {\bf Cauchy identity}, 
     which reads as
     \begin{align*}
     \sum_{\lambda} s_\lambda(x_1,x_2,...) \,s_\lambda(y_1,y_2,...) = \prod_{i,j} \frac{1}{1-x_iy_j}.
     \end{align*}
     This identity, as well as identities of this type for other special functions, will play a very important role for our purposes as
     it contains the core of the probability measures, which describe the laws of solvable models in the KPZ class. The Cauchy
     identity for Schur polynomials will emerge in Section \ref{geomLPP}, see \eqref{GTschur} through $\rsk$, in relation to the law of the solvable
     last passage percolation. The role of more general Cauchy identities in building integrable stochastic dynamics
     will be explored in Section \ref{sec:Cauchy}. 
     
     Here we would like to emphasise that the Cauchy identity is actually equivalent to the orthogonality relation {\rm (B)}. We will
     show this in the next subsection in the more general context of Macdonald polynomials. 
     
     \subsection{Macdonald polynomials}\label{Macdonald:intro} 
     Macdonald polynomials were defined by Macdonald in \cite{M88}, see also \cite{M98},
     as a family of symmetric polynomials, depending on two parameters $q,t$ in a way that they degenerate, in certain limits of $q,t$
     to several other families of symmetric polynomials that includes Schur, Hall-Littlewood, Jack, zonal etc. In particular, Macdonald
     proved the following theorem:
     \begin{theorem}[\cite{M88}]\label{thm:macdonald}
     Consider the inner product $\langle \cdot, \cdot \rangle_{(q,t)}$ defined via its values on power symmetric polynomials as
     \begin{align}\label{macdonald-inner}
     \langle p_\lambda, p_\mu \rangle = z_\lambda(q,t) \delta_{\lambda, \mu},\qquad \text{with} \qquad
     z_\lambda(q,t) := z_\lambda \prod_{i=1}^{\ell(\lambda)}\frac{1-q^{\lambda_i}}{1-t^{\lambda_i}},
     \end{align}
    and $z_\lambda$ as in \eqref{zlambda}.
     Then, for each partition $\lambda$, there exists a unique symmetric function $P_\lambda(x_1,x_2,...)=P_\lambda(x_1,x_2,...; q,t)$ 
     such that 
       \begin{align*}
      {\rm (A)}& \hskip 2cm P_\lambda = m_\lambda + \sum_{\mu<\lambda} u_{\lambda \, \mu}(q,t) \, m_\mu \qquad \text{and}\\
      {\rm (B)}& \hskip 2cm \langle P_\lambda, P_\mu \rangle_{(q,t)} = 0, \qquad \text{if} \qquad \lambda\neq \mu.
     \end{align*}
     In $\rm (A)$ the coefficients $u_{\lambda \, \mu}(q,t)$ are rational functions in $q,t$.
     \end{theorem}
     We notice that when $q=t$, then $z_\lambda(q,q)=z_\lambda$ and, thus, the inner product in \eqref{macdonald-inner} becomes the
     same as in the Schur case \eqref{Schur-inner}. So in this case, by the uniqueness, the Macdoland polynomials are identical to the Schur.
     
     In \cite{M88}, Macdonald proceeds to establish the existence and uniqueness of the polynomials in the above theorem 
     as the eigenfunctions of 
     an operator $D$, which is (i) self-adjoint with respect to the inner product $\langle\cdot, \cdot \rangle_{(q,t)}$, 
     (ii) triangular relative to the basis $(m_\lambda)$, i.e. $Dm_\lambda=\sum_{\mu\leq \lambda} c_{\lambda \mu} m_\mu$, for some
     coefficients $c_{\lambda \mu}$ with $c_{\lambda \lambda}\neq c_{\mu\mu}$ whenever $\lambda\neq \mu$. Notice that, by the 
     requirement of triangularity of the operator, $c_{\lambda\lambda}$ are its eigenvalues and the self-adjointness together with the 
     fact that all eigenvalues are different implies the orthogonality of its eigenfunctions with respect the $\langle\cdot, \cdot \rangle_{(q,t)}$.
     The uniqueness of the Macdonald polynomials is also a consequence of the above properties of the operator $D$, which was 
     provided in \cite{M88} as
     \begin{align*}
     D=\sum_{i=1}^n \Big( \prod_{i\neq j} \frac{tx_i-x_j}{x_i-x_j}\Big) T_{q,x_i}
     \end{align*}  
  with the operator $T_{q,x_i}$, for $i=1,...,n$, defined via its action on a function $f(x_1,...,x_n)$ as
  \begin{align*}
  (T_{q,x_i} f) (x_1,...,x_n)= f(x_1,...,qx_i,...,x_n).
  \end{align*}     
  The fact that these operators possess the required properties in not obvious and requires a check, see \cite{M88}, Section 2.
  The eigenvalue $c_{\lambda\lambda}$ is also explicit and given by
   $q^{\lambda_1} t^{n-1}+q^{\lambda_2} t^{n-2}+\cdots+q^{\lambda_n}$, see \cite{M98}, VI (4.15).
   
   We will further discuss Macdonald polynomials and some useful, explicit formulas in Section \ref{sec:Cauchy} in relation to the construction 
   of integrable stochastic systems. We want to close, here, with the Cauchy identity for Macdonald polynomials and showing 
   how this is equivalent to the orthogonal property (B) in Theorem \ref{thm:macdonald}. To this end, we introduce the dual Macdonald polynomials
   \begin{align*}
   Q_\lambda(\cdot ; q,t) :=\frac{P_\lambda(\cdot; q,t)}{\langle P_\lambda(\cdot; q,t), P_\lambda(\cdot; q,t) \rangle_{(q,t)}},
   \end{align*}
   so that by condition (B) of Theorem \ref{thm:macdonald} it holds that $P_\lambda$ and $Q_\lambda$ are bi-orthonormal, that is, 
   $\langle P_\lambda(\cdot; q,t), Q_\lambda(\cdot; q,t) \rangle_{(q,t)} =\delta_{\lambda \mu}$. For indeterminates
   $x=(x_1,...,x_n), y=(y_1,...,y_n)$, the Cauchy identity may be written as
   \begin{align*}
   \sum_{\lambda} P_\lambda(x;q,t) Q_\lambda(y;q,t) &= H(x,y) \qquad \text{with} \\
    H(x,y):= H(x,y;q,t) &:= \prod_{i,j} \frac{(tx_iy_j;q)_\infty}{(x_iy_j;q)_\infty},
   \end{align*}
   where in the last expression $(a;q)_\infty:=\prod_{i=0}^\infty (1-aq^i)$ is the {\bf $q$-Polchammer symbol}.
   The equivalence of the Cauchy identity to the bi-orthogonality of the Macdonald polynomials is established in the following theorem
   \begin{theorem}[\cite{M88}, (2.6)] The following are equivalent 
   \begin{align}
   {\rm (i)} &\hskip 1cm \langle P_\lambda(\cdot; q,t), Q_\mu(\cdot; q,t) \rangle_{(q,t)} =\delta_{\lambda \mu}, \\ 
   {\rm (ii)} & \hskip 1cm \sum_{\lambda} P_\lambda(x;q,t) Q_\lambda(y;q,t) = H(x,y) .
   \end{align}
   \end{theorem}
   \begin{proof}
   Recall the power symmetric polynomials $p_\lambda$ and the numerical function $z_\lambda(q,t)$ from \eqref{macdonald-inner}. We will
   show at the end of this proof that
   \begin{align}\label{aux-equiv}
   H(x,y) = \sum_\lambda z_\lambda(q,t)^{-1} p_\lambda(x) p_\lambda(y).
   \end{align}
   For the moment let us assume this and let us set $p^*_\lambda:=z_\lambda(q,t) p_\lambda$, so that by \eqref{macdonald-inner} we have that 
   $\langle p^*_\lambda, p_\lambda \rangle_{(q,t)}= \delta_{\lambda \mu}$. Assume that $P_\lambda$ and $Q_\lambda$ admit the
   following expansions in terms of power symmetric polynomials $P_\lambda= \sum_\rho a_{\lambda \rho } p^*_\rho$ and 
   $Q_\lambda= \sum_\rho b_{\lambda \rho } p_\rho$ with certain coefficients, which we gather in a matrix form,
   indexed by partition, $A=(a_{\lambda \rho})$ and $B=(b_{\lambda \rho})$. Such expansions will always exists since the power symmetric 
   polynomials form a basis.
   
   It is straightforward, given the orthonormality of $p^*_\lambda, p_\lambda$, from \eqref{macdonald-inner}, that
   \begin{align*}
   \langle P_\lambda(\cdot; q,t), Q_\lambda(\cdot; q,t) \rangle_{(q,t)}= \sum _\rho a_{\lambda \rho} b_{\lambda \rho} = AB^*,
   \end{align*}
   and, thus, (i) is equivalent to having $AB^*=I$, the identity matrix. On the other hand,  using \eqref{aux-equiv}, we have that 
   (ii) is equivalent to 
   \begin{align*}
   \sum_\lambda P_\lambda (x) Q_\lambda(y) = \sum_\rho p_\rho^*(x) p_\rho(y),
   \end{align*}
   which, by replacing in the left-hand side $P_\lambda, Q_\lambda$ with their expansion in terms power symmetric polynomials,
   is equivalent to 
   \begin{align*}
   \sum_{\lambda} a_{\lambda \rho} b_{\lambda \sigma} = \delta_{\rho \sigma}.
   \end{align*}
   In matrix form, this may be written as $B^*A=I$, which is clearly equivalent to $A^*B=I$, which, as we saw, is equivalent to (i).
   
   Thus, it only remains to check  \eqref{aux-equiv}. To this end, we first compute
   \begin{align*}
   \log H(x,y) 
   &= \sum_{i,j} \sum_{r=0}^\infty \big( \log (1-x_iy_j q^r)^{-1}  -  \log (1-tx_iy_j q^r)^{-1} \big) \\
   &= \sum_{i,j} \sum_{r=0}^\infty \sum_{n\geq 1} \frac{1}{n} (x_iy_j q^r)^n (1-t^n) \\
   &= \sum_{n=1}^\infty \frac{1}{n} \frac{1-t^n}{1-q^n} p_n(x) p_n(y).
   \end{align*}
   Then,
   \begin{align*}
   H(x,y) = \prod_{n=1}^\infty \exp \Big(  \frac{1}{n} \frac{1-t^n}{1-q^n} p_n(x) p_n(y) \Big),
   \end{align*}
   and by Taylor expanding the exponential,
    \begin{align*}
   H(x,y) 
   &= \prod_{n=1}^\infty \sum_{m_n=0}^\infty \frac{1}{m_n!}  \Big(  \frac{1}{n} \frac{1-t^n}{1-q^n} p_n(x) p_n(y) \Big)^{m_n}\\
   &= \prod_{n=1}^\infty \sum_{m_n=0}^\infty \frac{1}{n^{m_n} m_n! }  \Big(  \frac{1-t^n}{1-q^n}  \Big)^{m_n} p_n(x)^{m_n} p_n(y)^{m_n}.
    \end{align*} 
    Finally, we interchange the sums and the products and rearrange the summands in terms of partitions $\lambda$, viewing the 
    $m_n$'s as the number of parts of $\lambda$ with length equal to $n$ and keeping in mind that, by definition, 
    $z_\lambda^{-1}=\prod_{n\geq 1}\frac{1}{n^{m_n} m_n!}$ and 
    $p_\lambda(x)=p_{\lambda_1}(x) p_{\lambda_2}(x)\cdots=\prod_{n\geq 1} p_n(x)^{m_n}$.
   The last equality comes from grouping together the terms $p_{\lambda_i}(x)$ with $\lambda_i=n$, since the number of these
   terms is (by definition) $m_n$. 
   \end{proof}
   \subsection{Whittaker functions}\label{Whittaker-intro}
   Whittaker functions comprise another family of special functions which appear in many different fields of mathematics such as
   number theory and automorphic forms \cite{Gold06, Bum84},
   mirror symmetry \cite{Giv97, Rie12, L13} and integrable systems \cite{K79, KL01, GLO12}. 
   Their more recent appearance in probability \cite{BO11, O12, COSZ14, OSZ14} is a further manifestation of their ubiquity.
   In this text, we will encounter Whittaker functions in Section \ref{sec:loggama-Whit} where they will emerge via the geometric $\rsk$ 
   as expressing the Laplace transform of the partition function of a solvable polymer model.

  Whittaker functions, like Schur functions, can be associated to various classical Lie groups (see, for example, \cite{GLO12})
  but in order to keep the discussion simple, we will restrict to the case of Whittaker functions on
  the general linear group $GL_n(\R)$. These will be denoted by $\psi^{\mathfrak{g}\mathfrak{l}_n}_\alpha(x)$ for 
  $\alpha=(\alpha_1,...,\alpha_n)\in\C^n$ and $x=(x_1,...,x_n)\in \R_+^n$.
  In this case, probably the easiest way to define the Whittaker function $\psi^{\mathfrak{g}\mathfrak{l}_n}_\alpha(x)$ 
  is as eigenfunction
  of the differential operator (known as the {\it quantum Toda Hamiltonian})
  \begin{align*}
  -\Delta +\sum_{i=1}^{n-1} e^{x_{i}-x_{i+1}} ,
  \end{align*}
  with eigenvalue $-|\alpha|^2=-\sum_{i=1}^n\alpha_i^2$. We note that $x_{i+1}-x_i=\langle \mathfrak{a}_i,x \rangle$,
  where here $\langle \cdot, \cdot \rangle$ is the usual inner product on $\R^n$ and $\mathfrak{a}_i$ are the positive 
  roots of the Lie group $\mathfrak{g}\mathfrak{l}_n$. The definition of the Whittaker functions for other classical groups
  amounts to replacing $ \mathfrak{a}_i$ with the positive roots of the group in question. We mention that in Section 
  \ref{sec:loggama-Whit} we will be mostly
  working with the Whittaker functions in logarithmic variables defined through
  $\psi^{\mathfrak{g}\mathfrak{l}_n}_\alpha(x_1,...,x_n) =: \Psi^{\mathfrak{g}\mathfrak{l}_n}_\alpha(e^{x_1},...,e^{x_n})$.
  
  In Section \ref{sec:loggama-Whit} we will discuss and use the orthogonality properties of Whittaker functions (Theorem \ref{thm:Plancerel}) 
  and their Cauchy identity (also known as Bump-Stade identity \eqref{bumpstade}, see also Section  \ref{sec:Cauchy}).
  For the moment we would only like to mention that $GL_n(\R)$-Whittaker functions can be obtained as a limit of Macdonald 
  polynomials when $t=0$ and that in another limit they degenerate to (a continuous version of) Schur functions.
   Regarding the first, we have that for the following choice of parameters:
  \begin{align}\label{Whitt-scale}
&  t=0, \quad q=e^{-\epsilon},  \quad m(\epsilon)= -\lfloor \frac{1}{\epsilon}\rfloor, \quad \cA(\epsilon)= -\frac{\pi^2}{6}\frac{1}{\epsilon} 
  - \log\frac{\epsilon}{2\pi} \quad  \text{and} \notag\\
  &\\
& \text{for $k=1,...,n$} \colon \quad z_k=e^{ \epsilon \alpha_k}, \quad \lambda_k=(n-2k) m(\epsilon) +\frac{1}{\epsilon} x_k, \notag
  \end{align}
  it holds that $\epsilon^{\tfrac{n(n-1)}{2}} e^{\tfrac{(n-1)(n+2)}{2}\cA(\epsilon)} P_\lambda(z_1,...,z_n; q, 0)$ converges, as $\epsilon\to 0$ to 
  $\psi^{\mathfrak{g}\mathfrak{l}_n}_\alpha(x)$. We refer to \cite{BC14}, Theorem 4.1.7 for details. 
  The limit to Schur functions takes place when 
  $\epsilon\to 0$ in $\epsilon^{n(n+1)/2} \Psi^{\mathfrak{g}\mathfrak{l}_n}_{\epsilon \alpha}( e^{x_1/\epsilon}, ..., e^{x_n/\epsilon})$; 
  we refer for details to \cite{OSZ14}, Section 8, or \cite{BZ19}, Section 4.
  
  
   \section{The Robinson-Schensted-Knuth correspondence}\label{sec:RSK}
 In this section we introduce the Robinson-Schensted-Knuth correspondence. We do this in a gradual manner, starting from the Robinson-Schensted 
  correspondence and the combinatorial {\it row insertion algorithm} and then continue to its Robinson-Schensted-Knuth correspondence, which 
  we also express in terms of piecewise linear transformations. The latter formulation plays an important role in the geometric lifting, which we 
  present in Section \ref{sec:grsk}, and is also instrumental in connecting to weights of lattice paths. Before starting we urge the reader to recall the
  notions of partition, Young diagram and Young tableaux, which were presented in the previous section.
   \vskip 2mm
 \subsection{ Robinson-Schensted correspondence.}\label{sec:RS}  
  Young tableaux are of significance in representation theory in part because they classify the irreducible representations of the symmetric group. We refer to \cite{F97} for the use of Young tableaux in combinatorics, representation theory and geometry. Of importance to us is that Young tableaux encode information on quantities like the {\it longest increasing subsequence} in a permutation.
   This is done via the Robinson-Schensted ($\rs$) algorithm, which  
   gives a one-to-one correspondence between a permutation $\sigma\in S_N$ 
   and a pair of standard Young tableaux, which we will denote by $(P,Q)$. 
   The algorithm is as follows: Consider a permutation 
   \begin{equation*}
   \sigma=\left({\begin{array}{cccc}
   1&2&\cdots& N\\
   x_1&x_2&\cdots&x_N
   \end{array}
   }\right),
   \end{equation*}
   where we denote $x_i:=\sigma(i)$. Then,
   \begin{itemize}
   \item
   Starting from a pair of empty tableaux $(P_0,Q_0)=(\emptyset,\emptyset)$, assume that we have inserted the first $i$ 
   {\bf biletters} $j \choose x_j$, for $1\leq j\leq i \leq N$, of the permutation
   $\sigma\in S_N$ and we have obtained a pair of Young tableaux $(P_i,Q_i)$.
   \item Next, we {\bf (row) insert} the biletter $i+1\choose x_{i+1}$ as follows: If the number $x_{i+1}$ is larger or equal \footnote{in the case of a permutation the
   ``or equal'' condition is void but it becomes relevant in the Robsinson-Schensted-Knuth generalisation.}
    than all the numbers of the first row of $P_i$, then a box is appended at the end of the first row of $P_{i}$ and its content 
    is set to be $x_{i+1}$.
     This is then the tableau 
    $P_{i+1}$. Also a box is appended at the end of the first row of $Q_i$ and its content is set to be $i+1$, giving the tableau $Q_{i+1}$.
    If, on the other hand, there is a box in the first row of $P_i$ with content strictly larger than $x_{i+1}$, then the content of the first such box becomes $x_{i+1}$
    and the replaced content, call it $b$, drops down and is {\it row inserted} in the second row of $P_i$ following the same rules and creating (possibly) a
    cascade of dropdowns (called {\bf bumps}). Eventually a box will be appended at the end of a row in $P_i$ or below its last row, in which case it creates a new row, and the content of this box will be the last bumped letter. At the same, corresponding, location a
    box will be added at $Q_i$ and its content will be set to be $i+1$.
    \item We repeat the above steps until all biletters have been row inserted.
   \end{itemize}
   
   Let us see how this algorithm works via an example. Consider the permutation
   \begin{align*}
   \left(\begin{array}{ccccccc}
   1&2&3&4&5&6&7\\
   3&5&1&6&2&4&7
   \end{array}
   \right)
   \end{align*} 
   The sequence is as follows:
    \begin{align*}
   &(\emptyset, \emptyset) \xrightarrow[]{3} \young(3) \quad \young(1) \xrightarrow[]{5} \young(35) \quad \young(12) \xrightarrow[]{1} \young(15,3) \quad \young(12,3)
   \xrightarrow[]{6} \young(156,3) \quad \young(124,3) \xrightarrow[]{2}\\
   &  \xrightarrow[]{2} \young(126,35) \quad \young(124,35)  \xrightarrow[]{4} \young(124,356) \quad \young(124,356) 
    \xrightarrow[]{7} \young(1247,356) \quad \young(1247,356) .
   \end{align*}
  In words, we have that we start by row inserting $`3$' and creating a box with content `$3$', identified with tableau $P_1$, and a box with content `$1$' constructing tableau $Q_1$.
   Then `$5$' is row inserted in $P_1$ and since it is larger than `$3$' it bypasses the box with content `$3$' and sits in a new box in the right of `$3$', creating
   $P_2$. A box with content `$2$' is also created in the right of the box with content `$1$' in $Q_1$ creating tableau $Q_2$. Then `$1$' is row inserted to $P_2$ and being smaller
    than `$3$' it bumps `$3$' and sits in the first box of $P_2$. `$3$' is then row inserted in the second row and since this is empty, it creates a new box whose content becomes `$3$'. At the same time a new box in the second row of $Q_2$ is created whose content is `$3$', giving $Q_3$. The procedure continues
    in this way. 
   \vskip 2mm
   There are a few observations to be made from this example. 
   \begin{itemize}
\item   Tableaux $P$ and $Q$ have the same shape, i.e. the lengths of the successive rows in each tableau are equal.
  \end{itemize}
  This is a general fact. The tableaux that $\rs$ produces have the same shape. This can be easily seen
    as at any stage of the algorithm a box is created at the same location in both the $P$ and $Q$ tableau.
     \begin{itemize}
\item   Tableaux $P$ and $Q$ are actually equal. 
  \end{itemize}  
     This is not a general fact but a consequence of the fact that the permutation matrix associated to the above permutation, is symmetric,
     or that $\sigma=\sigma^{-1}$. In general, as we will state below, if $(P,Q)$ is the output of a permutation $\sigma$, then the output of permutation $\sigma^{-1}$ is $(Q,P)$. Thus, if $\sigma=\sigma^{-1}$, then $P=Q$.
       \begin{itemize}
\item 
    A third observation that we make is that the length of the first row of either 
    output tableau $P$ and $Q$ (which in this case is $4$) equals the length of the longest increasing
    subsequence in the permutation $(3,5,1,6,2,4,7)$, which, for example (as there are more than one such), is
     the sequence $(3,5,6,7)$. Moreover, the length of the second longest increasing subsequence
    $(1,2,4)$ equals the length of the second row of the output tableaux.
    \end{itemize}
     This is also not a coincidence and goes by the name of Greene's theorem \cite{G74} 
     (see Theorem \ref{thm:greene} below),
     an extension of Schensted's theorem \cite{S61} (see Theorem \ref{schensted} below):     
    \begin{theorem}[Schensted]\label{schensted} The $\rs$ correspondence is a bijection between permutations and pairs of standard Young tableaux $(P,Q)$ of the same shape.
    If $\sigma\in S_N$ and $(P,Q)= \rs(\sigma)$ is the image of $\sigma$ under RS-correspondence, then 
    $(Q,P)=\rs(\sigma^{-1})$, where $\sigma^{-1}$ is the inverse of permutation $\sigma$. In particular, if $\sigma=\sigma^{-1}$, then $P=Q$.
    \end{theorem}
    \begin{theorem}[Greene]\label{thm:greene}
    Let $\sigma\in S_N$ and $(P,Q)=\rs(\sigma)$. Then, the 
     length $\lambda_1$ of the first row of the output tableaux $P$ or $Q$ equals the length of the longest increasing subsequence in $\sigma$. Moreover,
    the sum $\lambda_1+\lambda_2+\cdots+\lambda_r$ of the lengths of the first $r$ rows 
    equals the maximum possible length of dijoint unions of $r$ increasing subsequences 
    in $\sigma$. 
    \end{theorem}
    We do not provide proofs of these theorems here as we will derive them later on as a corollary to a matrix formulation of the algorithm, which also covers the Robinson-Schensted-Knuth correspondence and its {\it geometric lifting}, 
    see (the end of) Section \ref{RSK_matrix}.
    \vskip 2mm
    \subsection{ Robinson-Schensted-Knuth correspondence.}\label{subsec:RSK}
   Bearing in mind that permutations are identified with matrices whose entries are either $0$ or $1$ and 
   no two $1$'s are in the same row or column (permutation matrices),
    one can see the $\rs$ correspondence as a bijection between permutation
   matrices and pairs of standard 
   Young tableaux. Knuth's generalisation constituted in extending $\rs$ as a bijection between matrices with nonnegative integer entries
   and pairs of {\it seminstandard} Young tableaux. We can think of a matrix $W=(w^i_j)_{\substack{1\leq i\leq n\\1\leq j\leq N}}$, 
   where $i$ indicates rows and $j$ columns, as a sequence of $n$ words
   \begin{align}\label{word}
   w^i:=1^{w^i_1}\cdots N^{w^{i}_N} := \underbrace{1\cdots1}_{w^i_1} \,\,\underbrace{2\cdots2}_{w^i_2} \,\cdots \,\underbrace{N\cdots N}_{w^i_N} 
   \end{align}
   with letters $1,2,...,N$, such that $w^i_j$ symbolises the number of letters $j$ in word $i$. Knuth's extension of $\rs$ correspondence,
   named Robinson-Schensted-Knuth ($\rsk$), consists of inserting, via the $\rs$ row-insertion, the letters of words 
   $w^1,w^2,...,w^n$ (in this order) with the letters
   of each word $w^i $ as in \eqref{word} being read from left to right.
   
   \vskip 2mm
  {\bf  Berenstein and Kirillov's $(max,+)$ formulation}.
   Berenstein and Kirillov \cite{BK95} adopted a different point of view of the $\rsk$ correspondence, which was to encode the combinatorial 
   transformations via piecewise linear transformations. This is a particularly useful approach for applications in the
   probabilistic models we are interested in and we will describe it now. The exposition here follows mainly the presentation in Noumi-Yamada \cite{NY04}. 
   \vskip 2mm
  Let us describe how the $P$ tableau is constructed. This is done by successive row insertions of words $w^1,w^2,...$ as described in the $\rs$
  correspondence. 
  Let us start by inserting $w^1$, that is the sequence of letters
  \begin{align*}
   \underbrace{1\cdots1}_{w^1_1} \,\,\underbrace{2\cdots2}_{w^1_2} \,\cdots \,\underbrace{N\cdots N}_{w^1_N}.
  \end{align*}
  Since the letters in $w^1$ are ordered from smaller to larger, the insertion of $w^1$ will produce the one-row tableau 
  \begin{align*}
  P_1=   
{\begin{tikzpicture}[baseline=1.1ex, scale=.6]
\draw (0.1,0.1) -- (10.5,0.1)--(10.5,0.9)--(0.1,0.9)--(0.1,0.1) ;
\draw[<->](0.2,1.2)--(2.3,1.2); \node at (1.3,1.7) {$w^1_1$};
\node at (0.4,0.5) {$1$}; \draw[dotted, thick] (0.7,0.5)--(2,0.5); \node at (2.2,0.5) {$1$}; \draw (2.5,0.1)--(2.5,0.9);
\draw[<->](2.6,1.2)--(4.7,1.2); \node at (3.7,1.7) {$w^1_2$};
\node at (2.8,0.5) {$2$}; \draw[dotted, thick] (3.1,0.5)--(4.3,0.5); \node at (4.5,0.5) {$2$}; \draw (4.8,0.1)--(4.8,0.9);
\draw[<->](7.8,1.2)--(10.2,1.2); \node at (9,1.7) {$w^1_N$};
\draw[dotted, thick] (5.0,0.5)--(7.3,0.5); \draw (7.5,0.1)--(7.5,0.9);  \node at (7.9,0.5) {$N$}; \draw[dotted, thick] (8.3,0.5)--(9.7,0.5); \node at (10.1,0.5) {$N$};
\end{tikzpicture}}
 \end{align*}
 \vskip 2mm
We note that we
can identify the row of a tableau with words by reading the letters from left to right. In the case of $P_1$, the tableau can be identified with the single word (recall the notation introduced in \eqref{word})
\begin{align*}
p^1:=1^{p^1_1}2^{p^1_2}\cdots N^{p^1_N}=1^{w^1_1} 2^{w^1_2}\cdots N^{w^1_N} = w^1.
\end{align*}
Next, we insert word $w^2$ into $P_1$ and this insertion will produce a new tableau $\tilde P_1$. We denote this schematically as
\begin{equation*}
\begin{array}{ccc}
& w^2 & \\
P_1 & \cross & \tilde P_1 .\\
&&
\end{array}
\end{equation*}
This insertion will change the first row $p^1$ of $P_1$ by (possibly) bumping some letters out of it and replacing them with letters from $w^2$. The bumped letters
will form a word, which will then be inserted in the second row of the tableau, which in the case of $P_1$ is $\emptyset$. We denote this schematically as
\begin{align*}
\begin{array}{ccc}
& w^2 & \\
p^1 & \cross & \tilde p^{\,\,1} ,\\
& v^2&
\end{array}
\end{align*}
with $p^{\,1}$ denoting the first row of $P_1$, $ \tilde p^{\,\,1}$ the first row of $\tilde P_1$ and $v^2$ the word that will form 
from the bumped down letters from $P_1$ after the insertion of $w^2$.

This picture is a building block of $\rsk$, since the row insertion of a word $w$ in a tableau $P$ consisting of
rows $p^1,p^2,...,p^n$, can be decomposed as
 \begin{align}\label{rs-iterate}
\begin{array}{ccc}
& w=:v^1 & \\
p^1 & \cross & \tilde p^{\,1} \\
& v^2&\\
p^2 & \cross & \tilde p{\, ^2} \\
&\vdots&\\
& v^n&\\
p^n & \cross & \tilde p^{\,n} .\\
&v^{n+1}&
\end{array}
\end{align}
This picture means that the letters that will drop down from $p^1$, after the insertion of $v^1=w$, will form a word $v^2$ which
will be inserted in $p^2$, forming a new row $\tilde p^{\,\,2}$, and will bump down letters which will then form a new word $v^3$ to be inserted into $p^3$ and so on.
\vskip 2mm
An important remark is that row $p^i$, which is constructed via the $\rs$ algorithm,  will only include letters with value larger
or equal to $i$. This is an easy consequence of the algorithm. For example, $p^2$ will not include $1$'s as 
`$1$' is the smallest possible letter and so when a `$1$' is inserted in the first row it will stay there, bumping out 
$2$'s, $\,3$'s,...
 \vskip 2mm
It will be important to have an explicit, algebraic expression of the transformation $\begin{array}{ccc}
& \ba& \\
\bx & \cross & \tilde\bx \\
& \bb&
\end{array}
$,
where variables $\bx:=i^{x_i}(i+1)^{x_{i+1}}\cdots N^{x_N}$ and $\ba:=i^{a_i}(i+1)^{a_{i+1}}\cdots N^{a_N}$, for $ i\geq 1$,
 are considered as input variables ($\bx$ corresponds to the generic case of an $i$-th row in a tableau) and
$\tilde\bx:=i^{\tilde x_i} (i+1)^{\tilde x_{i+1}}\cdots N^{\tilde x_N}$ and 
$\bb:=(i+1)^{b_{i+1}}{(i+2)}^{b_{i+2}}\cdots N^{b_N}$ are output variables, with $\tilde\bx$ being the
new row after the insertion of $\ba$ and $\bb$ being the bumped down letters (as explained in the previous paragraphs,
 $\bb$ will only have letters strictly larger than $i$). In particular, we want to express 
$\tilde x_i,\tilde x_{i+1},...$ and $b_{i+1},b_{i+2},...$ as piecewise linear transformations of $x_i,x_{i+1},...$ and $a_i,a_{i+1},...$
To do so, it will be more convenient to introduce cumulative variables
\begin{align*}
\xi_j:=x_i+\cdots + x_j, \quad \text{and}\quad \tilde\xi_j:=\tilde x_i+\cdots+\tilde x_j, \qquad \text{for} \,\, j\geq i.
\end{align*}
We derive the piecewise linear transformations as follows:
When inserted into $\bx$, the $a_i$ letters $i$ will bypass the already existing letters $i$ in $\bx$ and will be appended after the last $i$ in $\bx$. Thus, the new total number of letters $i$ will be
\begin{align*}
\tilde \xi_i=\xi_i+a_i.
\end{align*}
When this insertion is completed a number of $(i+1)$'s will be bumped off $\bx$. The number of these will equal
\begin{align}\label{bump1}
b_{i+1}= \min\big(\tilde \xi_i-\xi_i, \xi_{i+1}-\xi_i\big) =  \min\big(\tilde \xi_i, \xi_{i+1}\big)-\xi_i.
\end{align}
This is to be understood as follows: either $a_i$ is smaller than the number of $(i+1)$'s in $\bx$,  which equals 
$x_{i+1}=\xi_{i+1}-\xi_i$,
  and so there will only be $a_i=\tilde \xi_i-\xi_i$ number
of $(i+1)$'s bumped down, or $a_i$ is larger than or equal to $x_{i+1}$, in which case all of the $(i+1)$'s in $\bx$, the number of which equals $\xi_{i+1}-\xi_i$, will be bumped down.

We also record an alternative formula for the number of bumped $(i+1)$'s, which is 
\begin{align}\label{bump2}
b_{i+1}=a_{i+1}+x_{i+1} - \tilde x_{i+1} = a_{i+1} + (\xi_{i+1}-\xi_i) - (\tilde \xi_{i+1}-\tilde \xi_{i}),
\end{align}
where the first equality
 is to be understood as that the number of $(i+1)$'s, which will be bumped, equals the number of  $(i+1)$'s that existed in $\bx$ 
 (denoted by $x_{i+1}$) plus
the number of $(i+1)$'s that we inserted (denoted by $a_{i+1}$) 
minus the number of $(i+1)$'s that we finally see in $\tilde \bx$ (denoted by $\tilde x_{i+1}$). This is depicted in
 the following figures, where blocks marked with $i$ or $i+1$ indicate consecutive boxes occupied by $i$ or $i+1$.
 The first figure shows the case where the $i$'s inserted in $\bx$ do not bump out all $(i+1)$'s:
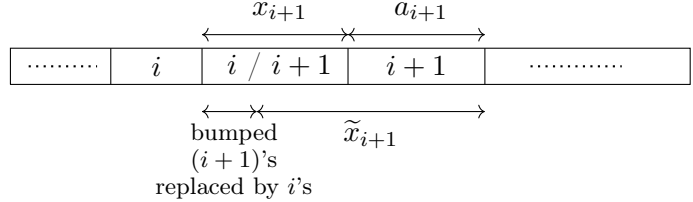
\begin{equation*}
{\begin{tikzpicture}[baseline=1.1ex, scale=.6]
\draw (0.1,0.1) -- (15,0.1)--(15,0.9)--(0.1,0.9)--(0.1,0.1) ;
\draw[dotted, thick] (0.5,0.5)--(2,0.5);
\draw (2.3,0.1)--(2.3,0.9);
\node at (3.3,0.5) {$i$};
\draw (4.3,0.1)--(4.3,0.9);
\node at (6,0.5) {$i$ / $i+1$};
\draw (7.5,0.1)--(7.5,0.9);
\node at (9,0.5) {$i+1$};
\draw (10.5,0.1)--(10.5,0.9);
\draw[dotted, thick] (11.5,0.5)--(13.5,0.5);
\draw[<->](4.3,1.2)--(7.5,1.2); \node at (6,1.7) {$x_{i+1}$};
\draw[<->](7.5,1.2)--(10.5,1.2); \node at (9.1,1.7) {$a_{i+1}$};
\draw[<->](4.3,-0.5)--(5.5,-0.5); \node at (5,-1) {\footnotesize{bumped}}; \node at (5,-1.6) {\footnotesize $(i+1)$'s};
\node at (5,-2.2) {\footnotesize{replaced by  $i$'s}};
\draw[<->](5.5,-0.5)--(10.5,-0.5); \node at (8,-1) {$\tilde x_{i+1}$};
\end{tikzpicture}}
\end{equation*}
and the next figure depicts the situation where the $i$'s inserted in $\bx$ have bumped out all $(i+1)$'s:
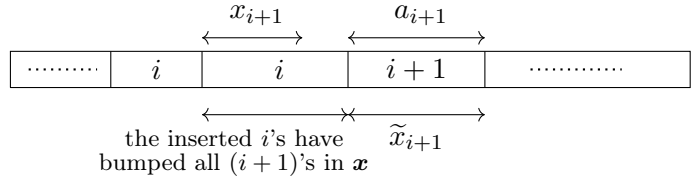
\begin{equation*}
{\begin{tikzpicture}[baseline=1.1ex, scale=.6]
\draw (0.1,0.1) -- (15,0.1)--(15,0.9)--(0.1,0.9)--(0.1,0.1) ;
\draw[dotted, thick] (0.5,0.5)--(2,0.5);
\draw (2.3,0.1)--(2.3,0.9);
\node at (3.3,0.5) {$i$};
\draw (4.3,0.1)--(4.3,0.9);
\node at (6,0.5) {$i$};
\draw (7.5,0.1)--(7.5,0.9);
\node at (9,0.5) {$i+1$};
\draw (10.5,0.1)--(10.5,0.9);
\draw[dotted, thick] (11.5,0.5)--(13.5,0.5);
\draw[<->](4.3,1.2)--(6.5,1.2); \node at (5.5,1.7) {$x_{i+1}$};
\draw[<->](7.5,1.2)--(10.5,1.2); \node at (9.1,1.7) {$a_{i+1}$};
\draw[<->](4.3,-0.5)--(7.5,-0.5); \node at (5,-1) {\footnotesize{the inserted $i$'s have}}; \node at (5,-1.6) {\footnotesize bumped all 
$(i+1)$'s in $\bx$};
\draw[<->](7.5,-0.5)--(10.5,-0.5); \node at (9,-1) {$\tilde x_{i+1}$};
\end{tikzpicture}}
\end{equation*}
in which case $\tilde x_{i+1}=a_{i+1}$ and so $b_{i+1}=a_{i+1}+x_{i+1}-\tilde x_{i+1}= x_{i+1}$.
\vskip 2mm
We now want to get an expression for $\tilde \xi_{i+1}$, which denotes 
the total number of numbers up to $i+1$ that exist in the output word 
$\tilde\bx$. Again, either the $i$'s that we inserted from $\ba$ did not bump all the $(i+1)$'s that existed in $\bx$ (first of the two figures above),
 in which case the $a_{i+1} $-many of $(i+1)$'s, which are inserted from $\ba$ will be appended at the end of the last $(i+1)$ in $\bx$, giving
 $\tilde \xi_{i+1}=\xi_{i+1}+a_{i+1}$, or the $a_i$-many $i$'s in $\ba$ bumped all the $(i+1)$'s in $\bx$ (second of the two figures above), in which case the new $(i+1)$'s from $\ba$ will
 be appended after the $i$'s in $\tilde\bx$. In this case, $\tilde \xi_{i+1}=\tilde \xi_i+a_{i+1}$. Altogether, we have that
 \begin{align}\label{bump2}
 \tilde \xi_{i+1}=\max\big(\tilde \xi_i,\xi_{i+1}\big)+a_{i+1}.
 \end{align} 

We can now iterate this procedure through diagram \eqref{rs-iterate}. The $\rsk$ row insertion via piecewise linear transformations can be 
summarised as
\begin{proposition}\label{RSK_maxplus}
Let $1\le i \le N$. Consider two {\it words} $\bx=i^{x_i}(i+1)^{x_{i+1}}\cdots N^{x_N}$ and $\ba=i^{a_i}(i+1)^{a_{i+1}}\cdots N^{a_N}.$ The {\it row insertion} of the word $\ba$ into the word $\bx$  denoted by  
\begin{align*}
\begin{array}{ccc}& \ba& \\ \bx & \cross & \tilde \bx \\& \bb& \end{array},
\end{align*}
  transforms  $(\bx,\ba)$ into a new pair $(\tilde \bx,\bb)$ with $\tilde \bx=i^{\tilde x_i}(i+1)^{\tilde x_{i+1}}\cdots N^{\tilde x_N}$ and $\bb=(i+1)^{b_{i+1}}\cdots N^{b_N}$, 
  which in cumulative variables 
  \begin{align}\label{cumul}
\xi_j=x_i+\cdots + x_j, \quad \text{and}\quad \tilde \xi_j=\tilde x_i+\cdots+\tilde x_j, \qquad \text{for} \,\, j\geq i,
\end{align}
  is encoded via
\begin{equation}\label{RSKrec}
\begin{cases}
\tilde \xi_i=\xi_i+a_i, &\\[4pt]
\tilde \xi_k =\max\big(\tilde \xi_{k-1}, \xi_k\big)+a_k, & i+1\le k\le N\\[5pt]
b_{k} = a_k+ (\xi_k-\xi_{k-1}) - (\tilde \xi_k-\tilde \xi_{k-1}), & i+1\leq k\leq N,
\end{cases}
\end{equation}
for $i<N$. If $i=N$, then $\tilde \xi_N=\xi_N+a_N$ and the output $\bb$ is empty and we write $\bb=\emptyset$. 
\end{proposition}
\vskip 2mm
It is worth noticing that the recursion
\begin{align*}
\tilde \xi_k =\max\big(\tilde \xi_{k-1}, \xi_k\big)+a_k,
\end{align*}
is actually the same as the recursion of last passage percolation \eqref{LPPrec}. 
To see this more clearly, we can consider the example of 
$\bx=1^{x_1}2^{x_2}\cdots N^{x_N}, \ba=1^{a_1}2^{a_2}\cdots N^{a_N}$ and
for $\xi_j=x_1+\cdots+x_j$ and $\tilde \xi_j=\tilde x_1+\cdots \tilde x_j$ we iterate as
\begin{align}\label{lpp-rec}
\tilde \xi_N &=\max\big(\tilde \xi_{N-1}+a_N, \xi_N+a_N\big)\notag\\
&=\max\Big(\max\big(\tilde \xi_{N-2}+a_{N-1}+a_N, \xi_{N-1}+a_{N-1}+a_N)\,,\, \xi_N+a_N\Big)\notag\\
&\,\,\vdots\notag\\
&=\max_{1\leq j \leq N } \big( \xi_j+a_j+\cdots+a_N)\notag\\
&= \max_{1\leq j \leq N } \big( x_1+\cdots+x_j+a_j+\cdots+a_N),
\end{align}
which, as shown in the figure below, is a last passage percolation on a two-row array
\begin{equation}\label{lpp-rec-diag}
{\begin{tikzpicture}[scale=.6]
\node at (-5,0.5) { {$\max_j$}}; \node at (-2.2,0.5){$\overset{\sum}{\substack{ \text{weights of nodes} \\ \text{along red path} } }$};
\draw (0,0) -- (10,0)--(10,1)--(0,1)--(0,0) ;
\draw (1,0)--(1,1); \draw (2,0)--(2,1);\draw (2,0)--(2,1); \draw (3,0)--(3,1); \draw (4,0)--(4,1); \draw (5,0)--(5,1); \draw (6,0)--(6,1);
\draw (7,0)--(7,1); \draw (8,0)--(8,1); \draw (9,0)--(9,1);  
\draw  [fill ] (0,0) circle [radius=0.1];\draw  [fill ] (1,0) circle [radius=0.1];\draw  [fill ] (2,0) circle [radius=0.1]; \draw  [fill ] (3,0) circle [radius=0.1];
\draw  [fill ] (4,0) circle [radius=0.1]; \draw  [fill ] (5,0) circle [radius=0.1];\draw  [fill ] (6,0) circle [radius=0.1];\draw  [fill ] (7,0) circle [radius=0.1];
\draw  [fill ] (8,0) circle [radius=0.1]; \draw  [fill ] (9,0) circle [radius=0.1];\draw  [fill ] (10,0) circle [radius=0.1];

\draw  [fill ] (0,1) circle [radius=0.1];\draw  [fill ] (1,1) circle [radius=0.1];\draw  [fill ] (2,1) circle [radius=0.1]; \draw  [fill ] (3,1) circle [radius=0.1];
\draw  [fill ] (4,1) circle [radius=0.1]; \draw  [fill ] (5,1) circle [radius=0.1];\draw  [fill ] (6,1) circle [radius=0.1];\draw  [fill ] (7,1) circle [radius=0.1];
\draw  [fill ] (8,1) circle [radius=0.1]; \draw  [fill ] (9,1) circle [radius=0.1];\draw  [fill ] (10,1) circle [radius=0.1];
\node at (10.5, 0.1) {$\ba$}; \node at (10.5, 0.9) {$\bx$};

\draw[red, ultra thick] (0,1)--(4,1)--(4,0)--(10,0);
\node at (4, 1.6) {$j$};
\end{tikzpicture}}
\end{equation}
This is an indication of the relevance of $\rsk$, and in particular this
 formulation, for last passage percolation and other models in the KPZ class. Moreover, this formulation is amenable to a generalisation which will be important
 in treating the positive temperature case relating to directed polymers. In Section \ref{RSK_matrix} we will prove this connection 
 in more detail.
 
\subsection{ Gelfand-Tsetlin parametrisation.}\label{sec:GT}
Gelfand-Tsetlin  ($\GT$)  patterns are triangular arrays of numbers $(z^i_{j})_{1\leq j \leq i\leq N}$ which interlace, meaning that
\begin{align}\label{interlace}
z^{i+1}_{j+1}\leq z^i_j\leq z^{i+1}_j,
\end{align}
and for this reason they are depicted as 
\begin{eqnarray}\label{Ppattern}
\begin{array}{cccccccccccc}
&&&&z^1_{1}&&&&&&\\
&&&z^2_{2}&&z^2_{1}&&&&&\\
&&z^3_{3}&&z^3_2&&z^3_{1}&&&&\\
&\iddots&\iddots&&&&\ddots&\ddots&&\\
z^N_{N}&z^N_{N-1}&&&\ldots&&&z^N_2&z^N_{1}&&&.\\
\end{array}
\end{eqnarray}
They provide a particularly useful parametrisation of Young tableaux:
given a Young tableau consisting of letters $1,2,...,N$ (not all of which have to appear in the tableau) the Gelfand-Tsetlin
variables $z^i_j$ are defined as
\begin{align}\label{GTvar}
z^i_j:= \sum_{k=j}^i \sharp\{\text{ $k$'s in the $j^{th}$ row}\}
\end{align}
Given this definition, the right inequality in \eqref{interlace} is immediate, while the left one is a consequence of the fact that  
entries along columns in a Young tableau are strictly increasing.
 \vskip 2mm
 The bottom row of a $\GT$ pattern is called the {\bf shape}, since $z^N_i$ equals the length of the $i$-th row of the corresponding tableau and thus
the collection of $z^N_1,z^N_2,...$ determines the shape of the tableau. We will denote the shape of a $\GT$ pattern $\sfZ$ by $sh(\sfZ)$ and similarly the
shape of a tableau $P$ by $sh(P)$. We will also often identify a $\GT$ pattern $\sfZ$ with the corresponding Young tableau $P$.
 \vskip 2mm
 Notice that definition \eqref{GTvar} is in agreement with definition \eqref{cumul} of the cumulative variables $\xi_j$ and 
 $\tilde \xi_j$ and so Gelfand-Tsetlin patterns 
  provide the natural parametrisation for the framework of $\rsk$ as described in Proposition \ref{RSK_maxplus}.
   Moreover, as we will see they provide a structure, which couples models 
   in the KPZ class, e.g. longest increasing subsequence or last passage percolation,
   with Random Matrices. To give a preliminary idea of how this comes about, 
   we point out that, as a consequence of Schensted's Theorem \ref{schensted}, 
    entry $z^N_1$ of a Gelfand-Tsetlin pattern is equal to the length of the first row of a Young Tableau.  
   On the other hand, it turns out that in certain situations {\it random } Gelfand-Tsetlin 
   patterns have a bottom row with law identical to the law of the eigenvalues of certain random matrices, see e.g. \cite{B01}. 
   Therefore, the element $z^N_1$ has a dual nature: it is an observable of models
   within the KPZ class and at the same time its distribution is identical (or closely related) to the distribution 
   of the largest eigenvalue of certain random matrices.  This coupling has played a central role in formulating the integrable structure
   of models in the KPZ universality and we will explore this in the coming sections.
   
   \vskip 2mm
   An important {\it invariant} of $\rsk$ is the {\bf type} of a tableau $P$, denoted by $ type(P)$. In $\GT$ parametrisation, this is defined to be the vector
   \begin{align*}
   \big(|z^i|-|z^{i-1}| \, \colon\, i=1,...,N\big), \quad \text{with} \quad |z^i|:=\sum_{j=1}^i z^{i}_j,
   \end{align*}  
   and the convention that $|z^0|=0$. Considering a pair of Gelfand-Tsetlin patterns $(\sfZ,\sfZ')$ as the output of $\rsk$ with input matrix $W=(w^{i}_{j}\colon 1\leq i \leq n, \, 1\leq j\leq N)$, that is $(\sfZ,\sfZ')=\rsk(W)$, with $\sfZ$ corresponding to the $P$ tableau and $\sfZ'$ to the $Q$ tableau in the $\rsk$ correspondence, then it holds that
   \begin{align}\label{type}
   |z^k|-|z^{k-1}|=\sum_{i=1}^n w^i_k.
   \end{align}
   This is due to the fact that both sides represent the number of letters $k$ inserted from $W$ via $\rsk$. This is clear for the right-hand side, since 
   (by definition) $w^{i}_j$ is considered as the number of letters $j$ in word $i$, while for the left-hand side this follows from
   \begin{align*}
   |z^k|-|z^{k-1}|&=\sum_{j=1}^k z^k_j - \sum_{j=1}^{k-1} z^{k-1}_j = z^k_k +\sum_{j=1}^{k-1} (z^{k}_j-z^{k-1}_j)\\
   &=\sharp\{\text{$k$'s in word $k$}\} + \sum_{j=1}^{k-1} \sharp\{\text{$k$'s in word $j$} \}.
   \end{align*}
   The {\it type} plays an important role in exactly solvable, via $\rsk$, probability models, as we will see in Sections \ref{geomLPP} and \ref{sec:loggamma}.
   
   \section{A geometric lifting of RSK - Kirillov's ``Tropical RSK''}\label{sec:grsk} 
   As we have seen in Proposition \ref{RSK_maxplus},
    $\rsk$ can be encoded in terms of piecewise linear recursive relations, using the
   $(\max,+)$ algebra. Kirillov \cite{K01} replaced the $(\max,+)$ in the set of $\rsk$'s piecewise linear relations with
   relations  $(+,\times)$,
   thus establishing a {\it geometric lifting} of $\rsk$, which he named {\it tropical} $\rsk$.  This name was given by Kirillov in honour of
   P. Sch\"utzenberger, see \cite{K01}, page 84, for some clues regarding the etymology of this name. 
   However, since the term tropical has been reserved for the opposite passage from the $(+,\times)$
   algebra to the $(\max,+)$, the term geometric $\rsk$ ($\grsk$) has now prevailed for the geometric lifting of the $\rsk$ correspondence. In this section we will present the construction of $\grsk$ following mostly a matrix reformulation by Noumi and Yamada \cite{NY04}  motivated by discrete integrable systems. The approach is closely related to that of
   Proposition  \ref{RSK_maxplus}. Let us start with the definition of the 
   {\it geometric row insertion}.  
   \subsection{Geometric $\rsk$ via a matrix formulation.}\label{RSK_matrix}
   We start with the following definition (compare to Proposition \ref{RSK_maxplus}):
   \begin{definition}\label{def:grsk}
Let $1\le i \le N$. Consider two {\it words} $\bx= (x_i,...,x_N)$ and $\ba=(a_i,...,a_N).$ We  define the {\bf geometric lifting of row insertion} or shortly {\bf geometric row insertion} of the word $\ba$ into the word $\bx$,  denoted by  
\begin{align*}
\begin{array}{ccc}& \ba& \\ \bx & \cross & \tilde \bx \\& \bb& \end{array},
\end{align*}
  as the transformation that takes  $(\bx,\ba)$ into a new pair $(\tilde \bx,\bb)$ with $\tilde \bx=(\tilde x_i, \cdots \tilde x_N)$ and
   $\bb=(b_{i+1}\cdots b_N)$, 
  which in cumulative variables 
  \begin{align}\label{g-cumul}
\xi_j=x_i\cdots x_j, \quad \text{and}\quad \tilde \xi_j=\tilde x_i\cdots \tilde x_j, \qquad \text{for} \,\, j\geq i, 
\end{align}
  is encoded via
\begin{equation}\label{gRSKrec}
\begin{cases}
\tilde \xi_i=\xi_i \cdot a_i, &\\[4pt]
\tilde \xi_k =a_k\big(\tilde \xi_{k-1}+ \xi_k\big), & i+1\le k\le N\\[5pt]
b_{k} = a_k\frac{\xi_k \,\tilde \xi_{k-1}}{\xi_{k-1} \,\tilde \xi_k}, & i+1\leq k\leq N.
\end{cases}
\end{equation}
for $i<N$.
If $i=N$, then $\tilde \xi_N=\xi_N \cdot a_N$ and the output $\bb$ is empty and we write $\bb=\emptyset$. 
\end{definition}
It was observed by Noumi and Yamada that the geometric row insertion $(\bx,\ba)\mapsto(\tilde{\bx},\bb)$, described
in Definition \ref{def:grsk} via relations \eqref{gRSKrec}, is equivalent to a system of equations, related to {\it discrete Toda systems}
(see \cite{NY04}, Lemma 2.2 and Remark 2.3):
\begin{equation}\label{eq:toda}
{\begin{split}
a_i x_i= \tilde x_i, &\quad \text{and} \quad a_j x_j=\tilde x_j b_j \quad \text{for}\,\, j\geq i+1,\\
\frac{1}{a_i}+\frac{1}{x_{i+1}}=\frac{1}{b_{i+1}}, &  \quad \text {and}\quad \frac{1}{a_j}+\frac{1}{x_{j+1}}=\frac{1}{\tilde x_j}+\frac{1}{b_{j+1}}
\quad \text{for}\,\, j\geq i+1.
\end{split}
}
\end{equation}
The derivation of the system of equations in \eqref{eq:toda} from \eqref{gRSKrec} is a matter of a simple algebraic manipulation.
\eqref{eq:toda} can be put into a matrix form as:
\begin{align}\label{matrixform}
&\begin{pmatrix}
1\\&\ddots\\ &&1\\&&&\bar a_i&1\\ &&&& \bar a_{i+1}&1\\ &&&&&\ddots&\ddots\\ &&&&&&& 1\\&&&&&&& \bar a_N 
\end{pmatrix}
\begin{pmatrix}
1\\&\ddots\\ &&1\\&&&\bar x_i&1\\ &&&& \bar x_{i+1}&1\\ &&&&&\ddots&\ddots\\ &&&&&&& 1\\ &&&&&&& \bar x_N 
\end{pmatrix}\\
&\qquad=\begin{pmatrix}
1\\&\ddots\\ &&1\\&&&\overline{\tilde x}_i&1\\ &&&& \overline {\tilde x}_{i+1}&1\\
 &&&&&\ddots&\ddots\\ &&&&&&& 1\\ &&&&&&& \overline {\tilde x}_N 
\end{pmatrix}
\begin{pmatrix}
1\\&\ddots\\ &&1\\&&&1&\\ &&&& \bar b_{i+1}&1\\  &&&&&\ddots&\ddots\\ &&&&&&& 1\\ &&&&&&& \bar b_N 
\end{pmatrix} \notag
\end{align}
where we have used the notation
\begin{align*}
\bar x :=\frac{1}{x},
\end{align*}
for a nonnegative real $x$. For a vector $\bx$ with nonnegative real entries, we will denote by $\bar\bx$ the vector 
whose entries are the inverses of the corresponding entries in $\bx$.
In the above matrices
the entries that are left empty are equal to zero. Moreover,
 in the first three matrices the upper left corner is an identity matrix of size $i-1$,
while in the fourth matrix the upper left corner is an identity matrix of size $i$.
We put this matrix equation in a more concrete notation as follows:
For a vector $\bx=(x_i,...,x_n)$ we define the matrix 
\begin{align}\label{Ematrices}
E_i(\bx):=\sum_{j=1}^{i-1} E_{jj}+\sum_{j=i}^N x_j E_{jj}+\sum_{j=i}^{N-1} E_{j,j+1},
\end{align}
where for $1\leq i,j\leq n$ we define the matrices $E_{ij}:=(\gd_{ai}\gd_{bj})_{1\leq a,b \leq n}$ with $\gd_{ab}$ being the Kronecker delta. In
other words, $E_{ij}$ has entry $(i,j)$ equal to $1$ and all others equal to zero. 
We will denote $E_1(\bx)$ simply by $E(\bx)$.
Then \eqref{matrixform} is written more concretely as
\begin{align}\label{E-eq}
E_i(\bar \ba) E_i(\bar \bx) = E_i({\overline {\tilde \bx}}) E_{i+1}(\bar \bb)
\end{align}
The entries of matrices $E_i(\bx)$, with $\bx=(x_i,...,x_N)$,
 can be readily read graphically from the following diagram 
  
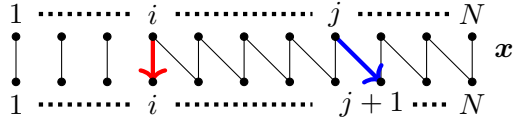
\begin{equation*}
{\begin{tikzpicture}[scale=.6]
\draw (10,0) -- (10,-1); \draw (9,0) -- (9,-1); \draw (8,0) -- (8,-1); \draw (7,0) -- (7,-1); \draw (6,0) -- (6,-1);
\draw (5,0) -- (5,-1); \draw (4,0) -- (4,-1); \draw (3,0) -- (3,-1); 
\draw (2,0) -- (2,-1); \draw (1,0) -- (1,-1);\draw (0,0) -- (0,-1);
 \draw(3,0)--(4,-1); \draw(4,0)--(5,-1); \draw(5,0)--(6,-1);
\draw(6,0)--(7,-1); \draw(7,0)--(8,-1); \draw(8,0)--(9,-1); \draw(9,0)--(10,-1);
\node at (10.7, -0.3) {$\bx$}; 
 \node at (7, 0.5) {$j$}; \node at (0, 0.5) {$1$}; \node at (0, -1.5) {$1$}; \draw[ultra thick, dotted] (0.5,0.5)--(2.5,0.5);
\draw[ultra thick, dotted] (0.5,-1.5)--(2.5,-1.5);
  \draw[ultra thick, dotted] (3.5,-1.5)--(6.5,-1.5);
 \draw[ultra thick, dotted] (3.5,0.5)--(6.5,0.5);  \draw[ultra thick, dotted] (7.5,0.5)--(9.5,0.5);
 \draw[ultra thick, dotted] (8.7,-1.5)--(9.5,-1.5); \node at (10, 0.5) {$N$}; \node at (10, -1.5) {$N$};
\draw[ultra thick, red, ->] (3,0)--(3,-1); \draw[ultra thick, blue, ->]  (7,0)--(8,-1);  \node at (7.8, -1.5) {$j+1$};
\node at (3, -1.5) {$i$}; \node at (3, 0.5) {$i$};
\foreach \i in {0,...,10} {\node[draw,circle,inner sep=1pt,fill] at (\i,0) {};
	}
\foreach \i in {0,...,10} {\node[draw,circle,inner sep=1pt,fill] at (\i,-1) {};
	}
\end{tikzpicture}} 
\end{equation*}
where on the diagonal edges and on the 
first $(i-1)$ vertical edges we assign the value $1$ and on the rest of the vertical edges we assign the values 
$x_i,x_{i+1},\dots, x_N$ in this order.  Then the $(k,\ell)$ entry of $E_i(\bx)$ is given by 
$E_i(\bx)_{(k,\ell)}=\sum_{\pi\colon (1,k)\to (2,\ell)} \sfw\sft(\pi)$, where the sum is over all down-right paths 
(consisting in this case of one step), along existing edges,
 starting from 
site $k$ in the top row to site $\ell$ in the bottom row and the weight of the path $\pi$ is given by the product of the weights
along the edges that path $\pi$ traces. Furthermore, one can easily check that for products of the form 
$E(\by^1,...,\by^k):=E_1(\by^1)E_2(\by^2)\cdots E_k(\by^k)$, where we understand that for $i=1,...,k $ the vector $\by^i=(y^i_i,...,y^i_N)$,
the entries can be read graphically from the following diagram:
  
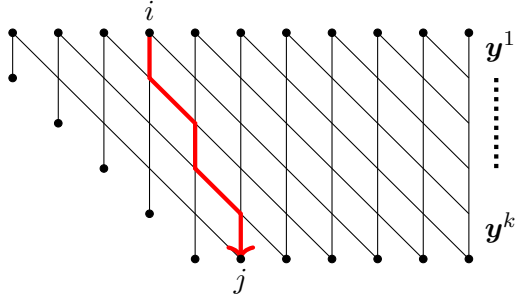
\begin{equation*}
{\begin{tikzpicture}[scale=.6]
\draw (10,0) -- (10,-5); \draw (9,0) -- (9,-5); \draw (8,0) -- (8,-5); \draw (7,0) -- (7,-5); \draw (6,0) -- (6,-5);
\draw (5,0) -- (5,-5); \draw (4,0) -- (4,-5); \draw (3,0) -- (3,-4); \draw (2,0) -- (2,-3); \draw (1,0) -- (1,-2);
\draw (0,0) -- (0,-1);

\draw(0,0)--(5,-5); \draw(1,0)--(6,-5); \draw(2,0)--(7,-5); \draw(3,0)--(8,-5); \draw(4,0)--(9,-5); \draw(5,0)--(10,-5);
\draw(6,0)--(10,-4); \draw(7,0)--(10,-3); \draw(8,0)--(10,-2); \draw(9,0)--(10,-1);
\node at (10.7, -0.3) {$\by^1$};  \node at (10.7, -4.3) {$\by^k$};
\node at (3, 0.5) {$i$}; 
\draw[ultra thick, red, ->] (3,0)--(3,-1)--(4,-2)--(4,-3)--(5,-4)--(5,-5); 
\node at (5, -5.5) {$j$}; 
\draw[ultra thick, dotted] (10.6,-1)--(10.6,-3);
\foreach \i in {0,...,10} {\node[draw,circle,inner sep=1pt,fill] at (\i,0) {};
	}
\foreach \i in {0,...,4} {\node[draw,circle,inner sep=1pt,fill] at (\i,-\i-1) {};
	}
\foreach \i in {5,...,10} {\node[draw,circle,inner sep=1pt,fill] at (\i,-5) {};
	}
\end{tikzpicture}} 
\end{equation*}
where a vertical edge connecting $(a,b)$ to $(a+1,b)$ (in matrix coordinates) is assigned the weight $y^a_b$ and all the
diagonal edges are assigned weight one. Entry $(i,j)$ of the matrix $E(\by^1,...,\by^k )$ is given by 
$E(\by^1,...,\by^k)_{(i,j)}=\sum_{\pi\colon (1,i)\to(k\wedge j +1,j)} \sfw\sft(\pi)$, where the sum is over all down-right paths, along existing edges,
 from site $(1,i)$ (in matrix coordinates) in the
top row to site $(k\wedge j+1, j)$,  $k\wedge j:=\min(k,j)$, along the lower border
and where the weight of a path is $\sfw\sft(\pi):=\prod_{{\be}\in \pi} w_{{\be}}$, with the product over all edges ${\be}$ that are traced by the path $\pi$.
 
More remarkably, minor determinants of matrices $ E(\by^1,...,\by^k)$ have also  a similar graphical representation. If we denote
by   $\det E(\by^1,...,\by^k)^{i_1,...,i_r}_{j_1,...,j_r} $ the determinant of the sub-matrix of $ E(\by^1,...,\by^k)$ which
includes rows $i_1<\cdots<i_r$ and columns $j_1<\cdots<j_r$ then
 \begin{equation}\label{eq:non-int-det}
  \det E(\by^1,...,\by^k)^{i_1,...,i_r}_{j_1,...,j_r} = \sum_{\pi_1,...,\pi_r \, \in \,\Pi^{\,i_1,...,i_r}_{\,j_1,...,j_r}} \sfw\sft(\pi_1) \cdots \sfw\sft(\pi_r),
  \end{equation}
  where $\Pi^{\,i_1,...,i_r}_{\,j_1,...,j_r}$ is the set of directed, non-intersecting paths, starting at locations $i_1,...,i_r$
   in the top row and ending at locations $j_1,...,j_r$ at the bottom border of the grid :
    
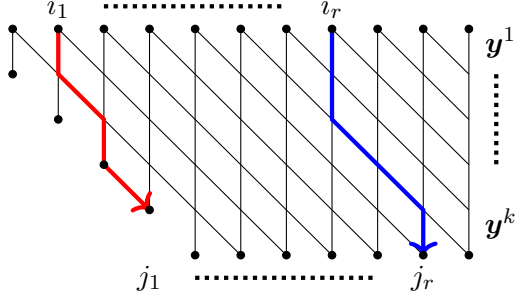
\begin{equation*}
{\begin{tikzpicture}[scale=.6]
\draw (10,0) -- (10,-5); \draw (9,0) -- (9,-5); \draw (8,0) -- (8,-5); \draw (7,0) -- (7,-5); \draw (6,0) -- (6,-5);
\draw (5,0) -- (5,-5); \draw (4,0) -- (4,-5); \draw (3,0) -- (3,-4); \draw (2,0) -- (2,-3); \draw (1,0) -- (1,-2);
\draw (0,0) -- (0,-1);

\draw(0,0)--(5,-5); \draw(1,0)--(6,-5); \draw(2,0)--(7,-5); \draw(3,0)--(8,-5); \draw(4,0)--(9,-5); \draw(5,0)--(10,-5);
\draw(6,0)--(10,-4); \draw(7,0)--(10,-3); \draw(8,0)--(10,-2); \draw(9,0)--(10,-1);
\node at (10.7, -0.3) {$\by^1$};  \node at (10.7, -4.3) {$\by^k$};
\node at (1, 0.5) {$i_1$}; \draw[ultra thick, dotted] (2,0.5)--(6,0.5); \node at (7, 0.5) {$i_r$};
\draw[ultra thick, red, ->] (1,0)--(1,-1)--(2,-2)--(2,-3)--(3,-4); \draw[ultra thick, blue, ->]  (7,0)--(7,-2)--(9,-4)--(9,-5);
\node at (3, -5.5) {$j_1$}; \draw[ultra thick, dotted] (4,-5.5)--(8,-5.5); \node at (9, -5.5) {$j_r$};
\draw[ultra thick, dotted] (10.6,-1)--(10.6,-3);
\foreach \i in {0,...,10} {\node[draw,circle,inner sep=1pt,fill] at (\i,0) {};
	}
\foreach \i in {0,...,4} {\node[draw,circle,inner sep=1pt,fill] at (\i,-\i-1) {};
	}
\foreach \i in {5,...,10} {\node[draw,circle,inner sep=1pt,fill] at (\i,-5) {};
	}
\end{tikzpicture}} 
\end{equation*}
Notice that the non-intersecting property and the orderings $i_1<\cdots<i_r$ and $j_1<\cdots j_r$
 enforces that the path starting at $i_a$ will end at $j_a$ for all $a=1,...,r$.
 
Since $E(\by^1,...,\by^k)_{(i,j)}=\sum_{\pi\colon (1,i)\to(k\wedge j+1,j)} \sfw\sft(\pi)$,
 \eqref{eq:non-int-det} is a consequence of the Lindstr\"om-Gessel-Viennot theorem \cite{L73, GV89, S90}.
 This theorem has played a central role in the developments
around {\it determinantal processes} as it gives determinantal expressions for the weights (or probability) of non-intersecting paths.
 Its precise statement is as follows:
\begin{theorem}[Lindstr\"om-Gessel-Viennot]\label{thm:LGV}
Let $G=(V,E)$ be a directed, acyclic graph with no multiple edges, with each edge $\be$ being assigned
 a {\it weigth} $\sfw\sft({\be})$. A path $\pi$
on $G$ is assigned a weight $\sfw\sft(\pi)=\prod_{\be\in\pi}\sfw\sft(\be) $. We say that two paths on $G$ are {\it non-intersecting} if they do not
share any vertex. Consider, now,  $(u_1,...,u_r)$ and $(v_1,...,v_r)$ two disjoint 
subsets of $V$ and denote by $\Pi^{\,u_1,...,u_r}_{\,v_1,...,v_r}$ the set of all $r$-tuples of non-intersecting paths $\pi_1,...,\pi_r$ 
that start from $u_1,...,u_r$ and end at $v_1,...,v_r$, respectively. We assume that $\{u_1,...,u_r\}$ and $\{v_1,...,v_r\}$
 have the property that for $i<j$ and $i'>j'$,
 any two paths $\pi\in \Pi^{u_i}_{v_j}$ and $ \pi'\in \Pi^{u_{i'}}_{v_{j'}}$, which start at $u_i,u_{i'}$ and end at $v_j, v_{j'}$, necessarily intersect.
Then
\begin{align*}
\det\Big( \sum_{\pi\in \Pi^{\,u_i}_{\,v_j}} \sfw\sft(\pi) \Big)_{1\leq i,j\leq r} = \sum_{\pi_1,...,\pi_r\,\in\, \Pi^{\,u_1,...,u_r}_{\,v_1,...,v_r}} \sfw\sft(\pi_1)\cdots\sfw\sft(\pi_r)
\end{align*} 
\end{theorem}
There is also a set of matrices dual to $E_i(\bx)$, which we now introduce, the entries and minor determinants 
of which are given in terms of the total weights of paths moving in the more standard up-right, directed fashion.
Let us start with the $i=1$ case,
 where we recall the convention that $E_1(\bx)=E(\bx)$, and define 
\begin{align}\label{HErelation}
H(\bx) :=D E(\bar{\bx})^{-1} D^{-1},\qquad \text{with} \qquad D=\text{diag}\big((-1)^{i-1}\big)_{i=1}^N.
\end{align}
An easy computation shows that $H(\bx)= \sum_{1\leq i\leq j \leq N} x_i x_{i+1}\cdots x_j E_{ij}$, that is, the $(i,j)$ entry of $H(\bx)$
equals $ x_i x_{i+1}\cdots x_j$ if $i \leq j$ and zero otherwise.

In general, for $k\geq1$, and $\bx=(x_k,...,x_n)$, in which case we note that 
$ H(\bx)=H(x_k,...,x_N)$ is a $k\times k$ matrix, we define the $N\times N$ matrix
\begin{align*}
H_k(\bx): =\left[ \begin{array}{cc}
I_{k-1} & {\bf 0}\\
{\bf 0} & H(\bx)
\end{array}
 \right]
\end{align*}
where $I_{k-1}$ is a $(k-1)\times (k-1)$ identity matrix.
Then using \eqref{HErelation} equation \eqref{E-eq} can be equivalently written as 
\begin{align}\label{H-eq}
H_i( \bx) H_i( \ba) = H_{i+1}(\bar \bb) H_i( {\tilde \bx}) 
\end{align}
Similarly to $E(\by^1,...,\by^k)$, products of the form $H(\by^1,...,\by^n):=H(\by^1)\cdots H(\by^n)$ or more generally $H(\by^1,...,\by^k):=H_k(\by^k)\cdots H_1(\by^1)$ 
   have the property that their entries and their minor determinants 
  are given via ensembles of non-intersecting paths.  This is again a consequence of the Lindstr\"om-Gessel-Viennot theorem: 
  denoting by $\det H(\by^1,...,\by^k)^{i_1,...,i_r}_{j_1,...,j_r}$ the minor determinant 
  of $H(\by^1,...,\by^k)$ consisting of rows $i_1<\cdots<i_r$ and columns $j_1<\cdots<j_r$, then
    \begin{equation*}
  \det H(\by^1,...,\by^k)^{i_1,...,i_r}_{j_1,...,j_r} = \sum_{\pi_1,...,\pi_r} \sfw\sft(\pi_1) \cdots \sfw\sft(\pi_r)
  \end{equation*}
  where the sum is over up-right, non crossing paths, starting at locations $i_1,...,i_r$ at the bottom border (including possibly the diagonal part) and ending at locations
  $j_1,...,j_r$  at the top row of the grid. 
   
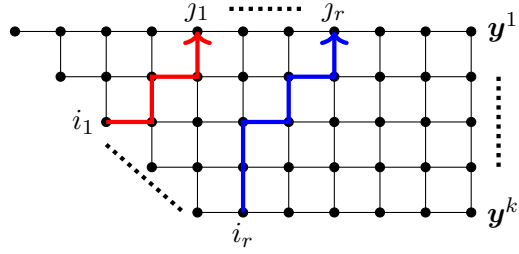
\begin{equation*}
{\begin{tikzpicture}[scale=.6]
\draw (0,0) -- (10,0) ; \draw (1,-1) -- (10,-1); \draw (2,-2) -- (10,-2); \draw (3,-3) -- (10,-3); \draw (4,-4) -- (10,-4);
\draw (10,0) -- (10,-4); \draw (9,0) -- (9,-4); \draw (8,0) -- (8,-4); \draw (7,0) -- (7,-4); \draw (6,0) -- (6,-4);
\draw (5,0) -- (5,-4); \draw (4,0) -- (4,-4); \draw (3,0) -- (3,-3); \draw (2,0) -- (2,-2); \draw (1,0) -- (1,-1);
\draw  [fill ] (0,0) circle [radius=0.1];\draw  [fill ] (1,0) circle [radius=0.1];\draw  [fill ] (2,0) circle [radius=0.1]; \draw  [fill ] (3,0) circle [radius=0.1]; \draw  [fill ] (4,0) circle [radius=0.1]; \draw  [fill ] (5,0) circle [radius=0.1];\draw  [fill ] (6,0) circle [radius=0.1];\draw  [fill ] (7,0) circle [radius=0.1]; \draw  [fill ] (8,0) circle [radius=0.1]; \draw  [fill ] (9,0) circle [radius=0.1];\draw  [fill ] (10,0) circle [radius=0.1];

\draw  [fill ] (1,-1) circle [radius=0.1];\draw  [fill ] (2,-1) circle [radius=0.1];\draw  [fill ] (3,-1) circle [radius=0.1]; \draw  [fill ] (4,-1) circle [radius=0.1]; \draw  [fill ] (5,-1) circle [radius=0.1]; \draw  [fill ] (6,-1) circle [radius=0.1];\draw  [fill ] (7,-1) circle [radius=0.1];\draw  [fill ] (8,-1) circle [radius=0.1]; \draw  [fill ] (9,-1) circle [radius=0.1]; \draw  [fill ] (10,-1) circle [radius=0.1];

\draw  [fill ] (2,-2) circle [radius=0.1];\draw  [fill ] (3,-2) circle [radius=0.1]; \draw  [fill ] (4,-2) circle [radius=0.1]; \draw  [fill ] (5,-2) circle [radius=0.1]; \draw  [fill ] (6,-2) circle [radius=0.1];\draw  [fill ] (7,-2) circle [radius=0.1];\draw  [fill ] (8,-2) circle [radius=0.1]; \draw  [fill ] (9,-2) circle [radius=0.1]; \draw  [fill ] (10,-2) circle [radius=0.1];

\draw  [fill ] (3,-3) circle [radius=0.1]; \draw  [fill ] (4,-3) circle [radius=0.1]; \draw  [fill ] (5,-3) circle [radius=0.1]; \draw  [fill ] (6,-3) circle [radius=0.1];\draw  [fill ] (7,-3) circle [radius=0.1];\draw  [fill ] (8,-3) circle [radius=0.1]; \draw  [fill ] (9,-3) circle [radius=0.1]; \draw  [fill ] (10,-3) circle [radius=0.1];

 \draw  [fill ] (4,-4) circle [radius=0.1]; \draw  [fill ] (5,-4) circle [radius=0.1]; \draw  [fill ] (6,-4) circle [radius=0.1];\draw  [fill ] (7,-4) circle [radius=0.1];\draw  [fill ] (8,-4) circle [radius=0.1]; \draw  [fill ] (9,-4) circle [radius=0.1]; \draw  [fill ] (10,-4) circle [radius=0.1];
\node at (10.7, 0.1) {$\by^1$};  \node at (10.7, -4.0) {$\by^k$};
\node at (4, 0.5) {$j_1$}; \draw[ultra thick, dotted] (4.7,0.5)--(6.3,0.5); \node at (7, 0.5) {$j_r$};
\draw[ultra thick, red, ->] (2,-2)--(3,-2)--(3,-1)--(4,-1)--(4,0);
\draw[ultra thick, blue, ->]  (5,-4)--(5,-2)--(6,-2)--(6,-1)--(7,-1)--(7,0);
\node at (1.5, -2) {$i_1$}; \draw[ultra thick, dotted] (2,-2.5)--(3.7,-4); \node at (5, -4.5) {$i_r$};
\draw[ultra thick, dotted] (10.6,-1)--(10.6,-3);
\end{tikzpicture}} 
\end{equation*}
Each vertex $(a,b)$ of the grid is assigned a weight $y^a_b$ and in this case the total weight of a path is $\sfw\sft(\pi)=\prod_{(a,b)\in \pi} y^a_b$.
\vskip 2mm
 Noumi and Yamada \cite{NY04} used these observations in order to give a matrix reformulation of geometric $\rsk$ 
 allowing the output of geometric RSK to be expressed in terms of 
 total weight (often also called partition functions) of ensembles of non-intersecting paths.
 The main theorem towards this is the following whose origins are in studies around total positivity  \cite{BFZ96} (but see also 
 \cite{K01}).
 
\begin{theorem}\label{thm:NY}
Given a matrix $ \sfX:=(x^i_j \colon 1\leq i \leq n, \,1\leq j \leq N)=: (\bx^1,...,\bx^n)^{\sfT}$ the matrix equation
\begin{align}\label{H-eq}
H(\bx^1) H(\bx^2) \cdots H(\bx^n) = H_k(\by^k) H_{k-1}(\by^{k-1}) \cdots H_1(\by^1) ,\qquad k=\min(n,N),
\end{align}
has a unique solution $(\by^1,...,\by^k)$ with $\by^i:=(y^i_i,...,y^i_N)$, given by
\begin{align}\label{H-uni}
y^i_i=\frac{\tau^i_i}{\tau^{i-1}_i},\quad \text{and}\quad y^i_j=\frac{\tau^i_j\,\tau^{i-1}_{j-1}}{\tau^{i-1}_j\,\tau^i_{j-1}}\quad \text{for} \quad i<j,
\end{align}
where 
\begin{align}\label{H-weight}
\tau^i_j :=\sum_{\pi_1,...,\pi_i\in \Pi^{1,...,i}_{j-i+1,...,j}} \sfw\sft(\pi_1)\cdots \sfw\sft(\pi_i)
\end{align}
 is the total weight (partition function) of an ensemble of $i$ non-intersecting, down-right paths $\pi_1,...,\pi_i$, along the entries of $X$, starting from $(1,1),...,(1,i)$ and ending at $(n,j-i+1),...,(n,j)$, respectively,
with the weight of a path $\pi_r$ given by $\sfw\sft(\pi_r):=\prod_{(a,b)\in \pi_r}  x^a_b$.
\end{theorem}
\begin{proof}
The fact that the left-hand side of \eqref{H-eq} can be written uniquely as a product of the form of the right-hand side of \eqref{H-eq} 
 follows from a more general result of Berenstein-Fomin-Zelevinsky \cite{BFZ96}, which states that any upper triangular matrix 
$A=(a^i_j)_{1\leq i,j\leq n}$ such that
 \begin{itemize}
 \item $a^i_j=0$, if $j<i$ or $j>i+m$ for some $m\leq n$, i.e. $A$ is a ``band'' upper triangular matrix . Moreover,
  we assume that $a^i_j=1$ for $j=i+m$,
 \item the minor determinants 
 \begin{align*}
 Q_{i,j}:=Q_{i,j}(A):= \det A^{1,...,j-i+1}_{i,i+1,...,j},
 \end{align*}
 are non-zero for all $i,j$ such that $i\leq j$ and $i\leq m$,
 \end{itemize}
 can be written uniquely in the form $H_k(\by^k) H_{k-1}(\by^{k-1}) \cdots H_1(\by^1)$.
 Let us note that the second bullet is essentially enforced by the first but we preferred to state it separately as
 the non-vanishing of these minor determinants is an important feature in this matrix representation as well as in terms of the path
 representation that follows.
  We refer for this general result and proofs
to \cite{NY04}, Propositions 1.5, 1.6 and Theorem 2.4. The proof of this statement is a clever linear algebra using the
Gauss decomposition but we prefer to omit it as we would like to go straight to the connection to path weights. 

 It is easy to check that the left-hand side of \eqref{H-eq} satisfies the above conditions.
Here we will only prove that the solution to \eqref{H-eq} is given via
\eqref{H-uni}, which shows the relation between the minor determinants of products \eqref{H-eq} and partition functions of paths.

From the graphical representation of the minor determinants of $H(\bx^1) H(\bx^2) \cdots H(\bx^n) $, we have that
\[
\tau^i_j=\det \big( H(\bx^1) H(\bx^2) \cdots H(\bx^n)  \big)^{1,...,i}_{j-i+1,...,j}.
\]
But by \eqref{H-eq} this is equal to $\det\big(H_k(\by^k) H_{k-1}(\by^{k-1}) \cdots H_1(\by^1)\big)^{1,...,i}_{j-i+1,...,j}$ and again from the
graphical representation we see that this equals
\begin{equation}\label{y-paths}
\sum_{\gamma_1,...,\gamma_r} \sfw\sft(\gamma_1)\cdots \sfw\sft(\gamma_r)\equiv 
\sum_{\gamma_1,...,\gamma_r}
\begin{tikzpicture}[baseline={([yshift=-.5ex]current bounding box.center)},vertex/.style={anchor=base,
    circle,fill=black!25,minimum size=18pt,inner sep=2pt}, scale=0.35]
\draw (0,0) -- (10,0) ; \draw (1,-1) -- (10,-1); \draw (2,-2) -- (10,-2); \draw (3,-3) -- (10,-3); \draw (4,-4) -- (10,-4);
\draw (10,0) -- (10,-4); \draw (9,0) -- (9,-4); \draw (8,0) -- (8,-4); \draw (7,0) -- (7,-4); \draw (6,0) -- (6,-4);
\draw (5,0) -- (5,-4); \draw (4,0) -- (4,-4); \draw (3,0) -- (3,-3); \draw (2,0) -- (2,-2); \draw (1,0) -- (1,-1);
\draw  [fill ] (0,0) circle [radius=0.1];\draw  [fill ] (1,0) circle [radius=0.1];\draw  [fill ] (2,0) circle [radius=0.1]; \draw  [fill ] (3,0) circle [radius=0.1]; \draw  [fill ] (4,0) circle [radius=0.1]; \draw  [fill ] (5,0) circle [radius=0.1];\draw  [fill ] (6,0) circle [radius=0.1];\draw  [fill ] (7,0) circle [radius=0.1]; \draw  [fill ] (8,0) circle [radius=0.1]; \draw  [fill ] (9,0) circle [radius=0.1];\draw  [fill ] (10,0) circle [radius=0.1];

\draw  [fill ] (1,-1) circle [radius=0.1];\draw  [fill ] (2,-1) circle [radius=0.1];\draw  [fill ] (3,-1) circle [radius=0.1]; \draw  [fill ] (4,-1) circle [radius=0.1]; \draw  [fill ] (5,-1) circle [radius=0.1]; \draw  [fill ] (6,-1) circle [radius=0.1];\draw  [fill ] (7,-1) circle [radius=0.1];\draw  [fill ] (8,-1) circle [radius=0.1]; \draw  [fill ] (9,-1) circle [radius=0.1]; \draw  [fill ] (10,-1) circle [radius=0.1];

\draw  [fill ] (2,-2) circle [radius=0.1];\draw  [fill ] (3,-2) circle [radius=0.1]; \draw  [fill ] (4,-2) circle [radius=0.1]; \draw  [fill ] (5,-2) circle [radius=0.1]; \draw  [fill ] (6,-2) circle [radius=0.1];\draw  [fill ] (7,-2) circle [radius=0.1];\draw  [fill ] (8,-2) circle [radius=0.1]; \draw  [fill ] (9,-2) circle [radius=0.1]; \draw  [fill ] (10,-2) circle [radius=0.1];

\draw  [fill ] (3,-3) circle [radius=0.1]; \draw  [fill ] (4,-3) circle [radius=0.1]; \draw  [fill ] (5,-3) circle [radius=0.1]; \draw  [fill ] (6,-3) circle [radius=0.1];\draw  [fill ] (7,-3) circle [radius=0.1];\draw  [fill ] (8,-3) circle [radius=0.1]; \draw  [fill ] (9,-3) circle [radius=0.1]; \draw  [fill ] (10,-3) circle [radius=0.1];

 \draw  [fill ] (4,-4) circle [radius=0.1]; \draw  [fill ] (5,-4) circle [radius=0.1]; \draw  [fill ] (6,-4) circle [radius=0.1];\draw  [fill ] (7,-4) circle [radius=0.1];\draw  [fill ] (8,-4) circle [radius=0.1]; \draw  [fill ] (9,-4) circle [radius=0.1]; \draw  [fill ] (10,-4) circle [radius=0.1];
\draw[ultra thick, red] (0,0)--(5,0);
\draw[ultra thick, blue] (1,-1)--(6,-1)--(6,0);
\draw[ultra thick, black] (2,-2)--(7,-2)--(7,0);
\draw[thick, dotted] (1.7,-2.5)--(2.3,-3.3);
\end{tikzpicture}
\end{equation}
where the summation is over non-intersecting, up-right  paths on the trapezoidal lattice with weights $\by$, that start from vertices $(1,1),...,(i,i)$ in the lower-left border
of the lattice and go to vertices $(1,j-i+1),...,(1,j)$ at the upper border.
 As seen in the figure in relation \eqref{y-paths}, there is only one such $i$-tuple of paths with total weight 
$\prod_{1\leq a \leq  i \,, \,a \leq b \leq j} y^a_b$. Writing similarly $\tau^{i-1}_j, \tau^{i-1}_{j-1}, \tau^{i}_{j-1}$ we see that
\begin{align*}
\frac{\tau^i_i}{\tau^{i-1}_i}=y^i_i,\quad \text{and for $i<j$}\quad \frac{\tau^i_j\,\tau^{i-1}_{j-1}}{\tau^{i-1}_j\,\tau^i_{j-1}} = y^i_j.
\end{align*}
\end{proof}
\vskip 2mm
Let us now describe how geometric $\rsk$ can be encoded in this matrix formulation. We will make reference to the following diagram
 (refer to Section \ref{subsec:RSK} and relation \eqref{rs-iterate} for a reminder on the notation):
\begin{equation*}
\begin{array}{cccccccc}
           &\bx^1&          &  \bx^2=:\bx^{2,1}     &        & \bx^3=:\bx^{3,1} & &      \\[2pt]
\emptyset      &\cross &\by^{1,1}   &\cross              & \by^{2,1} &\cross          & \by^{3,1}   & \cdots\\[3pt]
&        &         & \bx^{2,2}               &        & \bx^{3,2}        & &      \\[2pt] 
&        &\emptyset      &\cross              & \by^{2,2}&\cross          &\by^{3,2}   & \cdots\\[3pt]
&        &         &          &                  & \bx^{3,3}        & &      \\[2pt]
&        &         &          & \emptyset &\cross          & \by^{3,3}  & \cdots\\[3pt]
&        &         &          &                  &         & &      \\
&        &         &          &       &         & \emptyset & \cdots \\[3pt]
\end{array}
\end{equation*}
where $\bx^i=(x^i_1,...,x^i_N)$ for $i\geq 1$, is a sequence of words, which are successively row inserted via geometric $\rsk$. 
We should keep in mind the useful identification of entries $x^i_j$ with number of letters `$j$' in a `{\it word} ' 
 $\bx^i$ as in \eqref{word}.
\vskip 2mm
Let us now describe geometric $\rsk$ in this matrix
 language translating essentially from the language of $\rsk$ as described in Section  \ref{subsec:RSK}.  For this reason we will
 be using the terms {\it tableau, row insertion, word} etc.
 
Initially, we have an empty tableau $\emptyset$ to which we insert the first word $\bx^1$ as
$\crossrsk{\bx^1}{\emptyset}{\by^{1,1}}{\emptyset}$.
 Of course, in this situation  the output tableau will only have one row $\by^{1,1}$ and
 $\by^{1,1}=\bx^1$, which we trivially encode via the matrix equation 
\begin{align}\label{H-eq1}
H(\bx^1)=H(\by^{1,1}).
\end{align}
Next, to the single-row tableau $\by^{1,1}$ we insert $\bx^2$ as $\crossrsk{\bx^2}{\by^{1,1}}{\by^{2,1}}{\bx^{2,2}}$.
 Here, $\by^{2,1}$  corresponds to
the updated first row of the new tableau and $\by^{2,2}=\bx^{2,2}$ to its second row. We notice that $\bx^{2,2}$ is the word consisting of the dropdown letters after the insertion of $\bx^2$ into $\by^{1,1}$, which are then inserted into the empty second row and this is the reason why the second row of the updated tableau
 $\by^{2,2}$ equals $\bx^2$. This second set of insertions can be be encoded via a matrix equation, which is derived by multiplying 
 \eqref{H-eq1} on the right by $H(\bx^{2})$ and using \eqref{H-eq} to obtain
 \begin{align}\label{H-eq2}
 H(\bx^1) H(\bx^2) = H_1(\by^{1,1}) H(\bx^2) \stackrel{\eqref{H-eq}}{=} H_2(\by^{2,2})   H_1(\by^{2,1}).
 \end{align} 
 The fact that this matrix multiplication and the output variables $\by^{2,2}, \by^{2,1}$ give the output of geometric row insertion
 with input $\bx^1,\bx^2$ is a consequence of Definition \ref{def:grsk}, its matrix reformulation \eqref{E-eq}, \eqref{H-eq}
  and finally Theorem \ref{thm:NY}.
  
 In a similar way, we encode the third group of row insertions :
 \begin{align*}
 \crossrsk{\bx^3}{\by^{2,1}}{\by^{3,1}}{\bx^{3,2}}\qquad & \text{[word $\bx^3$ is row inserted to the first line $\by^{2,1}$ of the current tableau]}, \\
  \crossrsk{\bx^{3,2}}{\by^{2,2}}{\by^{3,2}}{\bx^{3,3}} \qquad & \text{[word $\bx^{3,2}$ formed by the dropped down letters are
    row  }\\
    &\quad \text{inserted to the second line $\by^{2,2}$ of the current tableau]}, \\
 \crossrsk{\bx^{3,3}}{\emptyset}{\by^{3,3}}{\emptyset} \qquad &\text{[the dropdown letters from the previous insertion form the new row $\by^{3,3}$]},
 \end{align*}
  via multiplying on the right \eqref{H-eq2} by $H(\bx^3)$
 and using successively relation \eqref{H-eq} as
 \begin{align*}
 H(\bx^1) H(\bx^2) H(\bx^3) &= H_2(\by^{2,2})   H_1(\by^{2,1}) H(\bx^3)  
 \stackrel{\eqref{H-eq}}{=}  H_2(\by^{2,2})   H_2(\bx^{3,2}) H_1(\by^{3,1})\\
& \stackrel{\eqref{H-eq}}{=}  H_3(\bx^{3,3})   H_2(\by^{3,2}) H_1(\by^{3,1})
 = H_3(\by^{3,3})   H_2(\by^{3,2}) H_1(\by^{3,1})
 \end{align*}
 We point out that the second equality above is a matrix representation of the diagram 
 $\crossrsk{\bx^3}{\by^{2,1}}{\by^{3,1}}{\bx^{3,2}}$, the third equality of the diagram 
 $ \crossrsk{\bx^{3,2}}{\by^{2,2}}{\by^{3,2}}{\bx^{3,3}}$ and the fourth of the (trivial) diagram 
 $ \crossrsk{\bx^{3,3}}{\emptyset}{\by^{3,3}}{\emptyset} $.
 
 This procedure continues during the first $N$ insertions at which stage the resulting tableau will have the full depth of $N$
 rows. After that, no additional rows will be created in the subsequent tableaux and the process continues as follows
\begin{equation*}
\begin{array}{cccccccc}
           &\bx^{N+1}=:\bx^{N+1,1}&           & \bx^{N+2}=:\bx^{N+2,1}    &        &  \bx^{N+3}=:\bx^{N+3,1} & &      \\[2pt]
\by^{N,1}        &\cross         &\by^{N+1,1}      &\cross              & \by^{N+2,1} &\cross          & \cdots  &\\[3pt]
           &\bx^{N+1,2}          &           &\bx^{N+2,2}            &        & \bx^{N+3,2}         & &      \\[2pt]
\by^{N,2}         &\cross         &\by^{N+1,2}         &\cross              &  \by^{N+2,2} &\cross          &\cdots  & \\[3pt]
           &\bx^{N+1,3}           &           & \bx^{N+2,3}              &        & \bx^{N+3,3}         & &      \\[2pt]
\by^{N,3}      &\cross         &\by^{N+1,3}         &\cross              & \by^{N+2,3}    &\cross          & \cdots  &\\[3pt]
     &\bx^{N+1,4}           &           & \bx^{N+2,4}             &        & \bx^{N+3,4}       & &      \\
  \vdots         &\vdots           & \vdots          & \vdots               & \vdots       & \vdots           &\vdots & \\[3pt]
               & \bx^{N+1,N}          &           & \bx^{N+2,N}             &        & \bx^{N+3,N}        & &      \\[2pt]
\by^{N,N}       &\cross         &\by^{N+1,N}          &\cross              &  \by^{N+2,N}     &\cross          & \cdots  & \\[3pt]
               &\emptyset           &           &\emptyset              &        & \emptyset       & &      \\
\end{array}
\end{equation*}
Overall, the above, two-step procedure is encoded via the matrix equation
\begin{align}\label{matrixRSK}
\begin{cases} \,\,\,
H(\bx^1) H(\bx^2)\cdots H(\bx^n) = H_n(\by^{n,n}) H_{n-1}(\by^{n,n-1}) \cdots  H_{1}(\by^{n,1}),& \quad \text{if}\,\, n\leq N \\
&\\
\,\,\,H(\bx^1) H(\bx^2)\cdots H(\bx^n) = H_N(\by^{n,N}) H_{N-1}(\by^{n,N-1}) \cdots  H_{1}(\by^{n,1}), & \quad \text{if}\,\, n \geq N  
\end{cases}
\end{align}
\vskip 2mm
We have seen how to encode (geometric) row insertion in a matrix formulation, thus producing the $P$ tableau of geometric $\rsk$.
We can state the geometric $\rsk$ as a one-to-one correspondence 
between input matrices  $\sfX=(x^i_j \colon 1\leq i \leq n, \,1\leq j \leq N)$
and two sets of variables $P:=(p^i_j\colon 1\leq i\leq N\wedge n, \,i\leq j \leq N)$ and 
$Q:=(q^i_j\colon 1\leq i\leq N\wedge n, \,i\leq j \leq n)$. These will be the analogues of the $P$ and $Q$ tableaux in the 
standard $\rsk$ correspondence. For a full proof of this theorem (in particular the reconstruction of $\sfX$ from $(P,Q)$),
we refer to \cite{NY04}, Section 3 and Theorem 3.8.
\begin{theorem}\label{thm:grsk}
Consider a matrix $\sfX:=(x^i_j \colon 1\leq i \leq n, \,1\leq j \leq N)$ with nonnegative entries and denote by
$ (\bx^1,...,\bx^n)$ its rows and by $(\bx_1,...,\bx_N)$ its columns. 
Then there exists a one-to-one correspondence between $\sfX$ and a set of variables $\bp^{\,i}:=(p^{\.i}_i,...,p^{\,i}_N)$ for $i=1,...,\min(n,N)$
and  $\bq^{\,i}=(q^{\,i}_i,...,q^{\,i}_n)$ for $i=1,...,\min(n,N)$, which are uniquely determined via equations
\begin{align}
H(\bx^1) H(\bx^{2}) \cdots H(\bx^n) &= H_k(\bp^k) H_{k-1}(\bp^{k-1}) \cdots H_1(\bp^1) ,\qquad k=\min(n,N),\label{grsk-eq1}\\
H(\bx_1) H(\bx_{2}) \cdots H(\bx_N) &= H_k(\bq^k) H_{k-1}(\bq^{k-1}) \cdots H_1(\bq^1) ,\qquad k=\min(n,N),\label{grsk-eq2}
\end{align}
Variables $(p^i_j)$ and $(q^i_j)$ are given in terms of the input variables $(x^i_j)$ via relations \eqref{H-eq} and \eqref{H-weight}.
\end{theorem} 

 An immediate consequence of the formulation of $\grsk$ as in Theorem \ref{thm:grsk} and in particular the matrix equations
 \eqref{grsk-eq1},  \eqref{grsk-eq2}
   is that if the input matrix $\sfX$ is symmetric then the $P$ and $Q$ tableaux of $\grsk$ are equal.
 \vskip 2mm
  {\bf Geometric $\rsk$ on (geometric) Gelfand-Tsetlin patterns.}  It is useful to put geometric $\rsk$ and Theorem \ref{thm:grsk} under a Gelfand-Tsetlin
 framework. 
To this end, if $\bp^1,...,\bp^k$ and $\bq^1,...,\bq^k$ are as in \eqref{grsk-eq1}, \eqref{grsk-eq2}, set
 \begin{align*}
 z^i_j:= p^j_j \,p^{\,j}_{j+1}\cdots p^{\,j}_{i-1}\,p^{\,j}_i &\quad\text{for} \quad 1\leq j\leq i\leq N, \,\,\text{and}\,\, j\leq n\wedge N,\\
 (z^i_j)':= q^{\,j}_j \,q^{\,j}_{j+1}\cdots q^{\,j}_{i-1}\,q^{\,j}_i &\quad\text{for} \quad 1\leq j\leq i\leq n, \,\,\text{and}\,\, j\leq n\wedge N.
 \end{align*}
 Then Theorem \ref{thm:grsk} establishes a bijection between matrices $\sfX:=(x^i_j \colon 1\leq i \leq n, \,1\leq j \leq N)$ with nonnegative
 entries and a pair $(\sfZ,\sfZ')=\grsk(\sfX)$. 
 We will call the arrays $\sfZ=(z^i_j\colon 1\leq j\leq i\leq N, j\leq n\wedge N)$ and 
 $\sfZ'=((z^i_j)'\colon 1\leq j\leq i\leq n, j\leq n\wedge N)$, {\bf geometric Gelfand-Tsetlin patterns}, even though in general
  they do not satisfy the interlacing constraints $z^{i+1}_{j+1}\leq z^i_j \leq z^{i+1}_j$ (however, they do degenerate to
 genuine Gelfand-Tsetlin patterns in the combinatorial limit described in the next paragraph). For short, we will often denote 
 geometric Gelfand-Tsetlin patterns by $\gGT$, while at other times, when it is clear from the context, we may omit the adjective
 ``goemetric''.
 
 Bearing in mind property \eqref{H-uni} we obtain that variables $z^i_j$ are given in terms of ratios of partition functions
\begin{align}\label{GT_polymers}
z^i_j=\frac{\tau_i^{\,j}}{\tau_{i-1}^{\,j}}=\frac{\sum_{\pi_1,...,\pi_j\,} \sfw\sft(\pi_1)\cdots \sfw\sft(\pi_j)}{\sum_{\pi_1,...,\pi_{j-1}} \sfw\sft(\pi_1)\cdots \sfw\sft(\pi_{j-1})},
\end{align}
 where the sum in the numerator is over directed, 
non-intersecting paths along entries of $\sfX$ starting from $(1,1),...,(1,j)$ and ending at $(n,i-j+1),...,(n,i)$, respectively,
and the denominator sum is 
over directed, non-intersecting paths starting from $(1,1),...,(1,j-1)$ and ending at $(n,i-j+2),...,(n,i)$, respectively.
In particular,  
\begin{align}\label{zN1_poly}
z^N_1=\sum_{\pi\colon(1,1)\to(n,N)} \sfw\sft(\pi)= \sum_{\pi\colon(1,1)\to(n,N)} \, \prod_{(a,b)\in \pi} x^a_b,
\end{align}
 which defines the {\bf partition function} of the {\bf directed polymer model}.
\vskip 2mm
{\bf Passage to standard (combinatorial) $\rsk$ setting. }
We will now see how the geometric $\rsk$ framework degenerates to the standard $\rsk$ framework described in Section 
\ref{subsec:RSK} and how in this way we can obtain 
via Theorem \ref{thm:grsk} both
Schensted's and Greene's theorems as well as the links between $\rsk$ and last passage percolation alluded to in \eqref{lpp-rec}
and \eqref{lpp-rec-diag}.

Replacing in \eqref{gRSKrec} 
 variables $\xi_k,  \tilde\xi_k, a_k, b_k$ by $e^{\xi_k/\epsilon}, e^{\tilde\xi_k/\epsilon}, e^{a_k/\epsilon}, e^{b_k/\epsilon}$, taking the
 logarithm on both sides of each relation therein and multiplying by $\epsilon$, the set of equations  \eqref{gRSKrec} 
 may be written as
  \begin{equation}
\begin{cases}
 \tilde \xi_i= \xi_i + a_i, &\\[4pt]
\tilde \xi_k =a_k+ \epsilon \log \big( e^{\tilde\xi_{k-1}/\epsilon}+ e^{\xi_k/\epsilon}\big), & i+1\le k\le N\\[5pt]
b_{k} = a_k+ (\xi_k -  \xi_{k-1})- (   \tilde \xi_k - \tilde\xi_{k-1}), & i+1\leq k\leq N.
\end{cases}
\end{equation}
Taking now the limit $\epsilon\to 0$ these reduce to the piecewise linear transformations \eqref{RSKrec} defining the standard
$\rsk$ correspondence. Replacing also the variables $x^i_j, p^i_j, q^i_j$ in Theorem \ref{thm:grsk} by 
$e^{x^i_j/\epsilon},  e^{p^i_j/\epsilon}, e^{q^i_j/\epsilon}$ we obtain
in the limit $\epsilon\to0$ the $\rsk$ correspondence, in the sense that variables $(p^i_j)$ and $(q^i_j)$ encode the $P$ and $Q$ tableaux of the standard $\rsk$.  

In particular, the solution to the degeneration, as $\epsilon\to0$, of problem \eqref{grsk-eq1} is given via the degeneration of relations  \eqref{H-uni}, 
\eqref{H-weight}  as:
\begin{align*}
p^i_i=\sigma^i_i - \sigma^{i-1}_i
\quad \text{and}\quad p^i_j=\sigma^i_j\ + \sigma^{i-1}_{j-1} - \sigma^{i-1}_j-\sigma^i_{j-1}\quad \text{for} \quad i<j,
 \end{align*}
 \begin{align*}
\text{with}\qquad \sigma^i_j :=\max_{\pi_1,...,\pi_i\in \Pi^{1,...,i}_{j-i+1,...,j}} \sum_{k=1}^i\sfw\sft(\pi_k)
\end{align*}
 being last passage percolation functionals corresponding to ensembles of $i$ non-intersecting, down-right paths $\pi_1,...,\pi_i$, 
 starting from $(1,1),...,(1,i)$ and ending at $(n,j-i+1),...,(n,j)$, respectively.
The weight of a path $\pi_r$ in this case is $\sfw\sft(\pi_r):=\sum_{(a,b)\in \pi_r}  x^a_b$.

Passing to the Gelfand-Tsetlin variables, we set
 \begin{align*}
 z^i_j:= p^j_j +p^{\,j}_{j+1}+\cdots +p^{\,j}_{i-1} +p^{\,j}_i = 
  \sigma_i^{\,j}-\sigma_{i-1}^{\,j}  
 \end{align*}
for $1\leq j\leq i\leq N, \,\,\text{and}\,\, j\leq n\wedge N$. From this we get that
 \begin{align*}
 z^N_1+\cdots z^N_j:=  \sigma_N^{\,j},  
 \end{align*}
 which in the case $j=1$ is Schensted's theorem (see also \eqref{lpp-rec-diag}) and for $j>1$ is Greene's theorem.
 \vskip2mm

\subsection{$\rsk$ and geometric $\rsk$ via local moves}\label{sec:localmove}
In this section we are going to show how $\rsk$ and $\grsk$ can be presented  in an alternative way
 as a bijection between matrices and a composition of a sequence
of local operations that involve only $2\times2$ subarrays of the matrices. This formulation plays a crucial role in
extracting useful properties. An example that we will present here is Theorem \ref{thm:jac}. The construction of 
$\rsk$ and $\grsk$ as a sequence of local operations also forms the basis of constructions of statistical mechanics models 
that capture multipoint correlations of observables in the KPZ class. Particular examples that we will discuss in Section 
\ref{LPP-Airyprocess} are the {\it polynuclear growth process} (PNG) and the {\it Airy line ensemble}. 

Let us start with the presentation of $\rsk$ and geometric $\rsk$  as bijections between matrices with nonnegative entries.
This is done by gluing together the output (geometric or standard) Gelfand-Tsetlin 
patterns $(\sfZ,\sfZ')$. In particular, if $\sfX=(x_{i,j})_{1\leq i\leq n, 1\leq j\leq N}$ we can consider $\rsk(\sfX)$ or $\grsk(\sfX)$
as the $n\times N$ matrix $\sfT=(t_{i,j})_{1\leq i \leq n, 1\leq j\leq N}$ with 
\begin{align}\label{GTglue}
t_{i,j}=\begin{cases}
z^{n-i+j}_{n-i+1}  &,\quad \text{for}\quad 1\leq i\leq n, \,1\leq j\leq N,\,\,\,\, \text{with} \,\,\,\, i-j\geq n-N\\
&\\
(z^{N+i-j}_{N-j+1})' &,\quad \text{for}\quad 1\leq i\leq n, \,1\leq j\leq N,\,\,\,\, \text{with} \,\,\,\, i-j\leq n-N.
\end{cases}
\end{align}
More suggestively, we can look at the following pictures
\begin{equation*}
 {\begin{tikzpicture}[baseline=1.1ex, scale=.7]
\draw (0,0) -- (6,0) -- (6,4) -- (0,4) -- (0,0);  
\draw[dashed] (6,0) -- (2,4);
\node at (2,2) {$\sfZ$};  \node at (5,3) {$\sfZ'$};
\node at (3,-1) {glued $\GT$ patterns when $n\leq N$};
\end{tikzpicture}}
\qquad \qquad
 {\begin{tikzpicture}[baseline=1.1ex, scale=.7]
\draw (0,0) -- (4,0) -- (4,6) -- (0,6) -- (0,0);  
\draw[dashed] (4,0) -- (0,4);
\node at (1.2,1.2) {$\sfZ$};  \node at (2.5,3.4) {$\sfZ'$};
\node at (2,-1) {glued $\GT$ patterns when $n\geq N$};
\end{tikzpicture}}
\end{equation*}
where on the dashed line we put the shape variables $(z^N_1,...,z^N_{n\wedge N}) = \big( (z^n_1)',...,(z^n_{n\wedge N})' \big)$,
with $z^N_1=(z^n_1)'$ occupying the lower-right corner and $z^1_1, (z^1_1)'$ occupying the lower-left and the upper-right corners,
respectively.
Notice that when $n< N$, then the number $n$ of inserted words is smaller than the number $N$ of letters in the alphabet and then
the $\sfZ$ pattern does not have full triangular shape. This is because the corresponding Young tableau from the $\rsk$ insertion of
the $n$ words will only have at most $n$ rows, if $n\leq N$. On the other hand, when $n\geq N$, then the Young tableau
 may have more than $n$ rows but, still, it necessarily has at most $N$ rows, since the maximum number of rows in a Young tableau,
 which is an output of $\rsk$ involving $N$ letters, is at most $N$. So, in the case $n\geq N$ we represent the output Gelfand-Tsetlin pattern 
$\sfZ$ as a complete triangle of length $N$.
Similar considerations hold for the $\sfZ'$ pattern with the roles of $n,N$ reversed since the output tableau corresponding to the $\sfZ'$ pattern
is the outcome of insertions of numbers $1,...,n$ (this will be the $\sfQ$ tableau).
\vskip 2mm
As a bijection between nonnegative matrices, $\grsk$  has a remarkable volume preserving property as noted in \cite{OSZ14} 
and stated in the following theorem:
\begin{theorem}\label{thm:jac}
Let $\sfX=(x_{i,j})_{1\leq i\leq n, 1\leq j\leq N}$ be a matrix with nonnegative entries and $\sfT=\grsk(\sfX)$, where 
$\sfT=(t_{i,j})_{1\leq i\leq n, 1\leq j\leq N}$ is produced
by gluing together the Gelfand-Tseltin patterns of $\grsk$, as in \eqref{GTglue}. Then the mapping
\begin{align*}
(\log x_{i,j}\colon 1\leq i\leq n, 1\leq j\leq N) \mapsto (\log t_{i,j}\colon 1\leq i\leq n, 1\leq j\leq N),
\end{align*}
has Jacobian $\pm1$.
\end{theorem}
We will prove this theorem below, after setting up the appropriate framework. 
This property is not obvious and the fact that it also holds in two other cases where the input matrix $\sfW$ is square and
has either symmetry with respect to the diagonal, i.e. $w_{ij}=w_{ji}$ \cite{OSZ14}, or the anti-diagonal, i.e. $w_{ij}=w_{n+1-j,n+1-i}$ \cite{BOZ21} ,
seems to indicate that there may be a deeper algebraic or geometric reason behind it. With regards to this, the framework on \cite{L13}, Section 2.4
might offer a useful perspective. With regards to the value of the above theorem in the analysis of solvable KPZ models, it
will play a very important role in Section \ref{sec:loggamma}, Theorem \ref{thm:log-gamma-laplace},
when we will compute the Laplace transform of a solvable directed polymer model in terms of Whittaker functions.

It appears to be difficult to derive Theorem \ref{thm:jac} from Noumi-Yamada's matrix formulation but it turns out to be easily checked once
a different point of view on the $\grsk$ is adopted. The latter amounts to a decomposition of $\grsk$ (as well as $\rsk$) 
into a set of {\it local moves}, which we will describe next. This local move description is close 
to the construction of the integrable interface called {\bf polynuclear growth model}, which plays an important role in encoding
multipoint correlations and which we will describe in Section \ref{LPP-Airyprocess}. In fact, the construction of the latter has been motivated
by the local rules that we will describe now.
 We also note that the local moves in the combinatorial setting can be identified with Fomin's construction 
of $\rsk$ \cite{F86, F95}, see also \cite{Kr06}. The geometric setting corresponds to a geometric lifting of this to the $(+,\times)$ algebra.

Crucial towards the point of view of local moves are the relations
\begin{align}\label{local_drop}
b_{i+1}=  \min\big(\tilde \xi_i, \xi_{i+1}\big)-\xi_i,
\end{align}
(see also \eqref{bump1}), which in the standard $\rsk$ setting record
the number of letters $i$ being bumped down from a row after the insertions of letters `$i$' from a new word, and relation
\begin{align}\label{local_new}
\tilde \xi_{i+1}=\max\big(\tilde \xi_i,\xi_{i+1}\big)+a_{i+1},
\end{align}
(see also \eqref{bump2}), which in the standard $\rsk$ setting records
the length up to letter `$i+1$' of a row in a Young tableau after the insertion of the string of letters `$i+1$' from the new word.
The geometric lifting of these relations (i.e. replacing $(\max,+)$ operations with $(+\times)$) may be written as
\begin{align*}
b_{i+1}= \frac{\tilde \xi_i \,\xi_{i+1}}{\xi_i(\tilde \xi_i+\xi_{i+1})} 
\qquad \text{and}\qquad 
\tilde \xi_{i+1}=a_{i+1}\,\big(\tilde \xi_i+\xi_{i+1}\big).
\end{align*}
Keeping these in mind as motivation, let us define the {\bf local moves}: The fundamental transformation is 
  the mapping $\ell_{i,j}$, for $2\leq i \leq n$ and $2\leq j\leq N$, which transforms a matrix $\sfX$ by 
   replacing only its $2\times 2$ submatrix (leaving all other entries unchanged)
\begin{align*}
\left(\begin{array}{cc} x_{i-1,j-1} & x_{i-1,j}\\ x_{i,j-1} & x_{i,j} \end{array} \right)
\end{align*}
by its image under the map 
\begin{align}\label{eq:localmove}
&\\
& \left(\begin{array}{cc} a & b\\ c & d\end{array} \right) \mapsto 
\left(\begin{array}{cc} b\wedge c-a & b\\ c & d+b\vee c \end{array}\right)
\qquad \text{or} \qquad \left(\begin{array}{cc} a & b\\ c & d\end{array} \right) \mapsto 
\left(\begin{array}{cc} \frac{b c}{a(b+c)} & b\\ c & d(b+ c) \end{array}\right) \notag
\end{align}
in the combinatorial (standard) $\rsk$ or geometric setting
(notice that the geometric lifting of $\min(b,c)=-\max(-b,-c)$ is $(b^{-1}+c^{-1})^{-1}$), respectively.
 In the case $j=1$, the transformation $\ell_{i,1}$ only changes the subarray 
\begin{equation*}
\left(\begin{array}{c} x_{i-1,1} \\ x_{i,1} \end{array}\right) 
\mapsto \left(\begin{array}{c} x_{i-1,1} \\ x_{i,1} + x_{i-1,1} \end{array}\right)
\qquad\text{or}\qquad 
\left(\begin{array}{c} x_{i-1,1} \\ x_{i,1} \end{array}\right) 
\mapsto \left(\begin{array}{c} x_{i-1,1} \\ x_{i,1}  \cdot x_{i-1,1} \end{array}\right)
\end{equation*}
 while in the case $i=1$ the transformation $\ell_{1,j}$ only changes the subarray $(x_{1,j-1} \,,\, x_{1,j})$ to $(x_{1,j-1} \,,\, x_{1,j} + x_{1,j-1})$
 in the combinatorial setting or to $(x_{1,j-1} \,,\, x_{1,j} \cdot x_{1,j-1})$ in the geometric setting.
 The transformation $\ell_{1,1}$ is the identity. 

Based on this, the row insertion of the $k^{th}$ row of the input matrix $\sfX$, which we denote by $R_k$,
is decomposed as
\begin{align*}
R_k=\rho^k_N\circ \cdots \circ \rho^k_2\circ \rho^k_1 
\end{align*}
with
\begin{align}\label{rho-comp}
\rho^i_j:=\begin{cases}
\ell_{1,j-i+1}\circ \cdots \circ \ell_{i-1,j-1}\circ \ell_{i,j}\quad& \text{if} \quad i\leq j,\\
\\
\ell_{i-j+1,1}\circ\cdots\circ\ell_{i-1,j-1} \circ \ell_{i,j} \quad & \text{if} \quad i\geq j.
\end{cases}
\end{align}
The fact that the composition $R_n\circ\cdots\circ R_1$ in the geometric setting is equivalent to Noumi-Yamada's matrix formulation
was shown in \cite{OSZ14}, Section 4.
A pictorial representation of $\rho^4_2$ and $\rho^4_5$ is shown in Figure \ref{pictorial-local2}.
\vskip 1mm
\begin{figure}
{\begin{tikzpicture}[scale=0.5]
\draw  [fill=red ] (2,12) circle [radius=0.2]; \draw  [fill=red ] (4,10) circle [radius=0.2]; \draw  [fill=red ] (6,8) circle [radius=0.2];
\draw  [fill=red ] (8,6) circle [radius=0.2];
\draw  [fill=blue ] (2,6) circle [radius=0.2]; \draw  [fill=blue ] (0,8) circle [radius=0.2];
\draw[dashed] (7.5,12.5)--(12.5,7.5); \draw[dashed] (-0.5,8)--(12.5,8);
\draw (7,7) circle [radius=1.6]; \draw (5,9) circle [radius=1.6]; \draw (3,11) circle [radius=1.6]; \draw (1,12) ellipse (1.8cm and 1cm);
\draw (1,7) circle [radius=1.6]; \draw (0,9) ellipse (1cm and 1.8cm);
\draw [thick, decorate,decoration={brace,amplitude=4pt},xshift=0.5cm,yshift=0pt]   (13,12) -- (13,8) node [midway,right,xshift=.1cm] 
{   $(\sfZ,\sfZ')-\GT$ patterns before update };
\draw [thick, decorate,decoration={brace,amplitude=4pt},xshift=0.5cm,yshift=0pt]   (13,4) -- (13,0) node [midway,right,xshift=.1cm] 
{   not yet inserted $\sfX$ entries };
\node at (18.5, 6) {$\longleftarrow\,\,\,$  currently inserted $\sfX$-row};
\node at (7,7) {$\ell_{4,5}$}; \node at (5,9) {$\ell_{3,4}$}; \node at (3,11) {$\ell_{2,3}$};\node at (0.9,12) {$\ell_{1,2}$};
\node at (1,7) {$\ell_{4,2}$}; \node at (0,9) {$\ell_{3,1}$};
\foreach \i in {0,1,...,6}
\foreach \j in {0,...,6}{
		\node[draw,circle,inner sep=1pt,fill] at (2*\i,2*\j) {};
}
\end{tikzpicture}
}
\caption{This figure shows the sequence of moves involved in the application of operations  $\rho^4_2$ and $\rho^4_5$.
The blue nodes indicate the entries that will be modified after application of $\rho^4_2$ and the red those that will be modified
after application of $\rho^4_5$. The circles indicate the entries that are involved in the $\ell$ operation as indicated in the center
of the circles.
 }\label{pictorial-local1}
\end{figure} 
 \begin{figure}
{\begin{tikzpicture}[scale=0.5]
\draw[dashed] (5.5,12.5)--(12.5,5.5); \draw[dashed] (-0.5,6)--(12.5,6);
\draw [thick, decorate,decoration={brace,amplitude=4pt},xshift=0.5cm,yshift=0pt]   (13,12) -- (13,6) node [midway,right,xshift=.1cm] 
{   updated patterns $(\tilde\sfZ,\tilde\sfZ')$ after $R_4$ };
\draw [thick, decorate,decoration={brace,amplitude=4pt},xshift=0.5cm,yshift=0pt]   (13,4) -- (13,0) node [midway,right,xshift=.1cm] 
{   not yet inserted $\sfX$ entries };
\foreach \i in {0,1,...,6}
\foreach \j in {0,...,6}{
		\node[draw,circle,inner sep=1pt,fill] at (2*\i,2*\j) {};
}
\end{tikzpicture}
}
\caption{ Figure \ref{pictorial-local1} showed two snapshots of the insertion of the fourth row via operation $R_4$.
 Before the insertion of the fourth row, the first three lines in Figure \ref{pictorial-local1}
comprised the $\sfZ$ and $\sfZ'$ patterns (glued together) and the rest of the lines comprised the rows $(x_{4,j})_{j\geq 1}, (x_{5,j})_{j\geq 1},...$
of the input matrix $\sfX$,
which were waiting to be inserted via operations $R_4,R_5,...$. This figure shows the form of the array after operation $R_4$, where 
now the first {\it four} rows comprise the current 
patterns $\sfZ$ and $\sfZ'$ (glued together) and the rest comprise the remaining row of the input matrix
$\sfX$, which are waiting to be inserted via operations $R_5,R_6,...$. 
}
\label{pictorial-local2}
\end{figure}
\vskip 1mm
Let us try to explain this decomposition. We do so by looking at the example of transformations $R_4$ and $\ell_{4,5}$ in the
combinatorial setting.
  We start by noticing that $R_4$ will eventually 
  transform the fourth row of the array in Figure \ref{pictorial-local1} to the (new) rightmost diagonal, $\tilde z_1$, of the new Gelfand-Tsetlin pattern $\tilde \sfZ$. In particular, after application of 
  $\rho^4_1, \rho^4_2, \rho^4_3, \rho^4_4$ the first four entries of the fourth row will be transformed to
  the updated values of the Gelfand-Tsetlin pattern: $\tilde z^{\,1}_1, \tilde z^{\,2}_1, 
  \tilde z^{\,3}_1$ and $\tilde z^{\,4}_1$, respectively.  Next, we want to apply $\rho^5$ and from \eqref{rho-comp} this will start with the application of
  $\ell_{4,5}$. The latter  will be applied as (see also Figure \ref{pictorial-local1})
\begin{align*}
\left({\begin{array}{cc}
z^4_1 & z^5_1 \\
&\\
\tilde z^{\,4}_1 & x_{4,5}
\end{array}
}
\right)
\mapsto
\left({\begin{array}{cc}
\min\big(\tilde z^{\,4}_1, z^{\,5}_1\big)-z^{\,4}_1 & z^{\,5}_1 \\
&\\
\tilde z^{\,4}_1 & x_{4,5}+\max\big( \tilde z^{\,4}_1\,,\,z^5_1 \big)
\end{array}
}
\right).
\end{align*}
Comparing to \eqref{local_drop} and \eqref{local_new} we see that the bottom-right entry in the right-hand side matrix above coincides with
$\tilde z^{\,5}_1$, i.e.
 the total length of letters up to `$5$' in the first row of the $P$ tableau after the 
 insertion of $x_{4,5}$-many $5$'s, which takes place during the insertion of the fourth row of the input matrix via $R_4$. 
 On the other hand, the top-left entry in the matrix of the right-hand side above records the number of letters `$5$' that have been bumped down
 and which will then be inserted via $\ell_{3,4}$ to the next row of the corresponding tableau. The bumping of $5's$ then continues,
  operations $\ell_{2,3}$ and $\ell_{1,2}$,  until the end of the 
 tableau is reached.
\vskip 2mm
The local move decomposition of the $\grsk$ makes the proof of Theorem \ref{thm:jac} easy:
 \begin{proof}[Proof of Theorem \ref{thm:jac}]
 It is straightforward and easy to check that the Jacobian of each transformation $\ell_{i,j}$ in logarithmic variables is $\pm1$.
 Therefore $\grsk$ is also volume preserving as a sequence of compositions of $\ell_{i,j}$ operations.
 \end{proof}
 
 One more advantage of the local move decomposition of $\rsk$ and $\grsk$ is that it can be generalised as a bijective 
 map between polygonal arrays of the form of Young tableaux type as: 
 \begin{equation*}
X=
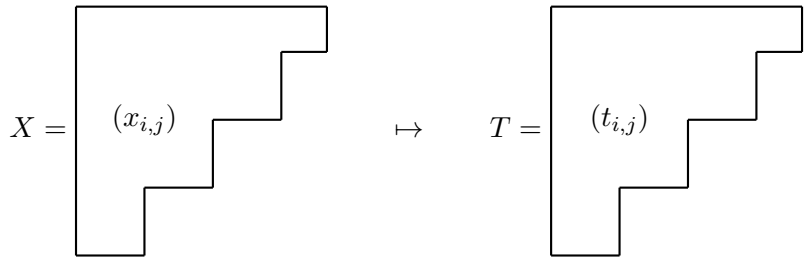
\begin{tikzpicture}[baseline={([yshift=-.5ex]current bounding box.center)},vertex/.style={anchor=base,
    circle,fill=black!25,minimum size=18pt,inner sep=2pt}, scale=0.3]
\draw[thick] (0,0)--(0,11); 
\draw[thick] (0,11)--(11,11);
\draw[thick] (11,11)--(11,9); 
\draw[thick] (11,9)--(9,9);
\draw[thick] (9,9)--(9,6);
\draw[thick] (9,6)--(6,6); 
\draw[thick] (6,6)--(6,3);
\draw[thick] (6,3)--(3,3); 
\draw[thick] (3,3)--(3,0);
\draw[thick] (3,0)--(0,0); 
 \node at (3,6) {$(x_{i,j})$};
\end{tikzpicture}
\qquad\mapsto\qquad 
T=\begin{tikzpicture}[baseline={([yshift=-.5ex]current bounding box.center)},vertex/.style={anchor=base,
    circle,fill=black!25,minimum size=18pt,inner sep=2pt}, scale=0.3]
\draw[thick] (0,0)--(0,11); 
\draw[thick] (0,11)--(11,11);
\draw[thick] (11,11)--(11,9); 
\draw[thick] (11,9)--(9,9);
\draw[thick] (9,9)--(9,6);
\draw[thick] (9,6)--(6,6); 
\draw[thick] (6,6)--(6,3);
\draw[thick] (6,3)--(3,3); 
\draw[thick] (3,3)--(3,0);
\draw[thick] (3,0)--(0,0); 
 \node at (3,6) {$(t_{i,j})$};
\end{tikzpicture}
\end{equation*}
This mapping has the property that the outer corners of the output array $T$ are equal to the last passage percolation times or polymer
partition functions of down-right paths from $(1,1)$ to the corresponding corners.
 More precisely, consider $\sfX=\big(x_{i,j}\colon (i,j)\in ind(\sfX)\big)$ a polygonal array with a suitable index set $ind(\sfX)$; if
  $(i,j)$ is such that $(i,j)\in ind(X)$ and $(i+1,j), (i,j+1)\notin ind(\sfX)$ then
 \begin{equation}\label{local_greene}
 \begin{split}
  t_{i,j}=\max_{\pi\colon (1,1)\to (i,j)} \,\sum_{(i,j)\,\in \,ind(\sfX)} x_{i,j}&\qquad \text{for} \quad \sfT=\rsk(\sfX)\qquad \text{or}\\
 t_{i,j}=\sum_{\pi\colon (1,1)\to (i,j)} \,\prod_{(i,j) \,\in\, ind(\sfX)} x_{i,j}&\qquad \text{for} \quad \sfT=\grsk(\sfX).
 \end{split}
 \end{equation}
 This is depicted in the following picture where the coloured vertices in the right-hand side are equal to the last passage percolation or 
 polymer partition functions of the paths with the same colour in the left-hand side:
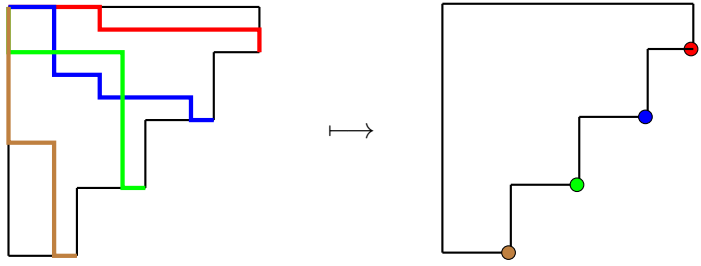
\begin{equation}\label{rsk-array}
\begin{tikzpicture}[baseline={([yshift=-.5ex]current bounding box.center)},vertex/.style={anchor=base,
    circle,fill=black!25,minimum size=18pt,inner sep=2pt}, scale=0.3]
\draw[thick] (0,0)--(0,11); 
\draw[thick] (0,11)--(11,11);
\draw[thick] (11,11)--(11,9);
\draw[thick] (11,9)--(9,9);
\draw[thick] (9,9)--(9,6);
\draw[thick] (9,6)--(6,6);
\draw[thick] (6,6)--(6,3);
\draw[thick] (6,3)--(3,3);
\draw[thick] (3,3)--(3,0);
\draw[thick] (3,0)--(0,0);
\draw[ultra thick, red] (0,11)--(4,11)--(4,10)--(11,10)--(11,9);
\draw[ultra thick, blue] (0,11)--(2,11) -- (2,8)--(4,8)--(4,7)--(8,7)--(8,6)--(9,6);
\draw[ultra thick, green] (0,11)--(0,9) -- (5,9)--(5,3)--(6,3);
\draw[ultra thick, brown] (0,11)--(0,5) -- (2,5)--(2,0)--(3,0);
\end{tikzpicture}
\qquad \longmapsto \qquad 
\begin{tikzpicture}[baseline={([yshift=-.5ex]current bounding box.center)},vertex/.style={anchor=base,
    circle,fill=black!25,minimum size=18pt,inner sep=2pt}, scale=0.3]
\draw[thick] (0,0)--(0,11); 
\draw[thick] (0,11)--(11,11);
\draw[thick] (11,11)--(11,9); \draw  [fill=red ] (10.9,9) circle [radius=0.3];
\draw[thick] (11,9)--(9,9);
\draw[thick] (9,9)--(9,6);
\draw[thick] (9,6)--(6,6); \draw  [fill=blue ] (8.9,6) circle [radius=0.3];
\draw[thick] (6,6)--(6,3);
\draw[thick] (6,3)--(3,3); \draw  [fill=green ] (5.9,3) circle [radius=0.3];
\draw[thick] (3,3)--(3,0);
\draw[thick] (3,0)--(0,0); \draw  [fill=brown ] (2.9,0) circle [radius=0.3];
\end{tikzpicture}
\end{equation}
 
 This property has been useful towards identifying joint, multi-point laws of last passage percolation and polymer models 
 \cite{J03, NZ17} (with the former leading to a full convergence result towards the Airy process). 
We will further discuss this, in relation to the polynuclear growth process, in Section \ref{LPP-Airyprocess}. 

\section{A solvable Last Passage Percolation model}\label{geomLPP}
We now have all the tools to analyse an exactly solvable model in the KPZ class. This is the last passage percolation with geometric variables
\footnote{the adjective geometric here has the more standard probabilistic interpretation and is not related to the way it was used to describe the geometric lifting of $\rsk$ earlier.}
In this section we will show how $\rsk$ allows to write explicitly the distribution of the last passage percolation time in terms of Schur functions.
The determinantal expression of the Schur functions will then set the stage for asymptotic analysis, which we will present in the following section.

To set things up, we consider a matrix $\sfW=(\bw^{\,i}_j)_{\substack{ 1\leq i \leq m,\\ 1\leq j \leq n}}$, where we assume that the entries are independent random variables with geometric distribution
\begin{align}\label{geom}
\bbP(\bw^i_j=w^i_j) = (1-p_iq_j) (p_iq_j)^{w^i_j} \,\ind_{w^i_j\in\{0,1,2,...\}},\qquad 1\leq i\leq m\,,\, 1\leq j\leq n,
\end{align}
where $p_i,q_j$ are parameters in $(0,1)$. The first question we want to ask is whether we can compute the law of
\begin{align}\label{def:geomlpp}
\tau_{m,n} :=\max_{\pi\in {\bf \Pi}_{m,n}} \sum_{(i,j)\in \pi} \bw^{\,i}_j,
\end{align}
where ${\bf\Pi}_{m,n}$ is the set of down-right paths going from site $(1,1)$ to site $(m,n)\in\N^2$ 
(using the matrix rather than the cartesian index notation). 
For simplicity, let us assume that $m=n=N$, although the general 
case can also be treated following similar reasoning. For conciseness we will also denote $\tau_{N,N}$ by $\tau_N$. 
The answer to this question is affirmative
 and the reason is that the geometric distribution fits the framework and the properties of $\rsk$. In particular, we can 
 answer the posed question by carrying out the following steps:
\vskip2mm
{\bf Step 1. Combinatorial analysis.} $\rsk$ gives a bijection between a matrix with nonnegative entries and a pair of Young tableaux $(P,Q)$ or equivalently a pair of $\GT$ patterns $(\sfZ,\sfZ')$. In particular,
 $\tau_{N}$ equals the length of the first row of the tableaux $P,Q$, which in $\GT$ parametrisation is $z^N_1=(z^N_1)'$. 
 
 This step may naturally bring up the question:
 ``since we are only interested in $z^n_1$ why do we need all the extra quantities $(z^i_j\,,\, (z^i_j)';\, 1\leq j\leq i\leq N)$ coming from $\rsk$ '' ? To give an answer to
 this, let us look back at the Hammersley process, Figure \ref{fig:hammer}. In order to find the length of the up-right path that collects the largest number
 of points, we have to count how many vertical or horizontal lines reach the upper or right side of the square. The difficulty lies in the fact that there
 are many cancellations of rays inside the square (when a horizontal and a vertical line meet they annihilate each other) that are not directly visible
 on the top or right sides of the square. We can try to recover 
 lost information as follows: from the annihilation points we draw a second (red) generation of vertical and horizontal rays, which
 again annihilate each other when they meet. Then from the annihilation points of the red rays we start a third (blue) generation of vertical and and horizontal rays, which
 again annihilate each other when they meet and so on as in the following figure
 
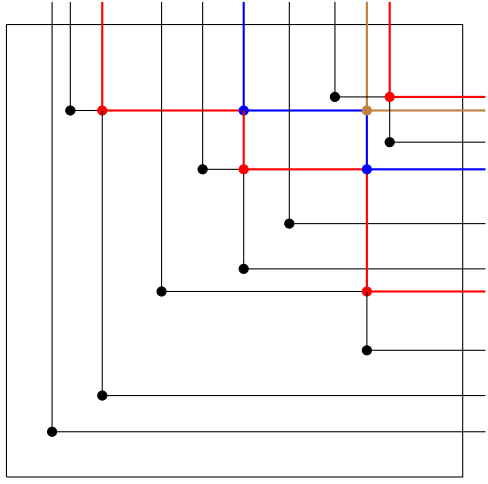
\begin{equation*}
\begin{tikzpicture}[scale=.6]
\draw (0,0) -- (0,10)--(10,10)--(10,0)--(0,0) ;
\draw  [fill ](1,1)  circle [radius=0.1]; \draw (1,1)--(1,10.5);  \draw (1,1)--(10.5, 1);
\draw  [fill ](1.4, 8.1)  circle [radius=0.1]; \draw (1.4,8.1)--(1.4,10.5); \draw (1.4,8.1)--(2.1,8.1); \draw  [red, fill ](2.1,8.1)  circle [radius=0.1] ; \draw [red, thick] (2.1,8.1)--(2.1,10.5); \draw [red, thick](2.1,8.1)--(5.2,8.1);
\draw  [blue, fill ](5.2,8.1)  circle [radius=0.1]; \draw [blue, thick]  (5.2,10.5)--(5.2,8.1) --(7.9,8.1)--(7.9,6.8)--(10.5,6.8);
\draw  [brown, fill ] (7.9,8.1)circle [radius=0.1]; \draw [brown, thick] (7.9,10.5)--(7.9,8.1)--(10.5,8.1);
\draw  [fill ](2.1, 1.8)  circle [radius=0.1]; \draw(2.1,1.8)--(2.1,8.1); \draw(2.1,1.8)--(10.5, 1.8);
\draw  [fill ](3.4, 4.1)  circle [radius=0.1]; \draw(3.4, 4.1) --(3.4,10.5); \draw(3.4,4.1)--(7.9, 4.1);
\draw  [fill ](6.2, 5.6)  circle [radius=0.1]; \draw  (6.2, 5.6)--(6.2,10.5); \draw (6.2, 5.6)--(10.5, 5.6);
\draw  [fill ](5.2, 4.6)  circle [radius=0.1]; \draw(5.2, 4.6)--(10.5, 4.6); \draw(5.2, 4.6)--(5.2, 6.8);
\draw  [fill ](4.3, 6.8)  circle [radius=0.1]; \draw (4.3, 6.8)--(5.2, 6.8);   \draw (4.3, 6.8)--(4.3, 10.5); \draw  [red, fill ] (5.2, 6.8) circle [radius=0.1]; \draw [red, thick ](5.2, 6.8)--(5.2,8.1);
 \draw [red, thick ](5.2, 6.8)--(7.9,6.8)--(7.9,4.1)--(10.5,4.1);
\draw  [red, fill ] (7.9,4.1) circle [radius=0.1];
\draw  [blue, fill ] (7.9,6.8) circle [radius=0.1]; 
\draw  [fill ](7.9, 2.8)  circle [radius=0.1]; \draw(7.9, 2.8)--(7.9,4.1); \draw(7.9, 2.8)--(10.5,2.8); 
\draw  [fill ](7.2, 8.4)  circle [radius=0.1]; \draw (7.2, 8.4)--(7.2,10.5);  \draw (7.2, 8.4)--(8.4,8.4);
\draw  [fill ](8.4, 7.4)  circle [radius=0.1]; \draw (8.4, 7.4)--(8.4, 8.4); \draw(8.4, 7.4)--(10.5, 7.4); 
\draw  [red,fill ](8.4, 8.4)  circle [radius=0.1];\draw [red, thick] (8.4, 8.4)--(10.5, 8.4); \draw [red, thick](8.4, 8.4)--(8.4, 10.5);
\end{tikzpicture}
 \end{equation*}
 Having all the generations of rays arriving at the two sides of the square is enough to restore the whole information. The information of the rays of all generations that
 arrive to the sides is actually the information that is contained in $\rsk$ (this description of cancelling rays is actually Viennot's construction of $\rsk$). Having now all the information the question comes down to whether
 the probability distribution of the image of $\rsk$ as well as the marginal over the quantity of interest $z^N_1$ are tractable. This is handled in the next two steps. 
\vskip 2mm
{\bf Step 2. Push forward law.} The law of the random weight matrix 
$\sfW$ from \eqref{geom} can be written explicitly in terms of $\GT$ variables as
\begin{align}\label{push1}
\bbP(\sfW=\{w^i_j\}) &= \prod_{i,j}(1-p_iq_j) \, \prod_i p_i^{\sum_j w^i_j}\, \prod_{j} q_j^{\sum_i w^i_j}\notag\\
&= \prod_{i,j}(1-p_iq_j) \, \prod_i p_i^{|(z^i)'|-|(z^{i-1})'|}\, \prod_{j} q_j^{|z^j|-|z^{j-1}|}.
\end{align}
Here we related $\sum_i w^i_j$ to the {\it type} of the $P$-tableau, $\big(|z^j|-|z^{j-1}|\big)_{j=1,...,N}$  as $\sum_i w^i_j = |z^j|-|z^{j-1}|$
  and the sum $\sum_j w^i_j$  to the {\it type} of the $Q$-tableau $|(z^i)'|-|(z^{i-1})'|$ as $\sum_j w^i_j =|(z^i)'|-|(z^{i-1})'|$.
   The former is just \eqref{type}, while the latter follows from the fact that 
 \begin{align*}
 \text{if $\rsk(\sfW)=(P,Q)=(\sfZ,\sfZ')$, then $\rsk(\sfW^\sft)=(Q,P)=(\sfZ',\sfZ)$}.
 \end{align*}
 \vskip 2mm
{\bf Step 3. Marginalisation and determinantal measures.}
 We are now ready to compute $\bbP(\tau_N\leq u)$. Using the previous two steps we have that
\begin{align*}
\bbP(\tau_N\leq u) =\sum_{\gl\colon \gl_1\leq u} \,\,
\sumtwo{(Z,Z')\, \text{pair of $\GT$ patterns}}{ \text{with shape $\gl$  } }\bbP\big(\rsk(\sfW) = (Z,Z')\big)
\end{align*}
 and by \eqref{push1} this equals 
\begin{align*}
&\prod_{i,j}(1-p_iq_j)\, \sum_{\gl\colon \gl_1\leq u} \,\,
\sumtwo{(Z,Z')\, \text{pair of $\GT$ patterns}}{ \text{with shape $\gl$  } }  \prod_i p_i^{|(z^i)'|-|(z^{i-1})'|}\, \prod_{j} q_j^{|z^j|-|z^{j-1}|} \\
&=  \prod_{i,j}(1-p_iq_j)\, \sum_{\gl\colon \gl_1\leq u} \,\,
\sumtwo{Z\colon  \text{$\GT$ pattern}}{ \text{with shape $\gl$  } }  \prod_{j} q_j^{|z^j|-|z^{j-1}|}
 \sumtwo{Z'\colon \text{$\GT$ pattern}}{ \text{with shape $\gl$  } }  \prod_i p_i^{|(z^i)'|-|(z^{i-1})'|},
\end{align*}
and now each of the two rightmost summands are recognised to be the 
 Schur functions, whose expression as generating series of Young tableaux \eqref{Schur-gen-tableau} 
 may be rewritten in the Gelfand-Tsetlin notation as
\begin{align}\label{GTschur}
s_\gl(q) := \sumtwo{Z\colon  \text{$\GT$ pattern}}{ \text{with shape $\gl$  } }  \prod_{j} q_j^{|z^j|-|z^{j-1}|}.
\end{align}
Thus,
the above induces that (this step is really a change of notation)
\begin{align}\label{schur1}
 \bbP(\tau_N\leq u)= \prod_{i,j}(1-p_iq_j)\, \sum_{\gl\colon \gl_1\leq u} s_\gl(q)  \,s_\gl(p).
\end{align}
We have, thus, computed the law of last passage percolation in terms of special functions, which furthermore possess many nice properties.
In particular, they can be written in terms of determinants and in fact there are more than one such formulae. For example,
 if $\lambda=(\gl_1,\gl_2,...)$ is a partition and $p_1,p_2,...,p_N$ are nonnegative parameters (or variables), then 
\begin{align}\label{schur_form}
s_\lambda(p)=\frac{\det\big(p_i^{\gl_j+N-j}\big)_{1\leq i,j\leq N}}{\det\big(p_i^{N-j}\big)_{1\leq i,j\leq N}},
\end{align} 
where in the denominator one recognises the Vandermonde determinant, which can be computed as $\Delta_N(p):=\prod_{1\leq i<j\leq N}(p_i-p_j)$.
\vskip 2mm
The next question we want to ask is whether we can perform asymptotic analysis. For this, we have
\vskip 2mm
{\bf Step 4. Fredholm determinants.} Relation \eqref{schur_form}
 allows to express \eqref{schur1} as a {\it Fredholm determinant}, in a form that is suitable to take the asymptotic limit
  and prove convergence to Tracy-Widom GUE distribution. 
  We will introduce the notion of a Fredholm determinant and some of its properties in the next section. 
  The significance of expressing \eqref{schur1} and other such probabilities 
  in terms of Fredholm determinants is that
doing so facilitates taking the limit of $N$ tending to infinity. In \eqref{schur1}, $N$ is the number of varables $\gl_1,...,\gl_N$, over which the
sum in \eqref{schur1} is taken. Thus, taking the limit $N\to\infty$ corresponds to the number of summations  
to infinity and the meaning of such limit is not clear at all at this stage. The key to 
resolving this difficulty is the notion of Fredholm determinants,
 which re-expresses such sums and integrals in a way 
 that the limit in $N$ becomes unambiguous and tractable. We will see how this is done in the next section.

\section{Determinantal measures, Fredholm determinants and asymptotics}\label{sec:Fred}
In this section we pick up from the closing paragraph of the previous section and we develop the framework that allows for
the asymptotic analysis of expressions like \eqref{schur1}. This framework usually goes under the name of {\it determinantal measures}
and {\it determinantal processes} and it is very general and widely applicable.
We will work with the example of expression \eqref{schur1} and we will arrive to Proposition \ref{prop:LPPcontour},
which re-expresses \eqref{schur1} in the form of a Fredholm determinant.
 In Section \ref{sec:dpp} we will generalise the notion of a determinantal measure to the notion of
{\it determinantal point process} and present a theorem which describes the structure of multipoint correlations.
This becomes useful in the analysis of joint laws of last passage percolation (but also of other other models that we do not
discuss here), which we describe in Section \ref{LPP-Airyprocess}.
In Section \ref{sec:asymptot} we describe the {steepest descent method}, which is the main tool for asymptotic analysis in this setting.
Picking up from Proposition \ref{prop:LPPcontour}, we will apply this method in order to derive that the asymptotic law of last passage percolation
is the Tracy-Widom GUE law, see Theorem \ref{thm:lpp-asympt}.
 Finally, in Section \ref{LPP-Airyprocess}, which focuses on multipoint statistics, we have included a discussion
on some recent important constructions of universal objects: the Airy line ensemble, the KPZ fixed point and the Directed Landscape.

\subsection{Determinantal measures and application to last passage percolation}\label{sec:det-meas}
We are now going to introduce the notion of determinantal measures and Fredholm determinants. 
We will demonstrate how two basic tools from determinantal calculus, the Cauchy-Binet or 
Andreief's identity and the Sylvester's identity, can be used to turn a determinantal measure into a Fredholm determinant.

In many statistical models we encounter probability measures of the form
\begin{align}\label{meas_part}
\mu_N(f):=\frac{Z_N(f)}{Z_N},
\end{align}  
where $\mu_N(f)$ denotes expectation of a functional $f\in L^2(\cX,\mu)$ on a measure space $(\cX,\mu)$ and
\begin{align}\label{partition}
Z_N(f):= \int_{\cX^N} \det\big(\phi_i(x_j)\big)_{1\leq i,j\leq N} \,  \det\big(\psi_i(x_j)\big)_{1\leq i,j\leq N} \,\,\,f(x_1)\cdots 
f(x_N)\,\, \mu(\dd x_1) \cdots \mu(\dd x_N)
\end{align} 
 the quantity $Z_N=Z_N(1)$ is typically known as the {\it partition function}.
Measures with determinants in this form in the right-hand side 
are known as {\it determinantal measures}. 
We note that the functions $\phi_i,\psi_i$ can be either non-negative, in which case, we have a genuine probability measure,
or they may also be allowed to take negative values, in which case we deal with signed measures.
For more regarding determinantal measures and processes we refer to \cite{B11, J05}. 
\vskip 2mm
Due to \eqref{schur_form} we see that the Schur measure
\begin{align}
\bbP(\lambda) := \prod_{i,j}(1-p_iq_j)\, s_\gl(q)  \,s_\gl(p).
\end{align}
 (see \eqref{schur1}) on partitions $\gl$ is a determinantal measure with $\phi_i(\lambda_j):=p_i^{\lambda_j+N-j}$
 and $\psi_i(\lambda_j):=q_i^{\lambda_j+N-j}$
 How to work with this measure will become more transparent in the analysis of the solvable last passage percolation model in Proposition \ref{prop:LPPcontour}, in particular see \eqref{LPPdet}. 
 This measure was introduced by Okounkov \cite{O01}.
 \vskip 2mm
  Determinantal probabilites such as \eqref{schur1} can be written in terms of objects called {\it Fredholm determinants} and this is crucial in obtaining asymptotics. 
We will exhibit this in the example of the Schur measure and last passage percolation with geometric weights.
\vskip 2mm
Let us first define the notion of Fredholm determinant.
Given an integral operator $K$ acting on $L^2(\cX,\mu)$ of a general measure space $(\cX,\mu)$ by
\begin{align*}
Kf(x)=\int_\cX K(x,y) f(y) \,\mu(\dd y),
\end{align*}
we define the Fredholm determinant associated to $K$ by
\begin{align}\label{Fred}
\det(I+K)_{L^2(\cX,\mu)}:= 1+\sum_{n=1}^\infty \frac{1}{n!}\ \int_{\cX^n} \det\big( K(x_i,x_j)\big)_{n\times n} \, \mu(\dd x_1) \cdots \mu(\dd x_n).
\end{align}
Here $I$ is the identity map.
Of course, one has to make sure that this infinite series is convergent. This is usually guaranteed by requiring that $K$ is a {\it trace class operator}. We refer to \cite{S79} for more details, but let us go through a quick sketch:
For a compact operator $K$ on a Hilbert space, 
say $L^2(\cX,\mu)$, we define its trace class norm as
$\|K\|_1:={\rm Tr}\sqrt{K^*K}$, where $K^*$ is the adjoint of $K$ and the square root can be defined via operator calculus, since
$K^*K$ is self-adjoint. In the case of a trace class norm operator one can obtain that \eqref{Fred} is well defined and the Fredholm determinant is bounded
by
\begin{align*}
|\det(I+K)_{L^2(\cX,\mu)}\,| \leq e^{\|K\|_1}.
\end{align*}
Moreover, one has the following continuity result 
\begin{align}\label{Fred-cont}
|\det(I+K_1)_{L^2(\cX,\mu)} -\det(I+K_2)_{L^2(\cX,\mu)}\,| \leq \|K_1-K_2\|_1\,\, e^{\|K\|_1+\|K\|_2+1}.
\end{align}
A consequence of this inequality is that if we would like to establish convergence of certain Fredholm determinants, it
is enough to establish the convergence of the corresponding operators in the trace class norm.
\vskip 2mm
A way to get a feeling about definition
\eqref{Fred} is to consider the case where $K$ is an $N\times N$ matrix and let $\lambda_1,...,\lambda_N$ denote its
eigenvalues. Then 
\begin{align}\label{Fred2}
\det(I+K)= \prod_{i=1}^N (1+\lambda_i) = 1 +\sum_{m=1}^N \sum_{1\leq i_1<\cdots <i_m\leq N} \gl_{i_1}\cdots \gl_{i_m}.
\end{align}
From the standard property of trace, $\sum_{i=1}^N \gl_i = {\rm Tr} K= \sum_x K(x,x)$, one sees immediately the identification of the first non-trivial terms in \eqref{Fred} and \eqref{Fred2}.
The rest of the terms have similar interpretation as traces of tensor products of $K$, see \cite{S79} for details.  
Without getting into details, we mention that Fredholm determinants formalise in some sense the inclusion-exclusion
principle. A flavour of this fact can be taken from expressions \eqref{eq:detpoint} and \eqref{eq:nopoint} later on.  
\vskip 2mm
We will now state two important tools that will allow to relate determinantal measures to Fredholm determinants.
\begin{proposition}[Cauchy-Binet or Andreief identity]\label{CB}
Consider a collection of functions $\big(\phi_i(\cdot)\big)_{1\leq i \leq N}$ and $\big(\psi_i(\cdot)\big)_{1\leq i \leq N}$, which belong to $L^2(\cX,\mu)$ of
 a measure space $(\cX,\mu)$. Then
\begin{align*}
&\frac{1}{N!}\int_{\cX^N} \det\big(\phi_i(x_j)\big)_{1\leq i,j\leq N} \,  \det\big(\psi_i(x_j)\big)_{1\leq i,j\leq N} \,\,\, \mu(\dd x_1) \cdots \mu(\dd x_N)\\
&\qquad =\det\Big( \int_\cX \phi_i(x) \psi_j(x) \, \mu(\dd x) \Big)_{1\leq i,j\leq N}.
\end{align*}
\end{proposition}
The following proposition is commonly known as Sylvester's identity:
\begin{proposition}\label{abba}
Consider general measure spaces $(\cX,\mu), (\cY,\nu) $ and
 trace class operators $A\colon L^2(\cY,\nu) \to L^2(\cX,\mu)$ and $B\colon L^2(\cX,\mu) \to  L^2(\cY,\nu) $. Then
\begin{align*}
\det\big(I+AB\big)_{L^2(\cX,\mu)} = \det\big(I+ BA\big)_{ L^2(\cY,\nu)}.
\end{align*}
\end{proposition}
Both propositions can be proven easily by expanding the determinants and essentially using Fubini; see \cite{AGZ10}. 
We can now prove a general result:
\begin{proposition}\label{prop:Fred}
Consider a collection of functions $\big(\phi_i(\cdot)\big)_{1\leq i \leq N}$ and $\big(\psi_i(\cdot)\big)_{1\leq i \leq N}$, which belong to $L^2(\cX,\mu)$ of
 a measure space $(\cX,\mu)$. Define the matrix
\begin{align}\label{gij}
\sfG_{ij}:=\int_\cX \phi_i(x) \psi_j(x)\, \mu(\dd x)
\end{align}
and assume that it is invertible. Define also the operator $K$ with kernel
\begin{align}\label{Fred_kern}
K(x,y):=\sum_{i,j} \psi_i(x) \big(\sfG^{-1}\big)_{ij} \,\phi_j(y).
\end{align}
Then, following notation \eqref{partition}, it holds that, for a general bounded function $g$ on $\cX$,
\begin{align*}
\frac{Z_N(1+g)}{Z_N}= \det\big(I+gK\big)_{L^2(\cX,\mu)}.
\end{align*}
\end{proposition}
\begin{proof}
The proof is a consequence of Propositions \ref{CB} and \ref{abba}. Let $f=1+g$. By the Cauchy-Binet identity
we have that 
\begin{align*}
Z_N(f) &= \det\Big(\int_\cX\phi_i(x) \psi_j(x) \, f(x)  \, \dd\mu \Big)_{1\leq i,j\leq N}\\
&=  \det\Big(\int_\cX\phi_i(x) \psi_j(x) + \int_\cX\phi_i(x) \psi_j(x) g(x)  \, \dd\mu \Big)_{1\leq i,j\leq N}\\
&= \det\Big(\sfG_{ij}+ \int_\cX\phi_i(x) \psi_j(x) g(x)  \, \dd\mu \Big)_{1\leq i,j\leq N},
\end{align*}
where by setting $g=0$ or equivalently $f=1$ we see that $Z_N=\det\sfG$. Using the multiplicativity of determinants, ie $\det(AB)=\det(A)\det(B)$ 
and denoting by $(\sfG^{-1})_{ij}$ the $(i,j)$ entry of matrix $\sfG^{-1}$, we see that
\begin{align}\label{detratio1}
\frac{Z_N(1+g)}{Z_N}
&=\det\Big(\gd_{ij} + \sum_k (\sfG^{-1})_{ik} \int_\cX\phi_k(x) \psi_j(x) g(x)  \, \dd\mu\Big)_{1\leq i,j\leq N}\notag\\
&=\det\Big(\gd_{ij} +  \int_\cX g(x)  \sum_k (\sfG^{-1})_{ik} \,\phi_k(x) \psi_j(x)  \, \dd\mu\Big)_{1\leq i,j\leq N}.
\end{align}
Now, we will use Proposition \ref{abba} with
\begin{align*}
A\colon  L^2(\cX,\mu) \to \ell^2(\{1,...,N\}) \quad \text{with} \quad A(i;x):= \sum_k (\sfG^{-1})_{ik} \,\phi_k(x)
\end{align*}
acting on functions in $f\in L^2(\cX,\mu)$ as $(Af)(i):=\int_{\cX} A(i;x) f(x) \mu(\dd x)$, with the result being
an element of $\ell^2(\{1,...,N\}$,
and
\begin{align*}
B\colon \ell^2(\{1,...,N\}) \to  L^2(\cX,\mu)  \quad \text{with} \quad B(x;i):=g(x)  \psi_i(x),
\end{align*}
acting on elements $h\in \ell^2(\{1,...,N\}$ as $(Bh)(x):=g(x)\sum_{i=1}^N \psi_i(x) \, h(i)$, with the result
being an element of  $L^2(\cX,\mu)$.
Then, by Proposition \ref{abba}, \eqref{detratio1} may be written as
\begin{align*}
\frac{Z_N(1+g)}{Z_N} = \det\Big(I+ AB\Big)_{\ell^2(\{1,...,N\})} = \det\Big(I+ BA\Big)_{L^2(\cX,\mu)}
\end{align*}
where the kernel of operator $BA$ may be written explicitly as
\begin{align*}
BA(x,y)&=\sum_{\ell=1}^N B(x;\ell) A(\ell,y) =g(x)  \sum_{1\leq \ell,k\leq N} \psi_\ell(x) \, \,(\sfG^{-1})_{\ell, k} \,\,\phi_k(y)\\
\end{align*}
for $x,y\in\cX$, completing the proof in view of \eqref{Fred_kern}.
\end{proof}
The difficulty that arises when one would like to apply concretely  Proposition \ref{prop:Fred} is to actually invert the
matrix $\sfG$. In some situations this can be done, as we will now see by applying this to 
 last passage percolation with geometric weights. We have
\begin{proposition}\label{prop:LPPcontour}
Consider a matrix $\sfW=(\bw^i_j)_{1\leq i,j\leq N}$ with distribution $\bbP$ as in \eqref{geom} and 
$\tau_N=\max_{\pi\in {\bf \Pi}_{N,N} }\sum_{(i,j)\in \pi} \bw^i_j$.
Then
\begin{align*}
\bbP\big( \tau_N\leq x\big) = \det\big(I+g_N K^{LPP}_N\big)_{L^2(\bbN)},
\end{align*}
with $g_N=\ind_{[x+N,\infty)}$ and for $t,s\in\N$ the kernel of operator $K^{LPP}_N$ is given by 
\begin{align}\label{KLPP}
K^{LPP}_N(t,s)=\frac{1}{(2\pi \iota)^2}\int_{\gamma_1} \dd \zeta \int_{\gamma_2} \dd \eta \,\,
\frac{\eta^s \zeta^t}{1-\zeta\eta} \,\,\prod_{j=1}^N \Big(\frac{1-\eta q_j}{\eta-p_j}\Big)
\prod_{i=1}^N \Big(\frac{1-p_i\zeta}{\zeta-q_i}\Big) 
\end{align}
where $\gamma_2$ is the circle in the complex plane with counter-clockwise orientation, centred at zero and radius one and $\gamma_1$ 
is the circle with counter-clockwise orientation, centred at zero of radius $r<1$. Without loss of generality, we assume that all $p_i,q_i, 1\leq i\leq N$ are small enough so that they are contained within contour $\gamma_1$. We also note that $\iota=\sqrt{-1}$.
\end{proposition}
\begin{proof}
Using \eqref{schur1} and \eqref{schur_form} we can write
\begin{align}\label{LPPdet}
\bbP\big( \tau_N\leq x\big) = \frac{\prod_{1\leq i,j\leq N}(1-p_iq_j) }{\Delta_N(p)\Delta_N(q)} \sum_{x\geq \gl_1\geq \gl_2\geq\cdots\geq \gl_N\geq 0} 
\det\big(p_i^{\gl_j+N-j}\big)_{1\leq i,j\leq N} \det\big(q_i^{\gl_j+N-j}\big)_{1\leq i,j\leq N},
\end{align}
where $\Delta_N(p)$ and $\Delta_N(q)$ are Vandermonde determinants.
To bring the sum in \eqref{LPPdet} into form \eqref{meas_part}, \eqref{partition}, we make the change of variables $t_j:=\gl_j+N-j$ and write it as
\begin{align*}
\sum_{x+N-1\geq t_1 > t_2 >\cdots >t_N\geq 0} 
\det\big(p_i^{t_j}\big)_{1\leq i,j\leq N} \det\big(q_i^{t_j}\big)_{1\leq i,j\leq N}
\end{align*}
 Noticing that this sum is symmetric in 
variables $t_1,...,t_N$ and that the summand vanishes if two of these are equal to each other (because the determinants do so in this case),
 we can extend the sum via symmetrization and write it as
\begin{align*}
\frac{1}{N!} \sum_{t_1,...,t_N\in \bbN}  \det\big(p_i^{t_j}\big)_{1\leq i,j\leq N} \det\big(q_i^{t_j}\big)_{1\leq i,j\leq N} \,\,\ind_{[0,x+N-1]}(t_1)\cdots \ind_{[0,x+N-1]}(t_N).
\end{align*}
Thus, we can write \eqref{LPPdet} in the form \eqref{meas_part} and \eqref{partition} with $\phi_i(t)=p_i^t$, $\psi_j(t)=q_j^t$, $t\in \{0,1,...\}$
and $\mu$ being the counting measure.
We will now use Proposition \ref{prop:Fred} and in this setting we compute
\begin{align}
\sfG_{ij}=\sum_{t\geq 0} (p_iq_j)^t = \frac{1}{1-p_iq_j}.
\end{align}
To invert this matrix we will use Cramer's formula, which states that
\begin{align}\label{cramer}
\big(\sfG^{-1}\big)_{ij}=\frac{(-1)^{i+j}\det \sfG^{\,ji}}{\det\sfG},
\end{align}
where $\sfG^{\,ji}$ denotes the minor matrix derived from $\sfG$ by deleting row $j$ and column $i$. One of the computable determinants goes under the name
of Cauchy determinant and is
\begin{align*}
\det\Big(\frac{1}{a_i-b_j}\Big)_{1\leq i,j\leq N} = (-1)^{\tfrac{n(n-1)}{2}} \frac{\Delta_N(a) \Delta_N(b)}{\prod_{1\leq i,j\leq N}(a_i-b_j)}
\end{align*}
and so 
\begin{align}\label{detg}
\det\sfG=\frac{\prod_{1\leq k<\ell\leq N}(p_k-p_\ell) \,\, \prod_{1\leq k<\ell\leq N}(q_k-q_\ell) }{\prod_{k,\ell=1}^N(1-p_kq_\ell)}
\end{align}
one can also compute $\det\sfG^{\,ji}$ observing that this is also a Cauchy determinant of the same type and so the same formula as in \eqref{detg}
will be valid, just without terms which contain variables $p_i$ and $q_j$. Thus, we obtain that
\begin{align*}
\big(\sfG^{-1}\big)_{ji} = \frac{\prod_{1\leq \ell\leq N}(1-p_jq_\ell) \, \prod_{1\leq k\leq N}(1-p_kq_i)}{(1-p_jq_i) 
\,\prod_{\ell\neq j} (p_j-p_\ell)\, \prod_{k\neq i} (q_i-q_k)}.
\end{align*}
Inserting this into \eqref{Fred_kern} with the choice
$\psi_i(t)=q_i^t$ and $\phi_j(s)=p_j^s$, we obtain that \eqref{LPPdet} can be written as a Fredholm determinant 
\begin{align*}
\bbP\big( \tau_N\leq x\big)  = \det\big(I+g_N K^{LPP}_N\big)_{L^2(\bbN)}
\end{align*}
with 
\begin{align*}
K^{LPP}_N(t,s) = \sum_{1\leq i,j \leq N} q_i^t p_j^s \,\, \frac{\prod_{1\leq \ell\leq N}(1-p_jq_\ell) \, \prod_{1\leq k\leq N}(1-p_kq_i)}{(1-p_jq_i) 
\,\prod_{\ell\neq j} (p_j-p_\ell)\, \prod_{k\neq i} (q_i-q_k)}.
\end{align*}
Using the Residue Theorem we can write this kernel in an integral form
\begin{align*}
K^{LPP}_N(t,s)=\frac{1}{(2\pi \iota)^2}\int_{\gamma_1} \dd \zeta \int_{\gamma_2} \dd \eta \,\,
\frac{\eta^s \,\zeta^t}{1-\zeta\eta} \,\,\prod_{j=1}^N \Big(\frac{1-\eta q_j}{\eta-p_j}\Big)
\prod_{i=1}^N \Big(\frac{1-p_i\zeta}{\zeta-q_i}\Big) 
\end{align*}
finishing the proof. 
 With regards to the application of the residue theorem we note that the function $\prod_{j=1}^N (\tfrac{1-\eta q_j}{\eta-p_j})$
 has poles $(p_i)$ included in the contout $\gamma_2$ and the function
$\prod_{i=1}^N (\tfrac{1-p_i\zeta}{\zeta-q_i})$ has poles $(q_i)$ included inside $\gamma_1$, while the choice of the contours excludes a pole from the case $\zeta\eta=1$.
\end{proof}
\vskip 2mm
{\bf Last Passage Percolation with exponential weights.}
If we scale the parameters of the $p_i,q_j$ of the geometric random variables $\bw^i_{j}$ in \eqref{geom}
 as $p_i\mapsto e^{-\gb_i \epsilon}$ and $q_j\mapsto e^{-\ga_j \epsilon}$ , then the rescaled geometric random variables $\epsilon \bw^i_j$
converge, as $\epsilon\to0$ to exponential variables with parameters $\alpha_j+\beta_i$. It, therefore, follows that last passage percolation with 
exponential variables with parameters $\alpha_j+\beta_i$ is also solvable and the kernel of the corresponding 
Fredholm determinant $K^{\text{exp}}_{N}(t,s)$ is derived from the limit
\begin{align}\label{Kexp_scale}
K^{\text{exp}}_{N,\ga\,,\gb}(t,s) =\lim_{\epsilon\to0} \epsilon^{-1}  K^{LPP}_{N,\epsilon\ga,\epsilon\gb}
\Big(\frac{t}{\epsilon},\frac{s}{\epsilon}\Big),
\end{align}
which can be computed after a change of variables $\zeta\mapsto e^{-\epsilon z}$ and $\eta\mapsto e^{\epsilon y}$ giving
\begin{align}\label{Fred_exp}
K^{\text{exp}}_{N,\ga,\gb}(t,s) = \frac{1}{(2\pi \iota)^2}\int_{\Gamma_1} \dd z \int_{\Gamma_2} \dd y \,\,
\frac{1}{z-y} \frac{e^{-tz}}{e^{-sy}}\,\,\prod_{j=1}^N \frac{\ga_j-y}{\gb_j+y}
\prod_{i=1}^N \frac{\gb_i+z}{\ga_i-z},
\end{align}
where contour $\Gamma_1$ is a straight, upwards oriented, 
vertical line with fixed, positive real part and with parameters $\ga_j, j=1,...,N$ lying on its right
and $\Gamma_2$ is the straight, upwards oriented, vertical line with real part zero.
The reason for scaling as in \eqref{Kexp_scale} comes from the series expansion of the
Fredholm determinant and scaling the integrals in that expansion with $\epsilon$ accordingly.

\subsection{Determinantal point processes}\label{sec:dpp}
The notion of determinantal measure \eqref{partition} can be extended to the notion of a {\it determinantal point process}.
 This notion is important in many aspects
but for us it is important because it allows to analyse joint laws of observables in solvable stochastic systems. An example that we will see later in Section \ref{sec:asymptot} is on the joint laws of last passage percolation times. 

Let $\cX$ be a complete, separable metric space. Typical examples are $\R^d$ or $\Z^d$. 
Consider a kernel $K(\cdot,\cdot)$ defined on $\cX\times \cX$.
Informally, a {\bf determinantal point process} on $\cX$ is a probability distribution $\bbP$ 
on all possible locally finite collections of points from $\cX$, such that
for any finite collection of locations $y_1,...,y_n\in\cX$, the probability (or probability density)
that these  locations are occupied by points of $\xi$ is given by $\det\big( K(y_i,y_j)\big)_{1\leq i,j\leq n}$. More formally, 
\begin{definition}\label{def:detpoint}
A determinantal point process is a 
probability distribution on the space of locally finite, counting measures $\xi$ on $\cX$, such that for any bounded, measurable function $g\colon \cX\to\R$,
supported on a bounded set $B\subset \cX$, it holds that
\begin{align}\label{eq:detpoint}
\E\big[ \prod_{i=1}^{\xi(B)} (1-g(x_i)\,) \big] = \sum_{n=0}^\infty \frac{(-1)^n}{n!} \int_{\cX^n} \det\big(K(x_i,x_j)\big)_{1\leq i,j\leq n} \prod_{i=1}^n g(x_i) \prod_{i=1}^n \mu(\dd x_i).
\end{align}
\end{definition}
Notice that the right hand side in \eqref{eq:detpoint} is the expansion of the Fredholm determinant $\det\big( I-gK\big)_{L^2(\cX)}$.
Setting $g(x)=1_{B}(x)$ for a subset $B\subset \cX$, we have that \eqref{eq:detpoint} takes the form
\begin{align}\label{eq:nopoint}
\bbP\big(\text{no point in B}\big) = \det\big(1-\ind_B K \big)_{L^2(\cX)}.
\end{align}
A very useful theorem, which captures the determinantal correlation structure of measures given as product of determinants is the following.
\begin{theorem}[\cite{J03}]\label{joh:multi}
Consider the space $\{1,...,N\}\times \R$ with fixed boundary points $(x^{(0)}_1,...,x^{(0)}_n)\in \{0\}\times\R^n$ 
and $(x^{(N+1)}_1,...,x^{(N+1)}_n)\in \{N+1\}\times\R^n$,
kernels $\phi_{r,r+1}(\cdot,\cdot) \colon \R^2\to\R, r\geq 0, $ and a reference measure $\mu(\dd x)$ on $\R$.
Then on the space of point configurations $\bx=(x^{(1)},...,x^{(N)})\in \{1,...,N\}\times \R $ 
with $x^r:=(x^{(r)}_1,...,x^{(r)}_n)\in\R^n$, for $r=1,...,N$, the measure  
\begin{align*}
\mu_{n,N}(\dd \bx):=\frac{1}{Z_{n,N}} \prod_{r=0}^{N} \det\big( \phi_{r,r+1}(x_i^{(r)},x_j^{(r+1)}) \big)_{1\leq i,j\leq n} 
\prod_{\substack{r=1,...,N \\ i=1,...,n}} \mu(\dd x_i^{(r)}),
\end{align*}
 where $Z_{n,N}$ is the normalisation,
defines a determinantal point process with kernel
\begin{align}\label{eq:joh:mullti}
K_{n,N}(r,x;s,y):= -\phi_{r,s}(x,y) +\sum_{i,j=1}^n \phi_{r,N+1}(x,x_i^{(N+1)}) \big[A^{-1}\big]_{i,j} \phi_{\,0,s}(x_j^{(0)}, y),
\end{align}
where denoting by $*$ the convolution, i.e. $(\phi*\psi)(x,y):=\int_\R \phi(x,z)\psi(z,y) \mu(\dd z)$, we have defined
\begin{align*}
&\phi_{r,s}(x,y):=\phi_{r,r+1}*\cdots * \phi_{s-1,s} (x,y), \qquad \text{for}\quad 0\leq r < s \leq N+1 \\
& \text{and} \qquad \phi_{r,s}(x,y):= 0, \hskip 3.2cm \text{if} \quad  r\geq s,
\end{align*}
and the matrix $A$ is defined as
\begin{align*}
A_{i,j} : = \phi_{\,0,N+1}(x^{(0)}_i, x^{(N+1)}_j) \qquad \text{for $1\leq i,j\leq n$}.
\end{align*}
\end{theorem}
We will see an application of this theorem in Section \ref{LPP-Airyprocess} when discussing the multi-point correlation structure
of the last passage percolation problem. For more applications of this theorem, we refer to the reviews \cite{J05, J17}

\subsection{Asymptotics and the Tracy-Widom law.}\label{sec:asymptot}
We will now give a sketch of how asymptotics are performed and how the $N^{1/3}$ scaling and the Tracy-Widom law emerge.
In a single phrase this can be described as
$$
 \text{{\it steepest descent method and Taylor expansion up to the third order.} }
 $$
 We will start by describing the method of steepest descent and then we will apply it to obtain the Tracy-Widom GUE asymptotic
 law for the rescaled exponential last passage percolation. 
 At the end of the section we will arrive at the following theorem:
 \begin{theorem}\label{thm:lpp-asympt}
 Let $\tau^{{\rm exp}}_N$ be the last passage percolation time from $(1,1)$ to $(N,N)$ on a lattice with exponential 
 weights ${\rm Exp(2\alpha)}$, for some parameter $\alpha>0$. Then, for $\sff=2/\alpha$ and $\sigma=2^{1/3}/\alpha$, it holds that
 the law of $\sigma^{-1}N^{-1/3} \big( \tau_N^{\rm exp} - \sff N\big)$ converges to the 
 one-point marginal of the ${\rm Airy_2}$ process, which coincides with the Tracy-Widom GUE distribution.
 More precisely, we have that
\begin{align*}
\lim_{N\to\infty}\bbP\big( \tau_N^{\text{exp}}\leq \sff N+\sigma N^{1/3} x\big) &=
\det\big(I+K_{\rm{Airy}_2}\big)_{L^2(x,\infty)},
\end{align*}
where the kernel $K_{{\rm Airy}_2}$ may be written explicitly as
\begin{align*}
K_{{\rm Airy}_2}(t,s)= \int_0^\infty Ai(\lambda+t) \,Ai(\lambda+s)\,\dd \lambda, \qquad t,s\in \R,
\end{align*}
with $Ai(\cdot)$ being the Airy function.
\end{theorem}
 \begin{remark}{\rm
 The Tracy-Widom GUE law describes the
limit of the probability that the largest eigenvalue of a random matrix from the
Gaussian Unitary Ensemble (GUE) is less than $x$, \cite{AGZ10}.
Even though the one-point marginal of the 
 $\text{Airy}_2$ process coincides with the Tracy-Widom GUE law, we presented the above theorem in term of the Airy process
 because, as we will see in the next subsection, this captures the correlations and joint law of the last passage percolation time from 
 $(1,1)$ to different lattice points $(x,y)$ on the line $x+y=N$.}
 \end{remark}
 \begin{remark}{\rm
 The Airy function $Ai(x)$ is the solution of the second order ODE $u''=xu$ subject to the condition that $u(x)\to0$ when $x\to\infty$.
 It admits the contour integral representation \eqref{airy-contour} that we will use below. For further information we refer to \cite{L72}.}
 \end{remark}

 Let us start by briefly describing the principles of the steepest descent method. A nice account of this, as well as other 
 classical asymptotic methods, can be found in \cite{AF03}. Another pedagogical account via examples on eigenvalue statistics
 of random matrices can be found in \cite{R12}. 
 
 Suppose that one is interested in the asymptotic behaviour of the integral
 \begin{align*}
 \int_{I} f(x) \,e^{N g(x)}\,\dd x,\qquad \text{as} \qquad N\to\infty,
 \end{align*}
 where $f,g$ are real functions and $I$ is an open interval in $\R$. If $g$ has a unique maximum, say $x_0\in I$,
  then Laplace asymptotics tell us that the main contribution to this integral
 will come from a neighbourhood around $x_0$. 
 In fact, if $g''(x_0)<0$, then a Taylor expansion (up to the second order) 
$g(x)=g(x_0) + \tfrac{1}{2}g''(x_0)\,(x-x_0)^2 +o\big( (x-x_0)^2\big) $ will give that the integral will be asymptotic to
\begin{align*}
f(x_0) \,e^{Ng(x_0)} \int_{(x_0-\epsilon,x_0+\epsilon)} e^{N\big(\tfrac{1}{2}g''(x_0)\,(x-x_0)^2 +o( (x-x_0)^2) \big)} \dd x
\,\approx\,
\Big(\frac{2\pi}{-Ng''(x_0)}\Big)^{1/2} f(x_0) \,e^{Ng(x_0)},
\end{align*} 
  where the last approximation follows from an evaluation of a Gaussian integral
  (the meaning of $\approx$ is that the ratio of the two sides converges to $1$). We now want to consider asymptotics when the
  integral is over a complex contour $\gamma$ with $f,g$ being complex valued, i.e.,
   \begin{align*}
 \int_\gamma f(z) \,e^{N g(z)}\,\dd z,\qquad \text{as} \qquad N\to\infty.
 \end{align*}
 The idea in this situation is to reduce to Laplace asymptotics by deforming the contour (using Cauchy's theorem) so that
 it passes through the critical point(s) of $g$ in a way such that the contribution along the contour away from the critical point is negligible. 
 This can, for example, be achieved if the real part of $g$ along the contour and away from the critical point,
 is strictly smaller than the real part of $g$ at the critical point, 
 while its imaginary part is (ideally) equal to zero.  
 
 Let us see how this idea can be applied in the example of last passage percolation with exponential weights. We will
 look at the Fredholm kernel we derived in \eqref{Fred_exp} and set the parameters $\alpha_i=\beta_j=a$ for all $i,j$. In this case, the kernel may be written as
 \begin{align*}
K^{\text{exp}}_{N,a}(t,s) = \frac{1}{(2\pi \iota)^2}\int_{\Gamma_1} \dd z \int_{\Gamma_2} \dd y \,\,
\frac{1}{z-y} \frac{e^{-tz}}{e^{-sy}}\,\, \Big(\frac{a-y}{a+y}\Big)^N
\Big(\frac{a+z}{a-z}\Big)^N,
\end{align*}
Since $\Re(z-y)>0$ we can write $(z-y)^{-1}=\int_0^\infty e^{-\lambda(z-y)}\,\dd \lambda $.  Setting 
\begin{align*}
g(z)=\log(a+z)-\log(a-z) 
\end{align*}
we have that
 \begin{align}\label{ker_exp}
K^{\text{exp}}_{N,a}(t,s) = \frac{1}{(2\pi \iota)^2} \int_0^\infty \dd\lambda \int_{\Gamma_1} \dd z 
\,\, e^{N g(z)-(\lambda+t)z} \,\, \int_{\Gamma_2} \dd y \, \,e^{-Ng(y)+(\lambda+s)y}.
\end{align}
Since we expect that $\tau_N^{\text{exp}}\sim \sff N + \sigma N^{1/3}\times (\text{fluctuations})$ for some constant $\sff$
\footnote{the macroscopic 
scale $N$ is indeed expected from applications of the sub-additive ergodic theorem, although the value of $\sff$ is not 
a priori known; even though it seems as if we cheat since we
``know'' the lower order term $N^{1/3}$ this would actually come naturally out of the asymptotic analysis that we are about to scketch} 
which represents the macroscopic behavior and a $\sigma$ which represents the standard deviation,
we want to compute the asymptotics of
 \begin{align*}
\bbP\big( \tau_N^{\text{exp}}\leq \sff N+\sigma N^{1/3} x\big) &= \det\big(I+ K^{\text{exp}}_N\big)_{L^2(\sff N+N^{1/3} x,\infty)} \\
 &=\det\Big(I+ K^{\text{exp}}_N(\cdot+\sff N+\sigma N^{1/3} x, \cdot+\sff N+\sigma N^{1/3} x)\Big)_{L^2(0,\infty)},
\end{align*}
which by the Fredholm expansion \eqref{Fred} and the continuity of Fredhom determinants \eqref{Fred-cont},
 amounts to computing the limit of
\begin{align}\label{scaleKexp}
\lim_{N\to\infty} \sigma N^{1/3}\, K^{\text{exp}}_N(\sigma N^{1/3} \,t+\sff N+\sigma N^{1/3} x, \sigma N^{1/3}\,s+\sff N+\sigma N^{1/3} x).
\end{align}
for $s,t\in\R$.
Given expression \eqref{ker_exp}, and setting
\begin{align*}
G(z)=\log(a+z)-\log(a-z)-\sff z,
\end{align*}
 computing \eqref{scaleKexp} amounts to computing the asymptotics of
 \begin{align}\label{Gamma12}
\frac{\big(\sigma N^{1/3}\big)^2}{(2\pi \iota)^2} \int_0^\infty \dd\lambda \int_{\Gamma_1} \dd z 
\,\, e^{N G(z)-\sigma N^{1/3}(\lambda+t+x)z} \,\, \int_{\Gamma_2} \dd y \, \,e^{-NG(y)+\sigma N^{1/3}(\lambda+s+x)y},
\end{align}
where we have scaled $\lambda$ as $\lambda\mapsto \sigma N^{1/3}\lambda$.
The two contour integrals have essentially the same structure, so let us focus on the $z$-integral.
We notice that $G(0)=0$ and if $\sff=2/a$, which will be our choice, then also $G'(0)=0$, so that $z=0$ is a critical point of $G$.
 However, at $z=0$ we also have that $G''(0)=0$, so in the Taylor 
expansion towards identifying the steepest descent contour we will need, this time, to go up to the third order;
thus, $NG(z)=N \,(G'''(0)z^3/6+o(z^3)\,)$ around $z=0$, with $G'''(0)=4/a^3$, which implies that
 \begin{itemize}
 \item
 the scaling $z\mapsto  N^{-1/3}z$ makes the leading term $G'''(0) Nz^3 / 6$ independent of $N$. Moreover,
  as one can estimate (we omit an elementary computation needed here),
   the contribution away from a neighbourhood of order $N^{-1/3}$ from the critical point is
 negligible, thus allowing us to localize to a neighbourhood within this scale. 
 Notice also that this change of variables turns the term
 $\sigma N^{1/3}(\lambda+t+x)z$ in \eqref{Gamma12} into $\sigma (\lambda+t+x)z$, making it invariant under scaling with $N$.
 These facts  make the emergence of the exponent $N^{1/3}$ clear.
 
 \item if $z=r e^{\pm \iota \pi/3}$ with $r$ positive real, then $z^3=-r^3<0$. Thus the appropriate steepest descent contour 
 (at least in an $O(N^{-1/3})$-neighbourhood around $z=0$) is 
 $$\gamma_1=\{re^{-\iota\pi/3} \colon r\in(-\infty,0)\} \cup \{re^{\iota\pi/3} \colon r\in(0,\infty)\},$$ traced upwards.
 \end{itemize}
 Changing the $\Gamma_1$ contour in \eqref{Gamma12} to $\gamma_1$ and ignoring the contribution outside
  a neighbourhood $B(0,N^{-1/3}R)$, for some $R>0$ which will eventually tend to infinity, we have the asymptotics 
 \begin{align*}
& \frac{\sigma N^{1/3}}{2\pi\iota}\int_{\gamma_1\cap B(0,N^{-1/3}R)} e^{N (G'''(0) z^3/6+o(z^3))-\sigma N^{1/3}(\lambda+t+x)z}\,\,\dd z \\
& =  \frac{\sigma}{2\pi\iota} \int_{\gamma_1\cap B(0,R)} e^{G'''(0) z^3/6-\sigma (\lambda+t+x)z}\,\,\dd z  
 \xrightarrow[R\to\infty]{}  \frac{\sigma}{2\pi\iota} \int_{\gamma_1} e^{ G'''(0) z^3/6-\sigma (\lambda+t+x)z}\,\,\dd z.
 \end{align*}
 Changing variables $z\mapsto (2 / G'''(0))^{1/3}\,z $ in the last integral and choosing $\sigma = (G'''(0) / 2)^{1/3}$, we have that the above integral equals
 \begin{align}\label{airy-contour}
 \frac{1}{2\pi\iota} \int_{\gamma_1} e^{ z^3/3- (\lambda+t+x)z}\,\,\dd z  = Ai(\lambda+t+x),
 \end{align}
 where $Ai(\cdot)$ is the so-called {\bf Airy function} given precisely by the last contour integral. 
 
 In the same manner, contour $\Gamma_2$ in \eqref{Gamma12}
  (which has to lie on the left of $\Gamma_1$) should be deformed to the contour $-\gamma_1$, the
same localization procedure around $z=0$ then leads to
\begin{align*}
 \frac{\sigma N^{1/3}}{2\pi\iota}\int_{\Gamma_2} e^{-NG(y)+\sigma N^{1/3}(\lambda+s+x)y}
 \xrightarrow[N\to\infty]{}  \frac{1}{2\pi\iota} \int_{-\gamma_1} e^{-z^3/3+(\lambda+s+x)z}\,\,\dd z = Ai(\lambda+s+x).
 \end{align*}
 Put together, we have that for $\sff=2/a$ and $\sigma=(G'''(0) / 2)^{1/3} = (2/ a^3)^{1/3}$
 \begin{align*}
&\lim_{N\to\infty}\sigma N^{1/3}\, K^{\text{exp}}_N(\sigma N^{1/3} \,t+\sff N+\sigma N^{1/3} x, \sigma N^{1/3}\,s+\sff N+\sigma N^{1/3} x) \\
&\hskip 0.5cm= \int_0^\infty Ai(\lambda+t+x) \,Ai(\lambda+s+x)\,\dd \lambda
=K_{{\rm Airy}_2}(t+x,s+x),
\end{align*}
Thus, we have shown that, for the
$\sff,\sigma$ as above
 \begin{align*}
\lim_{N\to\infty}\bbP\big( \tau_N^{\text{exp}}\leq \sff N+\sigma N^{1/3} x\big) &=
\det\big(I+K_{\rm{Airy}_2}\big)_{L^2(x,\infty)}.
\end{align*}

\subsection{Multipoint distributions and Airy processes.}\label{LPP-Airyprocess}

Let us start with a discussion on spatial correlations.
In the previous paragraphs we saw how to obtain the scaling limit of the law of the last passage percolation time from $(1,1)$
to a fixed point $(N,N)$. In particular, in the case of exponential passage times we just checked that
 \begin{align*}
\lim_{N\to\infty}\bbP\big( \tau_N^{\text{exp}}\leq \sff N+\sigma N^{1/3} x\big) &=
\det\big(I+K_{\rm{Airy}_2}\big)_{L^2(x,\infty)},
\end{align*}
with $\sff,\sigma$ explicitly computed.
The analogous result also holds in the case of geometrically distributed weights after the proper adjustments of the parameters $\sff, \sigma, \rho$
(see also Theorem \ref{thm:airy} below).
Furthermore, one is interested in spatial correlations and a rigorous statement of the asymptotics of the type of \eqref{GUEasy}. Cast in the framework of last passage percolation one is interested to know if
 there is a limit as $N\to\infty$ of the process
\begin{align}\label{2/3-LPP}
\Bigg\{
\frac{\tau_{N+\lfloor\rho N^{2/3}\,z\rfloor, N-\lfloor\rho N^{2/3}\,z\rfloor}-\sff N}{\sigma \,N^{1/3}}
\Bigg\}_{z\in\bbR},
\end{align} 
for suitable constants $\sff, \sigma, \rho$. .
It turns out \cite{J03} that (for explicitly computable constants $\sff,\sigma,\rho$)
the process \eqref{2/3-LPP} converges to the process $ \big\{ \Ai_2(z)-z^2\big\}_{z\in\bbR}$,
where $\Ai_2(z)$ is the so-called {\rm ${\bf Airy}_2$} process, 
originally introduced by Pr\"ahofer and Spohn \cite{PS02}.
The ``two\,'' here refers to the link with eigenvalue statistics of GUE 
random matrices which are characterised by a parameter $\beta=2$.
The ${\rm Airy}_2$ process is a stationary, continuous process and the distribution of its one-point marginal coincides with the Tracy-Widom GUE law.
Its finite dimensional  distributions are given by
\begin{align}\label{def:airy}
\bbP\big(\Ai_2(z_1)\leq \xi_1,...,\Ai_2(z_k)\leq \xi_k\big) =\det(I-\chi_\xi \,K_{{\rm Airy}_2}^{\rm ext} \,\chi_\xi)_{L^2(\{z_1,...,z_k\}\times \bbR)},
\end{align}
where for $\ell=1,...,k$ we define $\chi_\xi(z_\ell,u):=\ind_{\{u\geq \xi_\ell\}}$ and $K_{{\rm Airy}_2}^{\rm ext}$
 is the {\bf extended $\text{Airy}_2$ kernel,} defined by
\begin{align*}
K_{{\rm Airy}_2}^{\rm ext}(z,\xi;z', \xi'):=
\begin{cases}
\,\,\int_0^\infty e^{-\gl (z-z')} Ai(\xi+\gl) Ai(\xi'+\gl) \dd \gl, &\qquad \text{if} \,\,z\geq z',\\
&\\
-\int_{-\infty}^0 e^{-\gl (z-z')} Ai(\xi+\gl) Ai(\xi'+\gl) \dd \gl, &\qquad \text{if} \,\,z< z',
\end{cases}
\end{align*}
with $Ai(\cdot)$ being the Airy function. The expansion of the 
Fredholm determinant in \eqref{def:airy} may be written as 
\begin{align}\label{eq:airy_exp}
&\det(I-\chi_\xi\,K_{{\rm Airy}_2}^{\rm ext}\,  \chi_\xi)_{L^2(\{z_1,...,z_k\}\times \bbR)}\\
&= 1+ \sum_{n=1}^\infty \frac{(-1)^n}{n!} \sum_{s_1,...,s_n\in\{z_1,...,z_k\}} \int_{\xi_1}^\infty\cdots \int_{\xi_n}^\infty \det\big(  K_{{\rm Airy}_2}^{\rm ext} (s_i,y_i;s_j,y_j)  \big)_{n\times n}
 \dd y_1\cdots\dd y_n.  \notag
\end{align}
We can state more precisely the following theorem.
\begin{theorem}[\cite{J03}]\label{thm:airy}
Consider a matrix $\sfW=(w_{ij})_{i,j\geq 1}$ of passage times, geometrically distributed according to \eqref{geom} with $p_i=q_j=\sqrt{q}$ for $i,j\geq 1$. Then the
 process \eqref{2/3-LPP} of last passage times $\tau_{N+\lfloor\rho N^{2/3}\,z\rfloor, N-\lfloor\rho N^{2/3}\,z\rfloor}$ from $(1,1)$ to
  $\big(N+\lfloor\rho N^{2/3}\,z\rfloor, N-\lfloor\rho N^{2/3}\,z\rfloor\big)$
 with the choice 
\begin{align*}
\sff=\frac{2\sqrt{q}}{1-\sqrt{q}} , \qquad   \sigma= \frac{(\sqrt{q})^{1/3}(1+\sqrt{q})^{1/3}}{1-q},  \qquad \rho=\frac{1+\sqrt{q}}{1-\sqrt{q}} \, \sigma^{-1},
\end{align*}
converges in the topology of continuous functions on compact intervals to the process $ \big\{ \Ai_2(z)-z^2\big\}_{z\in\bbR}$,
where $\Ai_2(z)$ is the ${\rm  Airy}_2$ process.
\end{theorem}
We will sketch the proof of this theorem in the next paragraph, when we will introduce the polynuclear growth process. 
The analogous convergence of the joint laws of last passage percolation with exponential passage times was 
achieved by Pr\"ahofer-Spohn \cite{PS02}. Both \cite{PS02} and  \cite{J03} achieved these results via the analysis of the 
{\it polynuclear growth model} (PNG), which we will describe next.
\vskip 2mm
{\bf Polynuclear growth process.} 
This model can be viewed as a graphical representation of $\rsk$ via the local moves described in Section \ref{sec:localmove}. 
As we will see, this construction leads to an ensemble of non-intersecting
paths. The Lindstr\"om-Gessel-Viennot theorem can, then, be applied to give an {\it extended} determinantal measure 
from which a Fredholm determinant can be derived for the joint law of the last passage times following the general scheme of Section 
\ref{sec:dpp}. 

Let us first describe the construction of PNG. 
Assume we have a matrix of weights $\sfW=(w_{ij})_{i,j\geq 1}$ and let us only consider its triangular part $(w_{ij} \colon i+j\leq N+1)$. 
This is, actually, an array as in \eqref{rsk-array}, to which we can apply $\rsk$ via local moves as described there. For the purposes of the exposition here, it is better to rotate it $135^o$
counter-clockwise as in Figure \ref{fig:png1}. It is also more convenient to consider the change of coordinates 
$(x,t):=(i-j,i+j-1)$, where $t$ will play the role of time. The array $(w_{ij} \colon i+j\leq N+1)$ can be thought of as encoding the heights of the blocks ({\it plateaux})
 that will be nucleated along the interface.

We consider, initially at time $t=0$, an infinite number of straight lines at heights $0,-1,-2,...$ and the process starts at time $t=1$ with a plateau of height 
$w_{11}$ being created on top of the first line at location $x=0$. The plateau will then start growing at unit speed left and right and at time $t=2$ two new plateaux of
heights $w_{1,2}$ and $w_{2,1}$ will be created on top of the first (now deformed) line at locations $x=-1, x=1$, respectively. The plateaux continue to grow at unit
speed left and right and when two plateaux overlap, then the overlapping region drops down and builds on top of the next line, see Figure \ref{fig:png2}. After the drop-downs
have taken place, say at time $t$, new plateaux are created on top of the first line with heights $w_{ij}$, for $i+j-1=t$ and at locations $x=i-j$ for 
all $i,j\geq 1$ such that $i+j-1=t$.
The plateaux will continue to grow at unit speed left and right along all level lines and whenever an overlap occurs, 
a cascade starts down the lines. 
We note that
the height of the drop-down block at some location and the new height (after the new nucleation) of the line at the location where the overlap has taken place 
are encoded by the local transform \eqref{eq:localmove}. Indeed,  let us assume that the height of a certain line, at location $x$ is $a$ and the height of this 
line at 
locations $x\pm1$ is $b,c$, respectively. Then the overlap, after the extension of the blocks at $x\pm1$, will be $b\wedge c-a$ while the height at $x$ will be $b\vee c$ and
after the addition of the new block, at location $x$, of height, say, $d$, the height there will be $d+b\vee c$. This is exactly transform \eqref{eq:localmove}.

\begin{figure}[t]
	\begin{center}
		\begin{tikzpicture}[scale=.7]

		\draw[-,thick] (-3,0)--(-0.5,0); \draw[-,thick] (0.5,0)--(3,0); \draw[-,thick] (-3,-0.5)--(3,-0.5);\draw[-,thick] (-3,-1)--(3,-1);

		\draw[-,thick](-0.5,0)--(-0.5,1.5)--(0.5,1.5)--(0.5,0);
		\draw[<->,red] (1.5,0)--(1.5,1.5);
		\draw[->,green](0.5,0.7)--(1,0.7);
		\draw[->,green](-0.5,0.7)--(-1,0.7);
		\node at (2,1) {\small{{\bf $w_{11}$}}};
		
		\fill[black]  (0,0.5) circle [radius=0.1];
		\fill[black]  (1,1) circle [radius=0.1];
		\fill[black]  (-1,1) circle [radius=0.1];
		\fill[black]  (0,2) circle [radius=0.1];				
		\fill[black]  (2,2) circle [radius=0.1];
		\fill[black]  (-2,2) circle [radius=0.1];
		
		\draw[-,thick](8,0)--(10.5,0); \draw[-,thick](13.5,0)--(16,0); \draw[-,thick] (8,-0.5)--(16,-0.5); \draw[-,thick] (8,-1)--(16,-1);
		\draw[-,thick] (10.5,0)--(10.5,1.5)--(13.5,1.5)--(13.5,0);
		\draw[-,thick,red] (10.5,1.5)--(10.5,2.5)--(11.5,2.5)--(11.5,1.5);
		\draw[-,thick,red] (12.5,1.5)--(12.5,2)--(13.5,2)--(13.5,1.5);		
		\fill[black]  (12,0.5) circle [radius=0.1];
		\fill[black]  (13,1) circle [radius=0.1];
		\fill[black]  (11,1) circle [radius=0.1];
		\fill[black]  (12,2) circle [radius=0.1];				
		\fill[black]  (14,2) circle [radius=0.1];
		\fill[black]  (10,2) circle [radius=0.1];
		\draw[<->,thick,red] (9.5,0)--(9.5,2.5);
		\draw[<->,red,thick] (14.5,0)--(14.5,2);
		\draw[<->,red,thick] (10.3,1.5)--(10.3,2.5);
		\draw[<->,red,thick] (13.7,1.5)--(13.7,2);	
	    \draw[->,green](13.5,0.7)--(14,0.7);
	    \draw[->,green](10.5,0.7)--(10,0.7);
		\draw[->,green](13.5,2.2)--(14,2.2);
		\draw[->,green](10.5,2.6)--(10,2.6);
		\draw[->,green](12.5,2.2)--(12,2.2);
	    \draw[->,green](11.5,2.6)--(12,2.6);    
	    \node at (14.1,1.7) {\small{{\bf $w_{12}$}}};
	    \node at (10.5,2.8) {\small{{\bf $w_{21}$}}};
	    \node at (8.3,1.5) {\small{{\bf $w_{11}+w_{21}$}}};
	    \node at (15.7,1) {\small{{\bf $w_{11}+w_{12}$}}};	
		\end{tikzpicture}
	\end{center}  
	
	\caption{ \small  The PNG at times $t=1,2$. The dots correspond to the entries $(w_{ij})$ of the input array, which determine the nucleations and the heights of the plateaux. 
	Once created, the plateaux will expand left/right at unit speed.
	}\label{fig:png1}
\end{figure}
 \begin{figure}[t]
 	\begin{center}
 		\begin{tikzpicture}[scale=0.7]
 		\fill[black]  (0,0.5) circle [radius=0.1];
 		\fill[black]  (1,1) circle [radius=0.1];
 		\fill[black]  (-1,1) circle [radius=0.1];
 		\fill[black]  (0,2) circle [radius=0.1];				
 		\fill[black]  (2,2) circle [radius=0.1];
 		\fill[black]  (-2,2) circle [radius=0.1];
 		\draw[black, thick](-5,0)--(-2.5,0); \draw[black, thick](5,0)--(2.5,0); \draw[black, thick](-5,-0.5)--(5,-0.5);  \draw[black, thick](-5,-1)--(5,-1);
 		\draw[-,black, thick](-2.5,0)--(-2.5,2.5)--(0.5,2.5)--(0.5,1.5)--(2.5,1.5)--(2.5,0);
 		\draw[-,red,thick] (-5,-0.1)--(-2.6,-0.1)--(-2.6,1.5)--(-0.5,1.5)--(-0.5,2)--(2.6,2)--(2.6,0.1)--(5,0.1);
 		\draw[dashed,thick](-0.5,1.5)--(0.5,1.5);
 		\fill[lightgray](-0.5,1.5)--(0.5,1.5)--(0.5,2)--(-0.5,2)--(-0.5,1.5);

 		\draw[-,thick](8,0)--(9.5,0); \draw[-,thick](14.5,0)--(16,0); \draw[-,thick] (8,-0.5)--(11.5,-0.5);  \draw[-,thick] (12.5,-0.5)--(16,-0.5); \draw[-,thick] (8,-1)--(16,-1);
 		\draw[dotted,thick] (9.5,0)--(9.5,2.5)--(12.5,2.5)--(12.5,2)--(14.5,2)--(14.5,0);
 		\draw[-,thick,red] (9.5,2.5)--(9.5,3.5)--(10.5,3.5)--(10.5,2.5);
 		\draw[-,thick,red] (11.5,2.5)--(11.5,4)--(12.5,4)--(12.5,2.5);	
 		\draw[-,thick,red] (13.5,2)--(13.5,2.5)--(14.5,2.5)--(14.5,2);
 		\draw[-,thick] (11.5,-.5)--(11.5,0)--(12.5,0)--(12.5,-0.5);
 		\fill[lightgray](11.5,-0.5)--(11.5,0)--(12.5,0)--(12.5,-0.5)--(11.5,-0.5);			
 		\fill[black]  (12,0.5) circle [radius=0.1];
 		\fill[black]  (13,1) circle [radius=0.1];
 		\fill[black]  (11,1) circle [radius=0.1];
 		\fill[black]  (12,2) circle [radius=0.1];				
 		\fill[black]  (14,2) circle [radius=0.1];
 		\fill[black]  (10,2) circle [radius=0.1];
 		
 		\draw[<->,thick,red] (9.3,2.5)--(9.3,3.5);
 		\draw[<->,red,thick] (12.7,2.5)--(12.7,4);
 		\draw[<->,red,thick] (14.7,2)--(14.7,2.5);

		 \node at (8.8,3) {\small{{\bf $w_{31}$}}};
		 \node at (13.2,3.3) {\small{{\bf $w_{22}$}}};
		 \node at (15.2,2.3) {\small{{\bf $w_{13}$}}};	
 	
 		\end{tikzpicture}
 	\end{center}  
 	\caption{ \small  At time $t=3$, an overlap of height $\min(h_{12},h_{21})-h_{11}$ drops down and a new plateau is created at $x=0$.
 	}\label{fig:png2}
 \end{figure}
\vskip 2mm

 \begin{figure}[t]
 	\begin{center}
 		\begin{tikzpicture}[scale=0.55]
		\draw[-, thick]  (-10,0)--(-4.5, 0)--(-4.5,6)--(-3.5,6)--(-3.5, 4)--(-2.5, 4)--(-2.5,7)--(-1.5,7)--(-1.5, 3)--(-0.5,3)--(-0.5,6.5)--(0.5, 6.5)--(0.5, 4)--(1.5,4)--(1.5, 7)--(2.5, 7)--(2.5,4)--(3.5, 4)--(3.5, 5)--(4.5,5)--(4.5,0)--(10,0);
		\draw[-,thick] (-10,-0.5)--(-2.5, -0.5)--(-2.5,2)--(-1.5,2)--(-1.5, 1)--(-0.5,1)--(-0.5, 2.5)--(0.5, 2.5)--(0.5,1)--(1.5,1)--(1.5, 3)--(2.5,3)--(2.5, -0.5)--(10,-0.5);
		\draw[-,thick] (-10,-1)--(-0.5, -1)--(-0.5, 0.7)--(0.5, 0.7)--(0.5,-1)--(10,-1);
		\draw[-,thick] (-10,-1.5)--(10,-1.5);
		\draw[-,thick] (-10,-2)--(10,-2);
		\fill[black]  (-4,6) circle [radius=0.15]; \fill[black]  (-3,4) circle [radius=0.15]; \fill[black]  (-2,7) circle [radius=0.15]; \fill[black]  (-1,3) circle [radius=0.15];
		\fill[black]  (0,6.6) circle [radius=0.15]; \fill[black]  (1,4) circle [radius=0.15]; \fill[black]  (2,7) circle [radius=0.15]; \fill[black]  (3,4) circle [radius=0.15];
		\fill[black]  (4,5) circle [radius=0.15];   \fill[red]  (5,0) circle [radius=0.15];   \fill[red]  (-5,0) circle [radius=0.15];
		\fill[red]  (-4,-0.5) circle [radius=0.15]; \fill[red]  (-3,-0.5) circle [radius=0.15]; \fill[black]  (-2,2) circle [radius=0.15]; \fill[black]  (-1,1) circle [radius=0.15];
		\fill[black]  (0,2.5) circle [radius=0.15]; \fill[black]  (1,1) circle [radius=0.15];  \fill[black]  (2,3) circle [radius=0.15];  \fill[red]  (3,-0.5) circle [radius=0.15];
		 \fill[red]  (4,-0.5) circle [radius=0.15];  \fill[red]  (5,-0.5) circle [radius=0.15];   \fill[red]  (-5,-0.5) circle [radius=0.15];
		 \fill[red]  (-4,-1) circle [radius=0.15]; \fill[red]  (-3,-1) circle [radius=0.15]; \fill[red]  (-2,-1) circle [radius=0.15]; \fill[red]  (-1,-1) circle [radius=0.15];
		 \fill[black]  (0,0.7) circle [radius=0.15];  \fill[red]  (4,-1) circle [radius=0.15]; \fill[red]  (3,-1) circle [radius=0.15]; \fill[red]  (2,-1) circle [radius=0.15]; 
		 \fill[red]  (1,-1) circle [radius=0.15];   \fill[red]  (5,-1) circle [radius=0.15];   \fill[red]  (-5,-1) circle [radius=0.15];
\end{tikzpicture}
 	\end{center}  
 	\caption{ \small  A caricature of PNG after the insertion of blocks with heights encoded by an array $(w_{ij}\colon i+j\leq 6)$.
	The black dots show the parts of the lines where blocks of heights encoded by the weights $(w_{ij}\colon i+j\leq 6)$ have been added
	and the red dots indicate the parts that are still frozen.
 	}\label{fig:multi-png}
 \end{figure}		
		Figure \ref{fig:multi-png} shows an example of how the interface will look like. By the discussion around \eqref{rsk-array}, the heights of the {\it top} line
		at locations marked by the dots, encode the last passage times from $(1,1)$ to the corresponding entry $(i,j)$ in the array $\sfW$, for $(i,j)$ such that $i+j-1=N$
		(corresponding to the peaks of the top line) or $i+j-1=N-1$ (corresponding to the valleys of the top line). 
		Moreover, it is a consequence of Greene's theorem that the sum of the heights of the PNG lines above a site
		$x=i-j$, with $i+j=N+1$, is equal to the sum $\sum_{k\leq i, \ell\leq j} w_{k\ell}$. These
		and some more properties of PNG can be found in \cite{J03}, Section 3 or in \cite{NZ17} Proposition 2.6 regarding the 
		geometric lifting of PNG.
		We also note that if we let the process run for long time, we would
		observe that the line ensemble would have a parabolic curvature. This would correspond to the $-z^2$ drift in Theorem \ref{thm:airy}. 

It is a consequence of this graphical construction, or, equivalently, of $\rsk$, that the line ensemble representing PNG will be a
non-intersecting line ensemble. Moreover, using the properties of $\rsk$ in the same way we used them in Section \ref{geomLPP},
it follows that, when the entries of $(w_{ij}\colon i+j\leq N+1)$ are distributed according to the geometric distribution \eqref{geom},
then the total weight of the PNG ensemble after time $t=N$ (recall that $t=i+j-1$ and we have inserted the subarray 
$(w_{ij}\colon i+j\leq N+1)$) will have the product form
\begin{align}\label{png-prod}
\frac{1}{Z_N}\prod_{r=-N}^{N-1} \, \prod_{i=1}^N \phi_{r,r+1}(x^{(r)}_i,x^{(r+1)}_i)
\end{align}
with fixed boundary points $x^{(-N)}_i=x^{(N)}_i=-i$ for $i=1,...,N$, and the normalisation
 \begin{align*}
 Z_N:=\frac{\prod_{i=1}^N(1-p_i)^n \prod_{j=1}^N(1-q_j)^n}{\prod_{i+j\leq N+1}(1-p_iq_j)}.
 \end{align*}
In accordance to Figure \ref{fig:multi-png}, we have assumed $N=2n-1$ is odd.
The weights $\phi_{r,r+1}$ are 
given explicitly (again, we assume $N$ is odd; in the case of even $N$ the order of up and down steps is reversed), by
\begin{align}\label{pngweights1}
\phi_{-N+2j, -N+2j+1}(x,y)=
\begin{cases}
(1-p_{j+1}) p_{j+1}^{y-x}, \qquad &\text{if} \quad y\geq x\\
0, \qquad & \text{if} \quad y< x
\end{cases}\quad \text{for} \quad j=0,...,N-1,
\end{align}
and
\begin{align}\label{pngweights2}
\phi_{-N+2j+1, -N+2j+2}(x,y)=
\begin{cases}
(1-q_{N-j}) q_{N-j}^{y-x}, \qquad &\text{if} \quad y\leq x\\
0, \qquad & \text{if} \quad y> x
\end{cases}\quad \text{for} \quad j=0,...,N-1.
\end{align}
We refer for the details of this derivation to its original source, \cite{J03}, Section 3.

The product structure of the measure \eqref{png-prod}, together with the non-intersecting property implies, via the 
Lindstr\"om-Gessel-Viennot theorem \eqref{thm:LGV}, that the point process 
$\{x^{(r)}_i \colon r=-N,...,N, \, i=1,...,N\}$
that marks the heights of the PNG line ensemble
(see red and black dots in Figure \ref{fig:multi-png}) is a determinantal point process with law
\begin{align*}
\frac{1}{Z_N} \prod_{r=-N}^{N-1} \det\big( \phi_{r,r+1}(x^{(r)}_i,x^{(r+1)}_j)\big)_{1\leq i,j \leq n}.
\end{align*}
Theorem \ref{joh:multi} allows then to 
write the kernel of the determinantal process as in \eqref{eq:joh:mullti}. It is important to note that the specific structure of the weights  
\eqref{pngweights1}, \eqref{pngweights2} together with the fact that the $j^{th}$ PNG path starts and ends at the same height $-j+1$
allows, via the theory of T\"oplitz matrices \cite{BE09}, for the explicit inversion of matrix $A$
in \eqref{eq:joh:mullti}.
Let us record the (extended) kernel of the PNG point process, as given in Theorem 3.14 in \cite{J03}).
For $|u|,|v|<n$, $x,y\in \Z$, $N=2n-1$, we have that
\begin{align*}
K^{\rm PNG}_N(2u,x;2v,y)&=\frac{1}{(2\pi \iota)^2} \int_{\gamma_{r_2}} \frac{\dd z}{z} \int_{\gamma_{r_1}} \frac{\dd w}{w} \frac{w^y}{z^x} \frac{z}{z-w} G(z,w)\\
&\qquad - \frac{1}{2\pi\iota} \int_{\gamma_1} z^{y-x} G(z,z) \frac{\dd z}{z} \, \ind_{u < v}, \qquad \text{with} \\
G(z,w)&:=\frac{\prod_{j=u}^{n-1} \Big(\tfrac{1-q_{n-j}}{1-q_{n-j}z}\Big) \prod_{j=-n+1}^v \Big( \tfrac{1-p_{n+j}}{1-p_{n+j}/w}\Big)}
{\prod_{j=-n+1}^u \Big( \tfrac{1-p_{n+j}}{1-p_{n+j}/z}\Big) \prod_{j=v}^{n-1} \Big(\tfrac{1-q_{n-j}}{1-q_{n-j}w}\Big)},
\end{align*}
where the (counter-clockwise oriented) 
contours are such that $\gamma_1$ is the circle centred at $0$ and with radius $1$ and $\gamma_{r_1}, \gamma_{r_2}$
are the circles centred at zero and with radii $r_1, r_2$, respectively, satisfying $1-\epsilon < r_1<r_2<1+\epsilon$, for some $\epsilon>0$,
and the assumption that $\min(1/q_j \colon j\geq 1)>1+\epsilon$ and $1-\epsilon>\max(p_j\colon j\geq 1) $.

As already mentioned, the fact that the heights of the top line at locations $x=i-j$, for $i,j$ with $i+j=N+1$ record the last passage percolation
times from $(1,1)$ to $(i,j)$. Combining this with the fact that the PNG process is a determinantal process, Theorem \ref{joh:multi}
and Definition \ref{def:detpoint} allow to express the joint laws of the last passage percolation as 
\begin{align*}
\bbP\big( \tau_{i,j}\leq \xi_x \,\,\text{for}\,\, i+j=N+1, i-j=x, \, |x|<N\big)
= \det\big( I - \chi_\xi K_N^{\rm PNG} \chi_\xi \,\big)_{\ell^2\big( \{-N+1,...,N-1\}\times \Z\big)},
\end{align*}
where for $|x|<N$ we define $\chi_\xi(x,z):=\ind_{z\geq \xi_x}$.
The proof of Theorem \ref{thm:airy} will follow after taking the suitable limit of the above Fredholm determinant as this was described in 
Section \ref{sec:asymptot}.
\vskip 2mm
{\bf The Airy line ensemble.}
Theorem \ref{thm:airy} was restricted to the scaling limit of just the top line of the PNG line ensemble. One could also consider the $N^{1/3}, N^{2/3}$ scaling
of all the layers in PNG. After removal of the parabolic drift one arrives to either the {\it multi-layer Airy process} \cite{PS02} or the {\it Airy line ensemble}
\cite{CH14}, which is an ensemble of continuous, locally Brownian paths indexed by $\N$. 
\vskip 2mm
{\bf The KPZ fixed point.} The discussion, so far, concerned the process of last passage percolation times on a {\it triangular array}
$(w_{ij} \colon i,j\geq 1, i+j\leq N+1 )$ from $(1,1)$ to the edge sites $(i,j)$ with $i+j=N+1$. Equivalently, cast in the framework of the
corner growth process of Section \ref{sec:corner}, this concerns the statistics of the profile of the interface when initially it starts from having
the shape of a corner. With regards to TASEP (Section \ref{secTASEP}) this would correspond to the case when all particles occupy all
sites on the negative semi-axis. One, though, would also be interested to describe the growth process for an arbitrary initial interface
or, equivalently, for arbitrary initial positions of particles. Besides the case of corner initial conditions
 (also known as {\it wedge} initial conditions),
multi-point correlations were also studied in \cite{BFPS07} in the case of
{\it flat} initial condition, i.e. when the interface is a zig-zag curve, alternating between heights zero and one.
In the case of TASEP this would correspond to particles occupying every other
site of $\Z$. The breakthrough of extracting the joint distribution of the particles in TASEP, starting 
from arbitrary initial conditions was achieved in  \cite{MQR21}. There, starting from formulas in \cite{BFPS07},
new Fredholm determinant formulas, with kernels expressed in terms
of hitting probabilities of certain random walks, were derived for the law of the location of TASEP particles starting from
general initial conditions. In order to present this, we need to introduce some notation. 
   
   First, we define the kernels:
\begin{align*}
\cS_{-t,-n}(z_1,z_2)&:=\frac{1}{2\pi\iota} \int_{\gamma_r}  \frac{(1-w)^n}{2^{z_2-z_1} w^{n+1+z_2-z_1}} e^{t(w-1/2)} \dd w \,, \\
\bar{\cS}_{-t,-n}(z_1,z_2)&:=\frac{1}{2\pi\iota} \int_{\gamma_r}  \frac{(1-w)^{z_2-z_1+n-1}}{2^{z_1-z_2} w^n} e^{t(w-1/2)} \dd w
\end{align*}
where $\gamma_r$ is a circle with counter-clockwise orientation, centred at zero and with radius $r<1$.

We also introduce the random walk
$(\sfB_n)_{n\geq 0}$  with jumps strictly to the left and distributed according to a geometric random variable with parameter $1/2$.
More precisely, its transition kernel is $Q(x,y):=2^{y-x}\ind_{y<x}$. The $n$-step transition kernel is denoted by $Q^n(x,y)$.

Define, also, the kernel
\begin{align*}
\bar{\cS}^{{\rm epi}(X_0)}_{-t,n}(z_1,z_2)&:= \E_{\sfB_0=z_1} \big[ \bar{\cS}_{-t,n-\tau}(\sfB_\tau,z_2) \, \ind_{\tau<n}\big],
\end{align*}
where $\E$ is with respect to the geometric random walk $(\sfB_n)_{n\geq 0}$ and $\tau$ is the hitting time
\begin{align*}
\tau:=\min\{m\geq 0 \colon \sfB_m>X_0(m+1)\,\},
\end{align*}
given a sequence $X_0(1)>X_0(2)>\cdots$, which will represent the initial locations of particles in TASEP.
The theorem of Matetski-Quastel-Remenik is the following:
\begin{theorem}[\cite{MQR21}.]\label{fixed-point-thm}
Consider the TASEP with particles at initial positions
 $\infty>X_0(1)>X_0(2)>\cdots$,  at time zero. Denote the positions of the particles at time
$t$ by $X_t(1)>X_t(2)>\cdots$. Then, for any collection $1\leq n_1<n_2<\cdots < n_m$, we have that
\begin{align*}
\bbP\big( X_t(n_j) < a_j, \, j=1,...,m\big) = \det\big( I-\chi_a K_t^{\rm TASEP} \chi_a\big)_{\ell^2(\{n_1,...,n_m\}\times \Z)},
\end{align*}
with
\begin{align*}
K_t^{\rm TASEP}(n_i,\cdot; n_j,\cdot):=-Q^{n_j-n_i} \ind_{n_i<n_j} + (\cS_{-t,-n_i})^* \bar{\cS}^{{\rm epi}(X_0)}_{-t,n_j}
\end{align*}
and $\chi_a(n_j,x):=\ind_{x> a_j}$.
\end{theorem}
The path towards the above theorem was the construction of a bi-orthogonal set of functions, in the spirit of \cite{BFPS07}, via a solution to
a boundary value problem of a difference of operator related to the geometric transition kernel $Q(x,y)$ and with a boundary determined
by the initial condition of TASEP.
More recently, it was shown in \cite{MR22} that this approach is applicable to a wide range of determinantal models that admit Schutz-type \cite{S97}
determinantal formulas for their transition probabilities.  An alternative derivation of Theorem \ref{fixed-point-thm} via $\rsk$, which also covers the
case of TASEP with both particle and time inhomogeneous rates was provided in \cite{BLSZ22}.
\vskip 2mm
{\bf The Airy sheet and the Directed Landscape.} Two, more general, notions are those of the Airy sheet and the Directed Landscape established
\cite{DOV18} (see also \cite{CQR15} for proposing the notion of the Airy sheet). Similarly, in some sense, to the KPZ fixed point these notions
aim to be the universal objects that capture the scaling limits of the whole field of last passage percolation times  
$ \big\{ \tau\big( (x,n)\to (y,m)\big) \colon x,y\in \Z, n.m\in \N, \,\text{with}\,\, n<m \,\big\}$
(as well as other models in the KPZ class) when
 both the starting and end points vary. The Airy line ensemble and the ${\rm Airy}_2$ process that
we discussed earlier turn out to be marginals of these more universal objects.

The Airy sheet was constructed in \cite{DOV18} as the two-parameter scaling limit of a {\it semi-discrete} last passage percolation model,
also known as {\it Brownian last passage percolation}. Very briefly, given $n$ independent Brownian motions $B_1(\cdot),...,B_n(\cdot)$,
the Brownian last passage percolation is the quantity
\begin{align}\label{brown-lpp}
B\big[ (x,n) \to (y,1)\big]:=\sup_{x=t_0<t_1<\cdots < t_n=y} \, \sum_{i=1}^n \big( B_{n+1-i}(t_i)-B_{n+1-i}(t_{i-1})\big) .
\end{align}
One can think of the Brownian last passage percolation as the Brownian scaling limit for $N$ large (e.g. via Donsker's invariance principle) of the
usual last passage percolation on an array $(w_{ij}\colon 1\leq i\leq n, 1\leq j\leq N)$, where the entries of the $i^{th}$ row
$(w_{ij} \colon 1\leq j \leq N)$ represent the increments of the Brownian motion $B_i(\cdot)$. The theorem that establishes the existence of the
Airy sheet is the following:
\begin{theorem}[\cite{DOV18}]\label{thm:dov}
Let $B\big[(\cdot,\cdot)\to(\cdot,\cdot)\big]$ be the two-parameter family of Brownian last passage percolation. Then there exists
 a two-parameter process $\cS(\cdot,\cdot)\colon\R^2\to\R$ such that
\begin{align*}
B\big[ (2x/n^{1/3}, n)\to (1+2y/n^{1/3},1)\big] = 2\sqrt{n} +2(y-x)n^{1/6} + n^{1/2} \big( \cS+o_n\big)(x,y),
\end{align*}  
where $o_n$ are random functions that converge to zero in compact sets in the sense that, for some $a>1$ and
$K\subset \R^2$ compact,  $\bbE\big[  a^{\sup_K  |o_n|^{3/2}}\big]\to 0$.
\end{theorem}
The axiomatic definition of the Airy sheet was also given in \cite{DOV18} and is as follows:
\begin{definition}[Airy sheet - \cite{DOV18}]\label{def:dov1}
The Airy sheet is a random continuous function $\cS\colon \R^2\to\R$ such that
\begin{itemize}
\item[(i)] $\cS$ is invariant under spatial shifts, i.e. for any $t\in\R$, $\cS(\cdot+t,\cdot+t)$ has the same law as $\cS(\cdot,\cdot)$,
\item[(ii)] if $\mathfrak{h}=(\mathfrak{h}_1, \mathfrak{h}_2,...)$ is the Airy line ensemble minus a parabola,
 then the Airy sheet and the Airy line ensemble can
be defined in the same probability space, so that $\cS(0,\cdot)=\mathfrak{h}_1(\cdot)$ and for all $(x,y,z)\in\bbQ_+\times \bbQ^2$, there 
exists a random $K_{x,y,z}\in\Z$ such that, $a.s.$, for all $k\geq K_{x,y,z}$
\begin{align*}
\mathfrak{h}\big[ (-\sqrt{k/2x}, k) \to (z,1)\big] - \mathfrak{h}\big[ (-\sqrt{k/2x}, k) \to (y,1)\big] = \cS(x,z) -\cS(x,y),
\end{align*} 
\end{itemize}
where $\mathfrak{h}\big[ (x,m)\to (y,n)\big]$ denotes the last passage percolation on the Airy line ensemble (minus a parabola) 
in the same sense as in \eqref{brown-lpp}
but with Brownian motions therein replaced by $\mathfrak{h}_1,\mathfrak{h}_2,...$.
\end{definition} 
The Airy sheet appears as a marginal of a four-parameter process called the {\bf Directed Landscape}. This was also defined and constructed in
\cite{DOV18},
 \begin{definition}[The directed landscape - \cite{DOV18}]\label{defDOV2}
 Define the set of quadruples $\R^4_{\uparrow}:=\{(x,s;y,t) \in \R^4 \colon s<t\}$.
 The directed landscape is a random, continuous function $\cL\colon \R^4_{\uparrow} \to \R$ satisfying
 \begin{itemize}
 \item[(i)] The so-called metric composition law:
 \begin{align*}
 \cL(x,r;y,t) = \max_{z\in \R} \big[ \cL(x,r;z,s)+ \cL(z,s;y,t)\big],\qquad \text{for any} \,\,(x,r;y,t)\in \R^4_{\uparrow} \,\,\text{and}\,\, s\in(r,t),
 \end{align*}
 \item[(ii)] for any set of disjoint intervals $(t_i,t_i+s_i^3), i\geq 1$, it holds that $\cL(\cdot,t_i;\cdot, t_i+s_i^3)$ are independent Airy sheets $\cS_i$
 scaled as $s_i\cS(\cdot/s_i^2, \cdot / s_i^2)$.
 \end{itemize}
 \end{definition}
 Observing the convergence result in Theorem \ref{thm:dov}, which gives rise to the Airy sheet as the limit of (Brownian) last passage
 percolation with varying start and end points, and item (ii) in Definition \ref{def:dov1}, which regards the last passage percolation on the
 Airy line ensemble, one may be motivated to conclude that the two last passage percolations have some direct relation. 
 This is, indeed, the case
 and this was the key point established in \cite{DOV18}, which led to the above constructions. 
  As was observed by Dauvergne 
 (see \cite{C20}) this fact is a consequence of the Noumi-Yamada formulation of $\rsk$ and we will briefly explain this now in the discrete 
 setting.
 
 Let us start from the $\grsk$ setting; we can then easily pass to the $\rsk$ setting by tropicalisation (i.e. replace $(+,\times)$ with $(\max,+)$). 
 Assume an input matrix $\sfX=\big(x^i_{j} \colon 1\leq i \leq n, 1\leq j \leq N\big)$. As already seen in Theorem \ref{thm:grsk}, the output tableau $P$
 of $\grsk(\sfX)$ can be read through the matrix equation
 \begin{align*}
H(\bx^1) H(\bx^{2}) \cdots H(\bx^n) &= H_k(\bp^k) H_{k-1}(\bp^{k-1}) \cdots H_1(\bp^1) ,\qquad k=\min(n,N),\
\end{align*}
Without loss of generality, let us assume that $N\geq n$; this is also what would be needed if we would like to scale in $N$ and 
arrive to a Brownian last passage percolation via an application of Donsker's principle.
As we have already seen in Theorem \ref{thm:NY}, the $(i,j)$ entry ($j\geq i$) of either matrix in the above  equation admits a 
representation as the weight of directed paths. In particular, the equality between the $(i,j)$ entries of the above matrices 
is graphically given by
  \begin{equation}\label{grapheq}
  \sum_{\pi} \prod_{(k,\ell)\in \pi}
  {\begin{tikzpicture}[baseline={([yshift=-.ex]current bounding box.center)},vertex/.style={anchor=base,
    circle,fill=black!25,minimum size=18pt,inner sep=2pt}, scale=0.4]
    \draw (0,0) grid (10,-4);
  \foreach \i in {0,1,...,10}{
		\foreach \j in {0,...,4}{
		\node[draw,circle,inner sep=1pt,fill] at (\i,-\j) {};
	}}
	\draw[ultra thick, blue, ->]  (3,0)--(3,-1)--(5,-1)--(5,-2)--(6,-2)--(6,-3)--(7,-3)--(7,-4);
	\node at (-1, 0.1) {$\bx^1$};  \node at (-1, -4.0) {$\bx^n$};
	 \node at (7, -4.5) {$j$};  \node at (3, 0.7) {$i$};
\end{tikzpicture}} 
\,\, \,\,\, = \,\,
  \sum_{\pi} \prod_{(k,\ell)\in \pi}
{\begin{tikzpicture}[baseline={([yshift=-.5ex]current bounding box.center)},vertex/.style={anchor=base,
    circle,fill=black!25,minimum size=18pt,inner sep=2pt}, scale=0.4]
\draw (0,0) -- (10,0) ; \draw (1,-1) -- (10,-1); \draw (2,-2) -- (10,-2); \draw (3,-3) -- (10,-3); \draw (4,-4) -- (10,-4);
\draw (10,0) -- (10,-4); \draw (9,0) -- (9,-4); \draw (8,0) -- (8,-4); \draw (7,0) -- (7,-4); \draw (6,0) -- (6,-4);
\draw (5,0) -- (5,-4); \draw (4,0) -- (4,-4); \draw (3,0) -- (3,-3); \draw (2,0) -- (2,-2); \draw (1,0) -- (1,-1);
\draw  [fill ] (0,0) circle [radius=0.1];\draw  [fill ] (1,0) circle [radius=0.1];\draw  [fill ] (2,0) circle [radius=0.1]; \draw  [fill ] (3,0) circle [radius=0.1]; \draw  [fill ] (4,0) circle [radius=0.1]; \draw  [fill ] (5,0) circle [radius=0.1];\draw  [fill ] (6,0) circle [radius=0.1];\draw  [fill ] (7,0) circle [radius=0.1]; \draw  [fill ] (8,0) circle [radius=0.1]; \draw  [fill ] (9,0) circle [radius=0.1];\draw  [fill ] (10,0) circle [radius=0.1];

\draw  [fill ] (1,-1) circle [radius=0.1];\draw  [fill ] (2,-1) circle [radius=0.1];\draw  [fill ] (3,-1) circle [radius=0.1]; \draw  [fill ] (4,-1) circle [radius=0.1]; \draw  [fill ] (5,-1) circle [radius=0.1]; \draw  [fill ] (6,-1) circle [radius=0.1];\draw  [fill ] (7,-1) circle [radius=0.1];\draw  [fill ] (8,-1) circle [radius=0.1]; \draw  [fill ] (9,-1) circle [radius=0.1]; \draw  [fill ] (10,-1) circle [radius=0.1];

\draw  [fill ] (2,-2) circle [radius=0.1];\draw  [fill ] (3,-2) circle [radius=0.1]; \draw  [fill ] (4,-2) circle [radius=0.1]; \draw  [fill ] (5,-2) circle [radius=0.1]; \draw  [fill ] (6,-2) circle [radius=0.1];\draw  [fill ] (7,-2) circle [radius=0.1];\draw  [fill ] (8,-2) circle [radius=0.1]; \draw  [fill ] (9,-2) circle [radius=0.1]; \draw  [fill ] (10,-2) circle [radius=0.1];

\draw  [fill ] (3,-3) circle [radius=0.1]; \draw  [fill ] (4,-3) circle [radius=0.1]; \draw  [fill ] (5,-3) circle [radius=0.1]; \draw  [fill ] (6,-3) circle [radius=0.1];\draw  [fill ] (7,-3) circle [radius=0.1];\draw  [fill ] (8,-3) circle [radius=0.1]; \draw  [fill ] (9,-3) circle [radius=0.1]; \draw  [fill ] (10,-3) circle [radius=0.1];

 \draw  [fill ] (4,-4) circle [radius=0.1]; \draw  [fill ] (5,-4) circle [radius=0.1]; \draw  [fill ] (6,-4) circle [radius=0.1];\draw  [fill ] (7,-4) circle [radius=0.1];\draw  [fill ] (8,-4) circle [radius=0.1]; \draw  [fill ] (9,-4) circle [radius=0.1]; \draw  [fill ] (10,-4) circle [radius=0.1];
\node at (10.7, 0.1) {$\bp^1$};  \node at (10.7, -4.0) {$\bp^n$};
 \node at (7, 0.7) {$j$};
\draw[ultra thick, red, ->]  (3,-3)--(5,-3)--(5,-2)--(6,-2)--(6,-1)--(7,-1)--(7,0); \node at (2.5, -3) {$i$};
\end{tikzpicture}} 
\end{equation}
Each side is to be understood as the sum over all directed paths of the products of the weights on sites traced by the paths,
starting from site $(1,i)$ and ending at site $(n,j)$ in the left-hand side and starting from site $(i\wedge n,i)$ and ending at site $(1,j)$ 
in the right-hand side. 

In the zero temperature limit
$\sum_\pi\prod_{(k,\ell)\in\pi}$ will be replaced by $\max_\pi \sum_{(k,\ell)\in\pi}$ reducing to last passage percolation between the marked sites.
Switching to the last passage percolation setting, the variables $(p^i_j \colon 1\leq i \leq n\,,\, i\leq j\leq N \,)$ encode the increments of the Airy line ensemble.
More precisely, we can pass to the Gelfand-Tsetlin notation (see Section \ref{sec:GT})  $z^{i}_j:= p^j_j + p^{j}_{j+1}+\cdots + p^j_i$
and represent $(z^i_j \colon1\leq i \leq N\ \, n \leq j \leq i \,)$, which we represent
 graphically as in Figure \ref{fig:lppAiry}. Connecting the points with non-intersecting lines 
produces a PNG representation, with the jumps of the lines being given by the corresponding variable $p^i_j$. The optimising, red path in 
the right-hand side of \eqref{grapheq} is the red path in Figure \ref{fig:lppAiry}, which in agreement with \eqref{grapheq} starts from
the point $z_{i\wedge n}^i$ and ends at point $z^{j}_1$. Thus, equation \eqref{grapheq} shows that the value of the last passage
percolation on a rectangular array equals the value of last past percolation on the PNG line ensemble. 
In particular, the fact that, in a suitable, large $N$ limit, the last passage percolation on the grid reduces to the Brownian last passage percolation
and the fact that the PNG line ensemble converges to the Airy line ensemble,
 clarifies the link between Definition \ref{def:dov1} and Theorem \ref{thm:dov}.
\begin{figure}[t]
 	\begin{center}
 		\begin{tikzpicture}[scale=0.45]
		\fill[black]  (0,10) circle [radius=0.175]; \fill[black]  (0,8) circle [radius=0.175]; \fill[black]  (0,6) circle [radius=0.175]; \fill[black]  (0,5) circle [radius=0.175];
		\fill[black]  (2,9) circle [radius=0.15]; \fill[black]  (2,7) circle [radius=0.15]; \fill[black]  (2,5.5) circle [radius=0.15]; \fill[black]  (2,4) circle [radius=0.15];
		\fill[black]  (4,8) circle [radius=0.15]; \fill[black]  (4,6) circle [radius=0.15]; \fill[black]  (4,5) circle [radius=0.15]; \fill[black]  (4,3) circle [radius=0.15];
		\fill[black]  (6,7) circle [radius=0.15]; \fill[black]  (6,5.5) circle [radius=0.15]; \fill[black]  (6,3.5) circle [radius=0.15]; \fill[black]  (6,2) circle [radius=0.15];
		\fill[black]  (8,6.25) circle [radius=0.15]; \fill[black]  (8,4.5) circle [radius=0.15]; \fill[black]  (8,2.75) circle [radius=0.15];
		 \fill[black]  (10,5) circle [radius=0.15]; \fill[black]  (10,3.5) circle [radius=0.15];
		 \fill[black]  (12,4) circle [radius=0.15];
		 \draw[-, thick] (0,10)--(1,10)--(1,9)--(3,9)--(3,8)--(5,8)--(5,7)--(7,7)--(7,6.25)--(9,6.25)--(9,5)--(11,5)--(11,4)--(13,4)--(13,0)--(17,0);
		 \draw[-, thick] (0,8)--(1,8)--(1,7)--(3,7)--(3,6)--(5,6)--(5,5.5)--(7,5.5)--(7,4.5)--(9,4.5)--(9,3.5)--(11,3.5)--(11,-1)--(17,-1);
		 \draw[-, thick] (0,6)--(1,6)--(1,5.5)--(3, 5.5)--(3,5)--(5,5)--(5,3.5)--(7, 3.5)--(7, 2.75)--(9, 2.75)--(9,-2)--(17,-2);
		 \draw[-, thick] (0,5)--(1,5)--(1,4)--(3,4)--(3,3)--(5,3)--(5,2)--(7,2)--(7,-3)--(17,-3);
		 \draw[-, ultra thick] (0,-3)--(0,12);
		 \draw[<-, ultra thick, red]  (2,9)--(3,9)--(3,8)--(4,8)--(4,6)--(5,6)--(5,5.5)--(6,5.5)--(6,3.5)--(7,3.5)--(7,2.75)--(8,2.75);
		 \node at (-1,10) {\small{{ $z^N_1$}}};  \node at (-1,5.5) {\small{{ $z^N_n$}}}; \node at (-1,8.5) {$\vdots$}; \node at (-1,7.5) {$\vdots$};
		 \node at (12, 4.7) {\small{{ $z^1_1$}}};  \node at (10.4, 5.7) {\small{{ $z^2_1$}}};  \node at (10.4, 4.1) {\small{{ $z^2_2$}}};
		  \node at (6.35, 2.5) {\small{{ $z^n_n$}}};
\end{tikzpicture}
 	\end{center}  
 	\caption{ \small  This figure is an alternative depiction of the right-hand side of equation \eqref{grapheq}.
	It shows the last passage percolation on the PNG line ensemble,  where
	the red path maximises the sum of the jumps.
	We have only drawn half of the PNG line ensemble, which corresponds to the $P$ tableau as
given by the graphical equation \eqref{grapheq}. The graphical representation of the $Q$ tableau would occupy the left of the vertical line.
The fact that the paths go only upwards, as opposed to Figure \ref{fig:multi-png} has to do with the
fact that the input array here is a matrix $\sfX=(x^i_j \colon 1\leq i\leq n, 1\leq j\leq N )$, while in Figure \ref{fig:multi-png} the input is a triangular array. 
 	}\label{fig:lppAiry}
 \end{figure}		
\vskip 2mm		
{\bf Multi-time distributions.} 		
There has recently been important progress in computing temporal correlations of last passage percolation and TASEP models. 
Recent works \cite{BL19, J19, JR21, L20} have computed limits of the form
\begin{align*}
\lim_{N\to\infty}
\bbP\Big( \tau_{t_iN+\lfloor c_1(t_iN)^{1/3}\rfloor\, , \, t_iN-\lfloor c_1(t_iN)^{1/3} \rfloor } \leq c_2\,(t_iN) + c_3 u_i(t_iN)^{2/3} ,\,\,\text{for} \,\,i=1,...,k  \Big) 
\end{align*} 
for $0<t_1<\cdots <t_k$, $u_1,...,u_k\in\R$ and with $c_1,c_2,c_3$ suitable constants.
It is worth noting that the formulas that arise have the form of an integral of a Fredholm determinant as opposed to the spatial correlation
case, where, as we saw, it is captured by just Fredholm determinants.

\section{An exactly solvable Random Polymer model}\label{sec:loggamma}
\subsection{The log-gamma polymer and Whittaker functions}\label{sec:loggama-Whit}
The positive temperature solvable counterpart to last passage percolation  is the
{\bf log-gamma polymer model}. 
In this section we will show how this model can be analysed in a fashion that, largely, parallels the way that last passage percolation was analysed 
in the previous sections. In particular, $\grsk$ will be used instead of $\rsk$ and Whittaker functions will emerge at the place of Schur functions.
Some important differences between the two models will also be highlighted. Before proceeding, we mention a convention that we will follow on contour
integrals: lines, typically denoted by $\ell$, will be upwards directed, while circles, typically denoted by $C$ or $\gamma$ will have positive orientation.

To define the model, let us consider a matrix   $\sfW=(\bw^{\,i}_j \,\colon\, 1\leq i \leq m\,,\, 1\leq j \leq n)$.
The partition function of the polymer model is defined as
\begin{align*}
Z_{m,n}:=\sum_{\pi\,\in\,\Pi_{(m,n)}} \prod_{(i,j)\in \pi} \bw^{\,i}_j,
\end{align*}
 where $\Pi_{(m,n)}$ is the set of down-right paths from entry $(1,1)$ to entry $(m,n)$ (in matrix notation). We immediately notice that the polymer partition
 $Z_{m,n}$ fits the framework of $\grsk$, encoded as the rightmost entries $z^m_1=(z^m_1)'$
  of the output Gelfand-Tsetlin patterns $(\sfZ,\sfZ')=\grsk(\sfW)$, see \eqref{zN1_poly}.
 \vskip 2mm 
 The integrable distribution on $\sfW$ amounts to considering
 entries $(\bw^i_j)$ as independent random variables with an {\bf inverse-gamma} distribution
\begin{align}\label{log-gamma-measure}
\bbP(\bw^i_j=w_{ij}) = \frac{1}{\Gamma(\alpha_i+\beta_j)} 
 w_{ij}^{-\alpha_i-\beta_j} e^{-1/w_{ij}} \frac{\dd w_{ij}}{w_{ij}},\qquad 1\leq i\leq m\,,\, 1\leq j\leq n,
\end{align}
where for $a>0$, $\Gamma(a)=\int_0^\infty x^{a-1}e^{-x} \dd x$ is the Gamma function (which admits a meromorphic extension to the
complex plane except for the negative integers, including $0$)
and $\alpha_i,\beta_j$ real parameters with $\alpha_i+\beta_j>0$. 
The reason this model is called ``log-gamma'' and not ``inverse-gamma'' is due to the fact that
the inverse-gamma variables $\bw^i_j$ appear naturally  
when the more standard formulation \eqref{DPRMp2p} for the polymer partition function is used with  disorder $\omega(n,x)$
distributed as log-gamma variables.
\vskip 2mm
The log-gamma polymer was introduced by Sepp\"al\"ainen
\cite{S12} inspired by ideas from hydrodynamics \cite{BCS06}, \cite{BS10}. In particular, he observed that
the inverse-gamma variables are the unique solutions-in-law to the following set of transformations:
\vskip 1mm
\begin{itemize}
\item[] if $U,V,Y$ are independent, positive random 
variables, then the variables  
\begin{align}\label{burke}
U':=Y(1+UV^{-1}),\quad V':=Y(1+VU^{-1}),\quad Y':=(U^{-1}+V^{-1})^{-1}
\end{align}
have the same distribution as $(U,V,Y)$ if and only if the random variable $U$
 is such that $U^{-1}$ is distributed according to
a Gamma variable $\text{Gamma}(\theta,r)$, which has probability density function $\Gamma(\theta)^{-1} r^{\theta} x^{\theta-1} e^{-rx}$ 
on the positive real line. (we denote this as 
$U^{-1}\sim \text{Gamma}(\theta,r)$) and, moreover $V^{-1}$ and $Y^{-1}$ are distributed as
 $V^{-1}\sim \text{Gamma}(\mu-\theta,r)$ and $Y^{-1}\sim\text{Gamma}(\mu,r)$, for some parameters $0<\theta<\mu$. 
\end{itemize}
\vskip 1mm
This is Lemma 3.2 in \cite{S12}. The role of variables $U$ is played by ratios of partition functions $Z_{m,n}/Z_{m-1,n}$, of $V$ by ratios $Z_{m,n}/Z_{m,n-1}$, while the role of $Y$ is played by variables $\bw^m_n$ and
 equations \eqref{burke} are derived directly from the recursive equation 
$Z_{m,n}=\bw^m_n(Z_{m-1,n}+Z_{m,n-1})$, first by dividing both sides by $Z_{m-1,n}$, then by $Z_{m,n-1}$ and finally by both $Z_{m,n}$ and $\bw^m_n$.
\vskip 2mm
In terms of hydrodynamics, the ratios $Z_{m,n}/Z_{m-1,n}$ and $Z_{m,n}/Z_{m,n-1}$ can be thought of as a (geometric lifting of) ``flow'' through edges $\{(m-1,n),(m,n)\}$ and 
$\{(m,n-1),(m,n)\}$, respectively and the invariance of transformations \eqref{burke} indicate that this
``flow'' is stationary; a property related to Burke's theorem from queueing theory, \cite{BCS06}, \cite{OY01}. Ideas of relating
longest increasing subsequence problems (and thus also last passage percolation / polymer models)
 to hydrodynamics can be traced back to the work
of Aldous and Diaconis \cite{AD95} (see also \cite{CG06}).
 
\vskip 2mm
A dynamic analysis of the log-gamma polymer based on Noumi-Yamada's matrix formulation of $\grsk$ and 
{\it intertwining of Markov processes} was performed in \cite{COSZ14}. We will come back to the dynamic approach in Section \ref{sec:dynamics}.
Here we will expose a bijective approach based on the local moves formulation of $\grsk$ (Section \ref{sec:localmove}), which parallels the steps of the solution of last passage percolation.  Let us start by summarising the steps (compare also with the corresponding steps in Section \ref{geomLPP}) before exposing some of the details:

\begin{itemize}
\item[{\bf Step 1: Combinatorial Analysis.}]
 The first step is to analyse the combinatorial structure of the polymer and this is done with the use of $\grsk$. 
As we have seen, $\grsk$ maps an input matrix to two (geometric) Gelfand-Tsetlin patterns, the bottom-right corner of each one 
being identical to the random polymer partition function, see \eqref{zN1_poly}.
\item[{\bf Step 2: Push forward measure.}] Next we need to determine the push forward measure $\bbP\circ\grsk^{-1}$
 under $\grsk$. At this step the choice of an inverse-gamma measure is crucial as it maps to a tractable measure. 
\item[{\bf Step 3: Special functions and harmonic analysis.}] Once the push forward measure is identified, we compute its 
marginal on the shape of the geometric $\GT$ patterns and at this point emerge special functions, called {\bf Whittaker functions},
see Section \ref{Whittaker-intro}. 
These play the counterpart role of Schur functions in last passage percolation.
Whittaker functions are eigenfunctions
of integrable operators and possess a Plancherel (Fourier analysis) theory. 
This helps to compute the Laplace transform of the first coordinate of the
shape, which coincides with the polymer partition function.
\item[{\bf Step 4: Contour integrals.}] The Plancherel theory of Whittaker functions is combined with the evaluation of certain integrals
of Whittaker functions, which are expressed
 in terms of Gamma functions, leading to an $n$-fold contour integral for the Laplace transform of 
the polymer partition function that involves products of Gamma functions. Let us remark that even though this integral looks complicated, it is already a simplification as originally the polymer partition function is a function of $m\times n$ variables 
and so the computation of the Laplace transform involves an $m\times n$-fold integral.
\item[{\bf Step 5: Fredholm determinant and asymptotics.}]
 The last step is to turn the contour integral of the previous step into a Fredholm determinant. This essentially reduces
 the asymptotic analysis to that of contour integrals of a single variable, where the dimension $n$  of the contour integral in Step 4
 appears now as a parameter in the integrand. The asymptotic analysis can now be performed via the use of the Steepest Descent method 
 as described in Section \ref{sec:asymptot}. 
\end{itemize}
Steps 1-4 were performed in \cite{OSZ14} via a bijective approach and in \cite{COSZ14} via a dynamic approach, while Step 5
was performed in \cite{BCR13}. Notice that compared to the four-step solution to the last passage percolation model, the solvability of the
log-gamma polymer contains an extra step (Step 4). This step bypasses the lack of a determinantal structure in Whittaker functions,
which was present in the last passage percolation and Schur case.
Still, the absence or lack of a full understanding of a determinantal structure makes the derivation of the 
Fredholm determinant (Step 5), currently, ad hoc.
\vskip 2mm
Let us now present some details of the above outline. Step 1 has already been described in
Section \ref{sec:grsk}, see \eqref{zN1_poly}.
 To perform Step 2, we need to uncover some interesting properties of $\grsk$, which are summarised in
the following theorem. For simplicity we state the theorem for the case of a square input matrix $\sfW$. This theorem was proved in
\cite{OSZ14}.
\begin{theorem}\label{grsk_prop}
Let $(\sfZ,\sfZ')=\grsk(\sfW)$ be the output of $\grsk$ with input matrix $\sfW=(w_{ij}\colon 1\leq i,j\leq n)$, with
$\sfZ=(z^i_j\colon 1\leq j\leq i\leq  n)$ and $\sfZ'=((z^i_j)'\colon 1\leq j\leq i\leq  n)$ being the corresponding Gelfand-Tsetlin patterns.
Then it holds that
\begin{itemize}
\item[{\bf 1.}]   the products over columns and over rows of entries in the $\sfW$ matrix are expressed in terms of products and ratios in the 
Gelfand-Tsetlin patterns as 
\begin{align*}
\prod_{i=1}^n w_{ij} &=\frac{z^j_1\cdots z^j_j}{z^{j-1}_1\cdots z^{j-1}_{j-1}} \qquad \text{for}\qquad  j=1,...,n\quad\text{and} \\
\prod_{j=1}^n w_{ij} &=\frac{(z^i_1)'\cdots (z^i_i)'}{(z^{i-1}_1)'\cdots (z^{i-1}_{i-1})'} \qquad \text{for}\qquad  i=1,...,n.
\end{align*}
\item[{\bf 2.}] 
\begin{align*}
\sum_{1\leq i,j \leq n}\frac{1}{w_{ij}}&=\frac{1}{z^n_n}+\sum_{1\leq j\leq i\leq n}\Bigg(\frac{z^i_j}{z^{i+1}_{j}}+\frac{z^{i+1}_{j+1}}{z^i_j}\Bigg) +
\sum_{1\leq j\leq i\leq n}\Bigg(\frac{(z^i_j)'}{(z^{i+1}_{j})'}+\frac{(z^{i+1}_{j+1})'}{(z^i_j)'}\Bigg) \\
&=:\frac{1}{z^n_n}+\cE(\sfZ)+\cE(\sfZ')
\end{align*}
where the last line is the definition of the {\it energy} $\cE(\sfZ)$ of a geometric $\GT$ pattern $\sfZ$.
\item[{\bf 3.}] the mapping 
\begin{align*}
(\log w_{ij}\colon 1\leq i,j\leq n) \mapsto 
(\log z^i_j, \log (z^i_j)'\colon 1\leq j\leq i\leq n)
\end{align*}
has Jacobian $\pm 1$.
\end{itemize}
\end{theorem}
Item 1. in the above theorem is the geometric lifting of \eqref{type}, which 
relates the `{\it type}' of the geometric Gelfand-Tsetlin patterns to the input matrix $\sfW$.
Item 3. is the volume preserving property of $\grsk$, see Theorem \ref{thm:jac}.
Item 2 can be proved using the local moves and induction. 
\vskip 2mm
The push forward measure $\bbP\circ\grsk^{-1}$ can be now easily derived via Theorem \ref{grsk_prop}. 
For a (geometric) Gelfand-Tsetlin pattern $\sfZ=(z^i_j\colon 1\leq j\leq i \leq n)$, 
we define the {\bf type}, denoted by $type(\sfZ)$,  as the vector
\begin{align*}
\text{type}(\sfZ):=\Big(\frac{z^i_1\cdots z^i_i}{z^{i-1}_1\cdots z^{i-1}_{i-1}} \,;\, i=1,...,n\Big),
\end{align*} 
and for a vector $\alpha=(\alpha_1,...,\alpha_n)\in\mathbb{C}^n$ we also denote 
\begin{align*}
\text{type}(\sfZ)^{-\alpha}:=\prod_{i=1}^n\text{type}(\sfZ)_i^{-\alpha_i}.
\end{align*}
We can now write the push forward measure $\bbP\circ\grsk^{-1}$. We start by  noting that
\begin{align}\label{push}
&\bbP\Big( \bw^i_{j} \in \dd w_{ij} \,\,\,\, \text{for} \,\,\,\, 1\leq i,j\leq n \Big) = \frac{1}{\prod_{i,j}\Gamma(\alpha_i+\beta_j)} 
 \prod_{i,j}w_{ij}^{-\alpha_i-\beta_j} e^{-1/w_{ij}} \frac{\dd w_{ij}}{w_{ij}}\notag\\
 &=
 \frac{1}{\prod_{i,j}\Gamma(\alpha_i+\beta_j)} 
 \prod_i\Big(\prod_{j}w_{ij}\Big)^{-\alpha_i}  \prod_j\Big(\prod_{i}w_{ij}\Big)^{-\beta_j} e^{-\sum_{i,j}\frac{1}{w_{ij}}} 
 \prod_{i,j}\frac{\dd w_{ij}}{w_{ij}}\notag\\
 &=\frac{1}{\prod_{i,j}\Gamma(\alpha_i+\beta_j)}  \text{type}(\sfZ)^{-{\beta}}  \text{type}(\sfZ')^{-\alpha} \,
 e^{-1/z_{nn}-\cE(\sfZ)-\cE(\sfZ')} \prod_{1\leq j\leq i\leq n} \frac{\dd z^i_j}{z^i_j}\frac{\dd (z^i_j)'}{(z^i_j)'}.
\end{align}
The first equality in the above sequence is just the definition of the inverse-gamma probability law of the array $\sfW=\{\bw^i_{j}\}$, while the
last equality should be viewed as a change of variables from $\sfW\mapsto (\sfZ,\sfZ')$ via the use of the properties of $\grsk$ listed in
Theorem \ref{grsk_prop}.

We now observe from \eqref{push} that the push forward measure essentially factorizes over the $\sfZ$ and $\sfZ'$ variables, conditionally on fixing
 their (common) bottom rows, which are the {\bf shape} variables of each pattern 
 \begin{align}\label{geoshape}
sh(\sfZ):=(z^n_j\colon j=1,...,n)\qquad \text{and} \qquad sh(\sfZ'):=(\,(z^n_j)'\colon j=1,...,n)
\end{align}
 and which coincide. Therefore, the marginal of this measure on the
common shape of the geometric $\GT$ patterns $\sfZ,\sfZ'$ can be computed as we will demonstrate below.
Towards this task, we introduce the $GL_n(\bbR)$-Whittaker 
functions via Givental's \cite{Giv97} integral formula
\begin{align}\label{int-giv}
 \Psi^{\mathfrak{gl}_n}_{\gl}(x) :=  \int_{\gGT(x)}\text{type}(\sfZ)^{-{\gl}}  \,
 \exp\Big(-\cE(\sfZ) \,\Big) \prod_{1\leq j\leq i\leq n-1} \frac{\dd z^i_j}{z^i_j},
\end{align}
for a vector $x\in\R^n$,
with the integral running over all geometric Gelfand-Tsetlin patterns $\sfZ$ with shape (bottom row) equal to $x$. 
It is worth noticing the similarity of this integral formula to formula \eqref{GTschur} of Schur functions. In fact, it is not difficult to
check that $\epsilon^{\tfrac{n(n-1)}{2}} \Psi^{\mathfrak{gl}_n}_{\epsilon\gl}(e^{x_1/\epsilon},...,e^{x_n/\epsilon}) $ converges as $\epsilon\to0$ to 
a Riemann sum version of \eqref{GTschur}.
 We should remark at this point on a confusing notational convention between Whittaker
$\Psi^{\mathfrak{gl}_n}_{\gl}(x)$ and Schur $s_\gl(x)$ functions: the `shape' variables $x$ in $\Psi^{\mathfrak{gl}_n}_{\gl}(x)$ correspond to $\gl$ in $s_\gl(x)$ and the parameter
variables $\gl$ (also called `spectral variables') of $\Psi^{\mathfrak{gl}_n}_{\gl}(x)$ correspond to the argument $x$ in $s_\gl(x)$.
\vskip 2mm
Given, now, the integral formula \eqref{int-giv} and \eqref{push}, \eqref{geoshape}, we can write
\begin{align}\label{shape}
\bbP\big(sh(\sfZ)=sh(\sfZ'\big)\,\in\,\dd x)= \frac{1}{\Gamma_{\alpha,\beta}}  e^{-1/x_n} \Psi^{\mathfrak{gl}_n}_{\alpha}(x) \Psi^{\mathfrak{gl}_n}_{\beta}(x) \prod_{i=1}^n\frac{\dd x_i}{x_i},
\end{align}
where we denoted $\Gamma_{\alpha,\beta}:=\prod_{i,j}\Gamma(\alpha_i+\beta_j)$.
\eqref{shape} is a result of  integrating  \eqref{push} over all variables $(z^i_j)$ and
$(\,(z^i_j)'\,)$ except the shape variables $z^n_j=(z^n_j)', j=1,...,n$, which are fixed to be equal to $x_j$ for $j=1,...,n$ .

 Let us observe that integrating \eqref{shape} over all $x\in \bbR_+^n$ yields the identity
\begin{align}\label{bumpstade}
\int_{\bbR_+^n} e^{-1/x_n} \Psi^{\mathfrak{gl}_n}_{\alpha}(x) \Psi^{\mathfrak{gl}_n}_{\beta}(x) \prod_{i=1}^n\frac{\dd x_i}{x_i}
= \prod_{i,j}\Gamma(\alpha_i+\beta_j).
\end{align}
Even though the parameters $\alpha,\beta$ are real the above identity can be extended in the case of complex parameters such that $\Re(\alpha_i+\beta_j)>0$.
This type of identity is known in the number theory literature as the Bump-Stade identity and plays an important role in the theory
of automorphic forms. Identity \eqref{bumpstade}
 was conjectured by Bump \cite{Bum84} and was proven by Stade \cite{Sta02} via the use of Mellin transformations and Mellin-Barnes type integrals (these are special integrals of products and ratios of Gamma functions). Here we have obtained the Bump-Stade identity as a direct consequence of $\grsk$ and its properties.
Identity  \eqref{bumpstade} can also be viewed  as the analogue of Cauchy identity, which we will discuss later in Section \ref{sec:Cauchy}.
\vskip 2mm
Identity \eqref{bumpstade} is an important ingredient in Step 4, but first we need to complete Step 3 whose basic ingredient is
the Plancherel theory for $GL_n(\bbR)$-Whittaker functions which we now state.
\begin{theorem}\label{thm:Plancerel}
The integral transform
$$\hat f(\lambda) = \int_{(\R_{+})^n} f(x) \Psi^{\mathfrak{gl}_n}_\lambda(x) \prod_{i=1}^n \frac{\dd x_i}{x_i}$$
defines an isometry from $L_2((\R_{+})^n, \prod_{i=1}^n \dd x_i/x_i)$ onto 
$L^{sym}_2((\iota\R)^n,s_n(\lambda) \dd\lambda)$, where $L_2^{sym}$ is the space of $L_2$
functions which are symmetric in their variables (the variables $\lambda_1,\lambda_2,...$ are unordered), $\iota=\sqrt{-1}$ and 
\begin{align}\label{eq:sklyanin}
s_n(\lambda)=\frac1{(2\pi\iota)^n n!} \prod_{i\ne j} \Gamma(\lambda_i-\lambda_j)^{-1}
\end{align}
is the {\em Sklyanin} measure.
That is, for any two functions $f,g\in L_2((\R_{+})^n, \prod_{i=1}^n \dd x_i/x_i)$, it holds that
\[
\int_{(\R_{+})^n} f(x) g(x) \prod_{i=1}^n\frac{\dd x_i}{x_i} = \int_{(\iota\R)^n} \hat{f}(\lambda) \overline{\hat g (\gl)} s_n(\gl) \dd \gl.
\]
\end{theorem}
  We are now in a position, using the Plancherel theorem and identity \eqref{bumpstade}, to compute the Laplace transform of 
  the polymer partition function in a terms of a contour integral. 
  \begin{theorem}\label{thm:log-gamma-laplace}
  Let $Z_n=Z_{(n,n)}$ denote the point-to-point partition function of a log-gamma polymer with parameters 
  $(\alpha,  \beta)$ as in \eqref{log-gamma-measure}.
   Then its Laplace transform is given by the contour integral formula
  \begin{align}\label{laplace_integral}
  \bbE\Big[ e^{-sZ_{n}}\Big]= \int_{\ell_\delta^n} \dd \lambda \,
  s_n(\lambda) \prod_{1\leq j,j'\leq n} \Gamma(\lambda_j-\beta_{j'}) 
  \prod_{j=1}^n 
  \frac{s_n^{-\lambda_j}\prod_{1\leq i\leq n} \Gamma(\alpha_i+\lambda_j)}{s^{-\gb_j}\prod_{1\leq i\leq n} \Gamma(\alpha_i+\beta_j)},
  \end{align}
  where $s_n(\lambda)$ is the Sklyanin measure \eqref{eq:sklyanin}
  and the contour line $\ell_\gd$ is a vertical line with real part $\gd>0$ chosen so that the $\beta$ parameters lie to its left
  and the $\ga$ parameters lie to its right, so that the arguments of the Gamma functions 
  appearing in the above expression have positive real parts.
  \end{theorem} 
  \begin{proof}
  Let $(\sfZ,\sfZ')=\grsk(\sfW)$. We then have that the partition function $Z_n$ equals entry $z^n_1=(z^n_1)'$ of the $\sfZ$ and
  $\sfZ'$ geometric $\GT$ patterns. Therefore,
  \begin{align*}
  \bbE\Big[ e^{-sZ_{n}}\Big] &= \int e^{-sz^n_1} \,\,\bbP\circ\grsk^{-1}(\dd\sfZ,\dd\sfZ')\\
  & \stackrel{\eqref{push}}{=}\frac{1}{\Gamma_{\alpha,  \beta}} 
   \int e^{-sz^n_1} \text{type}(\sfZ)^{-{\beta}}  \text{type}(\sfZ')^{-\alpha} \,
 e^{-1/z_{nn}-\cE(\sfZ)-\cE(\sfZ')} \prod_{1\leq j\leq i\leq n} \frac{\dd z^i_j}{z^i_j}\frac{\dd (z^i_j)'}{(z^i_j)'}.
  \end{align*}   
  Integrating over all variables except the common shape, which we denote by $x\in\R^n$, we have
  \begin{align*}
   \bbE\Big[ e^{-sZ_{n}}\Big] 
   =\frac{1}{\Gamma_{\alpha,  \beta}} \int_{\bbR_+^n} e^{-s x_1-1/x_n} \Psi^{\mathfrak{gl}_n}_{\alpha}(x) \Psi^{\mathfrak{gl}_n}_{\beta}(x) \prod_{i=1}^n\frac{\dd x_i}{x_i},
  \end{align*}
  and using the Plancherel theorem \ref{thm:Plancerel} we write this as
  \begin{align}\label{Planch_appl}
  \frac{1}{\Gamma_{\alpha,  \beta}} \int_{i\bbR^n} 
   \dd \lambda  \, s_n(\lambda)
   \Big( \int_{\bbR_+^n} e^{-1/x_n} \Psi^{\mathfrak{gl}_n}_{\alpha}(x)  
   \Psi^{\mathfrak{gl}_n}_{\gl}(x) \prod_{i=1}^n\frac{\dd x_i}{x_i}  \Big)
   \Big( \int_{\bbR_+^n} e^{-s x_1} 
    \Psi^{\mathfrak{gl}_n}_{\beta}(x)  \Psi^{\mathfrak{gl}_n}_{-\gl}(x) \prod_{i=1}^n\frac{\dd x_i}{x_i} \Big),
  \end{align}
  The first parenthesis is given by \eqref{bumpstade} as
  \begin{align}\label{stad1}
  \int_{\bbR_+^n} e^{-1/x_n} \Psi^{\mathfrak{gl}_n}_{\alpha}(x)  
   \Psi^{\mathfrak{gl}_n}_{\gl}(x) \prod_{i=1}^n\frac{\dd x_i}{x_i}
   =\prod_{i,j} \Gamma(\alpha_i+\lambda_j).
  \end{align}
  The second parenthesis also reduces easily to \eqref{bumpstade} via the
  properties of Whittaker functions
  \begin{align}
&\hskip 2cm\Psi^{\mathfrak{gl}_n}_{\alpha}(s x)=s^{-\sum_i\alpha_i} \Psi^{\mathfrak{gl}_n}_{\alpha}(x) \qquad \text{and} \notag\\
&\Psi^{\mathfrak{gl}_n}_{\alpha}(x)=\Psi^{\mathfrak{gl}_n}_{-\alpha}(x') \quad \text{with}\quad x_i':=1/x_{n-i+1}
\quad \text{for} \quad i=1,2...,n, \label{eq:Whit_inv2}
\end{align}
both of which are easily checked via integral formula \eqref{int-giv}. Via the change of variables $x_i\mapsto1/x_{n-i+1}$
this yields  that
\begin{align}\label{stad2}
\int_{\bbR_+^n} e^{-s x_1}  \Psi^{\mathfrak{gl}_n}_{\beta}(x)  \Psi^{\mathfrak{gl}_n}_{-\gl}(x) \prod_{i=1}^n\frac{\dd x_i}{x_i} \notag
 &= \int_{\bbR_+^n} e^{-s /x_n}  \Psi^{\mathfrak{gl}_n}_{\beta}(x')  \Psi^{\mathfrak{gl}_n}_{-\gl}(x')\prod_{i=1}^n\frac{\dd x_i}{x_i}  \notag\\
 & \stackrel{\eqref{eq:Whit_inv2}}{=} \int_{\bbR_+^n} e^{-s /x_n}  \Psi^{\mathfrak{gl}_n}_{-\beta}(x)  \Psi^{\mathfrak{gl}_n}_{\gl}(x)\prod_{i=1}^n\frac{\dd x_i}{x_i}  \notag\\ 
 & \stackrel{\eqref{bumpstade}}{=}s^{\sum_{j=1}^n(\beta_j-\lambda_j)} \prod_{1\leq j,j'\leq n} \Gamma(\lambda_j-\beta_{j'}).
\end{align}
Inserted in \eqref{Planch_appl}, relations \eqref{stad1}, \eqref{stad2} lead to the desired formula
 after a contour shift in the $\dd \lambda$ integration by $\delta$, which makes sure that the contour integral crosses no poles.
 This contour shift can be justified thanks to the sufficient decay at $|\Im(\lambda_i)|\to\infty$ for $i=1,...,n$
 of the integrand in \eqref{laplace_integral} as 
 checked via the asymptotics of the Gamma function
 \begin{align*}
\lim_{b\to\infty}|\Gamma(a+\iota b)| e^{\frac{\pi}{2}|b|}|b|^{\frac{1}{2}-a}=\sqrt{2\pi},
\qquad a>0, \,b \in \mathbb{R} 
\end{align*}
Similarly, the asymptotics of the gamma function and \eqref{stad2} can be used to check the legitimate application 
of the Plancherel theorem in \eqref{Planch_appl}, which amounts to checking the $L^2(\dd \gl)$
 integrability of the two expressions in parentheses in \eqref{Planch_appl}.
  \end{proof}
  Let us note that we have presented the Laplace transform of $Z_n$ because in this case the emergence of Whittaker functions
 in the proof is more transparent. However, a similar formula holds for a general point-to-point partition function $Z_{(m,n)}$,
 with $m\geq n$. In particular,
   \begin{align*}
  \bbE\Big[ e^{-sZ_{(m,n)}}\Big]= \int_{\ell_\delta^n} \dd \lambda 
  \,s_n(\lambda) \prod_{1\leq j,j'\leq n} \Gamma(\lambda_j-\beta_{j'}) 
  \prod_{j=1}^n\frac{s^{-\lambda_j}\prod_{1\leq i\leq m} \Gamma(\alpha_i+\lambda_j)}{s^{-\gb_j}\prod_{1\leq i\leq m} \Gamma(\alpha_i+\beta_j)}.
  \end{align*}
  \vskip 2mm
  Having this formula as a starting point and employing intuition derived from Macdonald Processes \cite{BC14}, Borodin-Corwin-Remenik \cite{BCR13} were able to rewrite the Laplace transform of the point-to-point log-gamma polymer with parameters  
  $(\ga,\gb)$ in terms
  of a Fredholm determinant as 
  \begin{theorem}
  Let us use the abbreviation DPRM for \it{``directed polymer in random medium''}.
  Let $Z_{(m,n)}$ be the point-to-point partition function of a log-gamma DPRM with parameters $( \ga,\gb)$.
   Then its Laplace transform
  is given by the Fredholm determinant
   \begin{align}\label{BCR_Fredholm}
  \bbE\Big[e^{-u Z_{(m,n)} }\Big]=\det(I+K^{{\rm DPRM}}_{u,\ga,\gb})_{L^2(C_{\delta_1})},
 \end{align}
   where the kernel $K^{{\rm DPRM}}_{u,\ga,\gb}:L^2(C_{\delta_1})\to L^2(C_{\delta_1})$ equals 
 \begin{align*}
K^{{\rm DPRM}}_{u,\ga,\gb}(v,v')=\frac{1}{2\pi \iota} \int_{\ell_{\delta_2}} \frac{dw}{w-v'}\frac{\pi}{\sin(\pi(v-w))} 
 \frac{u^{w}}{u^v}\, \frac{\prod_{i=1}^m \Gamma(\ga_i-w) \prod_{j=1}^n  \Gamma(v+\gb_j)  } {
 \prod_{i=1}^m \Gamma(\ga_i-v) \prod_{j=1}^n \Gamma(w+\gb_j) } ,
\end{align*}
for $0<\gd_1<\gd_2$,
$\Re(\ga_i)>\gd_2$ for all $i$ and $|\gb_j|<\gd_1$ for all $j$. That is, all $\gb_j$ lie to the left of $\ell_{\gd_2}$
and all $\ga_i$ lie at the its right.
\end{theorem}
\begin{remark}{
\rm A very interesting, alternative Fredholm determinant formula for the Laplace transform of $Z_{(m,n)}$
(with a kernel closer in form to that in \eqref{KUexp})  has recently been obtained 
in \cite{IMS21a, IMS21b, IMS21c}. This is achieved via a novel framework that links to determinantal models 
(periodic and free boudary Schur measures \cite{B07, BBNV18, BBNV19}) via the use of crystal theory and {\it skew} $\rsk$ algorithms.  }
\end{remark}
It is instructive to compare the Fredholm determinant formula for the log-gamma polymer to the Fredholm
determinant for the last passage percolation with exponential weights \eqref{Fred_exp}. For this,
 we set the log-gamma parameters to be $(\epsilon \ga, \epsilon\gb)$ and $u=e^{-U/\epsilon}$. 
 Then, using standard asymptotics, the Laplace transform of the log-gamma partition function with parameters
 $(\epsilon \ga, \epsilon\gb)$ converges as
\begin{align*}
 \bbE\Big[\exp\big(- e^{-\frac{U}{\epsilon}}\,  Z^{(\epsilon \ga, \epsilon \beta)}_{(m,n)} ) \Big]=
 \bbE\Big[\exp\big(- e^{-\frac{1}{\epsilon} ( U-\epsilon\log Z^{(\epsilon \ga, \epsilon \beta)}_{(m,n)} ) }\Big] \xrightarrow[\epsilon\to 0]{} 
 \bbP\Big( \tau^{(\ga, \beta)}_{(m,n)} < U\Big),
\end{align*}
where 
\[
\tau^{( \ga,  \beta)}_{(m,n)}= \max_{\pi\colon (1,1)\to (m,n)} \sum_{(i,j)\in \pi} \go_{ij},
\qquad \text{with } \qquad \bbP(\go_{ij}\in \dd x)= (\ga_i+\beta_j) e^{-(\ga_i+\beta_j)x} \,\dd x.
\]
To see how the Fredholm determinant in \eqref{BCR_Fredholm} scales with $\epsilon$ we notice that when the parameters
 of the log-gamma are taken to be $(\epsilon\ga,\epsilon \gb)$, 
 we can choose the $C_{\gd_1}$ to be $\epsilon C_{\gd_1}$
 and also scale the imaginary part of the line $\ell_{\gd_2}$ by $\epsilon$, which with an abuse of notation we denote by  
 $\epsilon\ell_{\gd_2}$.
 By the expansion of the Fredholm determinant  we have that 
 \begin{align*}
 \det(I+K^{(\epsilon \ga,\epsilon\gb)}_{\exp(-U/\epsilon)})_{L^2(\epsilon C_{\delta_1})} &=
1+\sum_{k\geq 1} \frac{1}{k!}\int_{\epsilon C_{\delta_1}} \cdots   \int_{\epsilon C_{\delta_1}}\det \Big( K^{(\epsilon \ga,\epsilon\gb)}_{\exp(-U/\epsilon) }(v_i,v_j)\Big) \,\dd v_1\cdots \dd v_n \\
& = 1+\sum_{k\geq 1}\frac{1}{k!} \int_{ C_{\delta_1}} \cdots   \int_{ C_{\delta_1}}\det \Big( \epsilon K^{(\epsilon \ga,\epsilon\gb)}_{\exp(-U/\epsilon)}(\epsilon v_i, \epsilon v_j)\Big) \,\dd v_1\cdots \dd v_n
 \end{align*}
 and
 \begin{align*}
& \epsilon K^{(\epsilon \ga,\epsilon\gb)}_{\exp(-U/\epsilon)}(\epsilon v, \epsilon v') \\
 & = \frac{\epsilon}{2\pi \iota} \int_{\epsilon \ell_{\delta_2}} \frac{dw}{w-\epsilon v'}\frac{\pi}{\sin(\pi(\epsilon v-w))} 
 \frac{e^{-wU/\epsilon }}{e^{-vU/\epsilon}}\, \frac{\prod_{i=1}^m \Gamma(\epsilon \ga_i-w) \prod_{j=1}^n  \Gamma(\epsilon v+\epsilon
 \gb_j)  } {
 \prod_{i=1}^m \Gamma(\epsilon\ga_i-\epsilon v) \prod_{j=1}^n \Gamma(w+\epsilon \gb_j) } 
 \end{align*}
 and changing variables $w\mapsto\epsilon w$ and taking the limit $\epsilon \to 0$, using that $\Gamma(x)\sim 1/x$ for $x\sim0$,
  this converges to
 \begin{align*}
  K^{( \ga,\gb) }_{U, {\rm exp}}( v,  v') &=
 \frac{1}{2\pi \iota} \int_{ \ell_{\delta_2}} \frac{dw}{(w-v')(v-w)} 
 \frac{e^{-wU}}{e^{-vU}}\, \frac{\prod_{i=1}^m (\ga_i-v) \prod_{j=1}^n (w+\gb_j) }{\prod_{i=1}^m (\ga_i-w) \prod_{j=1}^n (v+\gb_j)  }\\
 &=  \frac{1}{2\pi \iota}  \int_{ \ell_{\delta_2}} \frac{dw}{(w-v')}  \int_0^\infty \dd \gl \,e^{\gl(v-w)}\,
 \frac{e^{-wU}}{e^{-vU}}\, \frac  {\prod_{i=1}^m (\ga_i-v) \prod_{j=1}^n (w+\gb_j) } {\prod_{i=1}^m (\ga_i-w) \prod_{j=1}^n (v+\gb_j)  },
 \end{align*}
 where in the second equality the use of $(v-w)^{-1}=\int_0^\infty e^{\lambda(v-w)}\dd \lambda$ is justified by the fact that
 $\Re(v-w)<0$, since $\delta_1<\delta_2$ and $v\in C_{\delta_1}, w\in\ell_{\delta_2}$.
   Considering, now, the operators 
 \[
 A(v,\gl):= e^{\gl v} \qquad\text{and} \qquad 
 B(\gl,v'):= \frac{1}{2\pi \iota}  \int_{ \ell_{\delta_2}} \frac{dw}{(w-v')} \,e^{-\gl w}\,
 \frac{e^{-wU}}{e^{-vU}}\, \frac {\prod_{i=1}^m (\ga_i-v) \prod_{j=1}^n (w+\gb_j) }{\prod_{i=1}^m (\ga_i-w) \prod_{j=1}^n (v+\gb_j)  } 
 \]
 and using the identity $\det(I+AB)=\det(I+BA)$ we can write the Fredholm determinant of $K^{( \ga,\beta) }_{U, {\rm exp}}( v,  v') $ on $L^2({C_{\gd_1}})$ as a Fredholm determinant on $L^2(\bbR_+)$ with respect to the kernel
  \begin{align*}
  \tilde K^{( \ga,\gb) }_{U, {\rm exp}}( t,  s) 
 &=  \frac{1}{(2\pi \iota)^2}  \int_{C_{\gd_1}} \dd v  \int_{ \ell_{\delta_2}} \frac{dw}{(w-v)} 
 \frac{e^{-(t+U)w}}{e^{-(s+U)v}}\, \frac{\prod_{i=1}^m (\ga_i-v) \prod_{j=1}^n (w+\gb_j) }{\prod_{i=1}^m (\ga_i-w) \prod_{j=1}^n (v+\gb_j)  } 
 \end{align*}
 which is equivalent (upon changing $t+U$ and $s+U$ to $t$ and $s$, respectively) to the Fredholm determinant on $L^2(U,\infty)$ with kernel 
  \begin{align}\label{KUexp}
  \tilde K^{( \ga,\gb) }_{U, {\rm exp}}( t,  s) 
 &=  \frac{1}{(2\pi \iota)^2}  \int_{C_{\gd_1}} \dd v  \int_{ \ell_{\delta_2}} \frac{dw}{(w-v)} 
 \frac{e^{-t w}}{e^{-s v}}\, \frac  {\prod_{i=1}^m (\ga_i-v) \prod_{j=1}^n (w+\gb_j) }{\prod_{i=1}^m (\ga_i-w) \prod_{j=1}^n (v+\gb_j)  }.
 \end{align}
 In this way we have recovered the formulas of the exponential last passage percolation.
\subsection{Variants of the log-gamma polymer}\label{sec:loggamma_var}
We will present here a variant of the log-gamma polymer and also discuss geometries other than the point-to-point.
\vskip 2mm
{\bf The strict-weak polymer.}
The strict-weak polymer was introduced by Sepp\"all\"ainen \cite{S10b} and was analysed in \cite{CSS15, OO15}. In this polymer model the paths are directed towards down and right but the right steps can only be south-east diagonal (instead of straight east
as in the usual polymer models, see \eqref{strict-weak-paths}) 
and the disorder is {\it gamma distributed} and associated to the edges rather than to the sites.
We will see, however, that the strict-weak polymer falls within the same structure as that of
 the log-gamma polymer. We will present here
a different point of view than that of \cite{CSS15, OO15}.

 Theorems \ref{thm:NY} and \ref{thm:grsk} have shown that, given a matrix $\sfW=\{\bw^i_j\colon 1\leq i\leq m, 1\leq j\leq n\}$,
(assume that $m\geq n$) then the $\sfZ=(z^i_j)_{1\leq j\leq i \leq n}$ Gelfand-Tsetlin pattern of $\grsk(\sfW)$ can be represented as the solution of the matrix equation
\begin{align}\label{theHeq}
H(\bw^1) H(\bw^{2}) \cdots H(\bw^m) &= H_n(\bp^n) H_{n-1}(\bp^{n-1}) \cdots H_1(\bp^1) ,
\end{align}
where $\bw^1,...,\bw^n$ represent the rows of $\sfW,\,\, \bp^k=(p^k_k,...,p^k_n)\in \R^{n-k+1}$ and 
\begin{align}\label{zvariables}
z^i_j=p^j_j \, p^j_{j+1}\cdots p^j_i, \qquad \text{for}\qquad 1\leq j\leq i\leq n.
\end{align} 
 Theorem \ref{thm:NY}, in particular, gives the representation of $z^n_1$ as the polymer partition 
function $\sum_{\pi\colon (1,1)\to(m,n)}\prod_{(i,j)\in \pi} \bw^i_j$, where the sum is over all paths that go from $(1,1)$ to $(m,n)$ by
moving either down or right at each step.
But one can also write equation \eqref{theHeq} in terms of the dual matrices $E_i$ defined in \eqref{Ematrices} via relation 
\eqref{HErelation} as
\begin{eqnarray}\label{Eequation}
E(\bar\bw^n)\cdots E(\bar\bw^1)=E_1(\bar\bp^1)\cdots E_N(\bar\bp^n),\qquad m\geq n.
\end{eqnarray}
where we recall the notation that, for a vector $\bx=(x_1, x_2,...)$, then $\bar\bx=(1/x_1, 1/x_2,...)$.
As we have seen in Section \ref{RSK_matrix}, the entries of either of the above matrix products have path representations. 
In particular,
  \begin{equation}\label{strict-weak-paths}
\Big(E(\bar\bw^n)\cdots E(\bar\bw^1)\Big)_{1m}=\sum_\pi \,\,
{\begin{tikzpicture}[baseline={([yshift=-.5ex]current bounding box.center)},vertex/.style={anchor=base,
    circle,fill=black!25,minimum size=18pt,inner sep=2pt}, scale=0.6]
\draw (10,0) -- (10,-8); \draw (9,0) -- (9,-8); \draw (8,0) -- (8,-8); \draw (7,0) -- (7,-8); \draw (6,0) -- (6,-8);
\draw (5,0) -- (5,-8); 

\draw(5,-6)--(6,-7); \draw(5,-5)--(7,-7); \draw(5,-4)--(8,-7); \draw(5,-3)--(9,-7); \draw(5,-2)--(10,-7); 
\draw(5,-1)--(10,-6); \draw(5,0)--(10,-5);\draw(6,0)--(10,-4); \draw(7,0)--(10,-3); \draw(8,0)--(10,-2); \draw(9,0)--(10,-1);
\node at (10.7, -0.3) {$\bar\bw^1$}; \node at (10.7, -1.3) {$\bar\bw^2$};  \node at (10.7, -7.3) {$\bar\bw^m$};
\draw[ultra thick, red, ->] (5,0)--(5,-1)--(6,-2)--(7,-3)--(7,-4)--(8,-5)--(9,-6)--(10,-7)--(10,-8); 
\draw[ultra thick, dotted] (10.6,-2.3)--(10.6,-6);
\foreach \i in {5,...,10} {\node[draw,circle,inner sep=1pt,fill] at (\i,0) {};
	}
\foreach \i in {5,...,10} {\node[draw,circle,inner sep=1pt,fill] at (\i,-8) {};
	}
\end{tikzpicture}} 
\end{equation}
where the sum is over all directed paths from the upper-left corner to the lower-right one, with the weights of the diagonal edges being
equal to $1$ and the vertical ones being equal to $w^i_j$.  Moreover,
\begin{equation*}
\Big(E_1(\bar\bp^1)\cdots E_N(\bar\bp^n) \Big)_{1m}=\sum_\pi\,\,
{\begin{tikzpicture}[baseline={([yshift=-.5ex]current bounding box.center)},vertex/.style={anchor=base,
    circle,fill=black!25,minimum size=18pt,inner sep=2pt}, scale=0.6]
\draw (10,0) -- (10,-6); \draw (9,0) -- (9,-5); \draw (8,0) -- (8,-4); \draw (7,0) -- (7,-3); \draw (6,0) -- (6,-2);
\draw (5,0) -- (5,-1); 

\draw(5,0)--(10,-5);\draw(6,0)--(10,-4); \draw(7,0)--(10,-3); \draw(8,0)--(10,-2); \draw(9,0)--(10,-1);
\node at (10.7, -0.3) {$\bar\bp^1$}; \node at (10.7, -1.3) {$\bar\bp^2$};  \node at (10.7, -5.3) {$\bar\bp^n$};
\draw[ultra thick, red, ->] (5,0)--(6,-1)--(7,-2)--(8,-3)--(9,-4)--(10,-5)--(10,-6); 
\draw[ultra thick, dotted] (10.6,-2.3)--(10.6,-4.3);
\foreach \i in {5,...,10} {\node[draw,circle,inner sep=1pt,fill] at (\i,0) {};
	}
\foreach \i in {5,...,10} {\node[draw,circle,inner sep=1pt,fill] at (\i,-\i+4) {};
	}
\end{tikzpicture}} 
\end{equation*}
where the set of paths and the assignment of weights is as above. However, in this case there is only one admissible path and its weight is $1/p^n_n$, which by \eqref{zvariables} equals $1/z^n_n$. Combining the two, we have that \eqref{Eequation}, evaluated on both sides at entry $(1,m)$, gives
\begin{align*}
\sumtwo{\pi\colon (1,1)\to(n,n)}{\pi\colon \text{strict-weak}} \, \prod_{\be\in\pi} \frac{1}{\bw_e}= \frac{1}{z^n_n},
\end{align*}
where the product is over edges traced by path $\pi$. 
If we assume that the weight matrix $\sfW=\{\bw^i_j\colon 1\leq i\leq m, 1\leq j\leq n\}$ is distributed according to the log-gamma
measure \eqref{log-gamma-measure}, then the left-hand side is the partition function of the strict-weak polymer and we see that it 
equals the inverse of the bottom-left corner of the Gelfand-Tsetlin pattern of $(\sfZ,\sfZ')=\grsk(\sfW)$, while the partition function
of the log-gamma polymer equals the bottom-right corner $z^n_1$ of the same Gelfand-Tsetlin pattern. Thus both the strict-weak
and the log-gamma polymer fall within the same integrable structure and the steps described in the previous section for the
analysis of the log-gamma polymer can be applied to analyse the strict-weak polymer and obtain Tracy-Widom
fluctuations.
\vskip 2mm
{\bf Multipoint correlations.}
Multipoint spatial correlations for {\it positive temperature models}, e.g. log-gamma polymer, KPZ equation etc.,
 are less understood compared to their zero temperature counterparts. However, progress has recently been made in this direction.
 
 A first indication of the convergence of the joint law of the solution to the KPZ equation 
 at spatially separated points (at the same time) to the two-point function
 of the $\text{Airy}_2$ process has appeared in the physics literature \cite{D13, D13b, D14}. This was done
  via the non rigorous Bethe ansatz method and under certain assumptions. In the mathematics literature, contour integral
 expressions for the joint Laplace transform of the point-to-point partition functions of the log-gamma polymer
  $Z_{(m_1,n_1)}, Z_{(m_2,n_2)}$, for  $m_1\leq n_1, m_2\geq n_2$  have been derived in \cite{NZ17} via the use of
  geometric $\rsk$ (in particular its local move formulation as described in Section \ref{sec:localmove}, see the end of that section)
 and using the general scheme described in the previous section.  The formula from \cite{NZ17} may be written as
\begin{align}\label{eq:2p_intro}
& \bbE\Big[e^{-u_1 Z_{(m_1,n_1)}-u_2Z_{(m_2,n_2)}}\Big] \notag \\
 = &\int_{(\ell_\delta)^{m_1} } \dd \lambda
    \,\,s_{m_1}(\lambda)  \prod_{1\leq i,i'\leq m_1} \Gamma(-\ga_i+\gl_{i'})
     \prod_{i=1}^{m_1}\frac{u_1^{-\gl_i}\, \prod_{j=n_2+1}^{n_1}\Gamma(\gl_i+\gb_j) }{ u_1^{-\ga_i}\, \prod_{j=n_2+1}^{n_1}\Gamma(\ga_i+\gb_j) } \notag\\
&\,\times   \int_{(\ell_{\delta+\gamma})^{n_2} } \dd \mu  
    \,\,s_{n_2}(\mu)  \prod_{1\leq j,j'\leq n_2} \Gamma(-\gb_j+\mu_{j'}) 
    \prod_{j=1}^{n_2}\frac{u_2^{-\mu_j} \, \prod_{i=m_1+1}^{m_2}\Gamma(\mu_j+\ga_i)  }{u_2^{-{\gb_j}} \,  \prod_{i=m_1+1}^{m_2}\Gamma(\gb_j+\ga_i) } \notag\\
 &\,\,   \quad\times \prod_{\substack{1\leq i\leq m_1 \\ 1\leq j\leq n_2 }} \frac{\Gamma(\gl_i+\mu_j)}{\Gamma(\ga_i+\gb_j)}\,,
  \end{align} 
  where $\ell_\delta, \ell_{\delta+\gamma}$ are the vertical lines $\delta+\iota \bbR$ and $\delta+\gamma+\iota \bbR$, with $\delta,\gamma>0$ (with $\ga<\gd$ and $\gb<\gd+\gamma$),
   of the complex plane and $s_n(\lambda)$ is the Sklyanin measure \eqref{eq:sklyanin}. The most interesting term in the contour
   integral is the last product involving $\Gamma(\gl_i+\mu_j)$, which couples the $\lambda$ and $\mu$ variables. Without this term
   the two integrals decouple to essentially the ones giving the Laplace transform of a single partition function, see \eqref{laplace_integral}. Thus this term contains all the correlations. Its presence, though,
    obscures the asymptotic analysis as 
   it does not lead to any (obvious) determinantal structure. 
    Under a certain assumption of convergence of some series, see  Section 4 of \cite{NZ17},
    convergence to the $\text{Airy}_2$ process was derived in \cite{NZ17} by analysing further 
    formula \eqref{eq:2p_intro}.
    
    An important contribution has been made by Dimitrov \cite{D20}, who for a {\it stochastic six vertex model}
    was able to establish fully the convergence of the two-point functions of its height function to the two-point function of the ${\rm Airy}_2$ process.
    This model possesses a parameter $t\in(0,1)$, which when sent to $1$ in a suitable fashion, leads to the convergence of the,
    suitably scaled, height function 
    to the partition function of the log-gamma polymer. Dimitrov, was able to establish the convergence to the ${\rm Airy}_2$ process
    under the assumption that the parameter $t$ is small enough. This allows, in combination with some fine analysis, 
    to tame the analogous series,
    whose convergence was taken as an assumption in \cite{NZ17}.    
    
    More recently there have been two more very important developments in the study of multipoint correlations \cite{QS20, V20}. 
    The first one has established that the large time limit of the multi-point laws of the solution, itself, to the KPZ equation
     are governed by those of the KPZ fixed point. The second shows that the large time limit of the multi-point laws of the Brownian polymers
     are governed by the directed landscape. Both approaches are novel in the sense that they break out of the circle of ideas involving 
     exact formulas and their analysis. We refer to the original works for details.

    \vskip 2mm
    {\bf The point-to-line geometry.}
    We briefly discuss the log-gamma polymer in the point-to-line geometry and some of its algebraic aspects,
    which link to orthogonal Whittaker functions. 
    The point-to-line partition function is
    \begin{align*}
Z^{\rm flat}_n:=\sum_{m+\ell=n+1} Z_{m,\ell}:=\sum_{m+\ell=n+1} \sum_{\pi\,\in\,\Pi_{m,\ell}} \prod_{(i,j)\in \pi} \bw^{\,i}_j,
\end{align*}
where $\Pi_{m,\ell}$ is the set of down-right (south-east) paths from $(1,1)$ to $(m,\ell)$. Pictorially, we have that
\begin{equation*}
Z^{\rm flat}_n= \sumtwo{\text{red paths}}{\text{ending on the diagonal}}
\begin{tikzpicture}[baseline={([yshift=-.5ex]current bounding box.center)},vertex/.style={anchor=base,
    circle,fill=black!25,minimum size=18pt,inner sep=2pt}, scale=0.35]
\draw[thick,dashed] (10.5,-.5) -- (.5,-10.5);
\foreach \i in {1,...,10}{
	\draw[thick] (1,-\i) grid (11-\i,-\i);
	\draw[thick] (\i,-1) grid (\i,-11+\i);
	
		\foreach \j in {\i,...,10}{
		\node[draw,circle,inner sep=1pt,fill] at (11-\j,-\i) {};
	}
}

\draw[thick,color=red,-] (1,-1) -- (2,-1) -- (2,-2) -- (3,-2) -- (3,-3) -- (3,-4) -- (3,-5) -- (4,-5) -- (5,-5) -- (5,-6);
\foreach \x in {(1,-1),(2,-1),(2,-2),(3,-2),(3,-3),(3,-4),(3,-5),(4,-5),(5,-5),(5,-6)}{
	\node[draw,circle,inner sep=1pt,fill,red] at \x {};
	}
\end{tikzpicture}
\end{equation*}
It turns out that the Laplace transform of $Z^{\rm flat}_{2n}$ 
(notice that we constrain ourselves to even polymer length $2n$) for the point-to-line log-gamma polymer with weight distribution
\begin{equation}
\label{eq:logGammaDistribution}
\frac{1}{\bw^{i}_j} \sim
\begin{cases}
{\rm Gamma}(\alpha_i + \beta_j) &1\leq i,j\leq n \, , \\
{\rm Gamma}(\alpha_i + \alpha_{2n-j+1}) &1\leq i\leq n\, , \,\, n < j\leq 2n-i+1 \, , \\
{\rm Gamma}(\beta_{2n-i+1} + \beta_j) &1\leq j\leq n \, , \,\, n < i \leq 2n-j+1 \, ,
\end{cases}
\end{equation}
given by (inverse) Gamma variables as in \eqref{log-gamma-measure}
for some ${\alpha},{\beta}\in\R_{+}^n$, 
 can also be expressed in terms 
of Whittaker functions but this time corresponding to the orthogonal group $SO_{2n+1}(\R)$, \cite{BZ19}.
 The formula, which remarkably 
parallels the one for the point-to-point case reads, is
  \begin{align}\label{flat_Laplace}
\bbE\Big[e^{-u \,Z^{\rm flat}_{2n}}\Big]&=\frac{u^{\sum_{i=1}^n(\ga_i+\gb_i)}}{\Gamma^{\rm flat}_{{\alpha},{\beta}}}
   \int_{\R_+^n} e^{-ux_1} \Psi^{\mathfrak{so}_{2n+1}}_{{\ga}}(x) \Psi^{\mathfrak{so}_{2n+1}}_{{\gb}}(x) 
   \prod_{i=1}^n\frac{\dd x_i}{x_i} \, ,
\end{align}
where $\Gamma^{\rm flat}_{{\alpha},{\beta}}$ is the normalisation of the measure \eqref{eq:logGammaDistribution}.
The $SO_{2n+1}(\R)-$Whittaker functions also admit a Givental type integral representation \cite{GLO12} over
geometric liftings of what we will call the {\bf ${\rm \boldsymbol{B}\boldsymbol{C}}$-Gelfand-Tsetlin pattern}.
We derive the terminology from the fact that $\boldsymbol{B}$ is typically used to refer to the orthogonal group $SO_{2n+1}(\R)$ and $\boldsymbol{C}$ 
is typically used to refer to the symplectic group $Sp_{2n}(\R)$, which both share similar structure with their associated Gelfand-Tsetlin patterns in the combinatorial setting. 
We consider ${\rm \boldsymbol{B}\boldsymbol{C}}$-Gelfand-Tsetlin patterns to be half-triangular arrays 
$(z^i_j \colon 1\leq j\leq i\leq 2n\,,\, 1\leq j\leq \lceil \tfrac{i}{2}\rceil)$
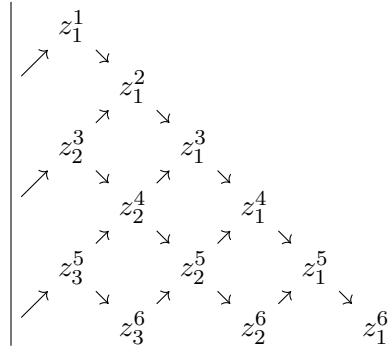
\begin{equation*}
\begin{tikzpicture}[scale=0.8]

\node (z11) at (1,-1) {$z^1_{1}$};
\node (z21) at (2,-2) {$z^2_{1}$};
\node (z22) at (0,-2) {};
\node (z31) at (3,-3) {$z^3_{1}$};
\node (z32) at (1,-3) {$z^3_{2}$};
\node (z41) at (4,-4) {$z^4_{1}$};
\node (z42) at (2,-4) {$z^4_{2}$};
\node (z43) at (0,-4) {};
\node (z51) at (5,-5) {$z^5_{1}$}; \node (z52) at (3,-5) {$z^5_{2}$}; \node (z53) at (1,-5) {$z^5_{3}$};
\node (z61) at (6,-6) {$z^6_{1}$}; \node (z62) at (4,-6) {$z^6_{2}$}; \node (z63) at (2,-6) {$z^6_{3}$};\node (z64) at (0,-6) {};

\draw[->] (z22) -- (z11);
\draw[->] (z11) -- (z21);
\draw[->] (z32) -- (z21);
\draw[->] (z21) -- (z31);
\draw[->] (z43) -- (z32);
\draw[->] (z32) -- (z42);
\draw[->] (z42) -- (z31);
\draw[->] (z31) -- (z41);
\draw[->] (z41) -- (z51); \draw[->] (z52) -- (z41); \draw[->] (z42) -- (z52); \draw[->] (z53) -- (z42); \draw[->] (z64) -- (z53);
\draw[->] (z51) -- (z61); \draw[->] (z62) -- (z51); \draw[->] (z52) -- (z62); \draw[->] (z63) -- (z52); \draw[->] (z53) -- (z63);

\draw (0,-0.6) -- (0,-6.3);

\end{tikzpicture}
\end{equation*} 
where the arrows indicate ``$\leq\,$'' in the standard, combinatorial setting (e.g. $z^{i+1}_j \to z^i_j \to z^{i+1}_{j+1}$ means
$z^{i+1}_j \leq z^i_j \leq z^{i+1}_{j+1}$ with the {\it wall} assigned the value $0$). In the geometric lifting / Whittaker function setting
arrows indicate the summation in the {\it energy} of the pattern as
\begin{equation*}
\mathcal{E}_{{\rm \boldsymbol{B}\boldsymbol{C}}}(\sfZ) = \sum_{a\to b} \frac{z_a}{z_b},
\end{equation*}
with the wall assigned the value $1$. 
This is similar to the energy of the usual geometric Gelfand-Tsetlin patters as in Theorem \ref{grsk_prop}.
The {\it type} of these patterns is the vector defined similarly as the one for the (full) Gelfand-Tsetlin patterns as
\begin{equation*}
\text{ type}(\sfZ)_i
:= \frac{\prod_{j = 1}^{\lceil i/2 \rceil} z^i_{j}}{\prod_{j = 1}^{\lceil (i-1)/2 \rceil} z^{i-1}_{j}}, 
\qquad \text{for } i=1, \dots, 2n
\, .
\end{equation*}
The $SO_{2n+1}(\R)$-Whittaker functions can be represented as (\cite{GLO12})
\begin{equation}
\label{eq:soWhittakerFn}
\Psi^{\mathfrak{so}_{2n+1}}_{\alpha}(x)
:= \int_{\GT_{\rm BC}(x)} 
\text{type}(\sfZ)^{\alpha^{\pm}}
\exp\Big(-\mathcal{E}_{\boldsymbol{B}\boldsymbol{C}}(\sfZ)\Big)
\,\,\,\,\prod_{\mathclap{\substack{1\leq i<2n, \\ 1\leq j\leq \lceil i/2 \rceil }}} \,\,\frac{\dd z_{i,j}}{z_{i,j}} \, ,
\end{equation}
where $\alpha\in \C^{n}$ and 
 $x \in\R_{+}^n$, and the integration is over all geometric $\boldsymbol{B}\boldsymbol{C}$-Gelfand-Testlin patterns
of depth $2n$ with positive entries and $(2n)$-th row equal to $x\in\R^n_+$, and
\[
\text{type}(\sfZ)^{{\alpha}^{\pm}}
:= \prod_{k=1}^n \text{type}(\sfZ)_{2k-1}^{\alpha_k} \text{type}(\sfZ)_{2k}^{-\alpha_k} 
\, .
\]
Formula \eqref{flat_Laplace} was derived in \cite{BZ19}. The main idea was to apply the geometric $\rsk$ in its 
{\it local moves formulation} to triangular arrays $\sfW=(\bw^i_j \colon i+j\leq 2n+1)$.
 Given that entries $t^i_j$ of array $\sfT=\grsk(\sfW )$ with $i+i=2n+1$ are the
point-to-point polymer partition functions, see \eqref{local_greene},
 the point-to-line partition function $Z_{2n}^{\rm flat}$ can be expressed as
$Z_{2n}^{\rm flat} = \sum_{i+j=2n+1} t^i_j$.
Thus,
\begin{align}\label{flat_laplace_integral}
\bbE\Big[e^{-u \,Z^{\rm flat}_{2n}}\Big] = \int \exp\Big(-u\sum_{i+j=2n+1} t^i_j\Big) \,\,\bbP\circ \grsk^{-1}(\dd \sfT).
\end{align} 
The next step is to perform the change of variables $t^i_j\mapsto (ut^i_j)^{-1}$ and then decompose the integration in \eqref{flat_laplace_integral} 
first over variables $(t^i_j: i<j)$ (for which we will use a notation $z'$ for reasons that will become clear as this will be identified as a geometric lifting
 of a symplectic $\GT$ pattern)
and $(t^i_j\colon i>j)$ (for which we will use a notation $z$)
 and then over the diagonal entries $t^i_i$ (for which we will use a notation $x$). Pictorially, the above integration, when $\bbP\circ \grsk^{-1}(\dd \sfT)$ is
 expanded out, as well as the resulting integration
 after the change of variables as in \eqref{push} and the decomposition described above, is as follows:
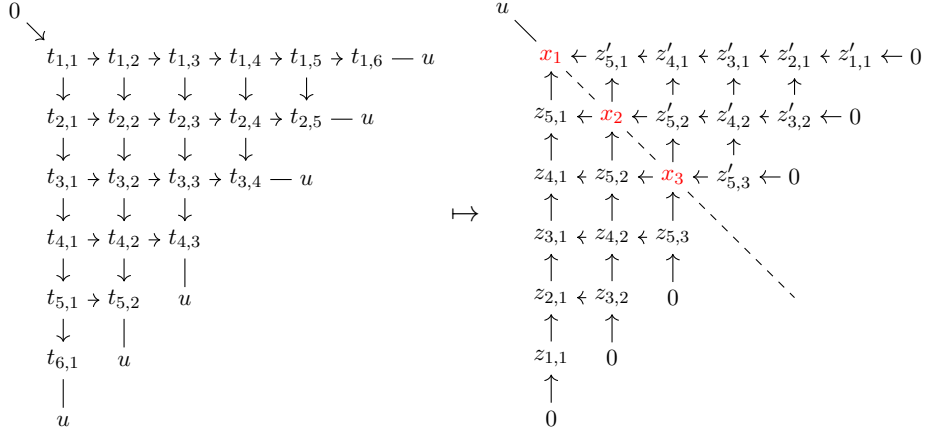
\begin{equation*}
\begin{tikzpicture}[baseline={([yshift=-.5ex]current bounding box.center)}, scale=0.8, every node/.style={transform shape}]
\node (t11) at (1,-1) {$t_{1,1}$};
\node (t12) at (2,-1) {$t_{1,2}$};
\node (t13) at (3,-1) {$t_{1,3}$};
\node (t14) at (4,-1) {$t_{1,4}$};
\node(t15) at (5,-1) {$t_{1,5}$};
\node (t16) at (6,-1) {$t_{1,6}$};

\node (t21) at (1,-2) {$t_{2,1}$};
\node (t22) at (2,-2) {$t_{2,2}$};
\node (t23) at (3,-2) {$t_{2,3}$};

\node (t24) at (4,-2) {$t_{2,4}$};
\node (t25) at (5,-2) {$t_{2,5}$};

\node (t31) at (1,-3) {$t_{3,1}$};
\node (t32) at (2,-3) {$t_{3,2}$};
\node (t33) at (3,-3) {$t_{3,3}$};
\node (t34) at (4,-3) {$t_{3,4}$};

\node (t41) at (1,-4) {$t_{4,1}$};
\node(t42) at (2,-4) {$t_{4,2}$};
\node(t43) at (3,-4) {$t_{4,3}$};

\node(t51) at (1,-5) {$t_{5,1}$};
\node (t52) at (2,-5) {$t_{5,2}$};

\node(t61) at (1,-6) {$t_{6,1}$};

\draw[->] (0.4,-0.4) -- (t11);
\node(o) at (0.2, -0.2) {$0$};
\node(u1) at (7,-1) {$u$};\node(u2) at (6,-2) {$u$};\node(u3) at (5,-3) {$u$};\node(u4) at (3,-5) {$u$};\node(u5) at (2,-6) {$u$};
\node(u6) at (1,-7) {$u$};

\draw (t16)--(u1); \draw (t25)--(u2); \draw (t34)--(u3); \draw (t43)--(u4); \draw (t52)--(u5); \draw (t61)--(u6);

\draw[->] (t11) -- (t12);
\draw[->] (t12) -- (t13);
\draw[->] (t13) -- (t14);
\draw[->] (t14) -- (t15);
\draw[->] (t15) -- (t16);

\draw[->] (t21) -- (t22);
\draw[->] (t22) -- (t23);
\draw[->] (t23) -- (t24);
\draw[->] (t24) -- (t25);

\draw[->] (t31) -- (t32);
\draw[->] (t32) -- (t33);
\draw[->] (t33) -- (t34);

\draw[->] (t41) -- (t42);
\draw[->] (t42) -- (t43);

\draw[->] (t51) -- (t52);

\draw[->] (t11) -- (t21);
\draw[->] (t21) -- (t31);
\draw[->] (t31) -- (t41);
\draw[->] (t41) -- (t51);
\draw[->] (t51) -- (t61);

\draw[->] (t12) -- (t22);
\draw[->] (t22) -- (t32);
\draw[->] (t32) -- (t42);
\draw[->] (t42) -- (t52);

\draw[->] (t13) -- (t23);
\draw[->] (t23) -- (t33);
\draw[->] (t33) -- (t43);

\draw[->] (t14) -- (t24);
\draw[->] (t24) -- (t34);

\draw[->] (t15) -- (t25);
\end{tikzpicture}
\mapsto
\begin{tikzpicture}[baseline={([yshift=-.5ex]current bounding box.center)}, scale=0.8, every node/.style={transform shape}]
\node (t11) at (1,-1) {\red{$x_{1}$}};
\node (t12) at (2,-1) {$z'_{5,1}$};
\node (t13) at (3,-1) {$z'_{4,1}$};
\node (t14) at (4,-1) {$z'_{3,1}$};
\node(t15) at (5,-1) {$z'_{2,1}$};
\node(t16) at (6,-1) {$z'_{1,1}$};

\node (t21) at (1,-2) {$z_{5,1}$};
\node (t22) at (2,-2) {\red{$x_{2}$}};
\node (t23) at (3,-2) {$z'_{5,2}$};
\node (t24) at (4,-2) {$z'_{4,2}$};
\node(t25) at (5,-2) {$z'_{3,2}$};

\node (t31) at (1,-3) {$z_{4,1}$};
\node (t32) at (2,-3) {$z_{5,2}$};
\node(t33) at (3,-3) {\red{$x_{3}$}};
\node(t34) at (4,-3) {$z'_{5,3}$};

\node (t41) at (1,-4) {$z_{3,1}$};
\node (t42) at (2,-4) {$z_{4,2}$};
\node (t43) at (3,-4) {$z_{5,3}$};

\node (t51) at (1,-5) {$z_{2,1}$};
\node(t52) at (2,-5) {$z_{3,2}$};

\node (t61) at (1,-6) {$z_{1,1}$};

\draw (0.4,-0.4) -- (t11);
\node(o) at (0.2, -0.2) {$u$};
\node(o1) at (7,-1) {$0$};\node(o2) at (6,-2) {$0$};\node(o3) at (5,-3) {$0$};\node(o4) at (3,-5) {$0$};\node(o5) at (2,-6) {$0$};
\node(o6) at (1,-7) {$0$};

\draw[<-]  (t16)--(o1);  \draw[<-]  (t25)--(o2); \draw[<-]  (t34)--(o3);  \draw[<-]  (t43)--(o4); \draw[<-]  (t52)--(o5); \draw[<-]  (t61)--(o6);
\draw[dashed] (t11)--(t22)--(t33)--(5,-5);

\draw[<-] (t11) -- (t12);
\draw[<-] (t12) -- (t13);
\draw[<-] (t13) -- (t14);
\draw[<-] (t14) -- (t15);
\draw[<-] (t15) -- (t16);

\draw[<-] (t21) -- (t22);
\draw[<-] (t22) -- (t23);
\draw[<-] (t23) -- (t24);
\draw[<-] (t24) -- (t25);

\draw[<-] (t31) -- (t32);
\draw[<-] (t32) -- (t33);
\draw[<-] (t33) -- (t34);

\draw[<-] (t41) -- (t42);
\draw[<-] (t42) -- (t43);

\draw[<-] (t51) -- (t52);

\draw[<-] (t11) -- (t21);
\draw[<-] (t21) -- (t31);
\draw[<-] (t31) -- (t41);
\draw[<-] (t41) -- (t51);
\draw[<-] (t51) -- (t61);

\draw[<-] (t12) -- (t22);
\draw[<-] (t22) -- (t32);
\draw[<-] (t32) -- (t42);
\draw[<-] (t42) -- (t52);

\draw[<-] (t13) -- (t23);
\draw[<-] (t23) -- (t33);
\draw[<-] (t33) -- (t43);

\draw[<-] (t14) -- (t24);
\draw[<-] (t24) -- (t34);

\draw[<-] (t15) -- (t25);
\end{tikzpicture}
\end{equation*}
In the figure, $t_a\to t_b$ denotes the terms $t_a/t_b$ in the exponent of the integrands
of \eqref{flat_laplace_integral} and same for $z_a\to z_b$ and $z_a'\to z_b'$ after the change of variables. When an entry is zero, e.g. $0\to t$, this is to be interpreted as $1/t$
(here $0$ plays the role of the ``wall'' in a symplectic Gelfand-Tsetlin pattern).
Finally, the dashed links: $t^i_{j}-u$ , $x_1-u$, in the above picture
denote multiplication
terms $ut^i_j$ and $ux_1$ in the exponent of the integrand. In the right-hand side of the above picture, 
one recognises two (geometric liftings) of
symplectic Gelfand-Tsetlin patterns and thus the integration over each one gives rise to the orthogonal Whittaker functions 
in \eqref{flat_Laplace}. Let us remark on the seemingly strange fact that $SO_{2n+1}$-Whittaker functions appear as integrals of {\it symplectic} Gelfand-Tsetlin patterns
associated to $Sp_{2n}$. The explanation for this is a duality between $SO_{2n+1}$ and $Sp_{2n}$ and the fact that, via the so-called Casselman-Shalika formula \cite{CS80},
 Whittaker
functions associated to a group are described as characters of a finite dimensional representation of the corresponding dual group. 
\vskip 2mm
Having formula \eqref{flat_Laplace}, one can use iteratively the Plancherel Theorem \ref{thm:Plancerel} 
(accompanied by the necessary estimates) as well as an identity (the Ishii-Stade identity) 
for integrals that involve products of an $SO_{2n+1}$and a $GL_n$ Whittaker function \cite{IS13}
and arrive at the
contour integral formula
\begin{equation}
\label{eq:P2LcontourInt}
\begin{split}
\bbE\Big[e^{- u Z^{\rm flat}_{2n}}\Big] 
= &\frac{u^{\sum_{k=1}^{n} (\alpha_k+\beta_k)}}
{\Gamma^{\rm flat}_{{\alpha},{\beta}}}
\int_{(\epsilon + \iota \R)^n} 
s_n({\rho}) \dd {\rho}
\int_{(\delta + \iota \R)^n}
s_n({\lambda}) \dd {\lambda} \,\,
u^{-\sum_{i=1}^n (\lambda_i+\rho_i)}  \\
&\times \frac{\prod_{1\leq i,j\leq n}
\Gamma(\lambda_i + \rho_j)
\Gamma(\lambda_i + \alpha_j)
\Gamma(\lambda_i - \alpha_j)
\Gamma(\rho_i + \beta_j)
\Gamma(\rho_i - \beta_j)}
{\prod_{1\leq i<j\leq n}
\Gamma(\lambda_i+\lambda_j)
\Gamma(\rho_i + \rho_j) }.
\end{split}
\end{equation}
One notices again the cross term $\prod_{i,j}\Gamma(\lambda_i+\rho_j)$ which bears resemblance to the cross term
in the formula for the joint Laplace transform \eqref{eq:2p_intro}. 

For completeness, we close this section with the statement of the Ishii-Stade identity \cite{IS13}.

\begin{theorem}
\label{thm:IshiiStade}
Let $\alpha,\beta, \in\C^n$, where $\Re(\alpha_i) > |\Re(\beta_j)|$ for all $i,j$. Then
\begin{equation}
\label{eq:IshiiStade}
\int_{\R_{+}^n}
\Psi_{-\alpha}^{\mathfrak{gl}_n}(x)
\Psi_{\beta}^{\mathfrak{so}_{2n+1}}(x)
\prod_{i=1}^n \frac{\dd x_i}{x_i}
= \frac{\prod_{1\leq i,j\leq n}
\Gamma(\alpha_i + \beta_j)
\Gamma(\alpha_i - \beta_j)}
{\prod_{1\leq i<j\leq n} \Gamma(\alpha_i+\alpha_j)} \, .
\end{equation}
\end{theorem}
The restriction on the parameters $\alpha, \beta$ in the above theorem is important as otherwise the integral diverges. 
Contrary to the Bump-Stade identity \eqref{bumpstade}, 
currently, there does not exist a combinatorial proof of the Ishii-Stade identity and it would be interesting if such a proof could be obtained.
 
\section{Dynamics on Gelfand-Tsetlin patterns and interacting particle systems}\label{sec:dynamics}
The purpose of this section is to explain how three fundamental algebraic concepts (Cauchy identity, Pieri rule, Branching rule)
form the main ingredients to build stochastic integrable dynamics. 
These concepts are closely related to the underlying structure of the special functions (Schur, Macdonald, Whittaker)
functions that we saw in Section \ref{basics} and their ramifications.
The stochastic integrability, here, amounts to
the fact that the constructed dynamics are Markovian (i.e. the future evolution does not depend on the past but only on the present),
and their transition probability laws as well as the fixed time marginals can be explicitly computed. Most prominent such examples are
one dimensional particle dynamics such as TASEP and several variations of it such as push-TASEP, q-TASEP etc.  Reference \cite{BP16b}
contains a large collection of such dynamics.  

The dynamic version of $\rsk$ and the evolution of its output tableaux provided the first motivation for a construction of integrable dynamics
of the nature that we will discuss. Transcribing $\rsk$ on Gelfand-Tsetlin patterns leads to the very important  coupling between particle
dynamics, like the ones mentioned above, and dynamics related to random matrix theory. Typically, the particle dynamics (TASEP, push-TASEP,
q-TASEP etc.) occupy one of the two diagonals of the Gelfand-Tsetlin pattern while its bottom row is related to random matrix dynamics, when 
the randomness is chosen in an integrable way (like geometric distribution, exponential, log-gamma). The right / left bottom corner
of the pattern will, then, have a dual nature: one as the largest / smallest particle of TASEP, push-TASEP, q-TASEP etc. and, at the same time,
the largest / smallest particle in the random matrix dynamics. This explains to some extent the link between KPZ models and random matrix theory
and justifies the interest for a framework that will produce integrable dynamics on Gelfand-Tsetlin patterns.
\subsection{Some motivating examples.}
We have seen that $\rsk$ (and geometric $\rsk$) can be considered as an evolution on Gelfand-Tsetlin patterns $\big(\sfZ(n)\big)_{n\geq 0}$, starting from
a given initial Gelfand-Tsetlin pattern $\sfZ(0)$ at time zero, via the dynamics of row insertion
\begin{align*}
\sfZ(n-1)\xrightarrow[]{w^n} \sfZ(n),
 \end{align*}
 as this was described in Section \ref{subsec:RSK} and summarised in Proposition \ref{RSK_maxplus} and Theorem \ref{thm:grsk}.
 If the sequence of words $w^n$ are random, then this evolution is a stochastic process, which is readily seen to be Markovian. We will see that this Markovian evolution can be described via an interacting particle system. Let us start with a simple example of a two column 
 input matrix $\sfW$, whose entries are either $0$ or $1$ and we assume that no row contains two $1$'s. For example,
 \begin{align}\label{exampleW}
 \sfW^{\,\sfT}=\left(
 \begin{array}{ccccc}
 0&1&1&0&0\,\cdots\\
 1&0&0&0&1\,\cdots
 \end{array}
 \right).
 \end{align}
We recall the convention \eqref{word} by which columns of the array correspond to letters and rows to words.
 In this case the first column will be associated to letter `$1$'
and the second column to letter `$2$'. So in the case of example \eqref{exampleW} the array $\sfW$ is associated to the sequence of letters $1221$, 
which can also be viewed as a concatenation of the words $(1)(2)(2)(\emptyset)(1)$.
 Let us simulate the dynamics of row insertion the way these were described in Section \ref{sec:RS}:
 Suppose we have inserted the first $n$ rows, $w^1,w^2,...,w^n$ of $\sfW$ and this insertion has resulted in a Gelfand-Tsetlin pattern 
 \begin{align*}
 \sfZ(n)=
 \begin{array}{ccccccc}
&&&z^1_{1}(n)&&&\\
&&&&&&\\
&&z^2_{2}(n)&&z^2_{1}(n)&&
\end{array},
 \end{align*}
 corresponding to a Young tableau
 
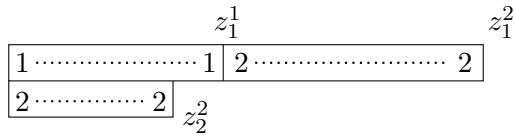
\begin{align*}
 {\begin{tikzpicture}[scale=.6]
\draw (0.1,0.1) -- (10.5,0.1)--(10.5,0.9)--(0.1,0.9)--(0.1,0.1) ;
 \node at (4.9,1.4) {$z^1_1$}; \node at (10.9,1.4) {$z^2_1$};
\node at (0.4,0.5) {$1$}; \draw[dotted, thick] (0.7,0.5)--(4.2,0.5);    \node at (4.5,0.5) {$1$}; \draw (4.8,0.1)--(4.8,0.9);
\node at (5.2,0.5) {$2$}; \draw[dotted, thick] (5.5,0.5)--(9.7,0.5); \node at (10.1,0.5) {$2$};
\draw(0.1,0.1)--(0.1,-0.7)--(3.7, -0.7)--(3.7,0.1);
\node at (0.4,-0.35) {$2$}; \draw[dotted, thick] (0.7,-0.35)--(3.1,-0.35); \node at (3.4,-0.35) {$2$};
\node at(4.2,-0.7) {$z^2_2$};
\end{tikzpicture}}
 \end{align*}
 where in the figure we recall that $z^1_1$ is the number of $1$'s in the first row, $z^2_1$ is the 
 total number of $1$'s and $2$'s in the first row and $z^2_2$ is the number of $2$'s in the
 second row.
 The evolution of $\sfZ(n)$ to $\sfZ(n+1)$ can be described as follows:
 \begin{itemize}
 \item if $w^{n+1}=(1,0)$, meaning that a `$1$' appears, then $z^1_1(n+1)=z^1_1(n)+1$, i.e. the number of $1$'s in the first row will increase by one, and 
 \begin{itemize}
 \item if $z^1_1(n)=z^2_1(n)$ (i.e. there are no $2$'s in the first row of the corresponding Young tableau), then $z^2_1(n+1)=z^2_1(n)+1$ and $z^2_2(n+1)=z^2_2(n)$, 
 \item if $z^1_1(n)<z^2_1(n)$ (i.e. there are $2$'s in the first row of the corresonding Young tableau), then $z^2_1(n+1)=z^2_1(n)$ and $z^2_2(n+1)=z^2_2(n)+1$. i.e. 
 a $2$ will be bumped from the first row and inserted into the second, thus increasing its length by one.
 \end{itemize}
 \item  if $w^{n+1}=(0,1)$, meaning that a `$2$' appears, 
 then $z^2_1(n+1)=z^2_1(n)+1$, which means that the inserted `$2$' will be appended at the end of the
 first row, thus increasing its length by one, while the lengths
  $z^1_1, z^2_2$ will remain unchanged, that is, $z^1_1(n+1)=z^1_1(n)$ and $z^2_2(n+1)=z^2_2(n)$.  
 \item if $w^{n+1}=(0,0)$, then nothing changes.
 \end{itemize}
 This type of dynamics can be immediately generalised to the general $\rsk$ framework where instead of two letters $\{1,2\}$, we have $N$ letters
 $\{1,2,...,N\}$, i.e. $\sfW$ is an $n\times N$ matrix,
  and the $\GT$ pattern consists of an array $\big(z^i_j\big)_{1\leq j\leq i \leq N}$ (we assume here that $n\geq N$). 
  The dynamics generated by $\rsk$ will be as follows
 (to have a more clear picture, keep in mind the simpler case of $n=2$ worked out earlier and the bumping process of the row insertion):
 \begin{itemize}
 \item particles $\big\{z^i_1\colon i=1,...,N\big\}$ are the only particles that jump on their own volition and this happens if and only
  if $w^{n+1}$ has a $1$ in the $i$-th coordinate, meaning that the letter `$i$'
   appeared in the $(n+1)$-th insertion. In this case $z^i_1$ jumps to the right by one, meaning
  that the total length of letters up to $i$ will increase by one, while $z^1_1,z^2_1,...,z^{\,i-1}_1$ will remain the same. The jump at $z^i_1$ will propagate at the lower entries
  of the Gelfand-Tsetlin pattern in a way which can be described inductively as follows:
  \item if for some $1\leq j\leq i\leq N$ particle $z^i_j$ jumps by one, i.e. $z^i_j(n+1)=z^i_j(n)+1$, then it will generate a jump of exactly one of its neighbours
   $z^{i+1}_j$ or $z^{i+1}_{j+1}$, as follows
  \begin{itemize}
  \item if $z^i_j(n)=z^{i+1}_j(n)$, then $z^{i+1}_j(n+1)=z^{i+1}_j(n)+1$, and 
  \item if  $z^i_j(n) < z^{i+1}_j(n)$, then $z^{i+1}_{j+1}(n+1)=z^{i+1}_{j+1}(n)+1$
  \end{itemize}
  \item this cascading procedure continues until we reach the bottom of the $\GT$ pattern. 
 \end{itemize}
 \vskip 2mm
 These dynamics have a number of interesting properties, as we will see below. Even though it is clear that $\big(\sfZ(n)\big)_{n\geq 0}$ is Markovian,
 it will turn out that in some {\it particular} situations also the bottom row $\big(z^N(n)\big)_{n\geq 0}$ is Markovian. This is entirely not obvious and 
 it is even more surprising 
 that the emerging Markov process turns out to be related to certain processes {\it ``conditioned not to meet''} and to processes
 related to Random Matrices, i.e. {\bf Dyson's Brownian Motion},  which describes the evolution of eigenvalues of random matrix ensembles whose entries are independent 
 Brownian motions up to the symmetries of the matrix ensemble (e.g hermitian or orthogonal).
 The process $\big\{z^i_i(n)\colon i=1,...,N\big\}_{n\geq 0}$ on the right diagonal of the pattern is also related to popular interacting particle systems, in particular, the
 process known as {\bf Push-TASEP} : in this process a particle attempts to jump but if it is blocked by other particles ahead of it, then it pushes all the 
 block of consecutive particles, lying ahead of it, to the right by one.  
We will explain these in the next section but 
 let us first look at a number of variations of the above dynamics:
 \vskip 2mm
 {\bf Poissonian-RSK Dynamics.} We can think of the input, $\{0,1\}$-matrix $\sfW$ in the above example as being continuous in the vertical direction
 with $1$'s appearing on its $i^{th}$ row as a Poisson point processes with rate $x_i$, independent of the process of $1$'s on the other rows. Each time a $1$ appears in the $i$-th row, a letter `$i$'
  is row inserted in the $\GT$ pattern and induces the continuous version of the discrete dynamics which were previously described.
 
 \vskip 2mm
 {\bf $q$-deformation of $\rsk$.} A $q$-deformation comes from perturbing algebraic quantities using a parameter $q\in(0,1)$ in such a way
 that we obtain a meaningful interpolation when varying $q$.
 We will describe here a $q$-deformation of $\rsk$ introduced by \cite{OP13}. Although at first this deformation might seem arbitrary, it is motivated by the structure of 
 special functions called $q$-{\it Whittaker} (i.e. Macdonald polynomials with parameter $t=0$, see Section \ref{basics})
 There also exist other deformations coming from the more general family of {\it Macdonald polynomials} (which generalize Schur and $q$-Whittaker functions)  \cite{BC14, BP16b}
 and we will see these in Section \ref{sec:Cauchy}.
 \vskip 2mm
 The O'Connell-Pei \cite{OP13} deformation is as follows:
 \begin{itemize}
 \item particles $\big\{z^i_1(t)\colon i=1,...,N\big\}_{t\geq 0}$ are the only particles that jump of their own volition. Jumps are of size one to the right and
 happen independently according to Poisson processes of rates $x_i$, respectively. 
 \item if a particle $z^i_j$, $1\leq j\leq i\leq N$, jumps (to the right by one), then, with probability
 \begin{align*}
 \sfR_j(z^i;z^{i+1}):=q^{z^{i+1}_j-z^i_j} \,\,\frac{1-q^{z^i_{j-1}-z^{i+1}_j}}{1- q^{z^i_{j-1}-z^i_j}},
 \end{align*}
 it also pushes particle $z^{i+1}_j$ to the right by one. 
 While with probability $\sfL_j(z^i;z^{i+1}):=1-  \sfR_j(z^i;z^{i+1})$ it pulls particle  $z^{i+1}_{j+1}$ to the right by one.
 \end{itemize}  
 We remark that:
 \begin{itemize}
\item If $z^{i}_j=z^{i+1}_j$ and particle $z^i_j$ jumps then $\sfR_j(z^i;z^{i+1})=1$, which means that necessarily $z^{i+1}_j$ is pushed to the right and
the interlacing is respected.

\item In the limit $q\to 0$ we have that $\sfR_j(z^i;z^{i+1})$ converges to $\ind_{z^i_j=z^{i+1}_j}$ and, thus, we have the reduction to the $\rsk$ dynamics.
\item If we make the choice $q=e^{-\epsilon}$ and change variables $z^i_j:=\epsilon^{-2} \tau-(i+1-2j)\epsilon^{-1}\log \epsilon +\epsilon^{-1}\zeta^i_j$
(this is motivated by the scaling \eqref{Whitt-scale}), for
$1\leq j\leq i\leq N$, we see that
\begin{align*}
\sfR_j(z^i;z^{i+1})\approx \epsilon e^{\zeta^i_j-\zeta^{i+1}_j}.
\end{align*}
If we also speed up time as $t:=\epsilon^{-2} \tau$ and change the rates of the Poisson processes to $x_i:=\epsilon \nu_i$,
 then the dynamics of the continuum Gelfand-Tsetlin pattern $(\zeta^i_j(t)\colon 1\leq j\leq i\leq N)$ are described by the system of coupled diffusions
\begin{equation}\label{brownianRSK}
{ \begin{split}
  \dd \zeta^1_1= \dd B^{(1)},\qquad \dd \gz^i_1(t)=\dd B^{(i)}+ \Big(\nu_i+e^{\gz^{i-1}_1-\gz^{i}_1}\Big)\,\dd \tau &,\qquad i=2,...,N ,\\
 &\\
  \dd \gz^{i}_j = \dd \gz^{i-1}_j + \Big(e^{\gz^{i}_{j+1}-\gz^{i}_j} -  e^{\gz^{i}_{j}-\gz^{i-1}_{j-1}}\Big)
  \,\dd\tau &, \quad 1<j\leq i\leq N,
 \end{split}
 }
 \end{equation}
  where we make the convention that terms which contain variables, which don't belong to the $\GT$ pattern, are excluded.
  This system of interacting diffusions was intoduced by O'Connell \cite{O12}. The analogue of Schensted's  theorem \ref{schensted} in this setting is
   \begin{align*}
 \gz^N_1(t)=\int\cdots\int_{0<t_1<\cdots<t_N=t} \exp\Big(\sum_{i=1}^N \big( B^{(i)}(t_i) -  B^{(i)}(t_{i-1}) \big)\Big)\dd t_1\cdots\dd t_{N-1},
 \end{align*} 
 with the right-hand side called the {\it Brownian directed polymer} or {\it O'Connell-Yor polymer} as it was
  originally introduced by O'Connell and Yor \cite{OY01}. 
 The role of the noise is played here by $N$ independent Brownian motions and the polymer paths are piecewise constant paths with (upwards) 
 jumps at times $t_1,...,t_{N-1}$, that start at level $1$ and time $0$ and end at level $N$ and time $t$. 
\end{itemize}
\subsection{The Pitman-Rogers theorem}\label{sec:PRthm}
Let us now come back to the question of whether the bottom row in the above variations of $\rsk$ dynamics evolves as a 
Markov process with respect to its own filtration (meaning the information generated by the history of the process).
 Since the whole Gelfand-Tsetlin pattern $\sfZ(t)$
 (we think of time $t$ either discrete or continuous) evolves as a Markov process,
its bottom row can be regarded as a function of $\sfZ(t)$, which is nothing other than the projection.  We are, then, led to
the following general and fundamental question:
\vskip 3mm
  {\it When is a function of a Markov process Markov itself, with respect to its own 
filtration ?}
\vskip 3mm
A criterion for this has been provided by Pitman and Rogers \cite{RP81}.

\begin{theorem}[Pitman-Rogers]\label{MarkovfunctionsProp}
Consider a discrete time Markov process $\sfZ(\cdot)$ on a measurable space  $(\mathcal{Z},\mu)$ with transition probability kernel
$\mathbf{\Pi}$ and 
a measurable function $\Phi:\cZ\to\cX$, with $\mathcal{X}$ a measurable  space. 
Assume that there exists a kernel $\mathbf P(\cdot,\cdot)\colon \mathcal{X}\times \mathcal{X} \to \R $ such that for almost every 
$x\in  \mathcal{X}$, $\mathbf P(x,\cdot)$ is a probability measure
 and a kernel $\mathbf K(\cdot,\cdot)\colon \mathcal{X}\times \mathcal{Z}\to\R$ satisfying:
\begin{enumerate}
\item[(i)] for all $x\in \cX$, $\bK(x,\Phi^{-1}(x)) = 1$,
\item[(ii)] the inter-twinning relation $\bK \bPi = \mathbf{P} \bK$ holds.
\end{enumerate}
If, for arbitrary $x\in \cX$, the initial distribution of the Markov process $\sfZ(\cdot)$ is 
 $\bK(x,\cdot) / \int_\cZ \bK(x,z) \mu(\dd z)$, then it holds that
\begin{enumerate}
\item[(i)] The process $\sfX(t)=\Phi(\sfZ(t))$ is Markov with respect to its own filtration $\cX_t:=\sigma\{\sfX_s\colon s<t\}$ with
transition probability kernel $\mathbf{P}$ and initial condition $\sfX(0)=x$,
\item[(ii)] For all $x\in \cX$
 and all bounded Borel functions $f$ on $\cZ$, $$\bbE\big[f(\sfZ(t)) \,\big|\, \sfX(s), \, s<t ,\, 
\sfX(t)=x\big] = (\bK f)(x).$$
\end{enumerate}
\end{theorem}
Let us remark that even though we stated the theorem for discrete time Markov processes, it is also valid for continuous times. In this
case we would need to check the intertwining for the continuous-time transition kernels $\mathbf P_t$ and $\mathbf \bPi_t$. Or it is enough
to check the intertwining for the infinitesimal generators of the processes, given by definition as
\begin{align*}
\mathcal L =\frac{\dd \bPi_t}{\dd t}\Big|_{t=0} \qquad \text{and}\qquad  L =\frac{\dd \mathbf P_t}{\dd t}\Big|_{t=0},
\end{align*}
in which case the intertwining may be written as
\begin{align}\label{continter}
\mathbf K \mathcal L= L  \mathbf K.
\end{align}
In the next sections we will describe how we can construct processes on Gelfand-Tsetlin patterns compatible with the Pitman-Rogers theorem.
    \subsection{The role of Pieri identity, Branching rule and Intertwining. }\label{sec:pieri}

One difficulty in applying the Pitman-Rogers theorem is that one is often given dynamics recorded via $\mathbf{\Pi}$ (for example the dynamics described above coming from $\rsk$,
 Poissonian-RSK, Brownian-RSK, $q$-RSK and corresponding degenerations) 
and is requested to solve the intertwining equation $\bK\mathbf \Pi = \mathbf P \bK$ to find both kernels $\bK$ and $\mathbf P$. 
Often, the solution to this problem is based on trial and error guess work, starting from low rank cases, i.e. 
particle systems on $\GT$ patterns of small depths: two, three etc., until a pattern is recognised. 
Nevertheless, in many cases this procedure can be guided by certain structures that underlie special functions of representation theoretic origins
(see Section \ref{basics}). 
Two such useful structures go by the names of the {\bf Pieri rule} and the {\bf Branching rule}. We will here explain their role through Schur functions and 
Poissonian-RSK dynamics.  
\vskip 2mm
{\bf Branching rule for Schur polynomials.} Schur polynomials $s_\gl(x)$ with
 parameters $x=(x_1,...,x_n)$ and shape variables
$\gl=(\gl_1,...,\gl_n)$  can be written as
\begin{align}\label{SchurBranch}
s_\gl(x)=\sumtwo{Z\colon  \text{$\GT$ pattern}}{ \text{with shape $\gl$  } }  \,\prod_{i=1}^n Q^{\,i}_{i-1}(z^i,z^{i-1}\,;\,x_i),
\quad \text{with}\quad Q^{\,i}_{i-1}(z^i,z^{i-1}\,;\,x_i):= x_i^{|z^i|-|z^{i-1}|}.
\end{align}
This is a rewriting of \eqref{Schur-gen-tableau}
and leads to the recursive structure 
\begin{align}\label{SchurBranch2}
s_\gl(x_1,...,x_n)=\sum_{\mu\colon \mu\prec \gl}  Q^{\,n}_{n-1}(\gl,\mu\,;\,x_n) \,s_\mu(x_1,...,x_{n-1})
\end{align}
where $\mu=(\mu_1,...,\mu_{n-1})$ is a partition and $\mu\prec\gl$ means that $\mu$ interlaces with $\gl$, that is
$\gl_1\geq\mu_1\geq \gl_2\geq \mu_2\geq\cdots \geq \mu_{n-1}\geq \gl_n$. We make the convention that, when $n=0$, then 
$s_\gl(x_1,...,x_n)=s_{\emptyset}(\emptyset):=1$.
Such a recursive structure is often found among
special functions arising in representation theory and algebraic combinatorics and it is known by the name {\it branching rule}.
 We refer to \cite{Bum04} for more on the representation theoretic background, but very briefly let us mention that
 if $G$ is a group and $H$ a subgroup, then branching rules
  describe how the irreducible representations of $G$ decompose to irreducibles when restricted to 
 the subgroup $H$. In the case of Schur functions,
 the branching rule shows how to decompose the irreducible representations of the group $GL(n)$ (the 
 {\it general linear group} of $n\times n$ invertible matrices) into irreducibles of $GL(n-1)$.
 \vskip 2mm
 An algebraically more common notation for \eqref{SchurBranch2} (which we will also use in later sections) is
 \begin{align}\label{SchurBranch3}
s_\gl(x_1,...,x_n)=\sum_{\mu\colon \mu\prec \gl}  s_{\gl/\mu}(x_n) \,s_\mu(x_1,...,x_{n-1})
\end{align}
where $\gl/\mu$ denotes a skew partition, defined from partitions $\gl=(\gl_1,\gl_2,...)$
 and $\mu=(\mu_1,\mu_2,...)$ as $\gl/\mu=(\gl_1-\mu_1,\gl_2-\mu_2,...)$. In the case that $\gl=(\gl_1,...,\gl_n)$ and $\mu=(\mu_1,...,\mu_{n-1})$ we use the
 convention that $\gl_n-\mu_n=\gl_n$. An example is now depicted below
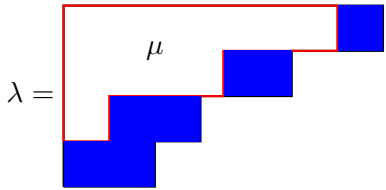
\begin{equation*}
\lambda= \begin{tikzpicture}[baseline={([yshift=-.5ex]current bounding box.center)},vertex/.style={anchor=base,
    circle,fill=black!25,minimum size=18pt,inner sep=2pt}, scale=0.3]
\draw[thick] (0,6)--(0,14)--(14,14)--(14,12)--(10,12)--(10,10)--(6,10)--(6,8)--(4,8)--(4,6)--(0,6);
 \draw[thick, red] (0,8)--(0,14)--(12,14)--(12,12)--(7,12)--(7,10)--(2,10)--(2,8)--(0,8);
 \node at (4,12) {$\mu$};
 \path [fill=blue] (12.05,12) rectangle (14,14);
  \path [fill=blue] (7.05,10) rectangle (10,12);
   \path [fill=blue] (2.05,8) rectangle (6,10);
    \path [fill=blue] (0.05,6) rectangle (4,8);
\end{tikzpicture}
\end{equation*}
In this figure the blue shaded area is the skew partition $\gl/\mu$. A skew partition $\gl/\mu$ is called a {\bf horizontal strip} if
 it contains no two boxes on top of each other, as for example in the following figures
 \begin{equation}\label{horizontal}
\lambda= 
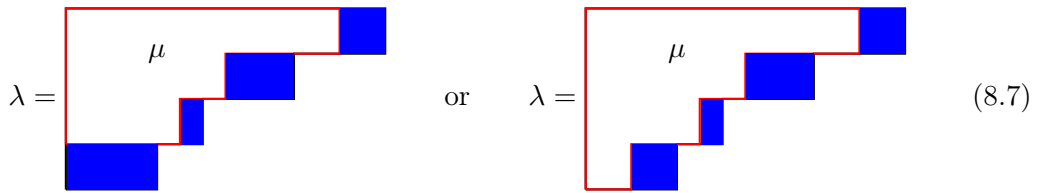
\begin{tikzpicture}[baseline={([yshift=-.5ex]current bounding box.center)},vertex/.style={anchor=base,
    circle,fill=black!25,minimum size=18pt,inner sep=2pt}, scale=0.3]
\draw[thick] (0,6)--(0,14)--(14,14)--(14,12)--(10,12)--(10,10)--(6,10)--(6,8)--(4,8)--(4,6)--(0,6);
 \draw[thick, red] (0,8)--(0,14)--(12,14)--(12,12)--(7,12)--(7,10)--(5,10)--(5,8)--(4,8)--(0,8);
 \node at (4,12) {$\mu$};
 \path [fill=blue] (12.05,12) rectangle (14,14);
  \path [fill=blue] (7.05,10) rectangle (10,12);
   \path [fill=blue] (0.05,6) rectangle (4,8);
    \path [fill=blue] (5.05,8) rectangle (6,10);
\end{tikzpicture}
\qquad\text{or}\qquad
\lambda= \begin{tikzpicture}[baseline={([yshift=-.5ex]current bounding box.center)},vertex/.style={anchor=base,
    circle,fill=black!25,minimum size=18pt,inner sep=2pt}, scale=0.3]
\draw[thick] (0,6)--(0,14)--(14,14)--(14,12)--(10,12)--(10,10)--(6,10)--(6,8)--(4,8)--(4,6)--(0,6);
 \draw[thick, red] (0,6)--(0,14)--(12,14)--(12,12)--(7,12)--(7,10)--(5,10)--(5,8)--(4,8)--(2,8)--(2,6)--(0,6);
 \node at (4,12) {$\mu$};
 \path [fill=blue] (12.05,12) rectangle (14,14);
  \path [fill=blue] (7.05,10) rectangle (10,12);
   \path [fill=blue] (2.05,6) rectangle (4,8);
    \path [fill=blue] (5.05,8) rectangle (6,10);
\end{tikzpicture}
\end{equation}
The summation in \eqref{SchurBranch3} is over skew partitions $\gl/\mu$ which are horizontal strips that correspond to the
first of the two pictures above. This is because $\gl/\mu$ will correspond to the part of the Young tableau which is filled with
letter $n$ and, by definition of a Young tableau, no two $n$'s can be on top of each other (recall that the entries of Young tableaux
are by definition strictly increasing along columns).
\vskip 2mm
Let us denote the kernel 
\begin{align}\label{SchurK0}
K(\gl,\sfZ):=  \prod_{i=1}^n Q^{\,i}_{i-1}(z^i,z^{i-1}\,;\,x_i)\, \ind_{z^n=\gl},
\end{align}
from which we readily have that 
\begin{align}\label{SchurK}
s_\gl(x)=\sum_{\sfZ\,\,\text{is $\GT$ pattern}} K(\lambda, \sfZ).
\end{align}
 As we will see, the kernel $K(\gl,\sfZ)$ will
play an important role in intertwining.
\vskip 2mm

{\bf The Pieri Rule for Schur polynomials.} Let us recall the complete homogeneous symmetric polynomial of degree $k$ in variables $x_1,...,x_n$ 
to be 
\begin{align*}
h_k(x_1,...,x_n):=\sum_{1\leq i_1\leq \cdots \leq i_k\leq n} x_{i_1} \cdots x_{i_k}.
\end{align*}
In the theory of Schur functions the Pieri rule is known as the identity
\begin{align*}
h_k s_\gl = \sum_{\nu \succ_k \lambda } s_\nu,
\end{align*}
where the notation $\nu\succ_k \gl$ means that $\nu/\gl$ is a {\it horizontal strip}, containing exactly $k$ boxes. 
We will restrict our attention on $k=1$, in which case $h_1(x_1,...,x_n)=\sum_{i=1}^n x_i$ and the Pieri rule may be written as
\begin{align}\label{Pieri1}
h_1 s_\gl = \sum_{\nu \succ_1 \lambda } s_\nu = \sum_{i=1}^n  s_{\lambda+{\be_i}} .
\end{align}
In fact, it is straightforward to check this special case of the Pieri rule via the representation of the Schur function as a sum over Gelfand-Tsetlin patterns \eqref{SchurBranch}. 
\vskip 2mm
{\bf Intertwining: the Poissonian-$\rsk$ case.}
Let us now see how the branching rule \eqref{SchurBranch2}, \eqref{SchurBranch3}, \eqref{SchurK} and the Pieri rule \eqref{Pieri1}
can be used to ``guess'' an intertwining in the case of Poissonian-$\rsk$ dynamics on a $\GT$ patterns.
We briefly recall the dynamics: on a $\GT$ pattern $\sfZ=(z^i_j\colon 1\leq j \leq i \leq n)$ only particles $z^i_1, i=1,2,...,n$ jump of their own volition and they do so at exponential
times at rate $x_i$, respectively. The jumps consist of one step to the right and trickle down the pattern as follows: if particle $z^i_j$ jumps and before the jump took place we had
$z^i_j=z^{i+1}_j$, then
  particle $z^{i+1}_j$ is pushed one step to the right along with $z^i_j$. 
  If  $z^i_j < z^{i+1}_j$, then particle $z^{i+1}_{j+1}$ is pulled one step to the right along with $z^i_j$.
   The jumps trickle down until they reach the bottom of the pattern.
 The Markov generator of this process can be easily written. Concretely, in the case  $n=2$, it may be written as
\begin{align*}
 \Pi(\sfZ,\tilde\sfZ) &= x_1\, \ind_{\{ (\tilde z^1,\tilde z^2)=(z^1+e_1,z^2+e_1)\}} \ind_{\{z^1_1=z^2_1\}} 
+ x_1 \, \ind_{\{ (\tilde z^1,\tilde z^2)=(z^1+e_1,z^2+e_2)\}} \ind_{\{z^1_1<z^2_1\}} \\
&\qquad+x_2 \ind_{\{ (\tilde z^1,\tilde z^2)=(z^1,z^2+e_1) \}} 
\end{align*}
where we have denoted the base vectors on $\R^2$ by $e_1:=(1,0)$ and $e_2:=(0,1)$.
Observe that
\begin{align}\label{sumpi}
\sum_{\tilde\sfZ }  \Pi(\sfZ,\tilde\sfZ) = x_1+x_2= h_1(x_1,x_2),
\end{align}
the complete, symmetric function of degree one, in variables $x_1,x_2$.
\vskip 2mm
Let us now write the Pieri rule  \eqref{Pieri1} by replacing the Schur function with the branching rule representation  \eqref{SchurK}  as
\begin{align*}
\sum_{i=1}^n s_{\gl+e_i}(x) \stackrel{ \eqref{SchurK}}{=}\sum_{i=1}^n \,\,\sum_{\tilde \sfZ } K(\lambda+e_i, \tilde \sfZ) 
\stackrel{\eqref{Pieri1}}{=} h_1 \sum_{\sfZ } K(\lambda, \sfZ)
\stackrel{ \eqref{sumpi}}{=}\sum_{\sfZ, \tilde \sfZ}  K(\lambda, \sfZ)  \Pi( \sfZ, \tilde \sfZ).
\end{align*}
An intertwining relation can be now guessed if we drop (seemingly rather arbitrarily)
 the summation over $\tilde \sfZ$ from the second and fourth term in the above equality and write
\begin{align*}
\sum_{i=1}^n  K(\lambda+e_i, \tilde \sfZ)  = \sum_{ \sfZ} K(\lambda, \sfZ)  \Pi( \sfZ, \tilde\sfZ),
\end{align*}
or, if we denote by $P$ the operator with $P(\lambda, \nu):= \sum_{i=1}^n \ind_{\nu=\lambda+e_i}$, then we can write this more concretely as
\begin{align}\label{guessinter}
P K = K  \Pi.
\end{align}
We have now arrived to a guess for an intertwining relation between the transition kernel $\Pi$ for the Poisson-RSK
 dynamics on the  Gelfand-Tsetlin pattern $\sfZ$ and a transition kernel for dynamics on its bottom row $\lambda$ .
 The intertwining kernel $K$ will be the kernel that appears in \eqref{SchurK0}, which 
 gives the combinatorial representation of Schur polynomials. Of course, it remains
 to check the validity of \eqref{guessinter}. This is typically straightforward and can be done by hand even though the check can be long. 
Let us just check this intertwining for $n=2$.
We compute 
\begin{align}\label{inter}
K \Pi(\gl,\tilde\sfZ)&=\sumtwo{Z\colon  \text{$\GT$ pattern}}{ \text{with shape $\gl$  } }  K(\gl,\sfZ) \Pi(\sfZ,\tilde\sfZ) \\
&= \sum_{z^1_1} x_2^{|z^2|-|z^1|} x_1^{|z^1|}  \, \ind_{\{z^2=\gl\}} \,
\Big( x_1\, \ind_{\{ (\tilde z^1,\tilde z^2)=(z^1+e_1,z^2+e_1)\}} \ind_{\{z^1_1=z^2_1\}} \,\notag\\
&\qquad \qquad+ x_1 \, \ind_{\{ (\tilde z^1,\tilde z^2)=(z^1+e_1,z^2+e_2)\}} \ind_{\{z^1_1<z^2_1\}} +
x_2 \ind_{\{ (\tilde z^1,\tilde z^2)=(z^1,z^2+e_1) \}} \Big)  \notag\\
&=x_2^{\gl_1+\gl_2-\tilde z^1_1 +1} x_1^{\tilde z^1_1} \,\, \Big( 
\ind_{\{ \tilde z^1_1=\gl_1+1 \}} \ind_{\{ \tilde z^2_1=\gl_1+1 \}}  \ind_{\{ \tilde z^2_2=\gl_2 \}} \notag\\
&\qquad \qquad +  \ind_{\{ \tilde z^2_1=\gl_1 \}} \ind_{\{\tilde z^2_2=\gl_2+1\}} \ind_{\{\gl_2+1\leq \tilde z^1_1\leq \gl_1 \}} 
+  \ind_{\{ \tilde z^2_1=\gl_1+1 \}} \ind_{\{ \tilde z^2_2=\gl_2 \}} \ind_{\{ \gl_2\leq \tilde z^1_1<\gl_1+1 \}}
 \Big)  \notag\\
 &=x_2^{\gl_1+\gl_2-\tilde z^1_1 +1} x_1^{\tilde z^1_1} \,\, \Big( 
\ind_{\{ \gl_2\leq \tilde z^1_1\leq \gl_1+1 \}} \ind_{\{ \tilde z^2_1=\gl_1+1 \}}  \ind_{\{ \tilde z^2_2=\gl_2 \}} 
 +  \ind_{\{ \tilde z^2_1= \gl_1 \}} \ind_{\{\tilde z^2_2=\gl_2+1\}} \ind_{\{\gl_2+1\leq \tilde z^1_1\leq \gl_1 \}}  \Big); \notag
\end{align}
where the last equality comes from combining the first and third terms in the parenthesis.
On the other hand we have 
\begin{align*}
P K (\gl,\tilde \sfZ)
&= \sum_{\mu\succ_1 \gl} P(\gl;\mu) \, K(\mu, \tilde \sfZ)
=K(\gl_1+1,\gl_2;\tilde \sfZ) + K(\gl_1,\gl_2+1;\tilde \sfZ)\\
&= x_2^{\gl_1+\gl_2-\tilde z^1_1 +1} x_1^{\tilde z^1_1} \Big( 
\ind_{\{ \gl_2\leq \tilde z^1_1\leq \gl_1+1 \}} \ind_{\{ \tilde z^2_1=\gl_1+1 \}}  \ind_{\{ \tilde z^2_2=\gl_2 \}} +
 \ind_{\{ \tilde z^2_1= \gl_1 \}} \ind_{\{\tilde z^2_2=\gl_2+1\}} \ind_{\{\gl_2+1\leq \tilde z^1_1\leq \gl_1 \}}   \Big)
\end{align*} 
which is the right-hand side of \eqref{inter}.
\vskip 2mm
The intertwining in the general $n$ case can be checked via induction. We omit the details. An example of such an inductive procedure (in the case
of the log-gamma polymer and Whittaker functions) can be found in \cite{COSZ14}, in the proof of Proposition 3.4.
  \vskip 4mm
  {\bf Normalization of kernels, Doob's transform and non-intersecting walks.}
   Notice that \eqref{guessinter} is not quite in the form of the Pitman-Rogers theorem since the kernels $K$ and $P$ 
  are not probability kernels. This 
  can be easily rectified by normalising $K$ and $P$ and defining 
    \begin{align}\label{Doobnorm}
  \bK(\gl,\sfZ) : = \frac{K(\gl,\sfZ)}{\sum_{Z} K(\gl,\sfZ)} = \frac{K(\gl,\sfZ)}{s_\gl(x)}
  \qquad \text{and} \qquad \mathbf P(\gl,\nu):=\frac{s_{\nu}(x)}{s_\gl(x)} P(\gl,\nu),
  \end{align}
for $\lambda\neq \nu$, with $s_\lambda(x), s_\nu(x)$ Schur functions corresponding
  to shapes $\lambda, \nu$ and parameters $x\in\R^n$.
  It is immediate to check that with this definition $\bPi, \mathbf P$ and $\bK$ satisfy the intertwining relationship
   $\mathbf P \, \bK=\bK\, \bPi$. 
  
  \vskip 2mm
  The form of $\mathbf P(\gl,\nu)$ points to some sort of {\bf Doob's $h$-transform}. We refer to \cite{RY13}, Chapter VIII for details on Doob's transform and applications but
  let us only say a couple of words about it: In general, given a positive kernel $P:\mathcal{A}\times\mathcal{A}\to\R_+$ on a state space $\mathcal{A}$,
   we can define a {\it probability} kernel $\bP$ on the same state space $\mathcal{A}$, which can serve as a Markov transition kernel. The probability
     kernel $\bP$ is defined as
     \begin{align*}
     \bP(a,b):=\frac{h(b)}{h(a)} P(a,b),\qquad \text{for} \qquad a,b\in \mathcal{A},
     \end{align*}
     where $h(\cdot)$ is a {\it ground state} eigenfunction for $P$ by which we mean that $\sum_{b} P(a,b)\,h(b)=h(a)$, for all $a\in\mathcal{A}$.
     The original motivation for this transform was to define Markovian processes conditioned on certain events (for example Brownian motion or random walks
     conditioned never to visit zero).
    \vskip 2mm
  Let us now see how the idea of Doob's transform can be applied in our setting. First, we notice that a direct consequence
  of Pieri's rule is that the function 
$\sfh(\lambda):=x_1^{-\gl_1}\cdots x_n^{-\gl_n} \, s_\gl(x)$ satisfies
 the equation 
\begin{align}\label{harmonic}
\sum_{i=1}^n x_i  \,\sfh(\gl+\be_i) = (\sum_{i=1}^n x_i) \,\sfh(\gl).
\end{align}
 This is easily seen by multiplying both sides of \eqref{Pieri1} by $x_1^{-\gl_1}\cdots x_n^{-\gl_n}$. The fact that
 $s_\gl(x)=0$, if $\gl$ does not belong to the {\it Weyl chamber} $\bW_n:=\{\gl\in \R^n\colon \gl_1\geq \gl_2\geq \cdots \geq \gl_n\}$,  implies that $h(\gl)=0$ if $\gl\notin\bW_n$. This, together with \eqref{harmonic} implies that $h(\cdot)$ is harmonic for the  
 generator of an
  $n$-dimensional random walk whose coordinates jump at rates $(x_i)_{i=1,...,n}$ and which is killed upon exiting the Weyl chamber $\bW_n$
  (i.e. Dirichlet boundary conditions on the complement of $\bW_n$). 
Let us  denote by $Q(\lambda,\nu)$ the transition kernel of this random walk, 
which takes values $Q(\lambda,\nu)=x_i$ if $\nu=\gl+e_i$ for $i=1,...,n$ and zero otherwise. Since for $\nu\neq\gl$ it holds that 
\begin{align*}
Q(\gl,\nu)=\frac{x_1^{\nu_1}\cdots x_n^{\nu_n}}{x_1^{\gl_1}\cdots x_n^{\gl_n}} P(\gl,\nu)
\end{align*}
for $P(\lambda, \nu):= \sum_{i=1}^n \ind_{\nu=\lambda+e_i}$ as defined in \eqref{guessinter},
we see  that
\begin{align}\label{810two}
\mathbf  P(\lambda,\nu):=\frac{\sfh(\nu)}{\sfh(\gl)} Q(\gl,\nu)=\frac{s_\nu(x)}{s_{\gl}(x)} P(\gl,\nu),
\end{align}
(as defined in \eqref{Doobnorm}) specifies a Markovian kernel of a random walk conditioned not to exit the Weyl chamber.
\vskip 2mm
If we consider the coordinates of the process $\gl=(\gl_1,...,\gl_n)$ defined via \eqref{810two} individually and actually ``separate'' them by considering
$\ell_i:=\gl_i+n-i$, so that $\ell_1> \ell_2>\cdots >\ell_n$, then $\ell_1,...,\ell_n$ evolve via \eqref{810two} as $n$ simple random walks conditioned never
to meet. This process is the discrete analogue of Dyson's Brownian motion \cite{Dy62, AGZ10}, which consists of $n$ Brownian motions conditioned never to meet. Dyson's
Brownian motion describes the evolution of the eigenvalues of Hermitian random matrices whose entries are independent Brownian motions, up to the
Hermitian symmetry.  This provides one more explanation of the relation between polymer models, RSK dynamics and random matrices.  

\subsection{The roles of Cauchy, skew-Cauchy identities and Macdonald polynomials}\label{sec:Cauchy}
The Cauchy identity states that
\begin{align*}
\sum_\gl s_\lambda(x) s_\gl(y) = \prod_{1\leq i,j\leq n} \frac{1}{1-x_iy_j},
\end{align*}
where $x=(x_1,...,x_n)\in\R^n$ and $y:=(y_1,...,y_n)\in\R^n$ and the sum is taken over all partitions 
$\lambda=(\gl_1,...,\gl_n)$ with $\gl_1\geq \gl_2\geq \cdots $.
This identity can extend to the case of Schur functions that correspond to vector parameters of different length, e.g.
$x\in\R^n$ and $y\in\R^m$ or even of infinite length, e.g. $x=(x_1,x_2,...)$ and $y=(y_1,y_2,...).$
As in the case of the branching and Pieri rules, the Cauchy identity has a representation theoretic background 
(see \cite{Bum04} for more details). 
It describes the decomposition of the {\it symmetric algebra }
over the tensor product representation of $GL(n)\times GL(n)$
(or $GL(n)\times GL(m)$ in general). Combinatorially the Cauchy identity is a consequence of the $\rsk$ correspondence
as apparent from \eqref{schur1} by letting $u\to\infty$, therein. Seen from this angle, the Cauchy identity
allows us to introduce a probability measure on partitions, called the {\it Schur measure} and introduced by
Okounkov \cite{O01}, which may be written as
\begin{align*}
\bbP(\lambda) = \prod_{i,j}(1-x_iy_j) \, s_\lambda(x) \,s_\gl(y).
\end{align*} 
Its marginal over $\gl_1$ is the law of last passage percolation with geometric weights, as manifested by
\eqref{schur1}. The analogue of the Cauchy identity in the case of Whittaker functions is the Bump-Stade identity
\begin{align*}
\int_{\bbR_+^n} e^{-1/x_n} \Psi^{\mathfrak{gl}_n}_{\alpha}(x) \Psi^{\mathfrak{gl}_n}_{\beta}(x) \prod_{i=1}^n\frac{\dd x_i}{x_i}
= \prod_{i,j}\Gamma(\alpha_i+\beta_j),
\end{align*}
which, as we already saw in Section \ref{sec:loggamma} (see relation \eqref{shape}), can analogously 
be derived via  the geometric $\rsk$ correspondence. Similarly to the Schur measure, the Bump-Stade identity 
can be used to define a probability measure on $\R^n$
\begin{align*}
\frac{1}{\prod_{i,j}\Gamma(\alpha_i+\beta_j)}
 \, e^{-1/x_n} \Psi^{\mathfrak{gl}_n}_{\alpha}(x) \Psi^{\mathfrak{gl}_n}_{\beta}(x) \prod_{i=1}^n\frac{\dd x_i}{x_i},
\end{align*}
the marginal of which on the first coordinate is the law of the partition function of the log-gamma polymer
(recall that the role of shape variables $\gl$ in Schur functions is played by the $x$ variables in the Whittaker functions).

Recently, an interesting approach of producing Cauchy identities based on integrable methods around the six-vertex model and Yang-Baxter relations
has been developed, e.g. \cite{BBF11, BW16, BWZ15, BP17, WZ-J16} that also leads to various integrable stochastic models.
\vskip 2mm
Here, we will focus on seeing how Cauchy identities, in conjunction with branching and Pieri rules, can be used to define
natural dynamics on Gelfand-Tsetlin patterns. We will expose this in the general framework of Macdonald functions
 as this was developed by Borodin-Corwin within {\it Madonald processes} \cite{BC14} . This setting has been
 subsequently generalised in several situations and broader frameworks, for example  \cite{BP14, BP16b, BM18, MP17} etc.
\vskip 2mm
{\bf Macdonald polynomials and useful identities.}
Macdonald polynomials were introduced by Macdonald \cite{M88}, see also \cite{M98}, as a new basis for  symmetric functions, 
which  generalise Schur polynomials in the sense that
they contain two parameters $q,t$ and when $q=t$ they reduce to Schur functions. Furthermore,
making different choices on the parameters $q,t$ Macdonald polynomials specialise to other special functions
e.g. Jack with parameter $\ga$ when $q=t^\ga$ and $t\to1$, zonal symmetric functions when $q=t^2$ and $t\to1$, Hall-Littlewood when $q=0$,  etc  (see \cite{M88, M98}). 
Let us present a combinatorial formula for Macdonald polynomials, which was not their original definition given by Macdonald, but 
has the structure of a branching rule in analogy to what is available for Schur, as well as Whittaker functions.

Let $x=(x_1,...,x_n)\in\R^n$ and $\lambda, \mu\in \R^n$ be partitions. The Macdonald polynomials 
$P_{\gl/\mu}(x)$ and $Q_{\gl/\mu}(x)$ corresponding to skew partition $\gl/\mu$ and parameter $x$ admit the
combinatorial representation 
\begin{align}\label{PQMac}
P_{\gl/\mu}(x) = \sum_{T\colon sh(T)=\gl/\mu} \psi_T \, x^T\qquad \text{and} \qquad 
Q_{\gl/\mu}(x) = \sum_{T\colon sh(T)=\gl/\mu} \phi_T  \,x^T,
\end{align}
where the summations are over all skew, semi-standard Young tableaux $T$ 
of shape $\gl/\mu$ and we have used the shorthand
notation 
\begin{align*}
x^T= x_1^{|\gl^{(1)}|}\, x_2^{|\gl^{(2)}| - |\gl^{(1)}| } \cdots  x_n^{|\gl^{(n)}| - |\gl^{(n-1)}| },
\end{align*}
with $|\gl^{(k)}|=\sum_{j\geq 1} \gl^{(k)}_j$ and $\mu=\gl^{(1)}\prec \gl^{(2)}\prec \cdots \prec \gl^{(n)}=\gl$.
 When $\mu=\emptyset$ the above expressions can be written in Gelfand-Tsetlin notation, 
 see Section \ref{sec:GT}, by identifying $\gl^{(k)}$
with the $k$-th row $z^k=(z^k_1,...,z^k_k)\in \R^k$ of a Gelfand-Tsetlin pattern with shape (i.e. bottom row) $\gl$.
 The weights $\psi_T$ and $\phi_T$ for tableaux $T$ with $sh(T)=\gl/\mu$ are given by
 \begin{align*}
 \psi_{\gl/\mu} = \prod_{i=1}^n \psi_{\gl^{(i)}/\gl^{(i-1)}} \quad \text{and} \quad \phi_{\gl/\mu} = \prod_{i=1}^n \phi_{\gl^{(i)}/\gl^{(i-1)}}.
 \end{align*}
 For $\gl/\mu$ a horizontal strip  (see \eqref{horizontal} and recall that $\gl^{(i)}/\gl^{(i-1)}$ is a horizontal strip since $\gl^{(i-1)} \prec \gl^{(i)}$) 
 we have that
\begin{equation}\label{phipsi}
 \begin{split}
 \psi_{\gl/\mu} &= \prod_{1\leq i\leq j\leq \ell(\mu)} \frac{f(q^{\mu_i-\mu_j} \,t^{j-i} )   f(q^{\gl_i-\gl_{j+1}} \,t^{j-i} )}
 { f(q^{\gl_i-\mu_j} \,t^{j-i}) f(q^{\mu_i-\gl_{j+1}} \,t^{j-i})},  \\
 \phi_{\gl/\mu} &= \prod_{1\leq i\leq j\leq \ell(\gl)} \frac{f(q^{\gl_i-\gl_j} \,t^{j-i} )   f(q^{\mu_i-\mu_{j+1}} \,t^{j-i} )}
 { f(q^{\gl_i-\mu_j} \,t^{j-i}) f(q^{\mu_i-\gl_{j+1}} \,t^{j-i})},
 \end{split}
 \end{equation}
 with
 \begin{align*}
 f(u)=\frac{(tu;q)_\infty}{(qu;q)_\infty},
  \end{align*}
 with the $q$-Polchammer symbol defined as $(a;q)_\infty:=\prod_{i= 1}^\infty (1-aq^{i-1})$. When $q=t$ we have that 
 $f(u)=1$ and then, evidently from \eqref{PQMac}, 
 \begin{align*}
 P_{\gl/\mu}(x) = Q_{\gl/\mu}(x) = \sum_{T\colon sh(T)=\gl/\mu}  \, x^T = s_{\gl/\mu}(x),
 \end{align*}
 which shows that in the case $q=t$ Macdonald polynomials are identical to Schur polynomials.
 
 We will also use the following ``coefficients'' when we later discuss Pieri identities:
 \begin{align*}
 \psi'_{\gl\,/\, \mu} = \prod_{i < j \colon \gl_i=\mu_i\,,\, \gl_{j}=\mu_{j}+1} 
 \frac{(1-q^{\,\mu_i-\mu_{j}} t^{\,j-i-1})(1-q^{\,\gl_i-\gl_j} t^{\,j-i+1}) }{(1-q^{\,\mu_i-\mu_j} t^{\,j-i})(1-q^{\,\gl_i-\gl_j}t^{\,j-i})},
 \end{align*}
 for $\gl/\mu$ a vertical $r$-strip. The latter means that the Young diagrams $\lambda'$ and $\mu'$ - obtained from
 $\lambda$ and $\mu$ by making the rows of $\gl,\mu$ be the columns of $\gl',\mu'$ -
  are such that $\gl'/\mu'$ is a horizontal strip (see \eqref{horizontal}) containing $r$ boxes.
  
 It will also be useful to distinguish 
 the case of Macdonald polynomials of one variable  $x_1\in\R$: in this case, if $\gl/\mu$ is a horizontal strip, then
 \begin{align}\label{skewQP}
 Q_{\gl/\mu}(x_1)= \phi_{\gl/\mu} \, x_1^{|\gl|-|\mu|} \qquad\text{and}\qquad 
 P_{\gl/\mu}(x_1)= \psi_{\gl/\mu} \, x_1^{|\gl|-|\mu|}\,.
 \end{align}
 \vskip 2mm
 We can then write the branching rule for Macdonald polynomials with parameter $x=(x_1,...,x_k)$ as
 \begin{align}\label{Macbranch}
 P_\gl(x)= \sum_{\gl^{(1)} \prec \cdots \prec \gl^{(k)}=\gl} 
 P_{\gl^{(1)}}(x_1) P_{\gl^{(2)} / \gl^{(1)}}(x_2) \cdots  P_{\gl^{(k)} / \gl^{(k-1)}}(x_k),
 \end{align}
 and similarly for $Q_\gl(x)$, in analogy to the branching formula for Schur polynomials \eqref{SchurBranch}.
 \vskip 2mm
{\bf Pieri identity for Macdonald polynomials.}

 Let $e_r:=e_r(x_1,...,x_n):=\sum_{1\leq i_1<\cdots<i_r\leq n} x_{i_1}\cdots x_{i_r}$ 
 be the elementary symmetric polynomial
 of degree $r$ in variables $x_1,...,x_n$
 and $g_r:=g_r(x_1,...,x_n):=Q_{(r)}(x_1,...,x_n)$ be
  the $Q$-Macdonald polynomial corresponding to a partition $(r)$,
  represented by a Young diagram of a single row containing $r$ boxes.  There are a number of Pieri identities
   (see \cite{M98}, Chapter  VI). We will just use the following two
 \begin{align}\label{macpieri}
 P_\mu \, g_r = \sum_{\gl\colon \gl \succ_r \,\mu } \phi_{\gl/\mu} P_\gl = \sum_{\gl\colon \gl {\succ_r} \,\mu} Q_{\gl/\mu}(1) P_\gl
 \qquad \text{and} \qquad P_\mu e_r = \sum_{\gl\colon \gl'\succ_r \,\mu'} \psi'_{\gl/\mu} P_\gl  ,
 \end{align}
  where in the second identity we recall the notation $\gl'$ which means the Young diagram whose rows are the columns of $\gl$.
 In the case that $|\gl/\mu|=1$, that is $r=1$ and $\gl\succ_1 \mu$, it holds that 
 $ \psi'_{\gl/\mu} = \tfrac{1-t}{1-q} \phi_{\gl/\mu} =  \tfrac{1-t}{1-q} Q_{\gl/\mu}(1)$.
 Given that $e_1=\sum_{i=1}^n x_i$, we will see that, similarly to the Schur case,
  either of the Pieri identities (which in the case $r=1$ 
 are identical up to a constant) can be used to define a Markovian evolution on Gelfand-Tsetlin patterns,
  whose total rate of change will be $\sum_{i=1}^n x_i$.
 \vskip 2mm
{\bf Cauchy and skew-Cauchy identity for Macdonald polynomials.}
The Cauchy identity for Macdonald polynomials may be written as
\begin{align}\label{CauchyM}
 \sum_{\kappa} P_{\kappa} (x) Q_{\kappa}(y) = H(x,y),
 \end{align}
where for $x=(x_1,...,x_n)$ and $y=(y_1,...,y_n)$ the summation is over all partitions $\kappa$ of length $\ell(\kappa)\leq n$. The right-hand side of this
 identity is 
 \begin{align}\label{H}
 H(x,y):=\prod_{i,j} \frac{(tx_iy_j;q)_\infty}{(x_iy_j;q)_\infty}.
 \end{align}
 As in other situations this Cauchy identity can be extended to vectors $x=(x_1,...,x_n), y=(y_1,...,y_m)$ of different dimensions or even
 of infinite dimensions $x=(x_1,x_2,...)$ and $y=(y_1, y_2,...)$.
  
 The Cauchy identity can be generalised to a skew-Cauchy identity in the form
 \begin{align}\label{skewCauchyM}
 \sum_{\mu} P_{\mu/\gl} (x) Q_{\mu/\nu}(y) = H(x,y) \sum_{\mu} Q_{\gl/\mu}(y) P_{\nu/\mu}(x).
 \end{align}
 In the case that $\gl=\emptyset$, then \eqref{skewCauchyM} becomes
 \begin{align}\label{skewCauchyM2}
 \sum_{\mu} P_{\mu} (x) Q_{\mu/\nu}(y) = H(x,y)  P_{\nu}(x),
 \end{align}
 which in a sense generalises Pieri's rule \eqref{macpieri}.
 We refer to \cite{M98}, Chapter VI.7, for details and proofs of the above Cauchy identities
 (in particular of \eqref{skewCauchyM} and \eqref{skewCauchyM2} as we have already seen a proof of \eqref{CauchyM} in Section \ref{basics})
  \vskip 2mm
 \subsection{Macdonald tailored dynamics on Gelfand-Tsetlin patterns}\label{sec:Macdynam}
 We now have the ingredients to present Borodin-Corwin's \cite{BC14} general construction of dynamics on Gelfand-Tsetlin patterns under which 
 the evolution of the bottom row is Markovian and arises from Pieri's rule via a Doob's transform in a similar way as this was done in the
 Schur case, see \eqref{Doobnorm}. The idea of this construction has been subsequently generalised, see for example \cite{BP16b, MP17}.
 This type of constructions can be traced back to works of Borodin-Ferrari \cite{BF08} and Diaconis-Fill \cite{DF90}.
\vskip 2mm
We start by defining a Markovian
 evolution kernel on a single row coming from the Pieri rule as
 \begin{align}\label{macPk}
 P_k(\mu,\nu) = \frac{1}{H(x_1,...,x_k;\rho)}\frac{P_\nu(x_1,...,x_k)}{P_\mu(x_1,...,x_k)} Q_{\nu/\mu}(\rho),
 \end{align}
 for partitions $\mu=(\mu_1,...,\mu_k) $ and $ \nu=(\nu_1,...,\nu_k)$ such that $\nu\succ \mu$. The parameter $\rho$ will play the role of
 time. We refrain from using the symbol $t$ for time as this is used here 
  for one of the two parameters of Macdonald polynomials.
 This is the analogue of \eqref{Doobnorm} in the Macdonald setting: it is a Doob's transform defining stochastic dynamics out
 of the eigenvalue problem induced by Pieri's rules \eqref{macpieri}, \eqref{skewCauchyM2}.
\vskip 2mm
We also define the so-called {\bf stochastic links}
 \begin{align}\label{stochlink}
 \Lambda^k_{k-1}(\mu,\nu)=\frac{P_\nu(x_1,...,x_{k-1})}{P_{\mu}(x_1,...,x_k)} \, P_{\mu/\nu}(x_k).
 \end{align}
 These will play the role of sampling the $(k-1)$-th row $\nu$ of a Gelfand-Tsetlin pattern knowing its $k$-th row $\mu$.
 The fact that $\Lambda^k_{k-1}$
 is a probability kernel follows immediately from the branching rule for Macdonald polynomials \eqref{Macbranch}.
 We remark that, by definition $P_{\mu/\nu}(x_k)=0$ (and thus also  $\Lambda^k_{k-1}(\mu,\nu)=0$), if $ \mu\nsucc \nu$.
 \vskip 2mm
 The significance of these definitions is that they lead to an intertwining between consecutive rows of Gelfand-Tsetlin patterns, 
 which can then be lifted to an intertwining on the whole pattern.  
\begin{proposition}\label{prop:inter2}
Let $k\geq 1$ and the operators $\Lambda^k_{k-1}$ and $P_k$ as in \eqref{stochlink} and \eqref{macPk}. Then
 \begin{align}\label{MacInter}
 \Delta^k_{k-1} := \Lambda^k_{k-1}P_{k-1} = P_{k} \Lambda^k_{k-1}.
 \end{align}
 \end{proposition}
 
 \begin{proof} Let $\lambda=(\gl_1,...,\gl_k)$ and $\nu=(\nu_1,...,\nu_{k-1})$ be partitions. Then, on the one hand we have
 \begin{align}\label{Minter1}
 \Lambda^k_{k-1} P_{k-1}(\lambda, \nu) 
&= \frac{1}{H(x_1,...,x_{k-1};\rho)}\sum_{\mu}
 \frac{P_\mu(x_1,...,x_{k-1})}{P_{\gl}(x_1,...,x_k)} \, P_{\gl/\mu}(x_k)\,\cdot\,
 \frac{P_\nu(x_1,...,x_{k-1})}{P_\mu(x_1,...,x_{k-1})} Q_{\nu/\mu}(\rho) \, \notag\\
 &=\frac{1}{H(x_1,...,x_{k-1};\rho)} \frac{P_\nu(x_1,...,x_{k-1})}{P_{\gl}(x_1,...,x_k)} \, 
 \sum_{\mu}
  P_{\gl/\mu}(x_k) \,Q_{\nu/\mu}(\rho) ,
 \end{align}
 and on the other that
 \begin{align}\label{Minter2}
 P_k\Lambda^k_{k-1}(\lambda, \nu) 
 &=\frac{1}{H(x_1,...,x_{k};\rho)} \sum_{\mu} 
 \frac{P_\mu(x_1,...,x_k)}{P_\lambda(x_1,...,x_k)} Q_{\mu/\lambda}(\rho) \,\cdot \,
 \frac{P_\nu(x_1,...,x_{k-1})}{P_{\mu}(x_1,...,x_k)} \, P_{\mu/\nu}(x_k) \notag\\
&=\frac{1}{H(x_1,...,x_{k};\rho)}  \frac{P_\nu(x_1,...,x_{k-1})}{P_\lambda(x_1,...,x_k)} 
 \sum_{\mu}  Q_{\mu/\lambda}(\rho) \, P_{\mu/\nu}(x_k) \notag\\
 &= \frac{H(x_k;\rho)}{H(x_1,...,x_{k};\rho)}\frac{P_\nu(x_1,...,x_{k-1})}{P_{\gl}(x_1,...,x_k)} \, 
 \sum_{\mu} P_{\gl/\mu}(x_k) \,Q_{\nu/\mu}(\rho) ,
 \end{align}
 where in the last equality we used the skew Cauchy identity \eqref{skewCauchyM}. We now see that
 \eqref{Minter2} 
 agrees with \eqref{Minter1} upon noticing that $H(x_1,...,x_{k};\rho) =H(x_1,...,x_{k-1};\rho) H(x_k;\rho)$.
 \end{proof}
 \vskip 2mm
Proposition \ref{prop:inter2} allows us now to build dynamics on Gelfand-Tsetlin patterns.
 In particular, define the transition matrix $\bPi_\rho(\sfZ,\tilde\sfZ)$ (we include the dependence on the time parameter $\rho$)
  between Gelfand-Tsetlin patterns $\sfZ$ and $\tilde \sfZ$ as
 \begin{align}\label{macdonaldPi}
 \bPi_\rho(\sfZ,\tilde\sfZ) = P_1(z^1,\tilde z^{\,1}) \,\prod_{k=2}^n \frac{P_k(z^k,\tilde z^{\,k}) \, \Lambda^k_{k-1}(\tilde z^{\,k}, 
 \tilde z^{\,k-1})}{\Delta^k_{k-1}(z^k,\tilde z^{\,k-1})},
 \end{align}
 if $\prod_{k=2}^n \Delta^k_{k-1}(z^k,\tilde z^{\,k-1})>0$ and zero otherwise. 
 Notice also that
 \begin{align}\label{no_interlace}
 \frac{P_k(z^k,\tilde z^{\,k}) \, \Lambda^k_{k-1}(\tilde z^k, \tilde z^{\,k-1}) }{ \Delta^k_{k-1}(z^k,\tilde z^{\,k-1})} = \delta_{\{\tilde z^{\,k}_j=z^k_j+1\}}
  \qquad \text{if} \qquad \tilde z^{\,k-1}_j=z^{\,k}_j+1 \qquad \text{for some $j\leq k$},
\end{align}
where $\delta$ is a delta function taking the value one if $\tilde z^{\,k}_j=z^k_j+1$ and zero otherwise. 
This ensures that the interlacing of the Gelfand-Tsetlin pattern is preserved.
We will now see that $\bPi_\rho$ intertwines with $P_n$ via the kernel
 \begin{align*}
 \bK(\gl,\sfZ)= \prod_{i=1}^n \Lambda^i_{i-1}(z^i , z^{i-1})\,\ind_{z^n=\gl}.
 \end{align*}
 Let us check this in the case $n=2$:
 \begin{align*}
 \bK\bPi_\rho (\gl,\sfZ) &= \sum_{\tilde\sfZ} \bK(\gl,\tilde \sfZ) \bPi_\rho(\tilde\sfZ,\sfZ)
 = \sum_{\tilde z^1} \Lambda^2_1(\gl,\tilde z^1)  
 P_1(\tilde z^1, z^1)  \frac{P_2(\gl, z^2) \, \Lambda^2_{1}( z^2, z^{1})}{\Delta^2_{1}(\gl, z^{1})} \\
&=P_2(\gl, z^2)\,\Lambda^2_{1}( z^2, z^{1}) = P_2\bK\,(\gl, \sfZ),
 \end{align*}
 by using $\sum_{\tilde z^1}\Lambda^2_1(\gl,\tilde z^1)   P_1(z^1,\tilde z^1) =\Delta^2_{1}(\gl, z^{1})$ from \eqref{MacInter},
 thus checking the intertwining property. The general $n$ case follows the same root by summing successively over $\tilde z^{\,n},...,\tilde z^{\,1}$
 and using \eqref{MacInter} to cancel each of the corresponding denominators $\Delta^k_{k-1}(\tilde z^k, z^{k-1})$.
 \vskip 2mm
 {\bf An example: the $q$-Whittaker $2d$ growth model}.
 In order to make the previous general construction more concrete and derive an example, we simplify (by just 
 performing the obvious cancellations)
 the terms in product \eqref{macdonaldPi} and write 
 \begin{align*}
  \frac{P_k(\mu,\nu) \, \Lambda^k_{k-1}(\nu,\gl) }{\Delta^k_{k-1}(\mu,\gl)}
  =\frac{P_{\nu/\gl}(x_k) Q_{\nu/\mu}(\rho)}{\sum_\nu P_{\nu/\gl}(x_k) Q_{\nu/\mu}(\rho) },
 \end{align*}
  and making an informal change of variables $\nu\mapsto \nu/\gl$ in the denominator and then using \eqref{skewCauchyM2} and \eqref{skewQP} we can write this as
 \begin{align*}
 \frac{P_{\nu/\gl}(x_k) \,Q_{\nu/\mu}(\rho)}{\sum_{\nu/\gl} P_{\nu/ \gl }(x_k) \, Q_{\small{(\nu / \gl) \big/ (\mu/\gl)}}(\rho) } 
 &= \frac{P_{\nu/\gl}(x_k) \,Q_{\nu/\mu}(\rho)}{ H(x_k;\rho) \, P_{\mu/\gl}(x_k) }
 =\frac{1}{H(x_k;\rho)} \frac{\psi_{\nu/\gl}}{\psi_{\mu/\gl}}  \, \phi_{\nu/\mu} \,(\rho  x_k)^{|\nu|-|\mu|}.
 \end{align*}
 Then \eqref{macdonaldPi} can be written (in Gelfand-Tsetlin pattern notation) as
 \begin{align}\label{macPisimple}
 \bPi_\rho(\sfZ ,\tilde \sfZ) = \frac{1}{H(x_1;\rho)} 
 \frac{\psi_{\,\tilde z^{\,1}}}{\psi_{\,z^1}} 
\, \, \phi_{\,\tilde z^{\,1} / \, z^1 } \,\,(\rho x_1)^{|\tilde z^{\,1} |- |z^1|} \,\,
 \prod_{k=2}^n \frac{1}{H(x_k;\rho)} \frac{\psi_{\, \tilde z^{\,k}/ \, \tilde z^{\,k-1}}}{\psi_{\,z^k/ \,\tilde z^{\,k-1}}}  \,\,
  \phi_{\,\tilde z^{\,k}/ \,z^k} \,\,(\rho x_k)^{|\tilde z^{\,k}|-|z^k|}.\notag\\
 \end{align}
 which we write in  the shorthand notation
  \begin{align}\label{macPisimple2}
 \bPi_\rho(\sfZ ,\tilde \sfZ) =  \prod_{k=1}^n U^{\,(k)}_{\rho} \big(z^k, \tilde z^{\,k}\,|\, z^{k-1},\tilde z^{\,k-1}),
 \end{align}
 with the obvious notation for $U^{\,(k)}_{\rho} (z^k, \tilde z^{\,k}\,|\, z^{k-1},\tilde z^{\,k-1})$ as derived from \eqref{macPisimple}.
 \vskip 2mm
 When the skew partition $\gl / \mu$ consists of a single box and the parameter $t$ of the Macdonald polynomials is set to
 zero, then another easy computation shows that the expressions for $\psi_{\gl/\mu}$ and $\phi_{\gl / \mu}$ simplify. More precisely, if we first set $t=0$ in \eqref{phipsi}, we obtain that $H(x;\rho)=1 / (\rho x;q)_\infty$ and that
 \begin{align*}
 \phi_{\gl / \mu} &= (q;q)^{-\ell(\gl)}_\infty \, \,\prod_{i=1}^{\ell(\gl)} 
 \frac{ (q^{\gl_i-\mu_i+1};q)_\infty \, (q^{\mu_i-\gl_{i+1}+1};q)_\infty }{ (q^{\mu_i-\mu_{i+1}+1};q)_\infty}, \\
  \psi_{\gl / \mu} &= (q;q)^{-\ell(\mu)}_\infty \,\,  \prod_{i=1}^{\ell(\mu)} 
 \frac{ (q^{\gl_i-\mu_i+1};q)_\infty \, (q^{\mu_i-\gl_{i+1}+1};q)_\infty }{ (q^{\gl_i-\gl_{i+1}+1};q)_\infty} ,
 \end{align*}
 and then if $\gl /\mu$ consists of
  a single box, that is, there exists $j$ such that $\gl_i=\mu_i$ for $i\neq j$ and $\gl_j=\mu_j+1$, 
 then the above may be written as
  \begin{align*}
 \phi_{\gl / \mu} = \frac{1-q^{\,\mu_{j-1}-\mu_{j} }}{1-q} \qquad\text{and}\qquad 
  \psi_{\gl / \mu} = \frac{1-q^{\,\mu_j-\mu_{j+1}+1}}{1-q}.
 \end{align*}
 Moreover, if two partitions $\gl,\nu$ are such that for some $j$ it holds that $\gl_i=\nu_i$ for $i\neq j$
 and $\gl_{j}=\nu_j+1$, then an easy cancellation gives that (still we consider the parameter $t=0$)
 \begin{align*}
 \frac{\psi_{\gl/\mu}}{\psi_{\nu/\mu}} = \frac{(1-q^{\mu_{j-1}-\nu_j})\,(1-q^{\nu_j-\nu_{j+1}+1})}{(1-q^{\nu_j-\mu_j+1})\,(1-q^{\nu_{j-1}-\nu_j})} .
 \end{align*}
 Thus, we obtain that, if in a Gelfand-Tsetlin pattern only the $j$-th
 coordinate $z^k_j$ of $z^k$ jumps by one (forming $\tilde z^{\,k}_j$) and all other entries remain the same,
 then 
  \begin{align}\label{onejump}
U^{\,(k)}_{\rho} (z^k, \tilde z^{\,k}\,|\, z^{k-1},\tilde z^{\,k-1})&=
 \frac{1}{H(x_k;\rho)} \frac{\psi_{\, \tilde z^{\,k}/ \, \tilde z^{\,k-1}}}{\psi_{\,z^k/ \,\tilde z^{\,k-1}}}  \,\,
  \phi_{\,\tilde z^{\,k}/ \,z^k} \,\,(\rho x_k)^{|\tilde z^{\,k}|-|z^k|} \notag\\
 &= \,\frac{\rho x_k}{(\rho x_k;q)_\infty}  \cdot
 \frac{ 1-q^{\,\tilde z^{\,k-1}_{j-1} -z^k_j }}{1-q^{\,z^{\,k}_{j} -\tilde z^{\,k-1}_j+1 }} 
 \,\cdot\frac{1-q^{\,z^{k}_{j} - z^{k}_{j+1}+1}}{1-q}.
 \end{align}

 In order to reduce ourselves to the situation of a ``single jump'' we look at continuous time dynamics, that is $\rho$ is continuous,
 and we compute the infinitesimal generator of $\bPi_\rho$ as
 \begin{align*}
 \mathcal L = \frac{\dd \bPi_\rho}{\dd \rho } \Big|_{\rho=0} =
 \sum_{i\leq n} \prod_{k\neq i} U^{\,(k)}_{\rho} (z^k, \tilde z^{\,k}\,|\, z^{k-1},\tilde z^{\,k-1})\, \,
 \frac{\dd U^{\,(i)}_{\rho} (z^i, \tilde z^{\,i}\,|\, z^{i-1},\tilde z^{\,i-1})}{\dd \rho } \, \Big|_{\rho=0} 
 \end{align*}
 To evaluate the derivative at $\rho=0$, we use \eqref{onejump} as well as the Taylor expansion $1/H(x;\rho)=1-(1-q)^{-1} (x\rho) + o(\rho)$ for 
 $\rho\to0$ and get
 \begin{align}\label{Uderi}
 &\frac{\dd U^{\,(i)}_{\rho} (z^i, \tilde z^{\,i}\,|\, z^{i-1},\tilde z^{\,i-1})}{\dd \rho } \, \Big|_{\rho=0}
 =\,x_i\Big\{ \ind_{|\tilde z^{\,i}|=|z^i|+1} \frac{\psi_{\, \tilde z^{\,i}/ \, \tilde z^{\,i-1}}}{\psi_{\,z^i/ \,\tilde z^{\,i-1}}}  \,\,
  \phi_{\,\tilde z^{\,i}/ \,z^i}   -\frac{\ind_{|\tilde z^{\,i}| =| z^i |} }{1-q}  \Big\}  \notag \\
  &\qquad=\frac{x_i}{1-q}\Big\{  \sum_{j=1}^i  \, \frac{ (1-q^{\,\tilde z^{\,i-1}_{j-1} -z^i_j }) \,(1-q^{\,z^{i}_{j} - z^{i}_{j+1}+1})}{1-q^{\,z^{\,i}_{j} -\tilde z^{\,i-1}_j+1} } 
   \, \ind_{\,\{\text{only $z^i_j$ jumps by one}\}} - \ind_{|\tilde z^{\,i}|=|z^i|} \Big\},
 \end{align}
 Since there is no jump in the $(i-1)$-th row, we have that $\tilde z^{\, i-1}= z^{i-1}$ and thus  \eqref{Uderi} may be written as
 \begin{align*}
 \frac{x_i}{1-q}\Big\{  \sum_{j=1}^i  \, \frac{ (1-q^{\, z^{\,i-1}_{j-1} -z^i_j }) \,(1-q^{\,z^{i}_{j} - z^{i}_{j+1}+1})}{1-q^{\,z^{\,i}_{j} -z^{\,i-1}_j+1} } 
   \, \ind_{\,\{\text{only $z^i_j$ jumps by one}\}} - \ind_{|\tilde z^{\,i}|=|z^i|} \Big\},
 \end{align*}
 Notice also that 
 \begin{align}\label{Uprod}
 \prod_{k\neq i} U^{\,(k)}_{\rho=0} (z^k, \tilde z^{\,k}\,|\, z^{k-1},\tilde z^{\,k-1}) = 0,
 \end{align}  
 unless $|\tilde z^{\, k}| = |z^{k}|$ for all $k\neq i$. Moreover, this  product will also be zero if $\psi_{\tilde z^{\, i+1}/\tilde z^{\,i}}=0$, which happens
 if the jump at row $z^i$ violates the interlacing i.e. leads to a violation of $\tilde z^{\, i+1} \succ \tilde z^{\,i}$.
  This means that if $z^i_j$ (for some $j$ with $j\leq i$) attempts to jump but $z^i_j=z^{i+1}_j$ then this will
 force particle $z^{i+1}_j$ to jump in order that $\psi_{\tilde z^{\, i+1}/\tilde z^{\,i}}\neq 0$ and thus the rate of this jump is also non-zero. 
 This jump will propagate in similar fashion on the whole
 string of particles $z^{i+1}_j, z^{i+2}_j,...$ which are equal to $z^i_j$ before $z^i_j$ jumped. 
 Finally, given also \eqref{no_interlace} which determines the rates of the above trickle down sequence of pushes,
  we see that if product \eqref{Uprod} is non-zero, then it will equal one.
 \vskip 2mm 
 Thus, the above construction defines the particle system (notice that the factor $(1-q)^{-1}$ that appears in \eqref{onejump} and \eqref{Uderi}
 can be ignored as it only corresponds to time change in the dynamics)
 \begin{definition}[$q$-Whittaker $2d$ growth model]
  Let $x_1,...,x_n$ be positive numbers. Each of the particles $z^k_j$ in a Gelfand-Tsetlin pattern $(z^k_j\colon 1\leq j\leq k\leq n)$
   jumps, independently of others, to the right by one step at rate
 \begin{align}\label{q-Whit-rate}
 x_k \, \frac{(1-q^{z^{k-1}_{j-1} - z^{k}_j} ) (1-q^{z^{k}_j  - z^{k}_{j+1}+1   })}{ 1-q^{z^{k}_j  - z^{k-1}_{j}+1   } },
 \end{align}
 and when it jumps it pushes along the string of particles $z^{k+1}_j, z^{k+2}_j,...$ with the property that
 $z^k_j= z^{k+1}_j = z^{k+2}_j = ...$ Notice that if $z^k_j=z^{k-1}_{j-1}$ then the jump of $z^k_j$ is suppressed (the rate in this case is equal to zero), which is consistent with presrving the interlacing property. We implicitly use the convention that terms which contain particles that are not included in the Gelfand-Tsetlin pattern are omitted from expression \eqref{q-Whit-rate}. 
 \end{definition}
 \vskip 2mm
 Let us remark that the $q$-Whittaker dynamics are different than the dynamics induced by $\rsk$. This is because in the latter the independent jumps
 only take place on the diagonal $z^k_1$ with $k=1,...,n$ and the jumps propagate to the rest of the Gelfand-Tsetlin pattern, while
 in the $q$-Whittaker dynamics each particle has its own independent exponential clock that initiates jumps.
 \vskip 2mm
 Notice that the rates of particles $(z^k_k \colon k=1,...,n)$ are just given by $x_k(1-q^{z^{k-1}_{k-1}-z^k_k})$, which means that the evolution
 $(z^k_k \colon k=1,...,n)$  is also Markovian: it is a $q$ deformation of TASEP, called $q$-TASEP. Thus, we see again that particle $z^n_n$
 has a double nature: on the one hand that of the smallest particle in a string of $q$-TASEP and on the other
 that of the smallest particle in a Dyson-like process as this is given by
 \eqref{macPk} for $k=n$. The fact that a Fredholm determinant formula can be derived (via 
 properties of Macdonald polynomials) for certain functionals of the latter ($q$-Laplace transform),
  immediately gives a Fredholm formula for the last particle in $q$-TASEP thus leading to Tracy-Widom asymptotics \cite{BC14}.
\vskip 2mm
\section{ Acknowledgements}\label{ackno} 
A large part of these notes was prepared and exposed during a series of lectures at the National Centre for Theoretical Sciences and Academia Sinica in Taipei.
I would like to warmly thank the institutes for the invitation to present this material and also for the kind hospitality and support through a Visiting Scholar
scheme. I  would also like to thank an anonymous referee, whose very detailed and conscientious comments have improved the presentation of the text as well as
Alexei Borodin, Ofer Busani, Leonid Petrov and Bruce Westbury
 for several useful comments. 
The work was also supported by EPRSC through grant EP/R024456/1.

\end{document}